\documentclass[11pt,twoside]{article}

\usepackage{amssymb,amsfonts,amsthm,color,mathrsfs}
\usepackage{amsmath}
\usepackage{mathtools} % dcases: such that every row is \displaystyle
\usepackage[Symbol]{upgreek}
\usepackage{txfonts}
\usepackage[nottoc,notlot,notlof]{tocbibind}
\usepackage[active]{srcltx}
\usepackage{hyperref}

\usepackage{citeref}  %Set the reference document to display the corresponding page number

\usepackage{cases} %Each case is numbered in the [case] statement.

\usepackage{xcolor} %Ability to add color

\usepackage[numbered,depth=subsubsection]{bookmark} %Bookmark and it's number

\usepackage{graphicx}
\usepackage{epsfig}

\usepackage{soul} % Highlight

\usepackage{txfonts}
\allowdisplaybreaks
\def\L{{\mathcal L}}

\def\cn{{\mathbb N}}

\def\cs{{\mathcal S}}

\def\cg{{\mathcal G}}

\def\ca{{\mathcal A}}

\def\az{\alpha}
\def\supp{{\mathop\mathrm{\,supp\,}}}
\def\dist{{\mathop\mathrm {\,dist\,}}}
\def\diam{{\mathop\mathrm {\,diam\,}}}

\def\kz{\kappa}

\def\r{\right}
\def\lf{\left}

\newtheorem{thm}{Theorem}[section]
\newtheorem{lem}[thm]{Lemma}
\newtheorem{prop}[thm]{Proposition}
\newtheorem{cor}[thm]{Corollary}

\newtheorem{rem}[thm]{Remark}

\numberwithin{equation}{section}

\begin{document}
\arraycolsep=1pt
\author{Bo Li, Tianjun Shen, Zihan Wang}
\title{{\bf Riesz transform on manifolds with mixed ends for $1<q<2$}
 \footnotetext{\hspace{-0.35cm} 2010 {\it Mathematics
Subject Classification}. Primary  42B20; Secondary  58J35, 43A85.
\endgraf{
{\it Key words and phrases: Riesz transform, heat kernel, bounded geometry, non-doubling measure}
\endgraf}}
\date{}}
\maketitle

\begin{center}
\begin{minipage}{11.5cm}\small
{\noindent{\bf Abstract}.
Let $M_1$, $\cdots$, $M_\ell$ be complete manifolds of the same dimension, where $2\le \ell\in\mathbb{N}$. Suppose that each $M_i$ satisfies two side Gaussian bounds. If some of these manifolds are parabolic
and there exists a constant
$1\le n_i\le 2$ such that for some  $x_i\in M_i$ with $1 \le r \le R < \infty$,
$$c_i\left(\frac{R}{r}\right)^{n_i}\le \frac{Vol_{M_i}(B(x_i,R))}{Vol_{M_i}(B(x_i,r))}\le C_i\left(\frac{R}{r}\right)^{n_i},$$
by assuming  that all the manifolds satisfy the relative connectedness of the annuli ($RCA$) condition introduced by Grigor'yan and Saloff-Coste, we show that the Riesz transform $\nabla \L^{-1/2}$
is bounded on $L^q(M)$ for each $1<q<2$ on the gluing manifold $M=M_1 \# M_2 \# \cdots \# M_\ell$.
}\end{minipage}
\end{center}

\tableofcontents

\section{Introduction} \hskip\parindent
The Riesz transform has played a central role in classical harmonic analysis.
Riesz transform bounds allow comparing the two corresponding first order homogeneous Sobolev spaces; see \cite{cd03}.
The aim of this paper is to give the $L^q$-boundedness of the Riesz transform for $1 < q < 2$, on a gluing manifold.
 Before describing our results in detail, let us introduce the setting.

\subsection{Setting} \hskip\parindent
In this paper, we consider a complete, connected and  non-compact Riemannian manifold
$M$, which is obtained by gluing  together several complete manifolds of the same dimension.

Let us recall the notion of the gluing manifold as in \cite{gri-sal09, gis18}.
Let $\left\{M_i\right\}_{i=1}^\ell$ be a finite family of non-compact Riemannian manifolds
and all $M_i$ are of the same dimension.
We say that a Riemannian manifold $M$ is a connected sum of the manifolds $M_i$, denoted by %write
$$
M=M_1 \# M_2 \# \cdots \# M_\ell,
$$
if, for some non-empty compact set $E_0 \subset M$ (called a central part of $M$),
the exterior $M \setminus E_0$ is a disjoint union of open sets $E_1, E_2, \ldots, E_\ell$,
such that each $E_i$ is isometric to $M_i \setminus K_i$, for some compact $K_i \subset M_i$.
In fact, we will always identify $E_i$ and $M_i \setminus K_i$. See  \cite[Figure 5]{gri-sal09}  and \cite[Figure 4]{gis18} for examples of the gluing manifold. %We say that the $E_i$'s are the ends of $M$ with respect to $E_0$.
We call the subsets $E_i$ the ends of $M$ with respect to $E_0$.
%Sometimes it is
%convenient to say that $M_i$'s are the ends of $M$.
%GS1999 said that "Sometimes it is
%convenient to say that M i ��s are the ends of M."

On the manifold $M$, we denote by $d$ the geodesic distance, by $\mu$ the Riemannian measure. We denote by $\L$ the  Laplace-Beltrami operator on $M$ which is non-negative and self-adjoint, and let $\{e^{-t \L}\}_{t>0}$ be the heat semigroup. The corresponding Riesz transform is then given by
$$
\nabla \L^{-1 / 2}=\frac{1}{\sqrt{\pi}} \int_{0}^{\infty} \nabla e^{-s \L} \frac{\,ds}{\sqrt{s}},
$$
where $\nabla$ denotes the Riemannian gradient.
Notice that as a consequence of integration by parts, the Riesz transform $\nabla \L^{-1 / 2}$ is always bounded on $L^2(M)$.

%Let us consider a gluing manifold $M$, given as
%$$
%    M=M_1 \# M_2 \# \cdots \# M_{\ell}=E_0 \cup \cup_{i=1}^{\ell} E_i,
%$$
%where each $M_i$ is a complete manifold and all $M_i$ are of the same dimension, $E_i=M_i \setminus K_i, K_i$ and $E_0$ are compact sub-manifolds. We say that the $E_i$'s are the ends of $M$ with respect to $E_0$. We refer the reader to \cite[\S 2.2]{gri-sal09}  and \cite[\S 2.2]{gis18} for figures and further details.
%\\--------------------------------\\

To move further, let us recall some basic notation. We denote by $B(x,r)$, $B_i(x_i,r)$ the open ball with centre $x\in M$, $x_i\in M_i$ and radius $r>0$ in $M$, $M_i$, and by $V(x,r)$, $V_i(x_i,r)$ their volume $\mu(B(x,r))$, $\mu_i(B_i(x_i,r))$, respectively.
We say that $M_i$ satisfies the volume  doubling property (in short is doubling)
if there exists  a constant  $C_{D}>1$ such that
$$ V_i(x,2r)\le C_{D}V_i(x,r),
\quad  \forall x \in M_i, \, r>0.
\eqno(D)$$
%for all  $x_i\in M_i$ and $r >0$.
 Let $\mathcal{L}_i$ denotes the Laplace-Beltrami operator on $M_i$.
The heat semigroup on $M_i$ has a smooth positive and symmetric kernel $h_{i,t}(x,y)$,  meaning that
$$e^{-t\mathcal{L}_i}f(x)=\int_{M_i} h_{i,t}(x,y)f(y)\,d\mu_i(y)$$
for each suitable  function $f$.
%One  says that the heat kernel satisfies the Li-Yau estimates  if
%there exist $C,c>0$ such that for all $t>0$ and $x,y\in M_i$,
%$$\frac{C^{-1}}{V_i(x,{\sqrt t})}\exp\lf\{-\frac{d^2(x,y)}{ct}\r\}\le h_{i,t}(x,y)\le
%  \frac{C}{V_i(x,{\sqrt t})}\exp\lf\{-\frac{d^2(x,y)}{ct}\r\}.\leqno(LY)
%$$

%We next deal with the case that some of these ends are parabolic.
We assume that
% each $M_i$
%is non-collapsed,
%satisfies
%the doubling condition and
the heat kernel $h_{i,t}(x,y)$ on $M_i$ satisfies both the upper and lower Gaussian bounds %. %That is,
%we assume that
%the heat  kernel $h_{i,t}(x,y)$ on $M_i$ satisfies  the Li-Yau estimates,
(also known as the Li-Yau estimates),
i.e.,
there exist constants $C,c>0$ such that %for all $x,y\in M_i$ and $t>0$,
%$$\frac{C^{-1}}{V_i(x,{\sqrt t})}\exp\lf\{-\frac{d(x,y)^2}{ct}\r\}
%%e^{-c\frac{d^2(x,y)}{t}}
%\le
%h_{i,t}(x,y)
%\le
%  \frac{C}{V_i(x,{\sqrt t})}\exp\lf\{-c\frac{d(x,y)^2}{t}\r\}.\eqno(LY)
%$$
$$\frac{C^{-1}}{V_i(x,{\sqrt t})} e^{-\frac{d(x,y)^2}{ct}}
%e^{-c\frac{d^2(x,y)}{t}}
\le
h_{i,t}(x,y)
\le
  \frac{C}{V_i(x,{\sqrt t})}e^{-c\frac{d(x,y)^2}{t}},
  \quad \forall x, y \in M_i, \, t>0.
  \eqno(LY)
$$
Notice that a two-sided Gaussian bound $(LY)$ for the heat kernel  is equivalent to the doubling property $(D)$ together with the $L^2$-Poincar\'{e} inequality $(P_2)$; see \cite{gri92,sal2} for more details.
Here the $L^2$-Poincar\'{e} inequality $(P_2)$ means that
there exists a constant $C_p >0$ such that for all balls $B=B(x,r)$ and $W^{1,2}(B)$ functions $f$, it holds
$$
\fint_{B}|f-f_B|^2\,d\mu\le
C_p r^2\fint_{B}|\nabla f|^2\,d\mu,\eqno(P_{2})$$
where $f_B$ denotes the mean (or average) of $f$ over $B$. Moreover,
it is well known that $(LY)$ is also equivalent to the parabolic Harnack inequality;
see \cite{sal92-2} for more details.

%\begin{defn}
We say that a manifold $M$ is
 non-collapsed, if the volume of each ball with radius one in $M$ has a positive bottom, i.e.,
$$\inf_{x \in M} V(x, 1) > 0.$$
%\end{defn}

Let us recall the relative connectedness of the annuli ($RCA$) condition as in \cite{gri-sal09}.
%\begin{defn}
Given a fixed reference point $o\in M$,
we say that a manifold $M$ satisfies the ($RCA$) condition, if there exists $A>1$ such that, for all $R>0$ large enough and for
any two points $x,y\in M$ both at distance $R$ from $o$, there is a continuous path $\gamma$ connecting $x$ to $y$ and staying in the annulus $B(o,AR)\setminus B(o,R/A)$.
%See \cite[Figures 2 and 3]{gis21} for typical positive and negative examples.

Here are some concrete examples. The Euclidean space $\mathbb{R}^n$ satisfies $(RCA)$ if and only if $n \ge 2$.
The manifolds $\mathcal{R}^n = \mathbb{R}^n \times \mathbb{S}^{Q-n}$ satisfy $(RCA)$ for all $1 \leq n \leq Q$, where $Q$ is a large integer.
By \cite[Theorem 2.5]{ca16}, $(RCA)$ can be deduced by an anchored Poincar\'{e} inequality and a reverse doubling condition.
%We refer the reader to \cite[Figures 2 and 3]{gis21} for typical positive and negative examples, and \cite[Examples 2.1 and 2.2]{gri-sal09} for more examples.
For typical positive and negative examples, see \cite[Figures 2 and 3]{gis21}. For more examples, we refer the reader to \cite[Examples 2.1 and 2.2]{gri-sal09}.

%We refer the reader to \cite{gri-sal09,gis21,ca16}.
 %Carron \cite[Theorem 2.5]{ca16} tells us that $(RCA)$ condition is ensured by an anchored Poincar\'{e} inequality and a reverse doubling condition.
  %  Finally, we would like to mention that $(RCA)$ can deduce $(LY)$ when the manifold has non-negative curvature; see \cite[Example 2.2]{gri-sal09}.
%�������⣬ ����֪������Щ�����£� ����������(RCA)�ġ�
%\end{defn}

\subsection{Main result and state of the art} \hskip\parindent

In the setting of the Euclidean space $\mathbb{R}^n$,
it is well known that the celebrated work of Riesz \cite{Ri28} (in one dimension) and
Calder\'{o}n and Zygmund \cite{CZ52} (in higher dimensions) showed that the Riesz transform is bounded on $L^p(\mathbb{R}^n)$
for $1 < p < \infty$, as well as of weak type $(1,1)$. %Their results implies that
%for all $p \in (1,+\infty)$,
%\begin{equation} \label{eq:Ep} \tag{$E_p$}
%C_p^{-1} \|\Delta^{1 / 2} f \|_p
%\leq \| |\nabla f| \|_p
%\leq C_p \|\Delta^{1 / 2} f\|_p,
%\quad \forall f \in \mathcal{C}_0^{\infty}(\mathbb{R}^n).
%\end{equation}
%This can be used to show the equivalence of different defined Sobolev norms
%$\|\Delta^{1 / 2} f \|_p$ and $\||\nabla f| \|_p$; see \cite{cd03}.
%%�ұߵ���Riesz�� ���ߵ��Ƿ����� Riesz�����Ƴ������� ��֮����Ȼ�� ��һ���������Է������ƣ� ��Auscher-Coulhon-2005-Lemma 0.1.
%%
%Note that the right-hand inequality of \eqref{eq:Ep} may be reformulated by saying that the Riesz transform $\nabla \Delta^{-1 / 2}$ is $L^p$ bounded.

It was asked by Strichartz \cite{str83} in 1983 on which non-compact Riemannian manifolds, and for which $p$, $1<p<\infty$,
the Riesz transform is still bounded. There are some typical results.
If the complete Riemannian manifolds has non-negative Ricci curvature, Bakry \cite{bak2} proves that the Riesz transform is $L^p$-bounded, $1<p<\infty$. Then, \cite{acdh} and \cite{CS10} improve this result to a complete non-compact Riemannian manifold satisfying volume doubling property $(D)$ and the gradient upper estimate of the heat kernel
$$
|\nabla_x p_t(x,y)| \leq \frac{C}{\sqrt t V(y, \sqrt t)},
\quad \forall x, y \in M, \, t>0.
\eqno(G)
$$

Another seminal result is due to Coulhon and Duong \cite{cd99} where they consider a complete
Riemannian manifold satisfying the volume doubling property and the Gaussian upper bound of the heat kernel
 \begin{equation} \tag{$UE$}
h_t(x, y) \lesssim \frac{1}{V(x, \sqrt{t})} e^{-c\frac{d(x, y)^2}{t}}, \quad \forall x, y \in M, \, t>0.
\end{equation}
 They prove that the Riesz transform is weak type (1,1) and $L^p$-bounded for all $1 < p \leq 2$.  It is therefore very natural to ask whether the above assumptions are necessary or not.
Chen et al. \cite{CCFR} showed further that the Gaussian upper bound can be relaxed; see also Li and Zhu \cite{lz17}.
For more research and development on the Riesz transform,
we
refer the readers to \cite{Al92,ac05,bf15,bk04,cch06,ch92,CMO15,cdl03,cl04,ji21,jl20,lh99,lx10,shz05,var88}
and references therein.

We are interested in the gluing manifold where the doubling condition fails
and focus on $L^q$ boundedness of $\nabla \L^{-1/2}$ for $1 < q <2$.
%Let us review some closely related results.
%Connected sums were first analysed in the context of heat kernel lower bounds by Benjamini, Chavel and Feldman [xxx];
%see previous results on connected sums in [xxx].
The study of the Riesz transform of connected sums of manifolds originates in Coulhon and Duong \cite[Section 5]{cd99}.
%They showed that the Riesz transform on the connected sum of two copies of $R^n$ , $n > 2$, is unbounded on $L^p$ for $p > n$.
For more related works on this topic, see \cite{ca07,cjks16,de15,hns19,hs19,Ni19}.

The following theorem is the main result of this paper.

\begin{thm}\label{main-result-parabolic}
  Let  $M=M_1\#M_2\#\cdots \# M_\ell$ ($2\leq \ell \in \mathbb{N}$) be a connected sum of manifolds, where
  $\{M_i\}_{1\le i\le\ell}$ is a sequence of complete, non-compact, connected and non-collapsed manifolds of the same dimension,
  and each $M_i$ satisfies $(LY)$ and $(RCA)$.
  % 1. Ӧ����ֻ��Ҫ���������㣨RCA��
  % 2. ��RCA�� ���Ƴ� ��LY�� ��Gri-Sal-09: Example 2.2 (ǰ��������M���зǸ�����)
  %3. (RCA) �����������㣨RCA�� �� Gri-Sal-09 p1970�� �ɴ˿�֪�� �ڱ���n_i>2�������µĶ˵�����Ȼ����(RCA)��
  %  Assume that each $M_i$ satisfies one of the following conditions:
  Assume that for each $M_i$, one of the following two conditions holds:
\begin{enumerate}
   \item[\rm{(i)}] There exist constants $n_i>2$ and $c_i>0$ such that, for some $x_i\in M_i$ and any $1 \le r \le R < \infty$,
$$c_i\left(\frac{R}{r}\right)^{n_i}\le \frac{V_i(x_i,R)}{V_{i}(x_i,r)}.$$
    %In this case, we say that $M_i$ is non-parabolic.
    \item[\rm{(ii)}] There exist constants $0\le n_i\le 2$, $c_i>0$ and $C_i>0$ such that, for some $x_i\in M_i$ and any $1 \le r \le R < \infty$,
$$c_i\left(\frac{R}{r}\right)^{n_i}\le \frac{V_i(x_i,R)}{V_i(x_i,r)}\le C_i\left(\frac{R}{r}\right)^{n_i}.$$
    %In this case, we say that $M_i$ is parabolic.
\end{enumerate}
%We assume that
%$$\max\{n_i:\,1\le i\le \ell\}>2,
%\quad
%1\le \min\{n_i:\,1\le i\le \ell\}<2,$$
%and there is some $i$ such that $n_i=2$.
Besides, suppose that there is some $1 \leq i \leq \ell$ such that $n_i =2$, and it holds
%$$\max\{n_i:\,1\le i\le \ell\}>2,
%\quad
%1\le \min\{n_i:\,1\le i\le \ell\}<2,
%\quad
%\{i:\, n_i = 2\} \neq \emptyset.
%$$
$$
1\le \min\{n_i:\,1\le i\le \ell\}<2
< \max\{n_i:\,1\le i\le \ell\}.
$$
%and there is some $i$ such that $n_i=2$.
Then the Riesz transform $\nabla \L^{-1/2}$
is bounded on $L^q(M)$ for each $1<q<2$.
\end{thm}

\begin{rem} \rm

  \begin{enumerate}
    \item[\rm{(i)}]
    Note that, for each $1 \leq i \leq \ell$,
    the assumption (i) in Theorem \ref{main-result-parabolic} implies
    the manifold $M_i$ is non-parabolic,
    while the assumption (ii) implies the manifold $M_i$ is parabolic.
    Here we say that a manifold is parabolic if any positive superharmonic function on the manifold is constant, and non-parabolic otherwise.
    %If $n_i >2$, $M_i$ is non-parabolic and if $n_i \le 2$, $M_i$ is parabolic.
    We refer the reader to \cite[Proposition 4.3]{gri-sal09} and \cite[p. 159]{gis18} for more about parabolicity.
    %See \cite{gri99, gri-sal09} for more details.
    %We refer the reader to \cite{gri99, gri-sal09} for detailed account of parabolicity. % The word in gs1999: Heat kernel on connected sums of Riemannian manifolds.
    \item[\rm{(ii)}]
    It is worth to note that $(RCA)$ is a very natural assumption. It implies that $M$ has exactly $\ell$ ``true ends''.
    Namely, for each manifold $M_i$ and any compact set $K \subset M_i$, $M_i \setminus K$ has only one unbounded connected component; see
    \cite[p. 1971]{gri-sal09} and \cite{ca16}.
    Since one of our key tools is the heat kernel estimates established in \cite[Theroem 6.6]{gri-sal09},
    each parabolic manifold $M_i$ satisfying $(RCA)$ is enough.  %for Theorem \ref{main-result-parabolic}.
    Nevertheless, it was shown that for any complete manifold satisfying $(LY)$ and the assumption (i) in Theorem \ref{main-result-parabolic},  satisfies $(RCA)$; see for instance \cite[p. 1970]{gri-sal09}.

    \item[\rm{(iii)}]
 Let us compare Theorem \ref{main-result-parabolic} with \cite[Therorem 1.3]{jiang-li-lin-2022}.
%    By assuming (stronger) two-side Gaussian bounds $(LY)$ of the heat kernel, and an additional property $(RCA)$ on each $M_i$,
    %By replacing the Gaussian upper bound $(UE)$ with the two-side Gaussian bounds $(LY)$, and assuming an additional property $(RCA)$,
 %   Theorem \ref{main-result-parabolic} extends the result of \cite{jiang-li-lin-2022} from $n_i \ge 2$ to $n_i \ge 1$.
%
%
Assuming each $M_i$ satisfies $(D)$ and $(UE)$, using the heat kernel estimates in \cite[Section 4]{gri-sal09}, Jiang, Li and Lin \cite{jiang-li-lin-2022} consider the case $n_i \ge 2$.
By assuming stronger conditions, in other words, an additional property $(RCA)$, and the two-sided Gaussian bounds $(LY)$ instead of the Gaussian upper bound $(UE)$ for the heat kernel, on each $M_i$,
  %the optimal heat kernel estimates in \cite[Section 6]{gri-sal09}  including the cases that some $n_i<2$ are available,
  %the optimal heat kernel estimates for the cases that some $n i < 2$ are  obtained in [31, Section 6] and available in our setting, and we will consider the case $n i �� 1$.
  the optimal heat kernel estimates for the cases that some $n_i < 2$ obtained in \cite[Section 6]{gri-sal09} are available,
%  and we shall consider the case $n_i \ge 1$.
  and we improve the range of index from $n_i \ge 2$ to $n_i \ge 1$.
  \end{enumerate}

\end{rem}

Throughout the paper, unless otherwise specified,
%we suppose
%the gluing manifold
%$M$ is assumed to satisfy the conditions in  Theorem \ref{main-result-parabolic}.
the gluing manifold
$M=M_1 \# M_2 \# \cdots \# M_\ell$ always satisfies the assumptions in  Theorem \ref{main-result-parabolic}.
%For each $1 \leq i\leq \ell$,
%Let $1 \leq i\leq \ell$.
%%Note that
%
%It can be verified that $M_i$ is non-parabolic if $n_i >2$, and parabolic if $1 \leq n_i \le 2$.
%
%
%Note that the manifold $M_i$ can be verified to be non-parabolic
%if it satisfies the assumption (i) in Theorem \ref{main-result-parabolic},
%and parabolic if it satisfies the assumption (ii) in Theorem \ref{main-result-parabolic}.
%
%
%Note that, for each $1 \leq i \leq \ell$, the manifold $M_i$ can be verified to be non-parabolic
%if it satisfies the assumption (i) in Theorem \ref{main-result-parabolic},
%and parabolic if it satisfies the assumption (ii) in Theorem \ref{main-result-parabolic}.
%
%

%

%Throughout the paper, unless otherwise specified, the manifold $M$ satisfies the assumption in  Theorem \ref{asp-1}.

\subsection{Plan of the paper} \hskip\parindent
%{\color{blue}The paper is organized as follows.
%%
%%{\color{blue}In Section 2, we study some basic estimates for the heat semigroup and the volume growth.
%% %and small time part of the Riesz transform.
%%%the heat kernel, and the local Riesz transform.
%%In Section 3, we provide some heat kernel estimates for later use,
%%and some results related to the heat kernel upper bounds.
%%}
%%
%In Sections 2 and 3, we provide some basic estimates for the heat kernel, %/semigroup,
%the volume growth %, the heat kernel
%and small time part of the Riesz transform.
%%the heat kernel, and the local Riesz transform.
%%In Section 3, we provide some heat kernel estimates for later use,
%%and some results related to the heat kernel upper bounds.
%%
%%
%In Section 4, we discuss the mapping properties for the heat semigroup and its time derivative.
%In Section 5, we give the proof of the main result.
%}
%
%
The paper is organized as follows.
In Section 2, we recall some basic notions and a priori heat kernel estimates.
In Section 3, we  study some basic properties for the the volume growth and give some estimates on the auxiliary function occurring in the priori heat kernel estimates.
In Section 4, we  give some further estimates on the heat kernel.
In Section 5, we collect some basic estimates for the heat semigroup, and in Section 6, we discuss the mapping properties for the heat semigroup and its time derivative.
In Section 7, we give the proof of the main result.

Throughout the paper,
%we denote by the letter $C$ (or $c$) a positive constant
the letters $C$, $c$, $c'$, $c''$  denote positive constants
which are independent of the
main parameters, but may vary from line to line.
When the value of a constant is significant, it will be explicitly stated.
%
%The symbol $A \lesssim B$ means that $A \leq C B$, $A \lesssim_{\alpha, \beta} B$ means that the implicit constant depends on $\alpha, \beta$,
%and $A \sim B$ means $c A \leq B \leq C A$, for some harmless constants $c, C>0$.
%
The symbol $A \lesssim B$ means that $A \leq C B$, and $A \sim B$ means $c A \leq B \leq C A$, for some harmless constants $c, C>0$,
Besides, the symbols $A \lesssim_{\alpha, \beta} B$ and $A \sim_{\alpha, \beta} B$ mean that the implicit constants depend on $\alpha, \beta$.
We use $\|\cdot\|_{p}$ to denote the $L^{p}(M)$ norm, and $\|\cdot\|_{p \rightarrow q}$ to denote the operator norm $\|\cdot\|_{L^{p}(M) \rightarrow L^{q}(M)}$.
Without mentioning it, we will repeatedly apply the inequality
%$r^{\eta}e^{-r} \lesssim_\eta 1$, $\eta \ge 0$, $r > 0$.
$r^{\alpha}e^{-r} \lesssim_\az 1$ for any $\alpha >0$ and all $r>0$.

\section{Preliminaries}\hskip\parindent
%In this section, we first provide the $L^p$-Davies-Gaffney estimates for the operators
%$e^{-t \L}$, $\nabla e^{-t \L}$ and $\L e^{-t \L}$.
%Then, some basic notation are introduced and some properties for the volume growth are established.
%
%
In this section, we first give some basic notation. Then,
%we shall introduce the weighted manifold and some heat kernel estimates established in \cite{gri-sal09}.
 following
\cite[Section 6]{gri-sal09}, we first construct a weighted manifold $\widetilde{M}$,
 and then state the optimal estimates for the heat kernel, %upper bound, % which is obtained with the help of $\widetilde{M}$.
 which is obtained
 %with the help of
 %through
 by this weighted manifold
 $\widetilde{M}$ and Doob's transform.
 %This bound is sharp up to the values of the constants.
Besides, some other heat kernel estimates in \cite{gri-sal09} are introduced.

\subsection{Some basic notation} \label{sec:not-basic}\hskip\parindent
 To begin with, we give some notation as in \cite[Subsections 4.3 and 4.4]{gri-sal09}.
 Let $B(x, r)$ (resp. $B_i(x, r)$ with $1 \leq i \leq \ell$) denote the geodesic ball in $M$ (resp. $M_i$).
 For each index $1 \leq  i \leq \ell$, fix a reference point $o_i \in \partial E_i$, and set
 \[V(x, r):=V(B(x, r)), \quad
 V_i(x, r):=V_i(B_i(x, r)), \quad
 V_i(r):=V_i(o_i, r), \quad
 V_0(r):=\min_{1\le i\le \ell} V_i(r),
 \]
% $V(x, r):=V(B(x, r))$,
% $V_i(x, r):=V_i(B_i(x, r))$,
% $V_i(r):=V_i(o_i, r)$,
% \[V_0(r):=\min_{1\le i\le \ell} V_i(r),
% \]
 %It will also be useful to set
% $
%    V_0(x,r) = V_0(r),
% $
% for all $x \in M$.
and
$$ F_i^{(r)} : = \{ x\in E_i: \dist(x, E_0) \leq 2r\}, \quad \forall  r\geq 1. $$%\, 1\leq i \leq \ell.$$

We also set
\[
 |x|:=\sup_{y \in E_0}\{d(x, y)\}, \quad \forall x \in M.
\]
%We can easily obtain
Observe that  $|x| \ge 1$. %; see \cite[p. 1922]{gri-sal09} for more details.
In fact, since $E_0$ is compact, there exist $x_0, y_0 \in E_0$ such that $d(x_0, y_0)=\diam(E_0)= 2$. Then, it holds by the triangle inequality that $d(x, x_0) \ge 1$ or $d(x, y_0) \ge 1$, which implies that $|x| \ge 1$.
%\href{https://math.stackexchange.com/questions/208943/the-diameter-of-a-compact-set}{stackexchange}.

%, i.e., $i_x = i$, if $x \in E_i$, $0 \leq i \leq \ell$.
%$$i_x = i, \quad \mbox{if} \, x \in E_i, \, 0 \leq i \leq \ell.$$
%\begin{equation*}
%    i_x
%    =
%    \begin{cases}
%        i, & \mbox{if} \, x \in E_i, i \geq 1, \\
%        0, & \mbox{if} \, x \in E_0.
%    \end{cases}
%\end{equation*}

In what follows, for simplicity of notions, we shall assume that
$\mbox{diam}(E_0)=2$, % ��ʦ�� ��������ֱ��������2�� ��������ʹ����Remark 2.5�������Ľ���$|x| \ge 1$, ����������$x$. ����������������������
$\mu(E_0)=1$ and
\begin{equation*}%\label{basic-assumption}
  1\le V_i(1)\le 10, \quad \forall\, 1\le i\le\ell,
\end{equation*}
and
\begin{equation*}
    1 \leq \mu(F_*) \leq 40 \ell,
\end{equation*}
where   $$F_*:= \{x \in M:  \dist(x ,E_0) \leq 100\}.$$ % ��ʦ�� ��������$\dist(x ,E_0) \leq 2$������$\dist(x ,E_0) \leq 2$. ��������ʹ�õ��Ľڿ�ͷ��$\mu(F^(R_i)_i)$��$R_i \ge 50$. ����$R_i \ge 50$ �Ļ��� ��7�ڣ���Ҫ������֤���� �������������ȽϷ��㡣 ����������������������

\subsection{Weighted manifold} \hskip\parindent
%We introduce the following harmonic function in \cite{gri-sal09}, which is used to constructing the weighted manifold.
%{\color{blue}We introduce the following harmonic function to construct the weighted manifold;
%see \cite{gri-sal09}.
%
%
%
%\begin{prop}%(\cite[Proposition 6.3]{gri-sal09})
% \label{prop:harmonic-function}
%  Let $M=M_1 \# \ldots \# M_\ell$ be a connected sum of complete manifolds $(M_i, \mu_i)$.
%  Assume that $M$ is non-parabolic and $M_i$ satisfies $(LY)$ for each $i=1, \ldots, \ell$.
%  Assume further that $M_i$ satisfies $(RCA)$ for each $i \in I$.
%  Then there exists a positive harmonic function $h$ on $M$ such that, for all $x \in M$,
%    $$
%        h(x) \sim 1+\left(\int_1^{|x|^2} \frac{ds}{V_{i_x}(\sqrt{s})}\right)_{+} .
%    $$
%\end{prop}
%%\begin{proof}
%%    For the proof, we refer the reader to \cite[Proposition 6.3]{gri-sal09}.
%%\end{proof}
%}
From \cite[Proposition 6.3]{gri-sal09}, we know that there exists a positive harmonic function $h$ on $M$ such that, for all $x \in M$,
    $$
        h(x) \sim 1+\left(\int_1^{|x|^2} \frac{ds}{V_{i_x}(\sqrt{s})}\right)_{+} .
    $$
%Let $(M=M_1 \# \ldots \# M_\ell, \mu)$ be a connected sum of complete non-compact manifolds $(M_i, \mu_i)$.
%Let $h$ be the above harmonic function.
%from Proposition \ref{prop:harmonic-function}.
We can then consider the weighted manifold $\widetilde{M}=(M, \tilde{\mu})$ where
$$
d \tilde{\mu}= h^2 d \mu .
$$
Moreover, by restricting $h$ to each $E_i=M_i \setminus K_i$ and extending the resulting function smoothly to a function $h_i$ defined on $M_i$, $1 \leq i \leq \ell$, %$i \in\{1, \ldots, \ell\}$,
we obtain the weighted manifold $\widetilde{M}_i=(M_i, \tilde{\mu}_i)$
such that
$$
\widetilde{M}=\widetilde{M}_1 \# \cdots \# \widetilde{M}_\ell,
$$
where $d \tilde{\mu}_i=h_i^2 d\mu_i$.
In the sequel, we will use a tilde $\enspace \widetilde{} \enspace $ to denote objects relative to the manifold $\widetilde{M}$.

\subsection{Some a priori heat kernel estimates} \hskip\parindent \label{sec:pri-es}
% ��ʦ�� ����������ԭ���� ��������ȷ��һ��a priori�Ƿ���ȷ. �ҿ���15��J. Math. Pures Appl.��ƪ�������ù�``Some a priori heat kernel estimates''��д��.
% Ȼ���Ҳ����£� ``a priori''������������, a��prioriһ�𹹳�һ�����ݴ�, ��ʾ�������ġ��� �����Խ����ġ�. ��Ϊ��������some�� ����a������Щ���֡� ���������������������ǶԵİɡ�
Define the function $\widetilde{H}(x,t)$ by
\begin{equation*} %\label{eq:def-H}
\widetilde{H}(x,t):=\min\left\{1, \, \frac{|x|^2}{\widetilde{V}_{i_x}(|x|)}+\left(\int_{|x|^2}^t\frac{\,d s}{\widetilde{V}_{i_x}(\sqrt s)}\right)_+\right\},
\end{equation*}
where $(\cdot)_+$ denotes the non-negative part and $i_x$ denotes the index of $E_i$ that $x$ belongs to.

For each $x, y \in M$, let us set
%$$d_+(x,y):= \inf\left\{\mbox{length}(\gamma_{xy}):\, \gamma_{xy}\cap E_0\neq \emptyset\right\},$$
$$d_+(x,y):= \inf\left\{\mbox{length}(\gamma):\, \gamma(0)=x, \, \gamma(1)=y, \, \gamma \cap E_0\neq \emptyset\right\},$$
and
$$d_\emptyset(x,y):=\begin{cases}
\inf\left\{\mbox{length}(\gamma):\, \gamma(0)=x, \, \gamma(1)=y, \, \gamma \cap E_0= \emptyset\right\},& \, x,y \,\mbox{belongs to the same end}, \\
\infty & \, \mbox{otherwise},
\end{cases}$$
where the infimum is taken over all curves $\gamma: \, [0,1] \rightarrow M$ connecting $x$ to $y$.

\begin{rem}\label{rem:new-dist} \rm
    The following statements hold (see \cite[p. 1948]{gri-sal09}).
    \begin{enumerate}
      \item [\rm{(i)}] For any $x, y \in M$, it holds  $d_{+}(x,y) \geq d(x,y)$, $d_{+}(x,y) \ge |x|+|y|-2 \diam(E_0)$ and $d_{\emptyset}(x,y) \geq d(x,y)$. %and moreover one of these inequalities must in fact be and equality.
            % $d_{+}(x,y) \ge |x|+|y|-2 \diam(E_0)$ should hold in any case. See \cite[p. 1952]{gri-sal09}
            % \cite[(2.3)]{jiang-li-lin-2022}, for $1\leq i \neq j \leq \ell$
            %See \cite[(2.14)]{jiang-li-lin-2022} and \cite[p. 1953]{gri-sal09}, for $1 \leq i \leq \ell$, $j=0$, also the following proof.
      \item [\rm{(ii)}]
      For any $1 \leq i, j\leq \ell$,  $i \neq j$, if $x \in E_i \cup E_0$, $y \in E_j \cup E_0$, then $d_+(x,y) = d(x,y)$ and $d_{\emptyset}(x,y) = \infty$.
    \end{enumerate}
\end{rem}
%\begin{proof}
%Pick a curve $\gamma_0$  connecting $x$ to $y$ such that $\mbox{length}(\gamma_0) = d_+(x,y)$.
%    Pick $x', y' \in E_0$ such that $|x| = \sup_{e \in E_0}\{d(x, e)\}= d(x, x')$ and $|y| = d(y, y')$. Fixed a point $z \in \gamma_0 \cap E_0$. The triangle inequality shows that
%    \begin{align*}
%        |x| + |y|
%        & = d(x, x') + d(y, y')
%        \leq d(x, z) + d(z, x') + d(y, z) + d(z, y') \\
%        & \leq d(x, z) + d(y, z) + 2 \mbox{diam}(E_0)
%        \leq \mbox{length}(\gamma_0) + 2 \mbox{diam}(E_0) \\
%        & \leq d_{+}(x,y) + 2 \mbox{diam}(E_0).
%    \end{align*}
%    Above, in the third inequality, we used the fact that
%    $d(x,z) \leq \mbox{length}(\gamma_{xz})$ for any curve $\gamma_{xz}$  connecting $x$ to $z$.
%    Similarly, we have $d(z, y) \leq \mbox{length}(\gamma_{yz})$ for any curve $\gamma_{yz}$  connecting $y$ to $z$. And, $\gamma_0$ is the connection of a curve connecting $x$ to $z$ and a curve connection $z$ to $y$. Therefore, we have
%    \[
%        d(x, z) + d(y, z) \leq \mbox{length}(\gamma_0).
%    \]
%%    On the other hand, we have
%%    \begin{align*}
%%        d_+(x,y) = d(x,z) + d(y,z)
%%        \leq |x| + |y|.
%%    \end{align*}
%    It completes the proof.
%\end{proof}

According to Grigor'yan and Saloff-Coste \cite[Proposition 6.5]{gri-sal09}, each weighted manifold $\widetilde{M}_i$ is a non-parabolic manifold and satisfies $(LY)$.
%Recall that a manifold is non-parabolic  if and only if at least one of its ends is non-parabolic (see \cite[Proposition 14.1]{gri99}).
Note that the gluing manifold $\widetilde{M}=\widetilde{M}_1 \# \cdots \# \widetilde{M}_\ell$ is non-parabolic if and only if
%there exists an $1 \leq i \leq \ell$ such that $\widetilde{M}_i$ is non-parabolic (see \cite[Proposition 14.1]{gri99}).
there exists at least one manifold $\widetilde{M}_i$ ($1 \leq i \leq \ell$) is non-parabolic (see \cite[Proposition 14.1]{gri99}).
%$\widetilde{M}_i$ is non-parabolic for some $1 \leq i \leq \ell$
%(see \cite[Proposition 14.1]{gri99}).
Therefore, with the help of \cite[Theorem 4.9]{gri-sal09}, one can obtain the estimates for the heat kernel $\tilde{h}_t(x,y)$ on $\widetilde{M}$,
which can be further transferred to the heat kernel $h_t(x,y)$ on $M$ by Doob's transform
 $h_t(x,y) = h(x) h(y) \tilde{h}_t(x,y);$
%see \cite[\S 6.2]{gri-sal09} for more details.%
%see \cite[(6.3)]{gri-sal09}.
see also \cite[Subsection 6.3]{gri-sal09}.

%Define a subset $I \subset\{1, \ldots, \ell\}$ by
%$$
%i \in I  \Leftrightarrow M_i \text { is parabolic}.
%%\Longleftrightarrow n_i \le 2.
%$$

\begin{thm}[\cite{gri-sal09}]
\label{thm:weighted-upper bound}
 Let $M=M_1 \# \cdots \# M_\ell$ be a connected sum of complete non-compact manifolds $M_i$.
 Assume that $M$ is non-parabolic and that $M_i$ satisfies $(LY)$ and $(RCA)$ for each $1 \leq i \leq \ell$.
  %Assume further that $M_i$ satisfies $(RCA)$ for each $i \in I$.
  Then the heat kernel ${h}_t(x, y)$ satisfies %for any $x, y \in M$ and $t >0$,
    %\begin{align*}%\label{heat-mixed}
%        h_t(x,y)
%        &\lesssim h(x)h(y)\left(\frac{\widetilde{H}(x,t)\widetilde{H}(y,t)}{\widetilde{V}_0(\sqrt t)}+\frac{\widetilde{H}(y,t)}{\widetilde{V}_{i_x}(\sqrt t)}+\frac{\widetilde{H}(x,t)}{\widetilde{V}_{i_y}(\sqrt t)}\right)
%        e^{-c\frac{d_+(x,y)^2}{t}} \nonumber\\
%       % \tag{\textsl{ME}}
%        & \ +\frac{h(x)h(y)}{\sqrt{\widetilde{V}_{i_x}(x,\sqrt t)\widetilde{V}_{i_y}(y,\sqrt t)}}
%        e^{-c\frac{d_{\emptyset}(x,y)^2}{t}}.
%    \end{align*}
%
    \begin{eqnarray*}%\label{heat-mixed}
        h_t(x,y)
        && \lesssim h(x)h(y)\left(\frac{\widetilde{H}(x,t)\widetilde{H}(y,t)}{\widetilde{V}_0(\sqrt t)}+\frac{\widetilde{H}(y,t)}{\widetilde{V}_{i_x}(\sqrt t)}+\frac{\widetilde{H}(x,t)}{\widetilde{V}_{i_y}(\sqrt t)}\right)
        e^{-c\frac{d_+(x,y)^2}{t}} \nonumber\\
       % \tag{\textsl{ME}}
        && \ +\frac{h(x)h(y)}{\sqrt{\widetilde{V}_{i_x}(x,\sqrt t)\widetilde{V}_{i_y}(y,\sqrt t)}}
        e^{-c\frac{d_{\emptyset}(x,y)^2}{t}},
        \quad \forall x, y \in M, \, t>0.
    \end{eqnarray*}
\end{thm}

% \begin{align}
%        h_t(x,y)
%        \lesssim h(x)h(y)\left(\frac{\widetilde{H}(x,t)\widetilde{H}(y,t)}{\widetilde{V}_0(\sqrt t)}+\frac{\widetilde{H}(y,t)}{\widetilde{V}_{i_x}(\sqrt t)}+\frac{\widetilde{H}(x,t)}{\widetilde{V}_{i_y}(\sqrt t)}\right)
%        e^{-c\frac{d_+(x,y)^2}{t}}
%         +\frac{h(x)h(y)}{\sqrt{\widetilde{V}_{i_x}(x,\sqrt t)\widetilde{V}_{i_y}(y,\sqrt t)}}
%        e^{-c\frac{d_{\emptyset}(x,y)^2}{t}}.
%    \end{align}

    %\begin{proof}
%      For the proof, we refer to \cite[Theorem 6.6]{gri-sal09}.
%    \end{proof}

   % Besides, we will use repeatedly the following estimates, which can be deduced
%    from \cite[Section 4]{gri-sal09}.

%One says that the heat kernel satisfies a Gaussian upper bound
%if there exist constants $C, c>0$ such that
%%$(U E)$
%%$$
%%h_{i, t}(x, y) \leq \frac{C}{V_i(x, \sqrt{t})} \exp \left\{-\frac{d^2(x, y)}{C t}\right\}, \quad \forall t>0, x, y \in M_i .
%%$$
%\begin{equation} \tag{$UE$}
%h_{i, t}(x, y) \leq \frac{C}{V_i(x, \sqrt{t})} \exp\lf\{-c\frac{d^2(x, y)}{t}\r\}, \quad \forall t>0, \, x, y \in M_i.
%\end{equation}
    %Besides, we also collect the following heat kernel upper bounds deduced from \cite[Section 4]{gri-sal09}.
    We also collect some heat kernel upper bounds deduced from \cite[Section 4]{gri-sal09}.
\begin{thm}%[Grigor'yan \& Saloff-Coste]
\label{thm:sim-heat-es}
   Under the assumptions of Theorem \ref{thm:weighted-upper bound}, we have the following estimates.
 \begin{enumerate}
       \item [\rm{(i)}]
  % Assume that %for each $i = 1, \dots, \ell$,
%   %{\color{blue}each} manifold
%   each $M_i$ satisfies $(D)$ and $(UE)$. Then
The small time heat kernel Gaussian upper bound holds, namely
% ��ʦ�� ���ǲ�ȷ����������˵�Բ��ԡ� ��Ϊ֮ǰû�м��������ı�����
    \begin{equation*} %\label{eq:small-time}
        h_t(x,y) \lesssim
        \frac{1}{V(x, \sqrt t)} e^{-c\frac{d(x,y)^2}{t}},
        \quad \forall x,y \in M, \, 0 < t\leq 1.
    \end{equation*}
    %for all $0 < t\leq 1$ and $x,y \in M$.
   % Assume that for each $i = 1, \dots, \ell$, each manifold $M_i$ satisfies $(D)$ and
% the heat kernel satisfies the Gaussian upper bound
%\begin{equation} \tag{$UE$}
%h_{i, t}(x, y) \lesssim \frac{1}{V_i(x, \sqrt{t})} e^{-c\frac{d^2(x, y)}{t}}, \quad \forall t>0, \, x, y \in M_i.
%\end{equation}
%   Then the small time heat kernel Gaussian upper bound holds, namely
%    \begin{equation} \label{eq:small-time}
%        h_t(x,y) \lesssim
%        \frac{1}{V(x, \sqrt t)} e^{-c\frac{d(x,y)^2}{t}},
%        \quad \forall 0 < t\leq 1, \, x,y \in M.
%    \end{equation}
   \item [\rm{(ii)}]  %Under the assumptions of Theorem \ref{thm:weighted-upper bound},
      Let $1 \leq i \leq \ell$. Then it holds
        % If $x, y \in E_i$, $1 \leq i \leq \ell$, then for any $t > 0$,  %the heat kernel $\tilde{h}_t(x, y)$
         %satisfies
        \begin{equation*}%\label{eq:heat-es-same-end}
            \tilde{h}_t(x,y)
            \lesssim \frac{\widetilde{H}(x,t)\widetilde{H}(y,t)}{\widetilde{V}_0(\sqrt t)} e^{-c\frac{|x|^2+|y|^2}{t}}
             + \frac{1}{\widetilde{V}_i(x, \sqrt t)} e^{-c\frac{d(x,y)^2}{t}},
            \quad  \forall x, y \in E_i, \, t >0.
        \end{equation*}
   \item [\rm{(iii)}]
   Let $0 \leq i,j \leq \ell$ with $i \neq j$. Then it holds
      % Moveover,
%       if $x \in E_i$, $y \in E_j$,  $i \neq j$, $0 \leq i,j \leq \ell$,  %and $0\leq j  \leq \ell$,
%       then for any $t \ge 1$,
    \begin{equation*}%\label{eq:heat-es-dif-end}
      \tilde{h}_t(x,y)
      \lesssim \left(\frac{\widetilde{H}(x,t)\widetilde{H}(y,t)}{\widetilde{V}_0(\sqrt t)} + \frac{\widetilde{H}(y,t)}{\widetilde{V}_i(\sqrt t)} +       \frac{\widetilde{H}(x,t)}{\widetilde{V}_j(\sqrt t)}  \right) e^{-c\frac{|x|^2+|y|^2}{t}},
      \quad \forall  x \in E_i, \, y \in E_j, \, t \ge 1.
    \end{equation*}
   %\item [\rm{(ii)}]  %Under the assumptions of Theorem \ref{thm:weighted-upper bound},
%       For the above weighted manifolds $\widetilde{M}$, the heat kernel $\tilde{h}_t(x, y)$
%         satisfies that  for any $t >0$ and $x, y \in E_i$ with $1 \leq i \leq \ell$,
%        \begin{equation}\label{eq:heat-es-same-end}
%            \tilde{h}_t(x,y)
%            \lesssim \frac{\widetilde{H}(x,t)\widetilde{H}(y,t)}{\widetilde{V}_0(\sqrt t)} e^{-c\frac{|x|^2+|y|^2}{t}}
%             + \frac{1}{\widetilde{V}_i(x, \sqrt t)} e^{-c\frac{d^2(x,y)}{t}}.
%        \end{equation}
%       Moveover, it holds that for any $t \ge 1$, $x \in E_i$ and $y \in E_j$ with $1 \leq i \leq \ell$, $0\leq j  \leq \ell$ and $i \neq j$.
%    \begin{equation}\label{eq:heat-es-dif-end}
%      \tilde{h}_t(x,y)
%      \lesssim \left(\frac{\widetilde{H}(x,t)\widetilde{H}(y,t)}{\widetilde{V}_0(\sqrt t)} + \frac{\widetilde{H}(y,t)}{\widetilde{V}_i(\sqrt t)} +       \frac{\widetilde{H}(x,t)}{\widetilde{V}_j(\sqrt t)}  \right) e^{-c\frac{|x|^2+|y|^2}{t}}.
%    \end{equation}
    \end{enumerate}
\end{thm}
\begin{proof}
  The claim (i) follows from \cite[%Thereoem 4.2 and
  Corollary 4.7]{gri-sal09}.
  %The estimate \eqref{eq:heat-es-same-end}
  The claim (ii)
  comes from \cite[(4.45)]{gri-sal09}.
  To get the claim (iii), %\eqref{eq:heat-es-dif-end},
  it suffices to use
   \cite[Theorem 4.9]{gri-sal09} and Remark \ref{rem:new-dist}. This completes the proof.
    %see also \cite[Remark 4.10]{gri-sal09}.
\end{proof}

\begin{rem} \label{rem:sma-ue} \rm
  Referring to %the proof of Lemma \ref{lem:volume-growth} (iii) and
  the setting of \cite[Corollary 4.7]{gri-sal09}, the assumptions of
  Theorem \ref{thm:sim-heat-es} (i)
  can be relaxed to that
   each $M_i$ satisfies $(D)$ and  the heat kernel satisfies the Gaussian upper bound
\begin{equation} \tag{$UE$}
h_{i, t}(x, y) \lesssim \frac{1}{V_i(x, \sqrt{t})} e^{-c\frac{d(x, y)^2}{t}}, \quad \forall x, y \in M_i, \, t>0.
\end{equation}
\end{rem}

Using  \cite[(4.42) and (4.45)]{gri-sal09}, %\cite[Corollary 4.16]{gri-sal09},
we obtain the following on-diagonal upper bounds of the heat kernel.
\begin{lem}%[\cite{gri-sal09}]
\label{lem:for-ultra}
  %The following statements holds.
     Let $1 \leq i \leq \ell$. It holds
     that
     $$
         h_t(x,x)
         \lesssim \frac{1}{V_0(\sqrt t)},
         \quad \forall x\in E_0, \, t >0
     $$
     and
     $$
        h_t(x,x)
        \lesssim
        \frac{1}{V_0(\sqrt t)}  + \frac{1}{V_{i}(x,\sqrt t)},
        \quad \forall x \in E_i, \, t>0.
     $$
      % \begin{equation*}% \label{eq:ca1}
%         h_t(x,x)
%          \lesssim
%          \frac{1}{V_0(\sqrt t)} +
%          \begin{cases}
%            \hfil 0, & \, \mbox{if} \, x \in E_0, \\%&  \mbox{if} \, i=0,\\%x, y \in E_0\\
%            \frac{1}{V_{i}(x,\sqrt t)},
%            &\, \mbox{if}\, x \in E_i \, (1 \leq i \leq \ell).
%            %& \mbox{if}\, 1\leq i \leq \ell.%x, y \in E_i , \, 1\leq i \leq \ell.
%          \end{cases}
%    \end{equation*}
%     ===================\\
%     For any $t>0$, it holds
%       \begin{equation*}% \label{eq:ca1}
%         h_t(x,x)
%          \lesssim
%          \frac{1}{V_0(\sqrt t)} +
%          \begin{cases}
%            \hfil 0, & \, \mbox{if} \, x \in E_0, \\%&  \mbox{if} \, i=0,\\%x, y \in E_0\\
%            \frac{1}{V_{i}(x,\sqrt t)},
%            &\, \mbox{if}\, x \in E_i \, (1 \leq i \leq \ell).
%            %& \mbox{if}\, 1\leq i \leq \ell.%x, y \in E_i , \, 1\leq i \leq \ell.
%          \end{cases}
%    \end{equation*}
\end{lem}

\section{Estimates on the volume growth and auxiliary functions}  \hskip\parindent  \label{sec:vol-aux}
%Ϊ��XXX�� ��һ������XXX�Ⱥ˽硣 or ���Ⱥ��Ͻ�Ӧ����������ƪ���ĵ������¡� (in our setting)
%or �����Ⱥ��е�һЩ�����Ľ硣
%����һ���У� ���ǽ����ȸ���һЩ����some properties on the volume growth
%of $\mu_i$. Then, we shall estimates the harmonic function $h$ and the weighted volume $\tilde{\mu}_i$.
%���󣬸����������Ⱥ��Ͻ��и�������$\widetilde{H}(x,t)$�Ĺ��ơ�
%����Щ���Ƶİ����£� ���ǿ���XXXX��
%In this section,
We shall give some estimates on the terms which occur in the priori heat kernel estimates.
First, we shall  give some properties on the growth of the volume ${V}_i$.
Then, we shall estimates the harmonic function $h$ and the weighted volume $\widetilde{V}_i$.
Finally, we will gives an estimate of the auxiliary function $\widetilde{H}$. % that occurs in the upper bound of the heat kernel.
%With the help of these estimates, we can apply the upper bound of the heat heat kernel in our setting.
\subsection{Basic properties on the volume growth} \hskip\parindent
Note that the doubling condition on $M_i$ ($1 \leq i \leq \ell$) implies that, there exists $N_i >0$,
which is of course not less than $n_i$, such that
\[\frac{V_i(x, R)}{V_i(x, r)} \lesssim \left(\frac{R}{r}\right)^{N_i}, \quad \forall x \in M_i, \, 0<r \leq R <\infty.\]
 Throughout the paper, we set
\begin{equation*} %\label{N-infty}
N_{\infty}:=\max\{N_1,\cdots, N_\ell\}.
\end{equation*}
Let $1 \leq i \leq \ell$.
By the standard trick of doubling property, for any $c' > 0$, it holds
%and for all $u, v \in M_k$ and $t > 0$, it holds
%% The world in GS-2009 P1944
    \begin{equation} \label{eq:change-centre}
     \frac{ V_i(x, \sqrt t)}{V_i(y, \sqrt t)}
     \lesssim \frac{ V_i(y, \sqrt t + d(x,y))}{V_i(y, \sqrt t)}
     \lesssim \lf(\frac{\sqrt t + d(x,y)}{\sqrt t}\r)^{N_i}
     \lesssim_{c', N_i} e^{c'\frac{d(x,y)^2}{t}},
     \quad
     \forall x,y \in M_i, \, t>0.
    \end{equation}
Sometimes, we will apply this trick without mentioning it.
%\subsection{Estimates on the harmonic function $h$}
Besides, since each $M_i$ is a connected doubling manifold, the doubling condition implies that there exists a constant $\delta_i>0$ such that
%for any $x\in M_i$ and $R>r>0$, it holds
\begin{equation}\label{rev-d}
c \left(\frac{R}{r}\right)^{\delta_i}\le \frac{V_i(x,R)}{V_i(x,r)},
 \quad \forall x\in M_i, \, 0<r<R<\infty;
\end{equation}
see for instance \cite[Remark 8.1.15]{hkst}.
Note that $\delta_i$ might be different from $n_i$ but is not larger than $n_i$.

 Next, we shall introduce the following volume growth properties which will be used repeatedly.  %and may not be mentioned sometimes.
\begin{lem}  \label{lem:volume-growth}
Let $1 \leq i \leq \ell$. The following statements are valid.
\begin{enumerate}
  \item[\rm{(i)}]
  %Since for each $M_i$, there exists $x_i\in M_i$ and $n_i \geq 1$, such that for all $1 \le r \le R < \infty$
%\[
%    \begin{dcases}
%            c_i\left(\frac{R}{r}\right)^{n_i}\le \frac{V_i(x_i,R)}{V_i(x_i,r)}, & \, n_i > 2,  \\
%            c_i\left(\frac{R}{r}\right)^{n_i}\le \frac{V_i(x_i,R)}{V_i(x_i,r)} \le c_i\left(\frac{R}{r}\right)^{n_i}, & \, n_i \leq 2,
%    \end{dcases}
%\]
%We find that the following inequality holds for $o_i$ as
For any $1 \le r \le R < \infty$,  it holds that
%\begin{equation*}
%    \begin{dcases}
%            \frac{V_i(R)}{V_i(r)} \gtrsim \left(\frac{R}{r}\right)^{n_i}, &\mbox{if} \,  n_i > 2,  \\
%            \frac{V_i(R)}{V_i(r)} \sim \left(\frac{R}{r}\right)^{n_i}, &\mbox{if} \, n_i \leq 2,
%    \end{dcases}
%\end{equation*}
%%\begin{equation*}
%%    \begin{dcases}
%%            \left(\frac{R}{r}\right)^{n_i} \lesssim \frac{V_i(R)}{V_i(r)}, &\mbox{if} \,  n_i > 2,  \\
%%            \left(\frac{R}{r}\right)^{n_i} \lesssim \frac{V_i(R)}{V_i(r)} \lesssim \left(\frac{R}{r}\right)^{n_i}, &\mbox{if} \, n_i \leq 2,
%%    \end{dcases}
%%\end{equation*}
%\[
%    \frac{V_i(R)}{V_i(r)} \gtrsim \left(\frac{R}{r}\right)^{n_i},
%    \quad \mbox{if} \,  n_i > 2,
%    \qquad
%    \frac{V_i(R)}{V_i(r)} \sim \left(\frac{R}{r}\right)^{n_i}, \quad \mbox{if} \, n_i \leq 2,
%\]
\[
   \frac{V_i(R)}{V_i(r)} \gtrsim \left(\frac{R}{r}\right)^{n_i},
    \quad \mbox{if} \,  n_i > 2,
\]
and
\[
    \frac{V_i(R)}{V_i(r)} \sim \left(\frac{R}{r}\right)^{n_i}, \quad \mbox{if} \, n_i \leq 2.
\]
%for any $R \geq r \geq  1$ and $1 \leq i \leq \ell$.
% In particular, for any $R \ge 1$, we have $V_i(R) \gtrsim R^{n_i}$, if $n_i >2$,
%    and $V_i(R) \sim R^{n_i}$, if $n_i \leq 2$.
 In particular, if $n_i \ge 2$, then %$R^2 /V_i(R) \lesssim r^2/V_i(r)$.
  $$\frac{R^2}{V_i(R)} \lesssim \frac{r^2}{V_i(r)}.$$
  \item[\rm{(ii)}]
  We have
  \begin{equation*}
    \mu(F_i^{(r)}) \sim V_i(2r) \sim V_i(r),  \quad \forall r \geq 1. %\, 1\leq i \leq \ell.
  \end{equation*}
  \item[\rm{(iii)}] The gluing manifold $M$ satisfies the local doubling volume property
$$
\frac{V(x, 2 r)}{V(x, r)} \lesssim 1, \quad \forall x \in M, \, 0<r \leq 1,
$$
and the volume is of at most polynomial growth in the sense that
$$
\frac{V(x, r)}{V(x, 1)} \lesssim r^{N_{\infty}}, \quad \forall x \in M, \, r \geq 1.
$$
\end{enumerate}
\end{lem}
\begin{proof}
%For the proof of (ii), we refer to \cite[Lemma 2.5 (ii)]{jiang-li-lin-2022}.
%Here, we need only prove (i).
%We only need to prove the claim (i), as the claims (ii) and (iii) has  been proved in \cite[Lemma 2.5]{jiang-li-lin-2022}.
According to \cite[Lemma 2.5]{jiang-li-lin-2022}, we only need to prove the claim (i).
For any %$1 \leq i \leq \ell$ and
$1 \le r \le R < \infty$, by the doubling property, we have
\[
    V_i(x_i, R+d(x_i,o_i))
    \lesssim V_i(o_i ,R+d(x_i,o_i))
    \lesssim \lf(\frac{R+d(x_i,o_i)}{R}\r)^{N_i} V_i(o_i, R),
\]
and hence%, as $x_i$ is a fixed point, it holds
\begin{eqnarray*}
\frac{V_i(R)}{V_i(r)}
&&\gtrsim \left(\frac{R}{R+d(x_i,o_i)}\right)^{N_i}\frac{V_i(x_i,R+d(x_i,o_i))}{V_i(x_i,r+d(x_i,o_i))} \\
&& \gtrsim \left(\frac{R}{R+d(x_i,o_i)}\right)^{N_i}\left(\frac{R+d(x_i,o_i)}{r+d(x_i,o_i)}\right)^{n_i}\\
&& \sim \left(\frac{R}{r}\right)^{n_i}\left(\frac{R}{R+d(x_i,o_i)}\right)^{N_i-n_i}\left(\frac{r}{r+d(x_i,o_i)}\right)^{n_i}\\
% Sugar water inequality
&& \gtrsim \left(\frac{R}{r}\right)^{n_i}\left(\frac{1}{1+d(x_i,o_i)}\right)^{N_i-n_i}\left(\frac{1}{1+d(x_i,o_i)}\right)^{n_i}
 %Notice that $x_i$ is fixed
\gtrsim \left(\frac{R}{r}\right)^{n_i},
\end{eqnarray*}
%by the fact that
since $x_i$ and $o_i$ are fixed points. %in {\color{blue}Therorem \ref{main-result-parabolic}}.
Similarly, when $n_i \leq 2$,
%by the same argument as above,
by the doubling property,
we have
\[
    V_i(x_i, r+d(x_i,o_i))
    \lesssim V_i(o_i, r+d(x_i,o_i))
    \lesssim \lf(\frac{r+d(x_i,o_i)}{r}\r)^{N_i} V_i(o_i, r),
\]
and hence
\begin{eqnarray*}
    \frac{V_i(R)}{V_i(r)}
    && \lesssim \left(\frac{r+d(x_i,o_i)}{r}\right)^{N_i} \frac{V_i(x_i, R+d(x_i, o_i)) }{ V_i(x_i, r+d(x_i,o_i)) } \\
    && \lesssim \left(\frac{r+d(x_i,o_i)}{r}\right)^{N_i} \left(\frac{R+d(x_i, o_i) }{ r+d(x_i,o_i) }\right)^{n_i} \\
    && \sim \left(\frac{R}{r}\right)^{n_i} \left( \frac{r+d(x_i,o_i)}{r}\right)^{N_i-n_i} \left(\frac{R+d(x_i,o_i)}{R}\right)^{n_i}  \\
%    Reverse-Sugar water inequality
    && \lesssim \left(\frac{R}{r}\right)^{n_i} \left( 1+d(x_i,o_i) \right)^{N_i-n_i} \left(1+d(x_i,o_i)\right)^{n_i}
     \lesssim \left(\frac{R}{r}\right)^{n_i}.
\end{eqnarray*}
%(ii) We refer to \cite[Lemma 2.5 (ii)]{jiang-li-lin-2022}.
This completes the proof.
\end{proof}
\subsection{Estimates on the weighted volume} \hskip\parindent
%{\color{blue}xxxxxxxxx
In this section, we shall first give some estimates on the harmonic function $h$ and then estimate the weighted volume.
%}
%
\subsubsection{Estimates on the harmonic function $h$ and the auxiliary function $\phi_i$}\hskip\parindent
%For each $0 \leq i \leq \ell$ and any $r > 0$, set
%\begin{equation*}
%    \phi_i(r) := 1+\left(\int_1^{r^2}\frac{1}{V_{i}(\sqrt s)}\,ds\right)_+
%\end{equation*}
%and note that for any $x \in M$,
%\[
%    h(x) \sim \phi_{i_x}(|x|).
%\]
In order to estimate the harmonic function $h$, let us introduce the following auxiliary function $\phi_i$. For each $1 \leq i \leq \ell$,
set
\begin{equation*}
    \phi_i(r) := 1+\left(\int_1^{r^2}\frac{1}{V_{i}(\sqrt s)}\,ds\right)_+,
    \quad \forall r >0,
\end{equation*}
and note that
\[
    h(x) \sim \phi_{i_x}(|x|),
    \quad \forall  x \in M.
\]

\begin{lem}\label{lem:phi-r-equal}
    Let $1\leq i \leq \ell$. We can obtain that for any $ 0 < r < 1$,
    $\phi_i(r)= 1$, and for any $r \ge 1$,
    \begin{equation*}
        \phi_i(r)
        \sim
        \begin{cases}
           1, & \, \mbox{if} \, n_{i} >2, \\
           \log(2+r), & \, \mbox{if} \, n_{i} =2, \\
           r^{2-n_{i}},  &  \, \mbox{if} n_{i}< 2.
        \end{cases}
    \end{equation*}
%    Moreover, for any $r \ge 1$, it holds
%    %if $r > \min(c_0, 1)$, where $c_0$ is the constant in Lemma \ref{rem:x-separ}, then
%    \[
%        \phi_i(r) \sim r^{2-n_{i}},  \quad \mbox{if} \, n_{i}< 2.
%    \]
    %In conclusion, for any $r \geq \min(c_0, 1)$, it holds
%    \begin{equation*}
%        \phi_i(r)
%        \sim
%        \begin{cases}
%           1, & \, \mbox{if} \ n_{i} >2, \\
%           \log(2+r), & \, \mbox{if} \ n_{i} =2, \\
%            r^{2-n_{i}}, & \, \mbox{if} \ n_{i}< 2.
%        \end{cases}
%    \end{equation*}
   % In fact, we set $r$ has a positive lower bound is enough, such that $r\geq c$.
\end{lem}
\begin{proof}
    For any $0 < r <1$, by the definition of $\phi_i$, it holds $\phi_i(r) =1$.
    It remains to consider the case $r \ge 1$.
    When $n_{i} >2$,
    %on the one hand,
    %note that $\phi_{i}(r) \geq 1$ holds by its definition,
%    and one has
%    $$  1+ \left(\int_1^{r^2} \frac{1}{V_{i}(\sqrt s)}\,ds \right)_+
%%    \leq  \int_1^{\infty}\frac{1}{V_{i}(\sqrt s)}\,ds
%    \leq 1+ \int_1^{\infty}\frac{1}{s^{n_i/2}}\,ds \lesssim 1.$$}
%%    Therefore, $\phi_{i}(r) \sim 1$.
    it holds that %for any $r>1$,
    $$ 1\leq \phi_i(r) = 1+ \int_1^{r^2} \frac{1}{V_{i}(\sqrt s)}\,ds
%    \leq  \int_1^{\infty}\frac{1}{V_{i}(\sqrt s)}\,ds
    \lesssim 1+ \int_1^{\infty}\frac{1}{s^{n_i/2}}\,ds \lesssim 1.$$
   When $n_{i}=2$, it follows that %for any $r \ge 1$,
    \begin{align*}
        \phi_i(r) = 1+ \int_1^{r^2}\frac{1}{V_{i}(\sqrt s)}\,ds
        \sim 1+ \int_1^{r^2} \frac{1}{s} \,ds
        %\sim 1+ \left. \log(s) \right|^{r^2}_1
        \sim 1+ \log r \sim \log(2+r).
    \end{align*}
    When $n_{i} < 2$, one has
     \begin{align*}
         \phi_i(r) = 1+  \int_1^{r^2}\frac{1}{V_{i}(\sqrt s)}\,ds
         \sim 1+ \int_1^{r^2} \frac{1}{s^{n_i/2}}\,ds
         %\sim 1 + \left.s^{1- n_i/2}\right|^{r^2}_1
         \sim r^{2-n_{i}}.
    \end{align*}
 %   =============DETAILS OF  $1+ \int_1^{r^2} \frac{1}{s^{n_i/2}}\,ds \gtrsim r^{2-n_{i}}$=============\\
%    Besides, it holds that for any $1 \leq r \leq 2$,
%    \[
%        r^{2-n_{i}} %\leq 2^{2}
%        \lesssim 1
%         \lesssim 1+\left( \int_1^{r^2}\frac{1}{V_{i}(\sqrt s)}\,ds \right)_+,
%    \]
%    and for any $r \ge 2$,
%    \[
%    r^{2-n_i}
%    \lesssim \int_{r^2/4}^{r^2} \frac{1}{s^{n_i/2}}\,ds
%  %  \lesssim \int_{r^2/4}^{r^2} \frac{1}{V_{i}(\sqrt s)}\,ds
%    \lesssim 1 + \left( \int_1^{r^2}\frac{1}{V_{i}(\sqrt s)}\,ds \right)_+.
%    \]
%    =====================END OF DETAILS===================================\\
  This completes the proof.
\end{proof}

\begin{rem}%\label{lem:doubling-phi}
    \rm
    For each $1 \leq i \leq \ell$, we can observe that the function $\phi_i$ satisfies the doubling property, i.e., there exists a constant $C>0$, such that for any $r > 0,$
    $$\phi_i(2r) \leq C \phi_i(r).$$
    Indeed,
%\begin{proof}
    when $r \leq 1$, by the definition of $\phi_i$, it holds $ \phi_i(2r) \leq C$ and $\phi_i(r) =1$, which implies that $ \phi_i(2r) \leq C \phi_i(r).$
   % When $0 < r < 2r \leq 1$, notice that $\phi_i(r) = \phi_i(2r)=1$. It holds that
%    $$\phi_i(2r) \leq \phi_i(r).$$
%
%    When $0 < r \leq 1 <2r$, notice that $\phi_i(r) =1$. By Lemma \ref{lem:phi-r-equal}, it holds
%    \[
%        \phi_i(2r) \leq \phi_i(2)
%        \leq C \leq C \phi_i(r).
%    \]
    %
   When $r \ge 1$,
   %note that
   %$  \log(2+2r) \leq  \log(4r) \leq 2\log2 +\log(r) \leq 3 \log(2+r)$.
   it follows from Lemma \ref{lem:phi-r-equal} that
   $ \phi_i(2r) \leq C \phi_i(r)$.
  % as desired.
   %which completes the proof.
   % When $r \ge 1$, by Lemma \ref{lem:phi-r-equal},
%    if $n_i <2$ or $n_i>2$,
%    \[
%        \phi_i(2r) \leq C_i \phi_i(r)
%    \]
%    holds.
%    As for $n_i =2$, we have
   % \[
%        \phi_i(2r)
%        \leq C_i \log(2+2r)
%        \leq C_i \log(4r)
%        \leq C_i (2\log2 +\log(r))
%        \leq C_i \log(2+r)
%        \leq C_i \phi_i(r),
%    \]
%    where $C_i$ is a positive constant depending on $i$. This completes the proof.
    %
%    Notice that when $r\leq 1$, $\phi_i(r)=1$.
%
%     When $r>1$, we have
%    \[
%        \phi_i(r) \sim
%        \begin{cases}
%            1, \quad & n_i>2, \\
%            \log(2+r), & n_i=2, \\
%            r^{2-n_i}, & n_i<2.
%        \end{cases}
%    \]
%
%
%
%
%
%
%    (i) $0 < r < 2r \leq 1$
%    $$\phi_i(2r)=1 \leq \phi_i(r).$$
%
%    (ii) $0 < r \leq 1 <2r$
%    $$\phi_i(2r) \leq \phi_i(2) \leq C \leq C \phi_i(r).$$
%
%    (iii) $1 <r <2r$
%
%    When $n_i \neq 2$, $\phi_i(2r) \leq C\phi_i(r)$ is clear.
%
%    When $n_i =2$, the logarithmic function is also doubling. Let us simply verify  the doubling property of the function $f(x) = \log(2+x)$, $x>0$.
%
%    If $x \geq 3$, then
%        \[
%            f(2x) = \log(2+2x) \leq \log(3x) \leq \log x^2 \leq 2 \log(2+x).
%        \]
%
%    If $x < 3$, then
%    \[
%        f(2x) = \log(2+2x) < \log8 \leq \frac{\log8}{\log2} \log2 \leq 3 \log(2+x).
%    \]
%    It completes the proof.
%\end{proof}
\end{rem}

%Let us start from the heat kernel estimate obtained by \cite[Section 6]{gri-sal09}.

\subsubsection{Estimates on the weighted volume $\widetilde{V}_i$} \hskip\parindent
 Let us give some estimates for  the weighted volume $\widetilde{V}_i$.
\begin{lem} \label{lem:V-tildeV}
Let $1\leq i \leq \ell$. For any $x \in M_i$ and $r>0$, it holds
\begin{equation}\label{eq:V-genernal V}
    \widetilde{V}_i(x,r) \sim (\phi^2_{i}(|x|) + \phi^2_i(r)) V_i(x,r)
\end{equation}
%\[
%    \widetilde{V}_i(x,r) \sim [\phi^2_{i}(|x|) + \phi^2_i(r)] V_i(x,r)
%\]
and
\begin{equation}\label{eq:V-genernal V-2}
    \widetilde{V}_i(r) \sim \phi^2_i(r) V_i(r).
\end{equation}
\end{lem}
\begin{proof}
     %The proof of this lemma was given in \cite[Theorem 4.8]{gri-sal02}; see also \cite[(6.6) \& (6.7)]{gri-sal09}.
     This lemma was given in \cite[(6.6) \& (6.7)]{gri-sal09}; see also \cite[Theorem 4.8]{gri-sal02}.
     Here we give a brief argument of this proof for completeness and the convenience of the reader.

    Notice that $h_i(y) \sim \phi_{i}(|y|)$ and
    $$\widetilde{V}_i(x,r) = \int_{B_i(x,r)} h^2_i(y) \, d\mu_i(y).$$
    %$$\widetilde{V}_i(x,r) = \int_{B_i(x,r)} h^2_i(y) \, d\mu_i(y),
%     %\quad \mbox{and} \quad
%     \qquad
%    h_i(y) \sim \phi_{i}(|y|).$$
%    %We first prove \eqref{eq:V-genernal V}.

 For any $y \in B_i(x,r)$,
% it holds $d(y,z) \leq d(x,z) + d(x,y) \leq d(x,z) + r$, $\forall z \in E_0$, and hence $|y| \leq |x| + r $.
 by the triangle inequality of the metric $d$, it holds
 $|y| \leq |x| + r $.
 By the doubling property of $\phi_i$, we have
\begin{align*}
    \widetilde{V}_i(x,r)
     %= \int_{B_i(x,r)} h^2_i(y) d\mu_i(y)
     \sim \int_{B_i(x,r)} \phi^2_i(|y|) \, d\mu_i(y)
    %& \leq C \phi_i^2(3|x|) V_i(x,r)
     \lesssim \phi^2_{i}(|x| + r) V_i(x,r)
    % \lesssim \phi^2_{i}(2\max\{|x|, r\}) V_i(x,r)
    \lesssim (\phi^2_{i}(|x|) + \phi^2_i(r)) V_i(x,r).
\end{align*}
One the other hand, if $|x| > r/2$, then for any $y \in B_i(x, r/4)$, it holds %$d(y,z) \geq d(x,z) - d(x,y) \geq d(x,z) - r/4$ for any $z \in E_0$, and hence
$|y| \geq |x| - r/4 \geq |x|/2$.
By the doubling properties of $\phi_i$ and $\mu_i$, %since $|x| > r/2$,
we have
\begin{align*}
    \widetilde{V}_i(x,r)
     %\geq \int_{B_i(x,\frac{r}{4})} h_i^2(y) d\mu_i(y)
    \gtrsim \int_{B_i(x, r/4)} \phi_i^2(|y|) \, d\mu_i(y)
    % \geq c \phi_i^2\left(\frac{|x|}{2}\right) V_i\left(x, \frac{r}{4}\right)
     \gtrsim \phi^2_{i}(|x|) V_i(x, r)
     \gtrsim (\phi^2_{i}(|x|) + \phi^2_i(r)) V_i(x, r).
\end{align*}
If $|x| \leq r/2$,
then for any $y \in B_i(x,r) \setminus B_i(x, 3r/4)$, it holds
%that
%$d(x,z) + d(y,z) \ge d(x,y) \ge 3r/4$ for any $z \in E_0$, and hence
$|y| \ge 3r/4 - |x| \geq r/4$.
By the doubling property of $\phi_i$ and the reverse doubling property of $\mu_i$ (see \eqref{rev-d}),
%since $|x| \leq r/2$,
we have
\begin{align*}
    \widetilde{V}_i(x,r)
     \gtrsim \int_{B_i(x,r) \setminus B_i(x, 3r/4)} \phi^2_i(|y|) \, d\mu_i(y)
     %\gtrsim \phi_i^2(r/4) (V_i(x,r) - V_i(x, 3r/4)) \\
     \gtrsim \phi_i^2(r)V_i(x,r)
     \gtrsim (\phi^2_{i}(|x|) + \phi^2_i(r)) V_i(x,r),
\end{align*}
and  hence %volume estimate
\eqref{eq:V-genernal V} follows readily.

Next, we turn to the proof of %will prove volume estimate
\eqref{eq:V-genernal V-2}.
%by the definition of $\phi_i$, we have $\phi_i(r) \geq 1$ for any $r>0$.
It holds
$
    |o_i| %= \sup_{y \in E_0} d(o_i, y)
    \leq \diam(E_0)=2,  %=2,
$
and hence
$
    \phi_i(|o_i|) \lesssim \phi_i(r)
$ for any $r>0$.
This together with \eqref{eq:V-genernal V} yields
$$\widetilde{V}_i(r) \sim (\phi^2_i(|o_i|) +\phi_i^2(r)) V_i(r) \sim \phi_i^2(r) V_i(r),$$
which gives the desired estimate  and completes the proof.
%
%By the definition of $\phi_i$, we have $\phi_i(r) \geq 1$ for any $r>0$.
%Combining this with
%$
%    |o_i| = \sup_{y \in E_0} d(o_i, y) \leq \diam(E_0) =2
%$
%%and the doubling property of $\phi_i$ (Lemma \ref{lem:doubling-phi})
%and \eqref{eq:V-genernal V}, one concludes that
%$$\widetilde{V_i}(r) \sim (\phi^2_i(|o_i|) +\phi_i^2(r)) V_i(r) \sim \phi_i^2(r) V_i(r),$$
%which finishes the proof.
%
%In the proof of \eqref{eq:V-genernal V}, we can deduce that
%if $r/2 \leq 1$, then $\widetilde{V_i}(r) \lesssim \phi^2(o_i) V_i(r)$. It together with $|o_i| \lesssim 1$ and $\phi(1) = 1 \lesssim \phi(r)$ gives us that
%$\widetilde{V_i}(r) \lesssim \phi_i^2(r) V_i(r)$. If $r/2 \geq 1$, then  $\widetilde{V_i}(r) \lesssim \phi^2(r) V_i(r)$.
%
%Therefore, $\widetilde{V_i}(r) \sim \phi^2(r) V_i(r)$ and it completes the proof.
\end{proof}

\begin{rem} \label{rem:thr}
Let $1 \leq i \leq \ell$. From Lemmas \ref{lem:phi-r-equal} and \ref{lem:V-tildeV}, we can observe the following simple facts.
  \begin{enumerate}
    \item[\rm(i)] \rm When $n_i > 2$, one has
%  \[
%    \widetilde{V}_i(x,r) \sim V_i(x,r)  \quad \mbox{and} \quad \widetilde{V}_i(r) \sim V_i(r),
%  \]
%  for any $x \in M_i$ and $r >0$.
  \[
    \widetilde{V}_i(x,r) \sim V_i(x,r),  \quad \forall x \in M_i, \, r >0.
  \]
  In particular, it holds $\widetilde{V}_i(r) \sim V_i(r)$ when $n_i >2$.
    \item[\rm(ii)]  %we have the following conclusion.
    When $r$ stays bounded and $x$ varies in a compact neighbourhood of $E_0$,
    it follows from \cite[(4.14)]{gri-sal09}  that
    \[
        V_i(x,r) \sim V(x,r) \sim V_0(r) \sim r^N
    \]
    and
    \[
        \widetilde{V}_i(x,r) \sim \widetilde{V}_i(r) \sim \widetilde{V}_0(r)  \sim r^N,
    \]
    where the positive integer $N$ is the topological dimension of the Riemannian manifold $M$.
    \item[\rm(iii)]  Since $\mu_i$ and $\phi_i$ is doubling,
   one can see that
   for any $x \in M_i$ and $r>0$,
    \[
        \widetilde{V}_i(x,2r) \leq C(\phi^2_{i}(|x|) + \phi^2_i(2r)) V_i(x,2r)
        \leq C (\phi^2_{i}(|x|) + \phi^2_i(r)) V_i(x,r) \leq C \widetilde{V}_i(x, r),
    \]
    which means
   that $\tilde \mu_i$ is a doubling measure on $\widetilde{M}_i$.
    %   {\color{blue} which means that %$\tilde{\mu}_i$ is doubling and
%    there exists a constant $\widetilde{N}_i$ such that
%    for any $x\in M_i$ and $0<r<R<\infty$, it holds
%\[
%\frac{\widetilde{V}_i(x,R)}{\widetilde{V}_i(x,r)}\le C_i\left(\frac{R}{r}\right)^{ \widetilde{N}_i}.
%\]}
%   This further implies that
    And the function $\widetilde{V}_0$ also satisfies the doubling property. %, i.e., there exists a constant $C >0$ such that for any $r>0$,
%    \[
%    \widetilde{V}_0(2r) \leq C \widetilde{V}_0(r).
%    \]
%    ``proof'': see gri-sal-2009, p.1945
%    \[
%        \widetilde{V}_0(2r) = \min\limits_{1\leq i\leq \ell} \widetilde{V}_i(2r)
%        \leq \min\limits_{1\leq i\leq \ell} C_i \widetilde{V}_i(r)
%        \leq \min\limits_{1\leq i\leq \ell} (\max\limits_{1\leq j\leq \ell} C_j \widetilde{V}_i(r))
%        \leq \max\limits_{1\leq j\leq \ell} C_j \min\limits_{1\leq i\leq \ell}  \widetilde{V}_i(r)
%        =  \max\limits_{1\leq j\leq \ell} C_j  \widetilde{V}_0(r).
%    \]
  \end{enumerate}
\end{rem}

Using the volume growth properties, the estimates of $\phi_i$, and the relation between the volume and the weighted volume, we obtain the following weighted volume growth properties which will be used repeatedly.
\begin{lem}\label{lem:tilde-V}
   Let $1\leq i \leq \ell$. %where $c_0$ is the constant in Remark \ref{rem:x-separ},
  The following statements hold true.
\begin{enumerate}
  \item[\rm(i)]
    For any $1 \le r \le R < \infty$, it holds that
    %\begin{eqnarray*}%\label{V-cond-general}
%    \frac{\widetilde{V}_{i}(R)}{\widetilde{V}_{i}(r)}
%     \sim \frac{\phi_i^2(R) V_{i}(R)}{\phi_i^2(r) V_{i}(r)}
%    \begin{cases}
%        \gtrsim \left(\frac{R}{r}\right)^{n_i} &\, \mbox{if} \, n_i>2, \\
%        \sim %\frac{R^{4-2n_i}}{r^{4-2n_i}} \left(\frac{R}{r}\right)^{n_i}=
%         \left(\frac{R}{r}\right)^{4-n_i} &\, \mbox{if} \, n_i<2, \\
%        \sim \frac{\log^2(2+R)}{\log^2(2+r)} \left(\frac{R}{r}\right)^2 &\, \mbox{if} \, n_i=2. \\
%    \end{cases}
%    \end{eqnarray*}
%   \begin{equation*}
%    \begin{dcases}
%       \frac{\widetilde{V}_{i}(R)}{\widetilde{V}_{i}(r)} \gtrsim \left(\frac{R}{r}\right)^{n_i}, &\, \mbox{if} \, n_i>2, \\
%       \frac{\widetilde{V}_{i}(R)}{\widetilde{V}_{i}(r)} \sim \frac{\log^2(2+R)}{\log^2(2+r)} \left(\frac{R}{r}\right)^2, &\, \mbox{if} \, n_i=2. \\
%       \frac{\widetilde{V}_{i}(R)}{\widetilde{V}_{i}(r)}\sim %\frac{R^{4-2n_i}}{r^{4-2n_i}} \left(\frac{R}{r}\right)^{n_i}=
%         \left(\frac{R}{r}\right)^{4-n_i}, &\, \mbox{if} \, n_i<2.
%    \end{dcases}
%    \end{equation*}
    \[
        \frac{\widetilde{V}_{i}(R)}{\widetilde{V}_{i}(r)} \gtrsim \left(\frac{R}{r}\right)^{n_i},
        \quad \mbox{if} \, n_i>2,
    \]
    and
    \[
       \frac{\widetilde{V}_{i}(R)}{\widetilde{V}_{i}(r)} \sim
        \begin{dcases}
        \frac{\log^2(2+R)}{\log^2(2+r)} \left(\frac{R}{r}\right)^2,
        & \, \mbox{if} \, n_i=2, \\
        %\frac{R^{4-2n_i}}{r^{4-2n_i}} \left(\frac{R}{r}\right)^{n_i}=
         \left(\frac{R}{r}\right)^{4-n_i},
         & \, \mbox{if} \, n_i<2.
        \end{dcases}
    \]
%    \[
%       \frac{\widetilde{V}_{i}(R)}{\widetilde{V}_{i}(r)} \sim
%        \frac{\log^2(2+R)}{\log^2(2+r)} \left(\frac{R}{r}\right)^2,
%        \mbox{if} \, n_i=2,
%    \]
%    and\[
%        \frac{\widetilde{V}_{i}(R)}{\widetilde{V}_{i}(r)} \sim
%         \left(\frac{R}{r}\right)^{4-n_i},
%          \mbox{if} \, n_i<2.
%    \]
  %   \begin{equation*}
%     \frac{\widetilde{V}_{i}(R)}{\widetilde{V}_{i}(r)}
%    \begin{cases}
%       \gtrsim \left(\frac{R}{r}\right)^{n_i}, &\, \mbox{if} \, n_i>2, \\
%       \sim \frac{\log^2(2+R)}{\log^2(2+r)} \left(\frac{R}{r}\right)^2, &\, \mbox{if} \, n_i=2. \\
%       \sim %\frac{R^{4-2n_i}}{r^{4-2n_i}} \left(\frac{R}{r}\right)^{n_i}=
%         \left(\frac{R}{r}\right)^{4-n_i}, &\, \mbox{if} \, n_i<2.
%    \end{cases}
%    \end{equation*}
  \item[\rm(ii)]
  For any $1 \le r \le R < \infty$, it holds that
    \begin{equation*}%\label{eq:R2/measure-R}
        %\frac{\widetilde{V}_{i}(R)}{\widetilde{V}_{i}(r)} \gtrsim \left(\frac{R}{r}\right)^2.
        \frac{R^2}{\widetilde{V}_{i}(R)} \lesssim \frac{r^2}{\widetilde{V}_{i}(r)}.
    \end{equation*}
  In particular, %when $r=1$,
  one has $R^2 / \widetilde{V}_i(R) \lesssim 1$ for all $R \ge 1$.
\end{enumerate}
\end{lem}
\begin{proof}
  %  Using Assumption \ref{asp-1}, Remark \ref{rem:change-r}, Lemma \ref{lem:phi-r-equal} and Lemma \ref{lem:V-tildeV}, we can get that
%    \begin{eqnarray*}
%    \frac{\widetilde{V}_{i}(R)}{\widetilde{V}_{i}(r)}
%     \sim \frac{\phi_i^2(R) V_{i}(R)}{\phi_i^2(r) V_{i}(r)}
%    \begin{cases}
%        \gtrsim \left(\frac{R}{r}\right)^{n_i} &\, \mbox{if} \, n_i>2, \\
%        \sim \frac{R^{4-2n_i}}{r^{4-2n_i}} \left(\frac{R}{r}\right)^{n_i}=
%         \left(\frac{R}{r}\right)^{4-n_i} &\, \mbox{if} \, n_i<2, \\
%        \sim \frac{\log^2(2+R)}{\log^2(2+r)} \left(\frac{R}{r}\right)^2 &\, \mbox{if} \, n_i=2. \\
%    \end{cases}
%\end{eqnarray*}
%    Therefore, we have
%    \begin{align*}
%    \frac{\widetilde{V}_{i}(R)}{\widetilde{V}_{i}(r)}
%    & \gtrsim\begin{cases}
%         \left(\frac{R}{r}\right)^{n_i} &\, \mbox{if} \, n_i>2, \\
%         \left(\frac{R}{r}\right)^{4-n_i} &\, \mbox{if} \, n_i<2, \\
%         \left(\frac{R}{r}\right)^2 &\, \mbox{if} \, n_i=2,
%    \end{cases}\\
%    &\gtrsim \left(\frac{R}{r}\right)^2,
%    \end{align*}
%    which deduces that
%   \[
%        \frac{\widetilde{V}_i(R)}{\widetilde{V}_i(1)} \gtrsim R^2
%   \]
%    and
%    $$\frac{R^2}{\widetilde{V}_i(R)} \lesssim 1.$$
%    This completes the proof.
%    \\--------------------\\
    It follows from Lemma \ref{lem:V-tildeV} that
    \[
    \frac{\widetilde{V}_{i}(R)}{\widetilde{V}_{i}(r)}
     \sim \frac{\phi_i^2(R) V_{i}(R)}{\phi_i^2(r) V_{i}(r)},
     \]
     which together with Lemmas \ref{lem:volume-growth} and \ref{lem:phi-r-equal} implies that (i) and (ii) hold.
\end{proof}

\begin{lem}\label{V0-general-equal}
%For any $r> \min\{c_0,1 \}$, we have
For any $r \ge 1$, we have
$$\widetilde{V}_0(r) \sim r^2\log^2(2+r).$$
\end{lem}
%\hl{if it holds, we have}
%
%    for any $r>0$,
%    $$\widetilde{V}_i(r) \gtrsim r^2\log^2(2+r)$$
%    and when $r>1$, the following equivalent characterization of $\phi(r)$ is given by $r > 1$,
%   \begin{eqnarray*}
%      V_i(r)
%       \sim \frac{1}{\phi^2(r)} \widetilde{V}_i(r)
%       \gtrsim
%   \begin{cases}
%         r^2\log^2(2+r) &\, \mbox{if} \, n_i>2, \\
%         r^2 &\, \mbox{if} \, n_i=2 \\
%         r^{2n_i-4}r^2\log^2(2+r)=r^{2n_i-2}\log^2(2+r)  &\, \mbox{if} \, 1\leq n_i<2, \mbox{then} \, 0\leq 2n_i-2 <2 \\
%    \end{cases}
%   \end{eqnarray*}
%   Since $n_i<2$, we have $n_i > 2n_i-2$. Therefore, above is reasonable.
%   \\When $r\leq 1$, $\phi(r) \sim 1$, therefore $$V_i(r) \sim \widetilde{V}_i(r) \sim r^2.$$
\begin{proof}
%When $r$ is separated from $0$, i.e., $r \gtrsim 1$,
By Lemma \ref{lem:tilde-V}, one has for any $r \ge 1$,
%\begin{eqnarray*}
%  \begin{cases}
%     \widetilde{V}_i(r) \gtrsim r^{n_i}, &\, \mbox{if} \, n_i>2, \\
%     \widetilde{V}_i(r) \sim  r^2 \log^2(2+r), &\, \mbox{if} \, n_i=2, \\
%     \widetilde{V}_i(r) \sim r^{4-n_i}, &\, \mbox{if} \, n_i<2,
%  \end{cases}
%\end{eqnarray*}
\[
     \widetilde{V}_i(r) \gtrsim r^{n_i}, \quad \mbox{if} \, n_i>2,
\]
and
\[
    \widetilde{V}_i(r) \sim
    \begin{cases}
       r^2 \log^2(2+r), & \, \mbox{if} \, n_i=2, \\
       r^{4-n_i}, & \, \mbox{if} \, n_i<2,
  \end{cases}
\]
which yields that $\widetilde{V}_0(r) \sim r^2 \log^2(2+r)$.
\end{proof}

Clearly, above Remark \ref{rem:thr} (ii) and Lemma \ref{V0-general-equal} lead to the following conclusion.

%\begin{cor}
%It can be seen that $\widetilde{V}_0$ is doubling, i.e., there exists a constant $C >0$, such that for any $r>0$
%\[
%    \widetilde{V}_0(2r) \leq C \widetilde{V}_0(r).
%\]
%{\color{blue} Doubling property can be seen by $\widetilde{V}_i(r)$ is doubling for any $i$; see \cite[P. 1945]{gri-sal09}.}
%\end{cor}
%\begin{proof}
%    When $1\leq r $, by Lemma \ref{V0-general-equal}, we get the desired estimate.
%
%    When $r  < 1$, it holds $2r < 2$. By Lemma \ref{rem:gs-4.14}, we get the desired estimate.
%\end{proof}

\begin{cor}\label{cor:vi<v0}
  Let $1 \leq i \leq \ell$. If $n_i \leq 2$, then we have $V_i(r) \lesssim \widetilde{V}_0(r)$ for any $r>0$.
\end{cor}

%Besides, Lemmas \ref{lem:tilde-V} and \ref{V0-general-equal} allows us to obtain the following corollary.
%We need the following inequality between $\widetilde{V}_0$ and $\widetilde{V}_{i}$, which is a corollary of Lemmas \ref{lem:tilde-V} and \ref{V0-general-equal}.
 Lemmas \ref{lem:tilde-V} and \ref{V0-general-equal} allows us to conclude that the function $\widetilde{V}_0$ grows more slowly than all functions $\widetilde{V}_i$.
%
%Using Lemmas \ref{lem:tilde-V} and \ref{V0-general-equal}, it follows that the function $\widetilde{V}_0$ grows slower than all functions $\widetilde{V}_i$.
\begin{cor}\label{lem:4-terms}
    Let $1 \leq i \leq \ell$. For any  $0 < r \le R<\infty$, we have
    \begin{equation} \label{eq:4-terms}
         \frac{\widetilde{V}_0(R)}{\widetilde{V}_0(r)} \lesssim \frac{\widetilde{V}_i(R)}{\widetilde{V}_i(r)}.
    \end{equation}
\end{cor}
\begin{proof}
  Set $$m_0 = \max\lf\{\exp \frac{2}{|n_i-2|}: 1\leq i \leq \ell, n_i \neq 2\r\}.$$

  When $r \leq m_0$, it follows from Remark \ref{rem:thr} (ii) that
  $ \widetilde{V}_0(r) \sim \widetilde{V}_i(r),$ which together with the fact
  $
    \widetilde{V}_0(R) \le \widetilde{V}_i(R)
  $ yields \eqref{eq:4-terms} holds.

  When $r \ge m_0$, %it follows from Lemmas \ref{lem:tilde-V} and \ref{V0-general-equal} that/ one has via xxx that/ using xxx, one finds
  by Lemmas \ref{lem:tilde-V} and \ref{V0-general-equal}, we have
%  \[
%    \frac{\widetilde{V}_0(R)}{\widetilde{V}_0(r)} \sim \frac{R^2\log^2(2+R)}{r^2\log^2(2+r)}
%    , \quad
%    \quad
%    \frac{\widetilde{V}_{i}(R)}{\widetilde{V}_{i}(r)}
%    \gtrsim
%    \begin{dcases}
%          \frac{R^2\log^2(2+R)}{r^2\log^2(2+r)}, &\, \mbox{if} \, n_i=2, \\
%             \left(\frac{R}{r}\right)^{2+|n_i-2|}, &\, \mbox{if} \, n_i \neq 2.
%    \end{dcases}
%  \]
  \begin{equation} \label{eq:vRvr}
    \frac{\widetilde{V}_{i}(R)}{\widetilde{V}_{i}(r)}
    \gtrsim
    \begin{dcases}
          \frac{R^2\log^2(2+R)}{r^2\log^2(2+r)}, &\, \mbox{if} \, n_i=2, \\
             \left(\frac{R}{r}\right)^{2+|n_i-2|}, &\, \mbox{if} \, n_i \neq 2,
    \end{dcases}
  \end{equation}
  and
  \begin{equation} \label{eq:v0Rv0r}
    \frac{\widetilde{V}_0(R)}{\widetilde{V}_0(r)} \lesssim  \frac{R^2\log^2(2+R)}{r^2\log^2(2+r)}.
  \end{equation}
  When $n_i=2$, it follows from \eqref{eq:vRvr} and \eqref{eq:v0Rv0r} that
  \[
    \frac{\widetilde{V}_0(R)}{\widetilde{V}_0(r)}
    \lesssim  \frac{R^2\log^2(2+R)}{r^2\log^2(2+r)}
    \lesssim \frac{\widetilde{V}_{i}(R)}{\widetilde{V}_{i}(r)}.
  \]
  %which implies that \eqref{eq:4-terms} holds.
  Otherwise $n_i \neq 2$,
  by the simple fact that
  the function
%  $$g: (0, \infty) \mapsto \frac{t^{|n_i-2|}}{\log^2(2+t)}$$
  $$g(t) := \frac{t^{|n_i-2|}}{\log^2(2+t)}$$
  is increasing on $(m_0, \infty)$,
  one has
  \[
    \frac{r^{|n_i-2|}}{\log^2(2+r)}
    \leq \frac{R^{|n_i-2|}}{\log^2(2+R)},
  \]
  which together with \eqref{eq:vRvr} and \eqref{eq:v0Rv0r} implies that
  \[
    \frac{\widetilde{V}_0(R)}{\widetilde{V}_0(r)}
    \lesssim  \frac{R^2\log^2(2+R)}{r^2\log^2(2+r)}
    \lesssim  \left(\frac{R}{r}\right)^{2+|n_i-2|}
    \lesssim \frac{\widetilde{V}_{i}(R)}{\widetilde{V}_{i}(r)}.
  \]
  This completes the proof.
\end{proof}

In our discussion we will also need the following result.
\begin{lem} \label{lem:phi/v-e}
Let $1 \leq i \leq \ell$ and $\az >0$.  It holds that for any $r >0 $ and $s \ge 1$,
\begin{equation*} %\label{eq:phi/v-e}
   %     \frac{\phi_{i}(r)}{\widetilde{V}_{i}(t)} e^{-c\frac{r^\az}{t^\az}}
%        \lesssim_{\az} \frac{1}{t^2}.
%        \qquad
        \frac{\phi_{i}(r)}{\widetilde{V}_{i}(s)} \exp\lf\{ -c \lf(\frac{r}{s}\r)^{\az}\r\}
        \lesssim_{\az} \frac{1}{s^2}.
\end{equation*}
\end{lem}
\begin{proof}
It follows from Lemma \ref{lem:V-tildeV} that
%\[
%    \frac{\phi_{i}(|x|)}{\sqrt{\widetilde{V}_{i}(\sqrt t)}} e^{ -c\frac{|x|^2}{t}}
%    \leq \frac{\phi_{i}(|x|)}{\phi_i(\sqrt t)} \frac{\phi_i(\sqrt t)}{\sqrt{\widetilde{V}_{i}(\sqrt t)}} e^{ -c\frac{|x|^2}{t}}
%    \lesssim \frac{1}{\sqrt{V_i(\sqrt t)}}.
%\]
\[
    \frac{\phi_{i}(r)}{\widetilde{V}_{i}(s)^{1/2}} \exp\lf\{ -c\lf(\frac{r}{s}\r)^\az\r\}
    = \frac{\phi_{i}(r)}{\phi_i(s)} \frac{\phi_i(s)}{\widetilde{V}_{i}(s)^{1/2}} \exp\lf\{ -c\lf(\frac{r}{s}\r)^\az\r\}
    \lesssim_\az \frac{1}{V_i(s)^{1/2}},
\]
%by the trivial inequality $r^{\alpha}e^{-r} \lesssim_\az 1$ for any $\alpha >0$ and all $r>0$,
%and
by the fact that $\phi_{i}(r)/\phi_i(s)$ has at most polynomial growth w.r.t. $r/s$ from Lemma \ref{lem:phi-r-equal}.
Then using Lemmas \ref{lem:volume-growth} and \ref{lem:tilde-V}, it holds
\begin{eqnarray*}
    \frac{\phi_{i}(r)}{\widetilde{V}_{i}(s)} \exp\lf\{ -c\lf(\frac{r}{s}\r)^\az\r\}
     \lesssim_\az \frac{1}{\widetilde{V}_i(s)^{1/2}} \frac{1}{V_i(s)^{1/2}}
  %  \lesssim
%        \begin{cases}
%        \frac{1}{t^{n_k/4} t^{n_k/4}},  &\, \mbox{if} \, n_k >2, \\
%          \frac{1}{\log(2 +\sqrt t) t^{1/2}  t^{1/2}}, &\, \mbox{if} \, n_k =2, \\
%        \frac{1}{t^{1-n_k/4} t^{n_k/4} }, &\, \mbox{if} \, n_k <2,
%        \end{cases}
     \lesssim_\az
        \begin{dcases}
        \frac{1}{s^{n_i}},  &\, \mbox{if} \, n_i >2, \\
          \frac{1}{s^2\log(2 +s) }, &\, \mbox{if} \, n_i =2, \\
        \frac{1}{s^2}, &\, \mbox{if} \, n_i <2,
        \end{dcases}
    %&& \lesssim \frac{1}{t},
\end{eqnarray*}
which means that
\[
    \frac{\phi_{i}(r)}{\widetilde{V}_{i}(s)} \exp\lf\{ -c\lf(\frac{r}{s}\r)^\az\r\}
    \lesssim_\az \frac{1}{s^2}.
\]
This completes the proof.
\end{proof}

\subsection{Estimates on the auxiliary function $\widetilde{H}$} \hskip\parindent
Recall that
$$\widetilde{H}(x,t)=\min\left\{1, \, \frac{|x|^2}{\widetilde{V}_{i_x}(|x|)}+\left(\int_{|x|^2}^t\frac{\,d s}{\widetilde{V}_{i_x}(\sqrt s)}\right)_+\right\}.$$
We have the following estimates for $\widetilde{H}(x,t)$.
\begin{lem} \label{H-equal-general}
%By the definition of $\widetilde{H}(x,t)$, it holds
%\[
%    \widetilde{H}(x,t) \leq 1.
%\]
%More precisely,
Let $1\leq i \le \ell$. For any $x \in E_i$ and $t>0$, the following statements hold.
% Notice that when $x \in M_i$, the $i_x$ in $\widetilde{H}(x,t)$ may not be $i$
\begin{itemize}
  \item [(i)] If $n_{i} \neq 2$, then
        $$
        \widetilde{H}(x,t)
        \sim \frac{|x|^2}{\widetilde{V}_{i}(|x|)}
        \sim
        \begin{dcases}
            \frac{|x|^2}{{V}_{i}(|x|)}, & \, \mbox{if}\, n_{i} > 2,\\%x, y \in E_0\\
            |x|^{n_i -2} , & \, \mbox{if}\, n_{i} < 2.%x, y \in E_i , \, 1\leq i \leq \ell.
        \end{dcases}
        $$
  %\item [(i)] If $n_{i}>2$, then
%        $$\widetilde{H}(x,t) \sim \frac{|x|^2}{\widetilde{V}_{i}(|x|)} \sim \frac{|x|^2}{{V}_{i}(|x|)} .$$

  \item [(ii)] If $n_{i}=2$, then
        $$\widetilde{H}(x,t) \lesssim \frac{1}{\log(2+|x|)}.$$

  %\item [(iii)] If $n_{i}<2$, then
%        $$\widetilde{H}(x,t) \sim \frac{|x|^2}{\widetilde{V}_{i}(|x|)} \sim |x|^{{n_i}-2}.$$
\end{itemize}
\end{lem}

\begin{proof}
   If $n_{i} \neq 2$, then
we deduce from Lemma  \ref{lem:tilde-V} (i) that
\begin{eqnarray*}
      \widetilde{H}(x,t)
       && \leq \frac{|x|^2}{\widetilde{V}_{i}(|x|)} +  \int_{|x|^2}^{\infty}\frac{\,d s}{\widetilde{V}_{i}(\sqrt s)}       %& \lesssim \frac{|x|^2}{{\widetilde{V}}_{i}(|x|)} + \int_{|x|^2}^{\infty} \frac{\,d s}{\widetilde{V}_{i}(\sqrt s)} \\
        \leq \frac{|x|^2}{{\widetilde{V}}_{i}(|x|)}
        + \frac{1}{\widetilde{V}_{i}(|x|)} \int_{|x|^2}^{\infty} \frac{\widetilde{V}_{i}(|x|)}{\widetilde{V}_{i}(\sqrt s)}  \,ds \\
       && \lesssim \frac{|x|^2}{{\widetilde{V}}_{i}(|x|)}
          + \frac{1}{\widetilde{V}_{i}(|x|)} \int_{|x|^2}^{\infty} \left(\frac{|x|}{\sqrt s}\right)^{2+ |n_i-2|} \,d s
        \lesssim \frac{|x|^2}{{\widetilde{V}}_{i}(|x|)}.
\end{eqnarray*}
Combining this with $|x|^2 / \widetilde{V}_i(|x|) \lesssim 1$ (cf. Lemma \ref{lem:tilde-V} (ii)), one gets%finds
%$\widetilde{H}(x,t)  \sim |x|^2 / \widetilde{V}_{i}(|x|).$
$$\widetilde{H}(x,t)  \sim \frac{|x|^2}{\widetilde{V}_{i}(|x|)}.$$
Then, it follows from Remark \ref{rem:thr} (i) and Lemma \ref{lem:tilde-V} (i) that the claim (i) holds.

     If $n_i=2$,
     %then one concludes from Lemma \ref{lem:tilde-V} (i) that
     then we use Lemma \ref{lem:tilde-V} (i) to obtain
     %then by Lemma \ref{lem:tilde-V} (i) that %, one concludes that
    \begin{eqnarray*}
        \widetilde{H}(x,t)
       && \leq   \frac{|x|^2}{\widetilde{V}_{i}(|x|)} + \int_{|x|^2}^{\infty}\frac{\,d s}{\widetilde{V}_{i}(\sqrt s)}
           \lesssim  \frac{1}{\log^2(2+|x|)}  + \int_{|x|^2}^\infty \frac{1}{s \log^2(2+ \sqrt s)} \, ds \\
        && \lesssim \frac{1}{\log^2(2+|x|)} + \int_{|x|}^\infty \frac{1}{(2+ s) \log^2(2 + s)} \, ds
          \lesssim \frac{1}{\log(2+|x|)},
    \end{eqnarray*}
   % \[
%      \int_{|x|^2}^\infty \frac{1}{s \log^2(2+ \sqrt s)} \, ds
%      \lesssim \int_{|x|}^\infty \frac{1}{s \log^2(2 +  s)} \, ds
%      \lesssim \int_{|x|}^\infty \frac{1}{(2+s) \log^2(2 +  s)} \, ds
%     \lesssim \left.\frac{1}{\log(2+s)}\right|_\infty^{|x|}
%    \]
    which completes the proof of Lemma \ref{H-equal-general}.
\end{proof}
From the estimates of $\phi_i$ and $\widetilde{H}$,
i.e., Lemmas \ref{lem:phi-r-equal} and \ref{H-equal-general}, we have the following:
\begin{cor}\label{cor:phi-H}
Let $1 \leq i \leq \ell $. For any $x \in E_i$, it holds
\[
    \phi_{i}(|x|) \widetilde{H}(x,t) \lesssim 1.
\]
\end{cor}

\section{Further estimates for the heat kernel}\hskip\parindent
%{\color{blue}In this section, we shall introduce the heat kernel upper bound
%from \cite[Subsection 6.3]{gri-sal09},
%%and then give some general formulas and estimates for computing the various terms in the upper bound.
%and then give some estimates for the various terms in this upper bound. %of the heat kernel.
%Besides, with the help of the heat kernel upper bound from \cite[Corollary 4.16]{gri-sal09}, we will give the ultracontractivity of the heat semigroup.
%In the end, %of this section,
%%by the small time heat kernel Gaussian upper bound,
%we shall show that the small time part of the Riesz transform is bounded on $L^p(M)$ for $1 < p<2$.
%
%We will deduce the estimates of the time derivative of the heat kernel.
%Besides, the $L^1$-estimate of the space derivative of the heat kernel is discussed.}
In this section, using the priori heat kernel estimates in Subsection \ref{sec:pri-es} and the results in Section \ref{sec:vol-aux}, we can  give some further estimates on the heat kernel and  its time derivative.
Besides, the $L^1$-estimate of the space derivative of the heat kernel is discussed.
\subsection{Heat kernel upper bounds}\label{sec:heat-ker-est} \hskip\parindent
 Theorem \ref{thm:sim-heat-es} allows us to obtain the following heat kernel upper bounds.
\begin{lem}\label{lem:es-away}
    Let $1 \leq i \leq \ell$ and $\beta \ge 16$. For any $\kz \geq 1$, it holds
    \[
        h_t(x,y) \lesssim_\beta \frac{1}{V_i(y, \sqrt t)} %\exp\left(-c\frac{d(x,y)^2}{t}\right),
        e^{-c\frac{d(x,y)^2}{t}},
        \quad \forall  x, y \in E_i, \, \dist(x,E_0) + \dist(y,E_0) \ge 2\kz, \, 0<t \leq \beta \kz^2.
    \]
   % where $F_i' = \{ x \in E_i; \dist(x, E_0) \leq R_i\}$.
\end{lem}
\begin{proof}
    Let us first estimate $h_t(x,y)$.
    Notice that  $V(y, \sqrt t)=V_i(y, \sqrt t)$ whenever $B_i(y, \sqrt t) \subset E_i$, otherwise $V(y, \sqrt t) \sim V_i(y, \sqrt t)$ for all  $y \in E_i$ and $0<t  < 1$; see \cite[p. 1945]{gri-sal09}. Hence according to Theorem \ref{thm:sim-heat-es} (i), we may assume that $t \geq 1$.
%
%    If $t \le 1$, by \cite[Corollary 4.7]{gri-sal09} and the standard track of doubling property (cf. \eqref{eq:change-centre}), we have
%    \[
%        h_t(x, y)
%        \lesssim \frac{1}{\sqrt{V(x, \sqrt t) V(y, \sqrt t)}} e^{-c \frac{d(x,y)^2}{t}}
%        \lesssim \frac{1}{V(y, \sqrt t) } e^{-c \frac{d(x,y)^2}{t}}.
%        %\lesssim \frac{1}{V_i(y, \sqrt t) } e^{-c \frac{d(x,y)^2}{t}}
%    \]
%    %where in the last inequality, we used the fact that $B(y,t) \subset E_i$.
%    Since $t \leq 1$, $y \in E_i \setminus F_i$, it holds $B(y,t) \subset E_i$  and hence
%    \[
%        h_t(x, y) \lesssim \frac{1}{V_i(y, \sqrt t) } e^{-c \frac{d(x,y)^2}{t}}.
%    \]
%
    %where $x,y \in E_i$ and $1 \leq t \leq \beta \kz^2$.
    Without loss of generality we may assume that $\dist(y,E_0) \ge \kz$.
    By Theorem \ref{thm:sim-heat-es} (ii) and Doob's transform,
    it holds that %it holds
    \begin{eqnarray*}
    h_{t}(x,y)
     \lesssim
               \frac{\phi_{i}(|x|) \phi_{i}(|y|) \widetilde{H}(x,t)\widetilde{H}(y,t)}{\widetilde{V}_0(\sqrt t)}   e^{-c\frac{|x|^2+|y|^2}{t}}
                        +   \frac{\phi_{i}(|x|) \phi_{i}(|y|)}{\widetilde{V}_{i}(x,\sqrt t)} e^{-c\frac{d(x,y)^2}{t}}
     =: \mathrm{I} + \mathrm{II}.
    \end{eqnarray*}
%    We begin with the estimate of the term $\mathrm{I}$.
    Aim now at the term $\mathrm{I}$.
    When $n_i > 2$,  the facts $\phi_i(|x|) \widetilde{H}(x,t) \lesssim 1$ and $\phi_i(|y|) \sim 1$
    (cf. Corollary \ref{cor:phi-H} and Lemma \ref{lem:phi-r-equal}) yield that
    \[
        \mathrm{I} \lesssim \frac{\widetilde{H}(y,t)}{\widetilde{V}_0(\sqrt t)}   e^{-c\frac{|x|^2+|y|^2}{t}}.
    \]
    By Lemma \ref{H-equal-general}, the fact $|y| \ge \dist(y, E_0) \ge \kz \gtrsim_\beta \sqrt t$
    and Lemma \ref{lem:volume-growth} (i), one obtains
    \[
    \widetilde{H}(y,t) \lesssim \frac{|y|^2}{V_i(|y|)}
         \lesssim_\beta \frac{t}{V_i(\sqrt t)},
    \]
    which together with Lemma \ref{V0-general-equal}  and \eqref{eq:change-centre} gives that
    \[
        \mathrm{I} \lesssim_\beta \frac{1}{t\log^2(2+\sqrt t)} \frac{t}{V_i(\sqrt t)} e^{-c\frac{|x|^2+|y|^2}{t}}
        \lesssim_\beta \frac{1}{V_i(\sqrt t)} e^{-c\frac{|x|^2+|y|^2}{t}}
        \lesssim_\beta \frac{1}{V_i(y, \sqrt t)} e^{-c\frac{d(x,y)^2}{t}}.
    \]
    %where in the last inequality we used \eqref{eq:change-centre}.
    When $n_i \leq 2$, one easily deduces from Corollary \ref{cor:vi<v0}, Corollary \ref{cor:phi-H}  and \eqref{eq:change-centre} that
    \[
        \mathrm{I}
         \lesssim  \frac{1}{V_i(\sqrt t)}  e^{-c\frac{|x|^2+|y|^2}{t}}
         \lesssim \frac{1}{V_{i}(y,\sqrt t)} e^{-c\frac{d(x,y)^2}{t}}.
    \]
    Next, for the term $\mathrm{II}$,
    the standard trick of doubling property (cf. \eqref{eq:change-centre}) together with Lemma \ref{lem:V-tildeV} implies
    \[
        \mathrm{II}
        \lesssim  \frac{\phi_{i}(|x|) \phi_{i}(|y|)}{\widetilde{V}_{i}(x,\sqrt t)^{1/2} \widetilde{V}_{i}(y,\sqrt t)^{1/2}} e^{-c\frac{d(x,y)^2}{t}}
       % \lesssim \frac{1}{{V}_{i}(x,\sqrt t)^{1/2} {V}_{i}(y,\sqrt t)^{1/2}} e^{-c\frac{d(x,y)^2}{t}}
        \lesssim \frac{1}{{V}_{i}(y,\sqrt t)} e^{-c\frac{d(x,y)^2}{t}}.
    \]
    Combining the above estimates of the terms $\mathrm{I}$ and $\mathrm{II}$, one has %concludes that
    \[
        h_{t}(x,y)
         \lesssim_\beta \frac{1}{V_{i}(y,\sqrt t)} e^{-c\frac{d(x,y)^2}{t}}.
        % \quad     \forall x,y \in E_i, \, \dist(x,E_0) + \dist(y,E_0) \ge2 \kz , \, 0 <t <\beta \kz^2.
    \]
    This completes the proof.
\end{proof}

\subsection{Estimate on the time derivative of the heat kernel} \label{sec:time-der-heat-es} \hskip\parindent
For each $1 \leq i \leq \ell$, recall that %we let $F_i^{(r)}$ be the set $\{x\in E_i:\,\dist(x, E_0)\le 2r\}$.
$F_i^{(r)} = \{x\in E_i:\,\dist(x, E_0)\le 2r\}$, $r \ge 1$.
If $\gamma$ is large enough, saying $\gamma > 100 \ell$, then we set in the sequel,
\[
    R_i:=R_i(\gamma) \geq 1
    \text{ such that }
    \mu(F_i^{(R_i)}) = \gamma.
\]
%for each $i$.
%Since we assumed $\mu(F_*) \leq 40 \ell$,
%%in \S \ref{sec:not-basic},
%%(cf. \S \ref{sec:not-basic}),  %and $\gamma > 100 \ell$,
%it follows that $R_i(\gamma) \geq 50$.
%For simplicity, we always write $F_i^{(R_i)}$ as $F_i$.
%
%Note that
It follows from the convention $\mu(F_*) \leq 40 \ell$
that $R_i(\gamma) \geq 50$.
% ��ʦ�� ����������ԭ���ǡ� ������2.1���ж�$\mu(F_*)$�Ķ��������޸ģ� Ȼ���ᵼ�µõ�����$R_i \ge 50$.
%By Lemma \ref{lem:volume-growth} (ii), for each $1 \leq i, j \leq \ell$, it holds that
For each $1 \leq i, j \leq \ell$, it holds by Lemma \ref{lem:volume-growth} (ii) that
\begin{equation}\label{eq:fifj-con}
    V_i(R_i) \sim \mu(F^{(R_i)}_i) =  \mu(F^{(R_j)}_j) \sim V_j(R_j), \quad
    \forall \gamma \gg 1. %\, 1 \leq i, j \leq \ell.
\end{equation}
See Figure 1.
\begin{figure}[ht]
\centerline{ \epsfig{file=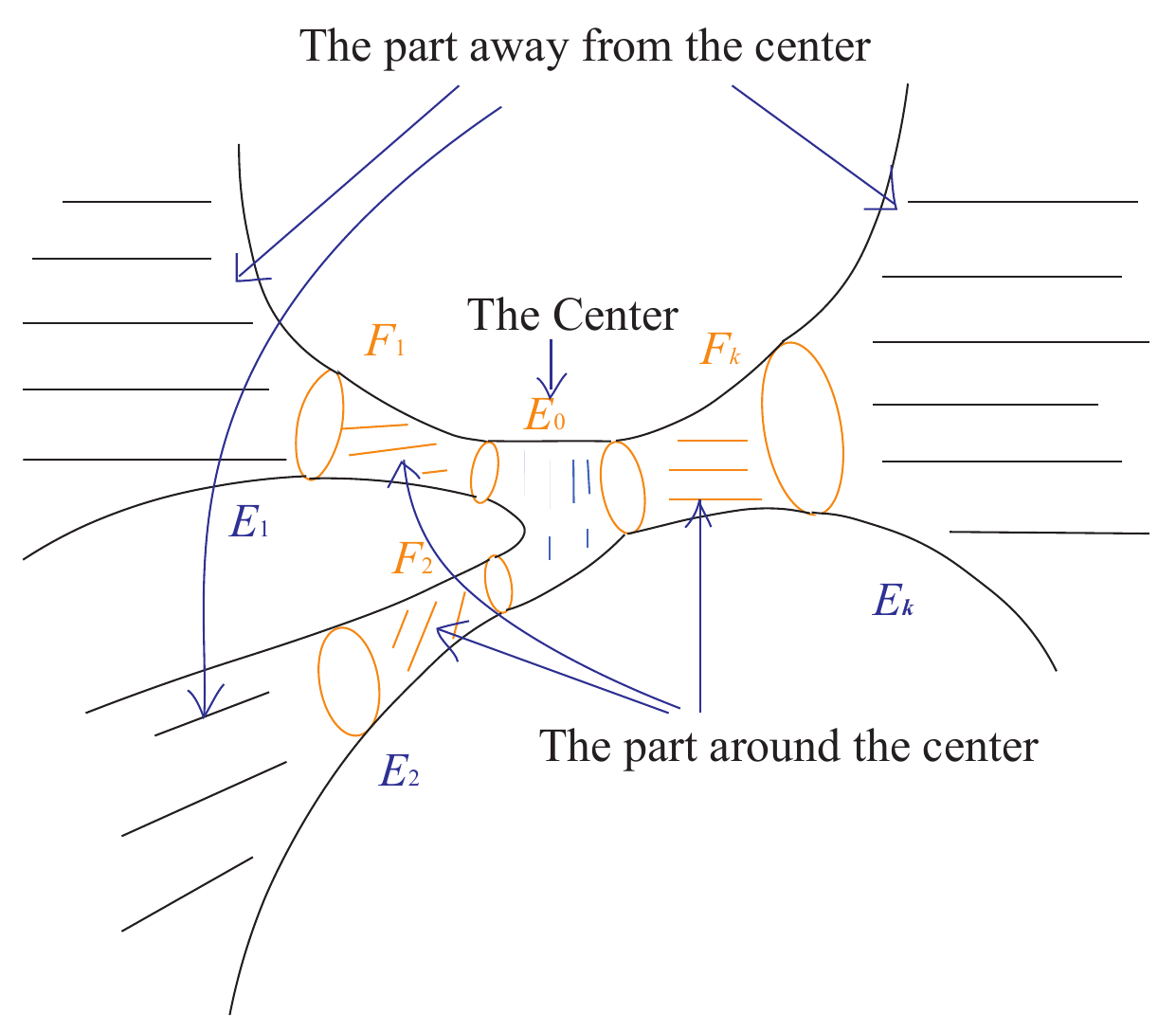, scale=0.6}
              }
             \caption{{Partition of the manifold}}
\end{figure}
% ��ʦ�� ����ͼ��������ԭ�����¡� ���ǲ�ȷ����The Center�� �е���ĸ"C"�Ƿ�ҪСд��
% Ȼ����֪������$F_i$�Ƿ���Ҫ�����ϽǱ�$(R_i)$. ��ʦ�� �Ҳ�֪������ͼ������ʲô�������ģ� ���Ծ�û���޸ġ�
%
%In this subsection, we will estimate the time derivative of the heat kernel. This will play a key
%role in the study of mapping properties of $t \L e^{-t\L}$. %(see Proposition \ref{mapping-time-heat-parabolic-1} in subsection \ref{sec:der-mapping-pr}).

In general setting, we take  the heat kernel $k_t(x,y)$ to be associated with a non-negative self-adjoint operator on a Riemannian
manifold. %,  but make no assumption of completeness or bounded geometry.
%We shall begin with
Let us give the following estimates for the time derivative of the heat kernel $k_t(x,y)$ %in general setting
obtained by Davies \cite[Theorem 4]{davies97}, with a minor modification of a symbol. %notation/symbol. %/character.

\begin{thm} \label{davies}
Suppose that $0 < \delta <1$, $0 < \epsilon < 1/8$, $x,y\in M$ and $t>0$.
Let $a,b,A$ be positive constants such that $0 <  A \leq \sqrt{ab}$ and
$$k_{(1-\delta) t}(x,x)\le a,\quad k_{(1-\delta)t}(y,y)\le b,\quad k_{s}(x,y)\le A$$
for all $ (1-\delta)t <s < (1+\delta)t$. Then for any $m\in\cn$, it holds
$$\left|\frac{\partial^m}{\partial t^m}k_t(x,y)\right|\le \frac{m!}{(\epsilon\delta t)^m} a^{\frac32 \epsilon} b^{\frac32 \epsilon} A^{1-3\epsilon}. $$
\end{thm}

The heat kernel upper bound (Lemma \ref{lem:es-away}) and the above theorem yields the following
estimates for the time derivative of the heat kernel on manifold $M$.

\begin{lem} \label{lem:es-away-time-der} %for any {\color{blue}$x \in E_i \setminus F^{(\kz)}_i$, $y \in E_i \setminus F^{(\kz)}_i$, $0 < t \leq \beta \kappa^2$}, it holds
   Let $1 \leq i \leq \ell$ and $\beta \ge 16$. For any $\kz \geq 1$, it holds
    \[
        |t\partial_th_t(x,y)| \lesssim_\beta \frac{1}{ V_i(y, \sqrt t)} %        \exp\left(-c\frac{d(x,y)^2}{t}\right),
        e^{-c\frac{d(x,y)^2}{t}},
        \quad \forall  x,y \in E_i \setminus F^{(\kz)}_i, \,
        0<t \leq \beta \kz^2.
    \]
    \end{lem}

\begin{proof}
    %Similarly,
    %Moreover,
    From Lemma \ref{lem:es-away}, we obtain that for any $x,y \in E_i \setminus F^{(\kappa)}_i$, $0 < t \leq \beta \kz^2$ and $t/2 < s < 3t/2$,
    $$
    h_{t/2}(x,x) \lesssim_\beta \frac{1}{V_{i}(x,\sqrt t)},
    \quad
    h_{t/2}(y,y) \lesssim_\beta \frac{1}{V_{i}(y,\sqrt t)},
    \quad
    h_s(x,y)
    \lesssim_\beta \frac{1}{V_{i}(y,\sqrt t)} e^{-c\frac{d(x,y)^2}{t}},
    $$
 %
%
%
%    \[
%    h_s(x,y)
%     +
%     h_{t/2}(x,x)
%     +
%    h_{t/2}(y,y)
%    \lesssim_\beta \frac{1}{V_{i}(y,\sqrt t)} e^{-c\frac{d(x,y)^2}{t}}.
%    \]
    which together with Theorem \ref{davies} implies
\begin{equation*}
    |t\partial_t h_t(x,y)| \lesssim_\beta  \frac{1}{V_{i}(y,\sqrt t)}%\exp\left(-\frac{d(x,y)^2}{ct}\right),
    e^{-c\frac{d(x,y)^2}{t}}.
    % \quad \forall x,y \in E_i \setminus F^{(\kz)}_i , \, 0<t \leq \beta \kz^2.
\end{equation*}
This completes the proof.
\end{proof}
Using the heat kernel upper bound (Theorem \ref{thm:sim-heat-es}) and Theorem \ref{davies},
we obtain the following estimates for the time derivative of the heat kernel on weighted manifold $\widetilde{M}$,
which will be used to deduce the mapping properties of the time derivative of the heat semigroup in Subsection \ref{sec:der-mapping-pr}.
\begin{prop}\label{time-heat-parabolic-2}
Let $1 \leq i , j \leq \ell$ and $0 < \epsilon < 1/8$. If $x\in  E_j\setminus F^{(R_j)}_j $, $y\in F^{(R_i)}_i$ and $t \ge R_j^2$,
then the following statements hold true.
\begin{enumerate}
\item[\rm{(i)}]
    When $n_i \neq 2$, $n_j \neq 2$ and $i=j$, it holds
   \begin{equation*}
   |t\partial_t\tilde h_t(x,y)|
    \lesssim_\epsilon
        \frac{\widetilde{V}_0(R_i)}{\widetilde{V}_0(\sqrt t) \widetilde{V}_i(R_i)}
        \left[1+ \left(\frac{\widetilde{V}_i(R_i)}{\widetilde{V}_0(R_i)}\right)^{\frac{3}{2}\epsilon}
\left(\frac{|y|^2}{\widetilde{V}_i(|y|)}\right)^{\frac{3}{2}\epsilon} \right]  e^{-c\frac{d(x,y)^2}{t}}.
\end{equation*}
\item[\rm{(ii)}]
    When $n_i \neq 2$, $n_j \neq 2$ and $i\neq j$, it holds
   \begin{equation*}
   |t\partial_t\tilde h_t(x,y)|
    \lesssim_\epsilon
       {\frac{\widetilde{V}_0(R_j)}{\widetilde{V}_0(\sqrt t)\widetilde{V}_j(R_j)}}\left(\frac{\widetilde{V}_j(R_j)}{\widetilde{V}_0(R_j)}\right)^{\frac 32\epsilon}\left(\frac{|y|^2}{\widetilde{V}_i(|y|)}+\frac{R_j^2}{\widetilde{V}_i(R_j)}\right)^{1-3\epsilon} e^{-c\frac{|x|^2+|y|^2}{t}}.
\end{equation*}
\item[\rm{(iii)}] When $n_i=2$ and $n_j<2$, it holds
\begin{equation*}
   |t\partial_t\tilde h_t(x,y)|
    \lesssim_\epsilon
     \frac{1}{\widetilde{V}_0(\sqrt t)} e^{-c\frac{|x|^2+|y|^2}{t}}.
\end{equation*}
\item[\rm{(iv)}] When $n_i=2$ and $n_j=2$, it holds
    \begin{equation*}
   |t\partial_t\tilde h_t(x,y)|
    \lesssim_\epsilon
        \frac{1}{\widetilde{V}_0(\sqrt t)} e^{-c\frac{d(x,y)^2}{t}}.
    \end{equation*}
\item[\rm{(v)}] When $n_i=2$ and $n_j>2$, it holds
       % \begin{aligned}
%            & \frac{R_j^2}{t} \frac{\log(2+R_j)}{ \log(2+\sqrt t)\log^{3\epsilon}(2+R_j)} \frac{1}{ \widetilde V_j(R_j)} \left(\frac{\widetilde V_j(R_j)}{R_j^2}\right)^{\frac 32\epsilon} \\
%            & \ \times \left(\frac{1}{\log(2+R_j)}+\frac{1}{\log(2+|y|)}\right)^{1-3\epsilon}e^{-c\frac{|x|^2+|y|^2}{t}},
%        \end{aligned}
\begin{align*}
   |t\partial_t\tilde h_t(x,y)|
    \lesssim_\epsilon
 \frac{\widetilde{V}_0(R_j)}{\widetilde{V}_0(\sqrt t)\widetilde V_j(R_j)}\left(\frac{\widetilde V_j(R_j)}{\widetilde{V}_0(R_j)}\right)^{\frac 32\epsilon} \left(\frac{1}{\log(2+R_j)}+\frac{1}{\log(2+|y|)}\right)^{1-3\epsilon} e^{-c\frac{|x|^2+|y|^2}{t}}.
\end{align*}
\item[\rm{(vi)}] When $n_i\neq 2$ and $n_j=2$, it holds
\begin{equation*}
          |t\partial_t\tilde h_t(x,y)|
    \lesssim_\epsilon
     \frac{1}{\widetilde{V}_0(\sqrt t)} \left(\frac{|y|^2}{\widetilde V_i(|y|)}+\frac{\widetilde{V}_0(R_j)}{\widetilde V_i(R_j)}\right)^{\frac 32 \epsilon}
      \left( \frac{|y|^2}{\widetilde V_i(|y|)} + \frac{\widetilde{V}_0(R_j)}{\log(2+R_j)\widetilde{V}_i(R_j)} \right)^{1-3\epsilon}
       e^{-c\frac{|x|^2+|y|^2}{t}}.
\end{equation*}
\end{enumerate}
\end{prop}
%
%
%======================End of this version==================================

\begin{proof}
  %It follows from Theorem \ref{thm:sim-heat-es} (ii), Theorem  \ref{thm:sim-heat-es} (iii) and Remark \ref{rem:thr} (iii)
It follows from Theorem \ref{thm:sim-heat-es} and Remark \ref{rem:thr} (iii)
  that for any $x \in E_j$, $y \in E_i$ and
  $t/2 \leq s  \leq 3t/2$,
 %  \begin{equation*} %label{eq:xx}
%    \tilde{h}_{t/2}(x,x)
%     \lesssim \frac{\widetilde{H}(x,t/2)}{\widetilde{V}_0(\sqrt{t})} + \frac{1}{\widetilde{V}_j(x, \sqrt t)},
%    \end{equation*}
%    \begin{equation*} %\label{eq:yy}
%      \tilde{h}_{t/2}(y,y)
%     \lesssim \frac{\widetilde{H}(y,t/2)}{\widetilde{V}_0(\sqrt{t})} + \frac{1}{\widetilde{V}_i(y, \sqrt t)}
%    \end{equation*}
%  \[
%      \tilde{h}_{r}(x,y)
%      \lesssim
%      \lf(\frac{\widetilde{H}(x,r) \widetilde{H}(y,r)}{\widetilde{V}_0(\sqrt{t})}  + \frac{1}{\widetilde{V}_i(x, \sqrt t)}\r) e^{-c\frac{d(x,y)^2}{t}},
%      \quad \mbox{if} \, i=j,
%  \]
%  and
%  \[
%    \tilde{h}_{r}(x,y)
%      \lesssim
%     \left(\frac{\widetilde{H}(x,r)\widetilde{H}(y,r)}{{\widetilde{V}}_0(\sqrt t)} + \frac{\widetilde{H}(y,r)}{\widetilde{V}_{j}(\sqrt t)}
%            + \frac{\widetilde{H}(x,r)}{\widetilde{V}_{i}(\sqrt t)} \right) e^{-c\frac{|x|^2+|y|^2}{t}},
%     \quad \mbox{if} \, i \neq j
%  \]
  \begin{equation*} %\label{eq:xy}
      \tilde{h}_{s}(x,y)
      \lesssim
       \begin{dcases}
            \lf(\frac{\widetilde{H}(x,s) \widetilde{H}(y,s)}{\widetilde{V}_0(\sqrt{t})}  + \frac{1}{\widetilde{V}_i(x, \sqrt t)}\r) e^{-c\frac{d(x,y)^2}{t}},
            & \,\mbox{if} \, i=j, \\
            \left(\frac{\widetilde{H}(x,s)\widetilde{H}(y,s)}{{\widetilde{V}}_0(\sqrt t)} + \frac{\widetilde{H}(y,s)}{\widetilde{V}_{j}(\sqrt t)}
            + \frac{\widetilde{H}(x,s)}{\widetilde{V}_{i}(\sqrt t)} \right) e^{-c\frac{|x|^2+|y|^2}{t}},
            & \,\mbox{if} \, i \neq j,
       \end{dcases}
    \end{equation*}
    which together with the facts $\widetilde{H}(x,t/2)\leq 1$ and $\widetilde{H}(y,t/2) \leq 1$ (see the definition of $\widetilde{H}$) further implies that
      \begin{equation*} %label{eq:xx}
    \tilde{h}_{t/2}(x,x)
     \lesssim \frac{\widetilde{H}(x,t/2)}{\widetilde{V}_0(\sqrt{t})} + \frac{1}{\widetilde{V}_j(x, \sqrt t)}
    \end{equation*}
    and
    \begin{equation*} %\label{eq:yy}
      \tilde{h}_{t/2}(y,y)
     \lesssim \frac{\widetilde{H}(y,t/2)}{\widetilde{V}_0(\sqrt{t})} + \frac{1}{\widetilde{V}_i(y, \sqrt t)}.
    \end{equation*}
%   Let us collect some estimates for the above terms %in the right-hand side of the above inequality.
%   in the upper bounds of $\tilde{h}_{t/2}(x,x)$,
%     $\tilde{h}_{t/2}(y,y)$ and $\tilde{h}_{r}(x,y)$.
   Assume now that $x\in  E_j\setminus F^{(R_j)}_j $, $y\in F^{(R_i)}_i$.
As a simple consequence of  %the estimates in Subsection \ref{sec:es-terms} including
%     Lemmas \ref{lem:tilde-V}, \ref{V0-general-equal},  \ref{lem:4-terms} and \ref{H-equal-general},
    Corollary \ref{lem:4-terms}, Lemmas \ref{H-equal-general}, \ref{V0-general-equal}, and the fact $|x| \ge R_j$,
     it holds that
       \begin{equation} \label{eq:1/w-V}
         \frac{1}{\widetilde{V}_k(\sqrt t)}
       \lesssim \frac{\widetilde{V}_0(R_j)}{\widetilde{V}_0(\sqrt t) \widetilde{V}_k(R_j)},
       %\lesssim \frac{1}{\widetilde{V}_0(\sqrt t)},
       \quad 1 \leq k \leq \ell,  \qquad \qquad \qquad \qquad \qquad
       \end{equation}
              \begin{equation*} \label{eq:hxr}
            \qquad \widetilde{H}(x,r)
            \lesssim
            \begin{dcases}
              \frac{R_j^2}{\widetilde{V}_j(R_j)}
             \lesssim \frac{\widetilde{V}_0(R_j)}{\log(2+R_j)\widetilde{V}_j(R_j)}
              \lesssim \frac{\widetilde{V}_0(R_j)}{\widetilde{V}_j(R_j)}
              \lesssim 1, &\, \mbox{if} \, n_j \neq 2, \\
             \frac{1}{\log(2+R_j)}
              \lesssim 1, &\, \mbox{if} \, n_j =2,
            \end{dcases}
            %\quad \forall r>0.
        \end{equation*}
        and
        \begin{equation*} \label{eq:hyr}
            \widetilde{H}(y,r)
            \lesssim
            \begin{dcases}
              \frac{|y|^2}{\widetilde{V}_i(|y|)} \lesssim 1, & \mbox{if} \, n_i \neq 2, \\
              \frac{1}{\log(2+|y|)} \lesssim 1, & \mbox{if} \, n_i =2.
            \end{dcases}   \qquad \qquad \qquad \quad \qquad
            %\quad \forall r>0,
          \end{equation*}
%Let $1\leq k \leq \ell$,
%it follows from  Lemma \ref{lem:4-terms} that
%for any $r >0$. Note that $\widetilde{H}(x,r)$ and $\widetilde{H}(y,r)$
%%obtain several bounds. /
%are controlled by several bounds.
%In the first case, when using \eqref{eq:hxr} and \eqref{eq:hyr}, we will specify which bound is used.
%Since the argument is the same,
%in the remaining cases,
%\eqref{eq:1/w-V}-\eqref{eq:hyr} are used without specification
%for concision.
%%are argued in the same method
%
%We shall use \eqref{eq:xy}-\eqref{eq:hyr} to control the terms  $\tilde{h}_{t/2}(x,x)$,
%     $\tilde{h}_{t/2}(y,y)$ and $\tilde{h}_{s}(x,y)$, and then get the estimates of $|t\partial_t\tilde h_t(x,y)|$ by Theorem \ref{davies}.
%
Below, we shall repeatedly apply the above  estimates %apply \eqref{eq:xx}-\eqref{eq:hyr}
without mentioning them. %Let $t/2 < s <3t/2$.
We will control %/estimate %control the terms
$\tilde{h}_{t/2}(x,x)$,
     $\tilde{h}_{t/2}(y,y)$ and $\tilde{h}_{s}(x,y)$, and then get the upper bounds of $|t\partial_t\tilde h_t(x,y)|$ by Theorem \ref{davies}.
 %    {\color{red} �����ǹ���a,b,c��ʱ���� ������ʽ��e, f, g���ᱻʹ��. �ⲻ�ᱻ��ʾ����������
%     �����ڹ���H(x,t), H(y,t)ʱ�� ���ǵļ������е�һ�����ᱻʹ�á�
%     �����棬 ���ǽ�ʹ��a, b, cȥ����xxx.
%     When estimating $\widetilde{H}(x,r)$ and $\widetilde{H}(y,r)$, one of their several different bounds will by used.
%     }
%Based on the different ranges of $n_i$ and $n_j$,
In the rest of the proof, %the letter $c'$ is a small enough positive constant.
the constant $c'>0$ is small enough.
Let us consider the following six cases,
%which are classified depending on the different ranges of $n_i$ and $n_j$.
%which are
which are distinguished by the varying ranges of $n_i$ and $n_j$.

\textbf{Case 1: \boldmath $n_i \neq 2$ and $n_j \neq 2$ with $i = j$.} %Let us begin with the proof of.
%\\  ================\\
%  By the standard trick of doubling property which will be used repeatedly, we have
%  \begin{align} \label{eq:doub-exa}
%    \frac{1}{\widetilde{V}_i(x, \sqrt t)}
%     &= \frac{1}{\widetilde{V}_i(o_i, \sqrt t)}  \frac{\widetilde{V}_i(o_i, \sqrt t)}{\widetilde{V}_i(x, \sqrt t)}
%    \leq \frac{1}{\widetilde{V}_i(o_i, \sqrt t)}  \frac{\widetilde{V}_i(x, \sqrt t + d(x, o_i))}{\widetilde{V}_i(x, \sqrt t)}  \notag \\
%     & \lesssim \frac{1}{\widetilde{V}_i(o_i, \sqrt t)} \lf( \frac{\sqrt t + d(x, o_i)}{\sqrt t} \r)^{N_\infty},
%    %\lesssim \frac{1}{\widetilde{V}_i(o_i, \sqrt t)} = \frac{1}{\widetilde{V}_i(\sqrt t)}.
%  \end{align} which together with
%   $
%    d(x,o_i) \leq d(x, E_0) + \diam{E_0} \leq 2R_i +1 \leq 3\sqrt t
%   $
%  and the fact $V_i(\sqrt t) \lesssim \widetilde{V}_0(R_i) / (\widetilde{V}_0(\sqrt t) \widetilde{V}_i(R_i))$ implies that
%  for any $x \in F^{(R_i)}_i$, $t \ge R_j^2$ and $i  = j$, it holds
%  \begin{equation} \label{eq:v_i}
%    \frac{1}{\widetilde{V}_i(x, \sqrt t)}
%    \lesssim \frac{1}{\widetilde{V}_i(\sqrt t)}
%    \lesssim \frac{1}{\widetilde{V}_0(\sqrt t)} \frac{\widetilde{V}_0(R_i)}{\widetilde{V}_i(R_i)}.
%  \end{equation}
%========================\\
  %For any $x \in  F^{(R_i)}_i$ and $t \ge R_i^2$,
  %By the fact that   $ d(x,o_i) \leq d(x, E_0) + \diam(E_0) \leq 2R_i +1 \leq 3\sqrt t,$
   Note that in this case, it holds $$ d(y,o_i) \leq \dist(y, E_0) + \diam(E_0) \leq 2R_i+2 \lesssim \sqrt t,$$
   and hence by the doubling property of $\tilde{\mu}_i$ and \eqref{eq:1/w-V}, % and the standard trick of doubling property (cf. \eqref{eq:change-centre}),
   we have
     \begin{equation} \label{eq:v_i}
    \frac{1}{\widetilde{V}_i(y, \sqrt t)}
    \lesssim \frac{1}{\widetilde{V}_i(\sqrt t)}
    \lesssim   \frac{\widetilde{V}_0(R_i)}{\widetilde{V}_0(\sqrt t) \widetilde{V}_i(R_i)}.
  \end{equation}
  %where the last inequality is due to \eqref{eq:1/w-V}.
    %For all $x \in  F^{(R_i)}_i$, $y \in E_i \setminus F^{(R_i)}_i$, $t \ge R_i^2$ and $s \in (t/2, 3t/2)$,
From this, %and applying the estimates $\widetilde{H}(x,t/2) \lesssim \widetilde{V}_0(R_i)/\widetilde{V}_i(R_i)$ and
%$\widetilde{H}(y, t/2) \lesssim |y|^2 / \widetilde{V}_i(|y|)$,
one deduces
%\begin{align*}%\label{eq:n-2-yy}
%      \tilde{h}_{t/2}(x,x)
%    &% \lesssim \frac{\widetilde{H}(x,t/2)^2}{\widetilde{V}_0(\sqrt{t})} + \frac{1}{\widetilde{V}_i(x, \sqrt t)}
%     \lesssim \frac{1}{\widetilde{V}_0(\sqrt{t})} \frac{\widetilde{V}_0(R_i)}{\widetilde{V}_i(R_i)} + \frac{1}{\widetilde{V}_i(y, \sqrt t)} e^{c'\frac{d(x,y)^2}{t}} \notag\\
%    & \lesssim \frac{\widetilde{V}_0(R_i)}{\widetilde{V}_0(\sqrt t) \widetilde{V}_i(R_i)}  e^{c'\frac{d(x,y)^2}{t}}
%      =: a(x,t).
%  \end{align*}
  \[
      \tilde{h}_{t/2}(x,x)
    % \lesssim \frac{\widetilde{H}(x,t/2)^2}{\widetilde{V}_0(\sqrt{t})} + \frac{1}{\widetilde{V}_i(x, \sqrt t)}
     \lesssim \frac{1}{\widetilde{V}_0(\sqrt{t})} \frac{\widetilde{V}_0(R_i)}{\widetilde{V}_i(R_i)} + \frac{1}{\widetilde{V}_i(y, \sqrt t)} e^{c'\frac{d(x,y)^2}{t}}
     \lesssim \frac{\widetilde{V}_0(R_i)}{\widetilde{V}_0(\sqrt t) \widetilde{V}_i(R_i)}  e^{c'\frac{d(x,y)^2}{t}}
      =: a(x,t)
  \]
  and
%  We deduce from  the fact $\widetilde{H}(y,t/2) \lesssim |y|^2/\widetilde{V}_i(|y|)$ and \eqref{eq:v_i} that
     \begin{equation*}%\label{eq:n-2-xx}
    \tilde{h}_{t/2}(y,y)
 %    \lesssim \frac{\widetilde{H}(y,t/2)^2}{\widetilde{V}_0(\sqrt{t})} + \frac{1}{\widetilde{V}_i(y, \sqrt t)}
    % \lesssim \frac{1}{\widetilde{V}_0(\sqrt t)} \frac{|y|^2}{\widetilde{V}_i(|y|)}
%        + \frac{1}{\widetilde{V}_i(\sqrt t)}  \\
     \lesssim \frac{1}{\widetilde{V}_0(\sqrt t)} \lf(\frac{|y|^2}{\widetilde{V}_i(|y|)}
              + \frac{\widetilde{V}_0(R_i)}{\widetilde{V}_i(R_i)}\r)
    =: b(y,t).
  \end{equation*}
     %Here and in the rest of the proof, the constant $c' >0$ is small enough.
  Similarly, %the facts  $\widetilde{H}(x,s) \lesssim \widetilde{V}_0(R_i) / \widetilde{V}_i(R_i)$ and  $\widetilde{H}(y,s)\lesssim 1$
    %one has via \eqref{eq:heat-es-same-end} that for any $s \in (t/2, 3t/2)$,
    %yield that
    using \eqref{eq:v_i} again
    %and applying the estimates $\widetilde{H}(x,s) \lesssim \widetilde{V}_0(R_i) / \widetilde{V}_i(R_i)$ and  $\widetilde{H}(y,s)\lesssim 1$
    , it holds
    \begin{equation*}%\label{eq:n-2-xy}
    \tilde{h}_{s}(x,y)
     %\lesssim \lf(\frac{\widetilde{H}(x,s) \widetilde{H}(y,s)}{\widetilde{V}_0(\sqrt{t})}  + \frac{1}{\widetilde{V}_i(y, \sqrt t)}\r) e^{-c\frac{d(x,y)^2}{t}}
    %\lesssim \frac{1}{\widetilde{V}_0(\sqrt{t})} \frac{|y|^2}{\widetilde{V}_i(|y|)} e^{\frac{-d(x,y)^2}{c't}} + \frac{1}{\widetilde{V}_i(\sqrt t)}e^{\frac{-d(x,y)^2}{c't}}\\
     \lesssim \frac{\widetilde{V}_0(R_i)}{\widetilde{V}_0(\sqrt t) \widetilde{V}_i(R_i)}  e^{-c\frac{d(x,y)^2}{t}}
     =: A(x,y,t).
  \end{equation*}
  Combining these estimates above with Theorem \ref{davies}, it follows that
 \begin{eqnarray*}
|t\partial_t\tilde{h}_t(x,y)|
&&\lesssim_\epsilon a(x,t)^{\frac{3}{2}\epsilon} b(y,t)^{\frac{3}{2}\epsilon} A(x,y,t)^{1-3\epsilon} \\
&&\lesssim_\epsilon
       \left(\frac{\widetilde{V}_0(R_i)}{\widetilde{V}_0(\sqrt t)\widetilde{V}_i(R_i)}\right)^{\frac{3}{2}\epsilon}
       \left[\frac{1}{\widetilde{V}_0(\sqrt t)}\left(\frac{|y|^2}{\widetilde{V}_{i}(|y|)}+\frac{\widetilde{V}_0(R_i)}{\widetilde{V}_i(R_i)}\right)\right]^{\frac 32\epsilon}
        \left(\frac{\widetilde{V}_0(R_i)}{\widetilde{V}_0(\sqrt t)\widetilde{V}_i(R_i)}\right)^{1-3\epsilon}  e^{-c\frac{d(x,y)^2}{t}}
        \\
&& \lesssim_\epsilon
\left(\frac{\widetilde{V}_0(R_i)}{\widetilde{V}_0(\sqrt t)\widetilde{V}_i(R_i)}\right)^{\frac{3}{2}\epsilon}
\left(\frac{\widetilde{V}_0(R_i)}{\widetilde{V}_0(\sqrt t)\widetilde{V}_i(R_i)}\right)^{\frac{3}{2}\epsilon}
\left( \frac{\widetilde{V}_i(R_i)}{\widetilde{V}_0(R_i)}
\frac{|y|^2}{\widetilde{V}_i(|y|)} + 1 \right)^{\frac{3}{2}\epsilon} \\
&& \ \times
        \left(\frac{\widetilde{V}_0(R_i)}{\widetilde{V}_0(\sqrt t)\widetilde{V}_i(R_i)}\right)^{1-3\epsilon}  e^{-c\frac{d(x,y)^2}{t}}\\
&&\lesssim_\epsilon  \frac{\widetilde{V}_0(R_i)}{\widetilde{V}_0(\sqrt t) \widetilde{V}_i(R_i)}
\left[1+ \left(\frac{\widetilde{V}_i(R_i)}{\widetilde{V}_0(R_i)}\right)^{\frac{3}{2}\epsilon}
\left(\frac{|y|^2}{\widetilde{V}_i(|y|)}\right)^{\frac{3}{2}\epsilon} \right]
e^{-c\frac{d(x,y)^2}{t}}.
%&& \lesssim_\epsilon  \frac{R_i^2}{t \widetilde{V}_i(R_i)}\left[ \left(\frac{\widetilde{V}_i(R_i)}{R_i^2}\right)^{\frac{3}{2}\epsilon}
%\left(\frac{|y|^2}{\widetilde{V}_i(|y|)}\right)^{\frac{3}{2}\epsilon} + 1\right]
%e^{-c\frac{d(x,y)^2}{t}},
\end{eqnarray*}

\textbf{Case 2: \boldmath $n_i \neq 2$ and $n_j \neq 2$ with $i \neq j$.}
%{\color{blue}  Now we turn to the proof of $i\neq j$}.
%In a very similar way,  for all $x \in F^{(R_i)}_i$, $y \in E_j \setminus F^{(R_j)}_j$, $t \ge R_j^2$ and $ s \in (t/2, 3t/2)$,
%$V_i(\sqrt t) \lesssim \widetilde{V}_0(R_j) / (\widetilde{V}_0(\sqrt t) \widetilde{V}_i(R_j))$ and
%$V_j(\sqrt t) \lesssim \widetilde{V}_0(R_j) / (\widetilde{V}_0(\sqrt t) \widetilde{V}_j(R_j))$
%[Using $\widetilde{H}(x,s) \lesssim |x|^2/ \widetilde{V}_i(|x|)$,
%$\widetilde{H}(y,s) \lesssim R_j^2 / \widetilde{V}_j(R_j)$ and
%$\widetilde{H}(y,s) \lesssim \widetilde{V}_0(R_j) / \widetilde{V}_j(R_j)$,]
Using the standard trick of doubling property (cf. \eqref{eq:change-centre}), it holds
\begin{equation}\label{c2xx}
    \tilde{h}_{t/2}(x,x)
   %  \lesssim \frac{\widetilde{H}(x,t/2)^2}{\widetilde{V}_0(\sqrt{t})} + \frac{1}{\widetilde{V}_j(x, \sqrt t)}
     \lesssim \frac{1}{\widetilde{V}_0(\sqrt{t})} \frac{\widetilde{V}_0(R_j)}{\widetilde{V}_j(R_j)} + \frac{1}{\widetilde{V}_j(\sqrt t)} e^{c'\frac{|x|^2}{t}} \\
     \lesssim \frac{\widetilde{V}_0(R_j)}{\widetilde{V}_0(\sqrt t) \widetilde{V}_j(R_j) }  e^{c'\frac{|x|^2}{t}}
      %:= b_*(x,t).
\end{equation}
%\begin{eqnarray}\label{c2xx}
%    \tilde{h}_{t/2}(x,x)
%    && \lesssim \frac{\widetilde{H}(x,t/2)^2}{\widetilde{V}_0(\sqrt{t})} + \frac{1}{\widetilde{V}_j(x, \sqrt t)}
%     \lesssim \frac{1}{\widetilde{V}_0(\sqrt{t})} \frac{\widetilde{V}_0(R_j)}{\widetilde{V}_j(R_j)} + \frac{1}{\widetilde{V}_j(\sqrt t)} e^{c'\frac{|x|^2}{t}} \nonumber \\
%    && \lesssim \frac{\widetilde{V}_0(R_j)}{\widetilde{V}_0(\sqrt t) \widetilde{V}_j(R_j) }  e^{c'\frac{|x|^2}{t}}
%      %:= b_*(x,t).
%\end{eqnarray}
%
and
%By the fact ${\widetilde{H}^2(y,t/2)}  \lesssim \widetilde{H}(y,t/2) \lesssim \widetilde{V}_0(R_j)/\widetilde{V}_j(R_j)$,
%\eqref{eq:change-centre} and  $\widetilde{V}_j(\sqrt t) \lesssim \widetilde{V}_0(R_j) / (\widetilde{V}_0(\sqrt t) \widetilde{V}_j(R_j))$, one concludes
%and %from ${\widetilde{H}^2(y,t/2)}  \lesssim \widetilde{H}(y,t/2) \lesssim \widetilde{V}_0(R_j)/\widetilde{V}_j(R_j)$ together with \eqref{eq:1/w-V} that
\begin{equation}\label{c2yy}
    \tilde{h}_{t/2}(y,y)
    % \lesssim \frac{\widetilde{H}(y,t/2)^2}{\widetilde{V}_0(\sqrt{t})} + \frac{1}{\widetilde{V}_i(y, \sqrt t)}
     \lesssim \frac{1}{\widetilde{V}_0(\sqrt t)}     + \frac{1}{\widetilde{V}_i(\sqrt t)} e^{c'\frac{|y|^2}{t}}  \\
     \lesssim  \frac{1}{\widetilde{V}_0(\sqrt t)} e^{c'\frac{|y|^2}{t}}.
    %:= a_*(y,t),
\end{equation}
Besides, one has
%One has via \eqref{eq:xy} that
\begin{eqnarray*}
h_s(x,y)
%&& \lesssim \left(\frac{\widetilde{H}(x,s)\widetilde{H}(y,s)}{{\widetilde{V}}_0(\sqrt t)} + \frac{\widetilde{H}(y,s)}{\widetilde{V}_{j}(\sqrt t)}
%    + \frac{\widetilde{H}(x,s)}{\widetilde{V}_{i}(\sqrt t)} \right) e^{-c\frac{|x|^2+|y|^2}{t}}\\
%&& \lesssim \left(\frac{1}{{\widetilde{V}}_0(\sqrt t)}  \frac{|y|^2}{\widetilde{V}_j(|y|)} \frac{|x|^2}{\widetilde{V}_i(|x|)} + \frac{1}{\widetilde{V}_{i}(\sqrt t)} \frac{|y|^2}{\widetilde{V}_j(|y|)} + \frac{1}{\widetilde{V}_{j}(\sqrt t)} \frac{|x|^2}{\widetilde{V}_i(|x|)}\right) e^{-\frac{|x|^2+|y|^2}{ct}}\\
&&\lesssim \left(\frac{1}{\widetilde{V}_0(\sqrt t)}  \frac{ \widetilde{V}_0(R_j)}{  \widetilde{V}_j(R_j)} \frac{|y|^2}{\widetilde{V}_i(|y|)}
   + \frac{1}{\widetilde{V}_{j}(\sqrt t)} \frac{|y|^2}{\widetilde{V}_i(|y|)}
   + \frac{1}{\widetilde{V}_{i}(\sqrt t)} \frac{R_j^2}{\widetilde{V}_j(R_j)}
    \right) e^{-c\frac{|x|^2+|y|^2}{t}} \\
&&\lesssim \left(\frac{1}{\widetilde{V}_0(\sqrt t)}  \frac{ \widetilde{V}_0(R_j)}{  \widetilde{V}_j(R_j)} \frac{|y|^2}{\widetilde{V}_i(|y|)}
   + \frac{\widetilde{V}_0(R_j)}{\widetilde{V}_0(\sqrt t)\widetilde{V}_{j}(R_j)} \frac{|y|^2}{\widetilde{V}_i(|y|)}
   + \frac{\widetilde{V}_0(R_j) }{\widetilde{V}_0(\sqrt t)\widetilde{V}_{i}(R_j)} \frac{R_j^2}{\widetilde{V}_j(R_j)}
    \right) e^{-c\frac{|x|^2+|y|^2}{t}}\\
&&\lesssim  \frac{\widetilde{V}_0(R_j)}{\widetilde{V}_0(\sqrt t)\widetilde{V}_j(R_j)}\left(  \frac{|y|^2}{\widetilde{V}_i(|y|)} + \frac{R_j^2}{\widetilde{V}_i(R_j)}  \right) e^{-c\frac{|x|^2+|y|^2}{t}}.
%&&\lesssim \left(\frac{1}{\widetilde{V}_0(\sqrt t)}  \frac{|x|^2}{\widetilde{V}_i(|x|)} \frac{ \widetilde{V}_0(R_j)}{  \widetilde{V}_j(R_j)}
%   + \frac{\widetilde{V}_0(R_j) }{\widetilde{V}_0(\sqrt t)\widetilde{V}_{i}(R_j)} \frac{R_j^2}{\widetilde{V}_j(R_j)}
%   + \frac{\widetilde{V}_0(R_j)}{\widetilde{V}_0(\sqrt t)\widetilde{V}_{j}(R_j)} \frac{|x|^2}{\widetilde{V}_i(|x|)} \right) e^{-\frac{|x|^2+|y|^2}{ct}}\\
%&&\lesssim  \frac{\widetilde{V}_0(R_j)}{\widetilde{V}_0(\sqrt t)\widetilde{V}_j(R_j)}\left(  \frac{|x|^2}{\widetilde{V}_i(|x|)} + \frac{R_j^2}{\widetilde{V}_i(R_j)}  \right) e^{-c\frac{|x|^2+|y|^2}{t}}
%:= B_*(x,y,t),
\end{eqnarray*}
Then, it follows from Theorem \ref{davies} that
\begin{eqnarray*}
|t\partial_t\tilde{h}_t(x,y)|
%&&\lesssim  (a_*(x,t)b_*(y,t))^{\frac{3}{2}\epsilon} B_*(x,y,t)^{1-3\epsilon}  \\
&&\lesssim_\epsilon
         \left(\frac{\widetilde{V}_0(R_j)}{\widetilde{V}_0(\sqrt t)\widetilde{V}_j(R_j)}\right)^{\frac{3}{2}\epsilon}
        \left(\frac{1}{\widetilde{V}_0(\sqrt t)}\right)^{\frac{3}{2}\epsilon}
         \left[\frac{\widetilde{V}_0(R_j)}{\widetilde{V}_0(\sqrt t)\widetilde{V}_j(R_j)} \left(\frac{|y|^2}{\widetilde{V}_i(|y|)} + \frac{R_j^2}{\widetilde{V}_i(R_j)}\right)\right]^{1-3\epsilon}
         e^{-c\frac{|x|^2+|y|^2}{t}}\\
%&& \lesssim_\epsilon\frac{\widetilde{V}_0(R_j)}{\widetilde{V}_0(\sqrt t)\widetilde{V}_j(R_j)}
%\left(\frac{|x|^2}{\widetilde{V}_i(|x|)} + \frac{R_j^2}{\widetilde{V}_i(R_j)}\right)^{1-3\epsilon}
%\left(\frac{\widetilde{V}_j(R_j)}{\widetilde{V}_0(R_j)}\right)^{\frac{3}{2}\epsilon}  e^{-c\frac{|x|^2+|y|^2}{t}}\\
&&\lesssim_\epsilon{\frac{\widetilde{V}_0(R_j)}{\widetilde{V}_0(\sqrt t)\widetilde{V}_j(R_j)}}\left(\frac{\widetilde{V}_j(R_j)}{\widetilde{V}_0(R_j)}\right)^{\frac 32\epsilon}\left(\frac{|y|^2}{\widetilde{V}_i(|y|)}+\frac{R_j^2}{\widetilde{V}_i(R_j)}\right)^{1-3\epsilon}
e^{-c\frac{|x|^2+|y|^2}{t}}.
%&& \lesssim_\epsilon{\frac{R_j^2}{t\widetilde{V}_j(R_j)}}\left(\frac{\widetilde{V}_j(R_j)}{R_j^2}\right)^{\frac 32\epsilon}\left(\frac{|y|^2}{\widetilde{V}_i(|y|)}+\frac{R_j^2}{\widetilde{V}_i(R_j)}\right)^{1-3\epsilon}
%e^{-c\frac{|x|^2+|y|^2}{t}}.
\end{eqnarray*}
%by Lemma \ref{V0-general-equal} and the fact $t \ge R_j^2$.
%which together with Lemma \ref{V0-general-equal} implies
%\[
%|t\partial_t\tilde{h}_t(x,y)|
%\lesssim {\frac{R_j^2}{t\widetilde{V}_j(R_j)}}\left(\frac{\widetilde{V}_j(R_j)}{R_j^2}\right)^{\frac 32\epsilon}\left(\frac{|x|^2}{\widetilde{V}_i(|x|)}+\frac{R_j^2}{\widetilde{V}_i(R_j)}\right)^{1-3\epsilon}
%e^{-c\frac{|x|^2+|y|^2}{t}}.
%\]
%Above, in the second term and last term of third inequality's right hand side, we used Remark \ref{rem:4-terms};
%This completes the proof of Proposition \ref{time-heat-parabolic}.

\textbf{Case 3: \boldmath $n_i=2$ and $n_j<2$.} In this case, by \eqref{eq:fifj-con} and Lemma \ref{lem:volume-growth}, it holds $R_i \lesssim R_j$,  and hence $$d(y,o_i) \leq \dist(y,E_0) + \diam(E_0) \lesssim \sqrt t.$$
%Notice that $\widetilde{H}(z,r)\lesssim 1$ for any $z \in M$ and $r>0$.
%We have for all $x \in F^{(R_i)}_i$, $y \in E_j \setminus F^{(R_j)}_j$, $t \ge R_j^2$ and $s \in(t/2, 3t/2)$,
We can easily deduce %from \eqref{eq:xx}-\eqref{eq:xy}
that
\[
          \tilde{h}_{t/2}(x,x)
     \lesssim
       \frac{1}{\widetilde{V}_0(\sqrt{t})} e^{c'\frac{|x|^2}{t}},
        \quad
     \tilde{h}_{t/2}(y,y)
     \lesssim \frac{1}{\widetilde{V}_0(\sqrt t)},
    \quad
        \tilde{h}_{s}(x,y)
     \lesssim
       \frac{1}{\widetilde{V}_0(\sqrt{t})} e^{-c\frac{|x|^2+ |y|^2}{t}},
\]
%
%\[
%    \tilde{h}_{s}(x,y)
%     \lesssim
%      \left(\frac{\widetilde{H}(x,s)\widetilde{H}(y,s)}{{\widetilde{V}}_0(\sqrt t)} + \frac{\widetilde{H}(y,s)}{\widetilde{V}_{i}(\sqrt t)}
%    + \frac{\widetilde{H}(x,s)}{\widetilde{V}_{j}(\sqrt t)} \right) e^{-c\frac{|x|^2+|y|^2}{t}}
%    \lesssim
%       \frac{1}{\widetilde{V}_0(\sqrt{t})} e^{-c\frac{|x|^2+ |y|^2}{t}}.
%\]
%By the standard trick of doubling property (cf. \eqref{eq:change-centre}), it holds
%\[
%     \tilde{h}_{t/2}(x,x)
%     \lesssim \frac{\widetilde{H}^2(x,t/2)}{\widetilde{V}_0(\sqrt{t})} + \frac{1}{\widetilde{V}_i(x, \sqrt t)}
%     \lesssim \frac{1}{\widetilde{V}_0(\sqrt t)},
%\]
%and
%\[
%      \tilde{h}_{t/2}(y,y)
%     \lesssim \frac{\widetilde{H}^2(y,t/2)}{\widetilde{V}_0(\sqrt{t})} + \frac{1}{\widetilde{V}_j(y, \sqrt t)}
%     \lesssim
%       \frac{1}{\widetilde{V}_0(\sqrt{t})} e^{c'\frac{|y|^2}{t}}.
%\]
%where $c' > 0$ is a small enough constant.
which together with Theorem \ref{davies} implies that
\begin{eqnarray*}
  |t\partial_t\tilde h_t(x,y)|
  \lesssim_\epsilon  \frac{1}{\widetilde{V}_{0}(\sqrt t)} e^{-c\frac{|x|^2+|y|^2}{t}}.
%  \sim_\epsilon  \frac{1}{t \log^2(2+\sqrt t)} e^{-c\frac{|x|^2+|y|^2}{t}}.
\end{eqnarray*}

\textbf{Case 4: \boldmath $n_i=n_j =2$.} In this case, by \eqref{eq:fifj-con} and Lemma \ref{lem:volume-growth}, it holds $R_i \sim R_j$, and hence $$d(y,o_i) \leq \dist(y,E_0) + \diam(E_0) \lesssim \sqrt t.$$
%Note that $|x| \lesssim \sqrt t$. Similar as \eqref{eq:v_i}, it holds
%$\widetilde{V}_i(x, \sqrt t) \gtrsim \widetilde{V}_i(\sqrt t) \gtrsim \widetilde{V}_0(\sqrt t)$. One has via \eqref{eq:heat-es-same-end} and \eqref{eq:heat-es-dif-end} that
%{\color{blue}In a very similar way},
%Using the standard trick of doubling property (cf. \eqref{eq:change-centre}),
By a simple calculation, one has %via \eqref{eq:xy}-\eqref{eq:yy} that %\eqref{eq:heat-es-dif-end} and \eqref{eq:heat-es-same-end} that
%$
%     \tilde{h}_{t/2}(x,x)
%     \lesssim 1/\widetilde{V}_0(\sqrt t),
%$
%\[
%      \tilde{h}_{t/2}(y,y)
%     \lesssim
%       \frac{1}{\widetilde{V}_0(\sqrt{t})} e^{c'\frac{d(x,y)^2}{t}}, \quad  \mbox{if} \, i=j,
%       \qquad
%       \tilde{h}_{t/2}(y,y)
%     \lesssim
%       \frac{1}{\widetilde{V}_0(\sqrt{t})} e^{c'\frac{|y|^2}{t}}, \quad \mbox{if} \, i \neq j.
%\]
%\[
%    \tilde{h}_{s}(x,y)
%     \lesssim
%       \frac{1}{\widetilde{V}_0(\sqrt{t})} e^{-c\frac{d(x,y)^2}{t}},
%    \quad \mbox{if} \,  i=j,
%    \qquad
%       \tilde{h}_{s}(x,y)
%     \lesssim  \frac{1}{\widetilde{V}_0(\sqrt{t})} e^{-c\frac{|x|^2+ |y|^2}{t}},
%      \quad \mbox{if} \, i \neq j,
%\]
\[
     \tilde{h}_{t/2}(y,y)
     \lesssim \frac{1}{\widetilde{V}_0(\sqrt t)},  \qquad \qquad \quad \,
\]
\[
    \qquad  \tilde{h}_{t/2}(x,x)
     \lesssim
       \begin{dcases}
       \frac{1}{\widetilde{V}_0(\sqrt{t})} e^{c'\frac{d(x,y)^2}{t}}, & \,  \mbox{if} \, i=j, \\
       \frac{1}{\widetilde{V}_0(\sqrt{t})} e^{c'\frac{|x|^2}{t}}, & \, \mbox{if} \, i \neq j,
       \end{dcases}
 \]
 and
 \[
    \qquad \tilde{h}_{s}(x,y)
     \lesssim
     \begin{dcases}
       \frac{1}{\widetilde{V}_0(\sqrt{t})} e^{-c\frac{d(x,y)^2}{t}}, & \mbox{if} \,  i=j, \\
       \frac{1}{\widetilde{V}_0(\sqrt{t})} e^{-c\frac{|x|^2+ |y|^2}{t}}, & \mbox{if} \, i \neq j.
     \end{dcases}
\]
Combining the estimates of $h_{t/2}(x,x)$, $h_{t/2}(y,y)$ and $h_{s}(x,y)$,
we can apply Theorem \ref{davies} to obtain
\[
    |t\partial_t\tilde h_t(x,y)|
    \lesssim_\epsilon  \frac{1}{\widetilde{V}_0(\sqrt t)} e^{-c\frac{d(x,y)^2}{t}}.
%    \lesssim_\epsilon  \frac{1}{t \log^2(2+\sqrt t)} e^{-c\frac{d(x,y)^2}{t}}.
\]

\textbf{Case 5: \boldmath $n_i=2$, $n_j>2$.}
%For all $x \in F^{(R_i)}_i$, $y \in E_j \setminus F^{(R_j)}_j$, $t \ge R_j^2$ and $s \in (t/2,3t/2)$,
%From $\widetilde{H}(x,s) \lesssim 1/\log(2+|x|) \lesssim 1$ and
%$\widetilde{H}(y,s) \lesssim \widetilde{V}_0(R_j)/(\log(2+R_j) \widetilde{V}_j(R_j))$,
Following the same argument as in \eqref{c2xx} and \eqref{c2yy}, one has
\[
     \quad \quad \tilde{h}_{t/2}(x,x)
     \lesssim \frac{\widetilde{V}_0(R_j)}{\widetilde{V}_0(\sqrt t) \widetilde V_j(R_j)} e^{c'\frac{|x|^2}{t}}
\]
and
\[
       \tilde{h}_{t/2}(y,y)
     \lesssim \frac{1}{\widetilde{V}_0(\sqrt{t})} e^{c'\frac{|y|^2}{t}}. \]
%One has via \eqref{eq:xy} that
Besides, it holds
\begin{eqnarray*}
     \tilde{h}_{s}(x,y)
    % &  \lesssim
%      \left(\frac{\widetilde{H}(x,s)\widetilde{H}(y,s)}{{\widetilde{V}}_0(\sqrt t)} + \frac{\widetilde{H}(y,s)}{\widetilde{V}_{j}(\sqrt t)}
%    + \frac{\widetilde{H}(x,s)}{\widetilde{V}_{i}(\sqrt t)} \right) e^{-c\frac{|x|^2+|y|^2}{t}} \notag \\
      &&  \lesssim
      \left(\frac{\widetilde{H}(x,s)}{{\widetilde{V}}_0(\sqrt t)} + \frac{\widetilde{H}(y,s)}{\widetilde{V}_{j}(\sqrt t)} \right) e^{-c\frac{|x|^2+|y|^2}{t}} \notag \\
    %& \lesssim \left(\frac{1}{ \widetilde{V}_0(\sqrt t)} \frac{\widetilde{V}_0(R_j)}{\log(2+R_j) \widetilde{V}_{j}(R_j)}
%        +  \frac{1}{\widetilde{V}_j(\sqrt t)} \frac{1}{\log(2+|y|)} \right. \\
%    &  \ \left. + \frac{1}{ \widetilde{V}_i(\sqrt t)} \frac{\widetilde{V}_0(R_j)}{\log(2+R_j) \widetilde{V}_{j}(R_j)}
%        \right)e^{-c\frac{|x|^2+|y|^2}{t}}  \\
     &&\lesssim \left(\frac{1}{ \widetilde{V}_0(\sqrt t)} \frac{\widetilde{V}_0(R_j)}{\log(2+R_j) \widetilde{V}_{j}(R_j)}
        % + \frac{1}{ \widetilde{V}_0(\sqrt t)} \frac{\widetilde{V}_0(R_j)}{\log(2+R_j) \widetilde{V}_{j}(R_j)}
        +  \frac{\widetilde{V}_0(R_j)}{\widetilde{V}_0(\sqrt t) \widetilde{V}_j(R_j)} \frac{1}{\log(2+|y|)} \right)e^{-c\frac{|x|^2+|y|^2}{t}} \nonumber\\
     &&\lesssim \frac{\widetilde{V}_0(R_j)}{\widetilde{V}_0(\sqrt t)\widetilde V_j(R_j)} \left(\frac{1}{\log(2+R_j)} + \frac{1}{\log(2+|y|)} \right)e^{-c\frac{|x|^2+|y|^2}{t}}.
     % \nonumber\\
%     &\quad \lesssim \frac{\widetilde{V}_0(R_j)}{\widetilde{V}_0(\sqrt t)\widetilde V_j(R_j)} \left(\frac{1}{\log(2+R_j)} + \frac{1}{\log(2+|x|)} \right)e^{-c\frac{|x|^2+|y|^2}{t}},
\end{eqnarray*}
Therefore, %if we choose suitably large constant in the estimate of $h_{t/2}(x,x)$ and $h_{t/2}(y,y)$,
we can use Theorem \ref{davies} to conclude that
\begin{eqnarray*}
|t\partial_t\tilde h_t(x,y)|
&& \lesssim_\epsilon
\left(\frac{ \widetilde{V}_0(R_j)}{\widetilde{V}_0(\sqrt t)\widetilde V_j(R_j)} \right)^{\frac 32\epsilon}
\left(\frac{1}{\widetilde{V}_0(\sqrt t)} \right)^{\frac 32\epsilon} \\
&& \ \times \left( \frac{\widetilde{V}_0(R_j)}{\widetilde{V}_0(\sqrt t)\widetilde V_j(R_j)} \right)^{1-3\epsilon}
 \left(\frac{1}{\log(2+R_j)}+\frac{1}{\log(2+|y|)} \right)^{1-3\epsilon} e^{-c\frac{|x|^2+|y|^2}{t}}\\
&&  \lesssim_\epsilon\frac{\widetilde{V}_0(R_j)}{\widetilde{V}_0(\sqrt t)\widetilde V_j(R_j)}\left(\frac{\widetilde V_j(R_j)}{\widetilde{V}_0(R_j)}\right)^{\frac 32\epsilon}\left(\frac{1}{\log(2+R_j)}+\frac{1}{\log(2+|y|)}\right)^{1-3\epsilon}e^{-c\frac{|x|^2+|y|^2}{t}}.
%&&  \lesssim_\epsilon\frac{R_j^2}{t} \frac{\log(2+R_j)}{\log(2+\sqrt t) \log^{3\epsilon}(2+R_j)} \frac{1}{ \widetilde V_j(R_j)} \left(\frac{\widetilde V_j(R_j)}{R_j^2}\right)^{\frac 32\epsilon} \\
%&& \ \times \left(\frac{1}{\log(2+R_j)}+\frac{1}{\log(2+|y|)}\right)^{1-3\epsilon}e^{-c\frac{|x|^2+|y|^2}{t}}. \\
%&&  \lesssim_\epsilon \frac{R_j^2}{t} \frac{\log(2+R_j)}{\log(2+\sqrt t)}
%\frac{\log^{3\epsilon-1}(2+R_j) + \log^{3\epsilon-1}(2+|y|)}{\log^{3\epsilon}(2+R_j)}
%\frac{1}{ \widetilde V_j(R_j)} \left(\frac{\widetilde V_j(R_j)}{R_j^2}\right)^{\frac 32\epsilon} e^{-c\frac{|x|^2+|y|^2}{t}}.
\end{eqnarray*}

\textbf{Case 6: \boldmath $n_i\neq 2$ and $n_j=2$.}
%for all $x \in F^{(R_i)}_i$, $y \in E_j \setminus F^{(R_j)}_j$, $t \ge R_j^2$ and $s \in (t/2, 3t/2)$,
%From $\widetilde{H}(x, s) \lesssim |x|^2 / \widetilde{V}_i(|x|)$, $\widetilde{H}(y,s) \lesssim 1$ and
%and $\widetilde{H}(y, s) \lesssim  \widetilde{V}_0(R_j)/ (\log(2+R_j)\widetilde{V}_i(R_j))$,
Using the standard trick of doubling property (cf. \eqref{eq:change-centre}), one has
\[
      \tilde{h}_{t/2}(x,x)
     %\lesssim \frac{\widetilde{H}(x,t/2)^2}{\widetilde{V}_0(\sqrt{t})} + \frac{1}{\widetilde{V}_j(x, \sqrt t)}
    \lesssim \frac{1}{ \widetilde{V}_0(\sqrt{t})} e^{c'\frac{|x|^2}{t}}
\]
%\begin{align*}
%     \tilde{h}_{t/2}(x,x)
%     & \lesssim \frac{\widetilde{H}^2(x,t/2)}{\widetilde{V}_0(\sqrt{t})} + \frac{1}{\widetilde{V}_i(x, \sqrt t)}
%     \lesssim \frac{1}{\widetilde{V}_0(\sqrt{t})} \frac{|x|^2}{\widetilde V_i(|x|)} + \frac{1}{\widetilde V_{i}(\sqrt t)} e^{c'\frac{|x|^2}{t}} \\
%     & \lesssim  \frac{1}{\widetilde{V}_0(\sqrt t)}\left(\frac{|x|^2}{\widetilde V_i(|x|)}+\frac{\widetilde{V}_0(R_j)}{\widetilde V_i(R_j)}\right)e^{c'\frac{|x|^2}{t}},
%\end{align*}
%By the fact $\widetilde{H}^2(y, t/2) \leq 1$ and \eqref{eq:change-centre}, it holds
and
\begin{eqnarray*}
     \tilde{h}_{t/2}(y,y)
      %\lesssim \frac{\widetilde{H}(y,t/2)^2}{\widetilde{V}_0(\sqrt{t})} + \frac{1}{\widetilde{V}_i(y, \sqrt t)}
     \lesssim \frac{1}{\widetilde{V}_0(\sqrt{t})} \frac{|y|^2}{\widetilde V_i(|y|)} + \frac{1}{\widetilde V_{i}(\sqrt t)} e^{c'\frac{|y|^2}{t}}
      \lesssim  \frac{1}{\widetilde{V}_0(\sqrt t)}\left(\frac{|y|^2}{\widetilde V_i(|y|)}+\frac{\widetilde{V}_0(R_j)}{\widetilde V_i(R_j)}\right)e^{c'\frac{|y|^2}{t}}.
\end{eqnarray*}
Besides, it holds
\begin{eqnarray*}
    \tilde{h}_{s}(x,y)
%     & \lesssim
%      \left(\frac{\widetilde{H}(x,s)\widetilde{H}(y,s)}{{\widetilde{V}}_0(\sqrt t)} + \frac{\widetilde{H}(y,s)}{\widetilde{V}_{i}(\sqrt t)}
%    + \frac{\widetilde{H}(x,s)}{\widetilde{V}_{j}(\sqrt t)} \right) e^{-c\frac{|x|^2+|y|^2}{t}} \\
     && \lesssim
      \left(\frac{\widetilde{H}(y,s)}{{\widetilde{V}}_0(\sqrt t)} + \frac{\widetilde{H}(x,s)}{\widetilde{V}_{i}(\sqrt t)} \right) e^{-c\frac{|x|^2+|y|^2}{t}} \\
    && \lesssim
      \left(\frac{1}{{\widetilde{V}}_0(\sqrt t)}\frac{|y|^2}{\widetilde V_i(|y|)}
           + \frac{\widetilde{V}_0(R_j)}{\widetilde{V}_{0}(\sqrt t) \widetilde{V}_i(R_j)} \frac{1}{\log(2+R_j)} \right) e^{-c\frac{|x|^2+|y|^2}{t}} \\
    &&\lesssim \frac{1}{\widetilde{V}_0(\sqrt t)}\left(\frac{|y|^2}{\widetilde V_i(|y|)} + \frac{\widetilde{V}_0(R_j)}{ \log(2+R_j)\widetilde{V}_i(R_j)} \right) e^{-c\frac{|x|^2+|y|^2}{t}}.
\end{eqnarray*}
%By the fact $\widetilde{H}^2(x, t/2) \lesssim \widetilde{H}(x, t/2) \lesssim |x|^2 / \widetilde{V}_i(|x|)$, \eqref{eq:change-centre} and
%$\widetilde{V}_i(\sqrt t) \lesssim \widetilde{V}_0(R_j) / (\widetilde{V}_0(\sqrt t) \widetilde{V}_i(R_j))$, one concludes
%where $c' > 0$ is a small enough constant.
Then, we deduce from Theorem \ref{davies} that
\begin{eqnarray*}%\label{eq:two-cases}
 |t\partial_t\tilde h_t(x,y)|
%&&\ \lesssim_\epsilon\left[\frac{1}{\widetilde{V}_0(\sqrt t)}\left(\frac{|x|^2}{\widetilde V_i(|x|)}+\frac{\widetilde{V}_0(R_j)}{\widetilde V_i(R_j)}\right)\right]^{\frac 32 \epsilon}
%\left(\frac{1}{\widetilde{V}_{0}(\sqrt t)}\right)^{\frac 32\epsilon} \\
%&&\quad\times
%\left[\frac{1}{\widetilde{V}_0(\sqrt t)}\left(\frac{|y|^2}{\widetilde V_i(|y|)} + \frac{\widetilde{V}_0(R_j)}{\log(2+R_j)\widetilde{V}_i(R_j)}  \right)\right]^{1-3\epsilon} e^{-c\frac{|y|^2+|y|^2}{t}} \nonumber\\
&& \lesssim_\epsilon\frac{1}{\widetilde{V}_0(\sqrt t)} \left(\frac{|y|^2}{\widetilde V_i(|y|)}+\frac{\widetilde{V}_0(R_j)}{\widetilde V_i(R_j)}\right)^{\frac 32 \epsilon}
 \left( \frac{|y|^2}{\widetilde V_i(|y|)} + \frac{\widetilde{V}_0(R_j)}{\log(2+R_j)\widetilde{V}_i(R_j)} \right)^{1-3\epsilon}e^{-c\frac{|x|^2+|y|^2}{t}}.
 %&& \ \lesssim_\epsilon
% \frac{1}{t\log^2(2+\sqrt t)}
% \left(\frac{|y|^2}{\widetilde V_i(|y|)} + \frac{R_j^2 \log^2(2+R_j)}{\widetilde V_i(R_j)}\right)^{\frac 32 \epsilon}
%\left( \frac{|y|^2}{\widetilde V_i(|y|)} + \frac{R_j^2 \log(2+R_j)}{\widetilde{V}_i(R_j)} \right)^{1-3\epsilon}e^{-c\frac{|y|^2+|y|^2}{t}} \\
%&&  \lesssim_\epsilon
% \frac{\log^{3\epsilon}(2+R_j) \log^{1-3\epsilon}(2+R_j)}{t\log^2(2+\sqrt t)}
% \left(\frac{|y|^2}{\widetilde V_i(|y|)} + \frac{R_j^2 }{\widetilde V_i(R_j)}\right)^{\frac 32 \epsilon} \\
%&& \ \times \left( \frac{|y|^2}{\widetilde V_i(|y|)} + \frac{R_j^2 }{\widetilde{V}_i(R_j)} \right)^{1-3\epsilon}e^{-c\frac{|x|^2+|y|^2}{t}}  \\
%&&  \lesssim_\epsilon
% \frac{1}{t\log(2+\sqrt t)}
% \left(\frac{|y|^2}{\widetilde V_i(|y|)} + \frac{R_j^2 }{\widetilde V_i(R_j)}\right)^{1-\frac 32 \epsilon} e^{-c\frac{|x|^2+|y|^2}{t}}.
\end{eqnarray*}

 Collecting the estimates above, we finish the proof.
\end{proof}

%{\color{blue}
%\begin{rem}
%  In fact, from the proof above, in Propositions \ref{time-heat-parabolic} and \ref{time-heat-parabolic-2},
%  provided x is near the center, i.e., $x \in c'F_i$, y is away from the center, i.e., $y \in E_j \setminus c'F_j$  and $t \ge c'R_j$ for some given positive constant $c'$,
%  the estimate of the time derivative of the heat kernel still holds.
%  The same is valid for the following mapping properties of the heat kernel, i.e., Propositions \ref{map-heat-parabolic} and \ref{mapping-time-heat-parabolic-1}.
%\end{rem}
%}

\subsection{$L^1$-estimate of the space derivative of the heat kernel } \label{sec:heat-ker-away-center-est}\hskip\parindent
%In this subsection, following an argument as in \cite[Lemma 2.4]{cd99},
%we obtain the $L^1$-estimate of the space derivative of the heat kernel,
%which is a counterpart to the main result in \cite[Subsection 2.3]{cd99}.
%Our primary tools are the heat kernel upper bound %\eqref{heat-mixed}
%and the Caccioppoli inequality which controls the spatial derivative part through the time derivative part; see for instance \cite{jl22}.
%It is worth to note that the method of \cite[Proposition 3.7]{jiang-li-lin-2022}, which  relies on the main results in \cite{gri95,gri97}, is also valid.
%
%In this subsection, following an argument as in \cite[Subsection 2.3]{cd99},
%we obtain the weighted estimate for the space derivative of the heat kernel, i.e., Proposition \ref{est-integral-diagonal-general},
%which can be considered as a counterpart to \cite[Lemma 2.4]{cd99}.
%
In this subsection, following an argument as in \cite[Lemma 2.4]{cd99},
we obtain the weighted estimate for the space derivative of the heat kernel.
Our primary tool %are the heat kernel upper bound %\eqref{heat-mixed}
%and
is the Caccioppoli inequality which controls the spatial derivative part through the time derivative part; see also \cite[Proposition 5.2]{jl22}.
It is worth to note that the method in the proof of \cite[Proposition 3.7]{jiang-li-lin-2022}, which  relies on the main results in \cite{gri95,gri97}, is also valid.

We have the following $L^1$-estimate of the space derivative of the heat kernel.

\begin{prop}\label{est-integral-diagonal-general}
Let $1 \leq i \leq \ell$ and $\beta \ge 16$. For any  $\kappa \ge 2$, it holds
\begin{eqnarray*}
\int_{E_i\setminus (F^{(\kz)}_i\cup B(y,r))}|\nabla_x h_t(x,y)|\,d\mu(x) \lesssim_\beta \frac{1}{\sqrt t}e^{-c\frac{r^2}{t}}, \quad
\forall y\in E_i\setminus F^{(\kappa)}_i, \, 0<t \le \beta \kappa^2, \, r>0.
\end{eqnarray*}
\end{prop}
\begin{proof}
Let $\psi_i$ be a Lipschitz function on $M$ with $\supp \psi_i\subset \{x \in E_i: \dist(x,E_i\setminus F^{(\kz)}_i)< \kz\}$ such that $\psi_i\equiv 1$ on $E_i\setminus F^{(\kz)}_i$ and $|\nabla \psi_i|\leq C/\kz$.
Note that it holds
$\dist(\supp \psi_i,\,E_0)>\kz,$
%Recall $F^{(\kz/2)}_i=\{x\in E_i:\,\dist(x,\,E_0)\le \kz\}$.
and hence $$\supp \psi_i \subset E_i \setminus F^{(\kz/2)}_i.$$
% ��ʦ�� �������ǲ�ȷ���Ƿ������ڶ���psi��ʱ��ֱ��˵֧���������������档
Repeating the Caccioppoli
argument, we have
\begin{eqnarray}\label{eq:caci-1}
\int_{M}|\nabla_x h_t(x,y)|^2\psi_i(x)^2 e^{c'\frac{d(x,y)^2}{t}}\,d\mu(x)
&& =\int_M \nabla_x h_t(x,y) \cdot \nabla_x  (\psi_i(x)^2e^{c'\frac{d(x,y)^2}{t}}  h_t(x,y))\,d\mu(x) \nonumber \\
&& \ -\int_M \nabla_x h_t(x,y) \cdot \nabla_x  (\psi_i(x)^2e^{c'\frac{d(x,y)^2}{t}}) h_t(x,y)\,d\mu(x) \nonumber \\
&& \le \int_M \L h_t(x,y) \psi_i(x)^2e^{c'\frac{d(x,y)^2}{t}}h_t(x,y)\,d\mu(x)
    \nonumber \\
&& \ +  \int_M |\nabla_x h_t(x,y)|  |\nabla_x  (\psi_i(x)^2e^{c'\frac{d(x,y)^2}{t}})| h_t(x,y) \,d\mu(x).
\end{eqnarray}
%\begin{eqnarray}\label{eq:caci-1}
%&& \int_{M}|\nabla_x h_t(x,y)|^2\psi_i(x)^2 e^{c'\frac{d(x,y)^2}{t}}\,d\mu(x) \nonumber \\
%&& \ =\int_M \nabla_x h_t(x,y) \cdot \nabla_x  \Big(\psi_i(x)^2e^{c'\frac{d(x,y)^2}{t}}  h_t(x,y)\Big)\,d\mu(x) \nonumber \\
%&& \ \ -\int_M \nabla_x h_t(x,y) \cdot \nabla_x  \Big(\psi_i(x)^2e^{c'\frac{d(x,y)^2}{t}}\Big) h_t(x,y)\,d\mu(x) \nonumber \\
%&& \ \le \int_M \L h_t(x,y) \psi_i(x)^2e^{c'\frac{d(x,y)^2}{t}}h_t(x,y)\,d\mu(x)
%    \nonumber \\
%&& \ \ +  \int_M |\nabla_x h_t(x,y)|  |\nabla_x  (\psi_i(x)^2e^{c'\frac{d(x,y)^2}{t}})| h_t(x,y) \,d\mu(x).
%\end{eqnarray}
%
%\begin{eqnarray}\label{eq:caci-1}
%&&\int_{M}|\nabla_x h_t(x,y)|^2\psi_i(x)^2 e^{c'\frac{d(x,y)^2}{t}}\,d\mu(x) \nonumber \\
%&&\ =\int_M \nabla_x h_t(x,y) \cdot \nabla_x  (\psi_i(x)^2e^{c'\frac{d(x,y)^2}{t}}  h_t(x,y))\,d\mu(x)
% -\int_M \nabla_x h_t(x,y) \cdot \nabla_x  (\psi_i(x)^2e^{c'\frac{d(x,y)^2}{t}}) h_t(x,y)\,d\mu(x) \nonumber \\
%&&\ \le \int_M \L h_t(x,y) \psi_i(x)^2e^{c'\frac{d(x,y)^2}{t}}h_t(x,y)\,d\mu(x)
%+  \int_M |\nabla_x h_t(x,y)|  |\nabla_x  (\psi_i(x)^2e^{c'\frac{d(x,y)^2}{t}}) h_t(x,y)| \,d\mu(x).
%\end{eqnarray}
By the basic inequality $2ab \leq a^2 + b^2$ for all $a, b>0$, it holds
\begin{eqnarray*}%\label{eq:caci-2}
&&   \int_M |\nabla_x h_t(x,y)|  |\nabla_x  (\psi_i(x)^2e^{c'\frac{d(x,y)^2}{t}}) | h_t(x,y) \,d\mu(x) \nonumber \\
&& \ \leq   2 \int_M |\nabla_x h_t(x,y)|  |\psi_i(x)|   \left(|\nabla_x\psi_i(x)| e^{c'\frac{d(x,y)^2}{t}}+ |\psi_i(x)| \frac{d(x,y)}{t}e^{c'\frac{d(x,y)^2}{t}}\right) h_t(x,y) \,d\mu(x)\nonumber  \\
&& \ \leq \frac{1}{2}\int_{M}|\nabla_x h_t(x,y)|^2\psi_i(x)^2e^{c'\frac{d(x,y)^2}{t}}\,d\mu(x) \nonumber \\
&& \ \ +  2\int_{M}\left(|\nabla_x\psi_i(x)| + |\psi_i(x)| \frac{d(x,y)}{t} \right)^2 h_t(x,y)^2 e^{c'\frac{d(x,y)^2}{t}} \,d\mu(x). %\nonumber  \\
%&&\ \le \frac{1}{2}\int_{M}|\nabla_x h_t(x,y)|^2\psi_i(x)^2e^{c'\frac{d(x,y)^2}{t}}\,d\mu(x)
%  + 4\int_{M}|\nabla_x \psi_i(x)|^2 h_t(x,y)^2e^{c'\frac{d(x,y)^2}{t}}\,d\mu(x)  \nonumber  \\
%&& \ \ +  4\int_{M}\frac{d(x,y)^2}{t^2}  \psi_i(x)^2 h_t(x,y)^2e^{c'\frac{d(x,y)^2}{t}}\,d\mu(x).
\end{eqnarray*}
Combining this estimate with \eqref{eq:caci-1}, we obtain that
\begin{eqnarray*}
\int_{M}|\nabla_x h_t(x,y)|^2\psi_i(x)^2 e^{c'\frac{d(x,y)^2}{t}}\,d\mu(x)
&&  \le 2 \int_M \L h_t(x,y) h_t(x,y) \psi_i(x)^2e^{c'\frac{d(x,y)^2}{t}} \,d\mu(x) \\
%&& \ +4\int_{M}\left(|\nabla_x\psi_i(x)| e^{c'\frac{d(x,y)^2}{t}}+ |\psi_i(x)| \frac{d(x,y)}{t}e^{c'\frac{d(x,y)^2}{t}}\right)^2 |h_t(x,y)|^2 \,d\mu(x),
&& \ +4\int_{M}\left(|\nabla_x\psi_i(x)| + |\psi_i(x)| \frac{d(x,y)}{t}\right)^2 h_t(x,y)^2 e^{c'\frac{d(x,y)^2}{t}} \,d\mu(x),
%&& \le 2\left(\int_{E_i\setminus F^{(\kz/2)}_i}|\L h_t(x,y)|^2 e^{c'\frac{d(x,y)^2}{t}}\psi_i(x)^2\,d\mu(x)\right)^{1/2}
%        \left(\int_{E_i\setminus F^{(\kz/2)}_i}|\psi_i(x) h_t(x,y)|^2 e^{c'\frac{d(x,y)^2}{t}} \,d\mu(x)\right)^{1/2} \\
%&& \ + 8 \int_{M}|\nabla_x \psi_i(x)|^2 h_t(x,y)^2e^{c'\frac{d(x,y)^2}{t}}\,d\mu(x)
%        + 8\int_{M}\frac{d(x,y)^2}{t^2}  \psi_i(x)^2 h_t(x,y)^2e^{c'\frac{d(x,y)^2}{t}}\,d\mu(x),
\end{eqnarray*}
which  together with the properties of $\psi_i$
%the facts $|\psi_i(x)| \lesssim 1$, $|\nabla \psi_i(x)| \lesssim 1/\kz \lesssim 1/\sqrt t$ and $\supp \psi_i(x) \subset E_i \setminus F^{(\kz/2)}_i$
implies that
\begin{eqnarray*}
 \int_{E_i\setminus F^{(\kz)}_i}|\nabla_xh_t(x,y)|^2 e^{c'\frac{d(x,y)^2}{t}}\,d\mu(x)
&&    \le \int_{M}|\nabla_xh_t(x,y) |^2\psi_i(x)^2 e^{c'\frac{d(x,y)^2}{t}}\,d\mu(x)\nonumber \\
%&& \ \lesssim
%     \frac{1}{t} \int_{E_i\setminus F^{(\kz)}_i} |t\partial h_t(x,y)| |h_t(x,y)| e^{c'\frac{d(x,y)^2}{t}} \,d\mu(x)
%  +
%   \int_{E_i\setminus F^{(\kz)}_i}\left( \frac{1}{\kz} e^{c'\frac{d(x,y)^2}{t}}+ \frac{d(x,y)}{t} e^{c'\frac{d(x,y)^2}{t}}\right)^2 |h_t(x,y)|^2 \,d\mu(x) \\
&&  \lesssim
     \frac{1}{t} \int_{E_i\setminus F^{(\kz/2)}_i} |t\partial_t h_t(x,y)| h_t(x,y) e^{c'\frac{d(x,y)^2}{t}} \,d\mu(x) \\
&& \  +
     \int_{E_i\setminus F^{(\kz/2)}_i} \lf(\frac{1}{t} + \frac{d(x,y)^2}{t^2}\r) h_t(x,y)^2   e^{c'\frac{d(x,y)^2}{t}} \,d\mu(x).
\end{eqnarray*}
% This fact, together with Lemma \ref{lem:es-away}, yields that for small enough $c'>0$  and any $\kz \ge 2$,
%This combined with Lemma \ref{lem:es-away} yields that, if $c'>0$ is small enough, then% and any $\kz \ge 2$,
From this, Lemma \ref{lem:es-away} and Lemma \ref{lem:es-away-time-der}, we get that, if $c'>0$ is small enough, then
%
%$0<t \le \beta \kappa^2$,  $y\in E_i\setminus F^{(\kappa)}_i$, %and any $y \in  E_i \setminus F^{(\kz)}_i$,
%\begin{eqnarray*}%\label{est-on-diagonal-gradient}
%&& \int_{E_i\setminus F^{(\kz)}_i}|\nabla_xh_t(x,y)|^2 e^{c'\frac{d(x,y)^2}{t}}\,d\mu(x) \nonumber \\
%&& \ \lesssim_\beta  \frac{1}{t} \int_{E_i \setminus F^{(\kz/2)}_i} \frac{1}{V_i(y,\sqrt t)^2} e^{-c''\frac{d(x,y)^2}{t}} \, d\mu(x)
%   +  \int_{E_i \setminus F^{(\kz/2)}_i} \lf(\frac{1}{t} + \frac{d(x,y)^2}{t^2}\r) \frac{1}{V_i(y,\sqrt t)^2} e^{-c''\frac{d(x,y)^2}{t}}\, d\mu(x) \notag \\
%%&& \leq C \int_{E_i \setminus F^{(\kz/2)}_i} \frac{1}{V_i(y,\sqrt t)^2} e^{-\frac{d(x,y)^2}{ct}} d\mu(x)\left(\frac 1t+\frac{1}{\kz^2}\right) \nonumber \\
%&& \ \lesssim_\beta \frac{1}{tV_i(y,\sqrt t)^2} \int_{M_i}  e^{-c''\frac{d(x,y)^2}{t}} \, d\mu(x).
%   % \lesssim_\beta \frac{1}{tV_i(y,\sqrt t)},
%%    \quad
%%\forall 0<t \le \beta \kappa^2, \, y\in E_i\setminus F^{(\kappa)}_i,
%\end{eqnarray*}
\begin{eqnarray*}%\label{est-on-diagonal-gradient}
 \int_{E_i\setminus F^{(\kz)}_i}|\nabla_xh_t(x,y)|^2 e^{c'\frac{d(x,y)^2}{t}}\,d\mu(x)
&&  \lesssim
     \int_{E_i \setminus F^{(\kz/2)}_i} \lf(\frac{1}{t} + \frac{d(x,y)^2}{t^2}\r) \frac{1}{V_i(y,\sqrt t)^2} e^{-c''\frac{d(x,y)^2}{t}}\, d\mu(x) \notag \\
%&& \leq C \int_{E_i \setminus F^{(\kz/2)}_i} \frac{1}{V_i(y,\sqrt t)^2} e^{-\frac{d(x,y)^2}{ct}} d\mu(x)\left(\frac 1t+\frac{1}{\kz^2}\right) \nonumber \\
&&  \lesssim_\beta \frac{1}{tV_i(y,\sqrt t)^2} \int_{M_i}  e^{-c''\frac{d(x,y)^2}{t}} \, d\mu(x).
   % \lesssim_\beta \frac{1}{tV_i(y,\sqrt t)},
%    \quad
%\forall 0<t \le \beta \kappa^2, \, y\in E_i\setminus F^{(\kappa)}_i,
\end{eqnarray*}
%for any $0<t \le \beta \kappa^2$,  $y\in E_i\setminus F^{(\kappa)}_i$, where we used the trivial inequality $e^{-c'' \frac{d(x,y)^2}{t}} \lesssim \frac{t}{d(x,y)^2}$.
%{\color{gray}and the fact that
%\begin{align}\label{eq:1/v-e}
%     \int_{M_i}  e^{-c''\frac{d(x,y)^2}{t}} d\mu(x)
%     & \leq \int_{B_i(y, \sqrt t)} e^{-c''\frac{d(x,y)^2}{t}} d\mu(x)
%          + \sum_{i=1}^{\infty} \int_{B_i(y, 2^i \sqrt t) \setminus B_i(y, 2^{i-1} \sqrt t)} e^{-c''\frac{d(x,y)^2}{t}} d\mu(x) \\
%     & \lesssim V_i(y, \sqrt t).
%\end{align}}
Note that from a standard method of decomposition in annuli, it holds %for any $\az >0$,
\begin{eqnarray}\label{eq:1/v-e}
     \int_{M_i}  e^{-\az \frac{d(x,y)^2}{t}} d\mu(x)
     %&\leq \int_{B_i(y, \sqrt t)} e^{-c''\frac{d(x,y)^2}{t}} d\mu(x)
%          + \sum_{i=1}^{\infty} \int_{B_i(y, 2^i \sqrt t) \setminus B_i(y, 2^{i-1} \sqrt t)} e^{-c''\frac{d(x,y)^2}{t}} d\mu(x) \notag \\
     && \lesssim_{\az} V_i(y, \sqrt t),
     \quad \forall \az >0.
\end{eqnarray}
Hence, one has
\begin{equation*}% \label{gra}
    \int_{E_i\setminus F^{(\kz)}_i}|\nabla_xh_t(x,y)|^2 e^{c'\frac{d(x,y)^2}{t}}\,d\mu(x)
    \lesssim_\beta \frac{1}{tV_i(y,\sqrt t)}.
  %  \quad
%\forall 0<t \le \beta \kappa^2, \, y\in E_i\setminus F^{(\kappa)}_i.
\end{equation*}
%where in the penultimate inequality we used the fact that %a standard method of decomposition to deduce that
%for any $c >0$,
%\begin{align}\label{eq:1/v-e}
%     \int_{M_i} \frac{1}{V_i(y,\sqrt t)} e^{-c\frac{d(x,y)^2}{t}} d\mu(x)
%    & \lesssim   \int_{M_i} \frac{1}{V_i(y,\sqrt t+d(x,y))} \left(\frac{\sqrt t+d(x,y)}{\sqrt t}\right)^{N_i} e^{-c\frac{d(x,y)^2}{t}} d\mu(x) \nonumber \\
%    &  \lesssim  \int_{M_i} \frac{1}{V_i(y,\sqrt t+d(x,y))} \frac{\sqrt t}{\sqrt t+d(x,y)}  d\mu(x)
%   % & \quad \lesssim   \int_{B_i(y,\sqrt t)} \frac{1}{V_i(y,\sqrt t+d(x,y))} \frac{\sqrt t}{\sqrt t+d(x,y)}  d\mu(x) \nonumber \\
%%    & \quad \quad    + \frac{1}{V_i(y,\sqrt t)}  \sum_{k=1}^{\infty} \int_{B_i(y,2^k\sqrt t) \setminus B_i(y,2^{k-1}\sqrt t)} \frac{1}{V_i(y,\sqrt t+d(x,y))} \frac{\sqrt t}{\sqrt t+d(x,y)}  d\mu(x) \nonumber \\
%      \lesssim 1.
%\end{align}
%Next, %since  $d(x,y) \geq r$ provided $x \in E_i \setminus (F^{(\kz)}_i \cup B(y,r))$,
%for any $y \in E_i \setminus F^{(\kz)}_i$ and $x \in E_i \setminus (F^{(\kz)}_i \cup B(y,r))$,
%we write%decompose $e^{-c'\frac{d(x,y)^2}{t}}$ as
%$$e^{-c'\frac{d(x,y)^2}{t}} = e^{-c\frac{d(x,y)^2}{t}}e^{-c\frac{d(x,y)^2}{t}} \leq e^{-c\frac{r^2}{t}} e^{-c\frac{d(x,y)^2}{t}}.$$
%By the H\"older inequality, \eqref{est-on-diagonal-gradient} and a standard method of decomposition (cf. \eqref{eq:1/v-e}), we have for any $y \in  E_i \setminus F^{(\kz)}_i$,
Combining this estimate with the H\"older inequality,
we conclude that for any $r>0$,
\begin{eqnarray*}
&&\int_{E_i\setminus (F^{(\kz)}_i\cup B(y,r))}|\nabla_xh_t(x,y)| \,d\mu(x) \\
&& \ \le \left(\int_{E_i\setminus F^{(\kz)}_i}|\nabla_xh_t(x,y)|^2 e^{c'\frac{d(x,y)^2}{t}}\,d\mu(x)\right)^{1/2}
\left(\int_{E_i\setminus (F^{(\kz)}_i\cup B(y,r))} e^{-c'\frac{d(x,y)^2}{t}}\,d\mu(x)\right)^{1/2}\\
&& \ \lesssim_\beta \frac{1}{\sqrt{tV_i(y, \sqrt t)}}  e^{-c\frac{r^2}{t}}
          \left(\int_{M_i}  e^{-c\frac{d(x,y)^2}{t}}\,d\mu(x)\right)^{1/2}
 \lesssim_\beta \frac{1}{\sqrt t} e^{-c\frac{r^2}{t}},
 %\quad \forall 0<t \le \beta \kappa^2, \, y\in E_i\setminus F^{(\kappa)}_i,
\end{eqnarray*}
%where in the last inequality we used \eqref{eq:1/v-e}.
where the last inequality is due to \eqref{eq:1/v-e}.
%where we have used in the second inequality
%$e^{-c'\frac{d(x,y)^2}{t}} = e^{-c\frac{d(x,y)^2}{t}}e^{-c\frac{d(x,y)^2}{t}} \leq e^{-c\frac{r^2}{t}} e^{-c\frac{d(x,y)^2}{t}},$
%\[
%\int_{E_i\setminus (F^{(\kz)}_i\cup B(y,r))} e^{-c'\frac{d(x,y)^2}{t}}\,d\mu(x)
%=
%\int_{E_i\setminus (F^{(\kz)}_i\cup B(y,r))} e^{-c\frac{d(x,y)^2}{t}}e^{-c\frac{d(x,y)^2}{t}}d\mu(x)
%\leq e^{-c\frac{r^2}{t}}
%\int_{E_i\setminus (F^{(\kz)}_i\cup B(y,r))} e^{-c\frac{d(x,y)^2}{t}}\,d\mu(x),\]
%and \eqref{eq:1/v-e} in the last one.
This finishes the proof.
\end{proof}

%{\color{blue}
%\begin{rem}
% In fact, from the above proof,  provided that
%   $y \in E_i\setminus F_i^{(\delta)},$ $\sqrt t \le C \delta$ and $\delta > 2$, then
%    $$\int_{E_i\setminus (F_i^{(\delta)} \cup B(y,r))}|\nabla_x h_t(x,y)|\,d\mu(x) \lesssim \frac{1}{\sqrt t}e^{-\frac{r^2}{ct}}$$
%  still holds.
%  Indeed,
%  We just need to replace Lemma \ref{lem:es-away} with Remark \ref{rem:es-away-1} in the proof
%  and reconstruct $\psi_i$ as follows:
%  $\psi_i\equiv 1$ on $E_i\setminus F_i^{(\delta)}$ with $\supp \psi_i\subset \{x\in E_i;\dist(x,E_i\setminus F_i)< \delta\}$ and $|\nabla \psi_i|\le C/\delta$.
% \end{rem}
%}

%\subsection{Estimates of terms in the upper bound} \label{sec:es-terms} \hskip\parindent
%    %In this subsection, we will give some estimates for the terms in the heat kernel upper bound \eqref{heat-mixed}.
%    In this subsection, we will give some estimates
%  %  for the terms in the right-hand side of \eqref{heat-mixed}.
%    for the terms in the heat kernel upper bounds. %upper bound of heat kernel in Theorem \ref{thm:weighted-upper bound}.
%
%
%\subsubsection{Estimates of $\widetilde{H}(x,t)$} \hskip\parindent

\section{Basic estimates for the heat semigroup} \hskip\parindent
%In this section, we first provide the $L^p$-Davies-Gaffney estimates for the operators
%$e^{-t \L}$, $\nabla e^{-t \L}$ and $\L e^{-t \L}$.
%Then, we will give the ultracontractivity of the heat semigroup.
 In this section, we collect some basic estimates for the heat semigroup. %, including the Davies-Gaffney estimates and the ultracontractivity.
\subsection{Davies-Gaffney estimates} \label{sec:DG-es} \hskip\parindent
 In this part, the $L^p$-Davies-Gaffney estimates for the operators
$e^{-t \L}$, $\nabla e^{-t \L}$ and $\L e^{-t \L}$ are provided.
For $1\le p<\infty$, we say that an operator $T$ satisfies the $L^p$-Davies-Gaffney estimate, if there exist constants $C,c>0$ such that
for any closed sets $E,F\subset M$, it holds
\begin{eqnarray*}
\|T(f\chi_E)\|_{L^p(F)}\le C
%\exp\left(-c\frac{\dist(E,F)^2}{t}\right)
e^{-c\frac{\dist(E,F)^2}{t}}
\|f\|_{L^p(E)}.
\end{eqnarray*}
When $p=2$, we shall say that $T$ satisfies the Davies-Gaffney estimate for short.

%We need the following observation, which is a direct consequence of
In what follows we will need the following observation, which is a direct consequence of \cite[Therorem 4.1]{cd03} and \cite[Propposition 3.6]{CS10}.
%The following result is a direct consequence of \cite[Therorem 4.1]{cd03} and \cite[Propposition 3.6]{CS10}, which is needed in what follows.
\begin{prop}\label{mapping-gradient-heat}
The operator $\nabla e^{-t\L}$ is bounded on $L^p(M)$ for $1<p\le 2$ with
$$
\|\nabla e^{-t\L}\|_{p\to p}
\lesssim_p \frac{1}{\sqrt t},
\quad
\forall  t>0.$$
\end{prop}
\begin{proof}By \cite[Theorem 4.1]{cd03}, for any $f\in C^\infty_0(M)$ and $1 < p \leq 2$, we have
$$\|\nabla e^{-t\L}f\|_{p}^2
\lesssim_p \|e^{-t\L}f\|_{p}\|\L e^{-t\L}f\|_{p}
\lesssim_p \frac{1}{t}\|f\|_{p}^2,$$
%which completes the proof.%implies the proposition.
where the last inequality is due to the classical Littlewood-Paley-Stein theory (cf. \cite{ste70}).
This completes the proof.
\end{proof}
%We have
%$$\|e^{-t\L}\|_{m \rightarrow m} \lesssim 1, $$ for any $1\leq m \leq \infty$;
%see Davies [Heat kernels and Spectral theory, Theorem 1.3.3].

The following result was proved in \cite[p. 930, (3.1)]{acdh}; see also \cite{CS}.

\begin{prop}\label{davies-gaffney-off}
The operators  $e^{-t\L}$, $\sqrt t\nabla e^{-t\L}$ and $t\L e^{-t\L}$ satisfy the Davies-Gaffney estimate.
\end{prop}

Using the Riesz-Thorin interpolation theorem, we deduce the following $L^p$-Davies-Gaffney estimates.
\begin{cor}\label{davies-operators}
The operators $e^{-t\L}$, $\sqrt t\nabla e^{-t\L}$ and $t\L e^{-t\L}$  satisfy $L^p$-Davies-Gaffney estimate for $1<p\le 2$.
\end{cor}
\begin{proof}
Note that $e^{-t\L}$  and $t\L e^{-t\L}$ are $L^p$-bounded for all $1 < p < \infty$,
and by Proposition \ref{mapping-gradient-heat}, $t\nabla e^{-t\L}$ is bounded on $L^p(M)$ for $1<p\le 2$.
From this, together with the Riesz-Thorin theorem and Proposition \ref{davies-gaffney-off}, we deduce the desired
estimate for $1<p<2$ as
\begin{eqnarray*}
&&\|e^{-t\L}(f\chi_E)\|_{L^p(F)}+\|\sqrt t\nabla e^{-t\L}(f\chi_E)\|_{L^p(F)}+\|t\L e^{-t\L}(f\chi_E)\|_{L^p(F)}\\
&& \ \le
%C(p)\exp\left(-c(p)\frac{\dist(E,F)^2}{t}\right)\|f\|_{L^p(E)}.
C(p) e^{-c(p)\frac{\dist(E,F)^2}{t}} \|f\|_{L^p(E)}.
\end{eqnarray*}
This completes the proof.
\end{proof}
From Corollary \ref{davies-operators} and \cite[Proposition 3.1]{Au07},
it follows that the composition of $\sqrt{t} \nabla e^{-t \L}$ and $t \L e^{-t \L}$ satisfies:

\begin{cor} \label{cor:com-DG}
Let $1<p \leq 2$. The operator $t^{3 / 2} \nabla \mathcal{L} e^{-t \mathcal{L}}$ satisfies the $L^p$-Davies-Gaffney estimate.
\end{cor}

\subsection{The ultracontractivity of the heat semigroup with large time} \hskip\parindent
We have the following ultracontractivity of the heat semigroup.
\begin{prop} \label{prop:ultcon}
Let $1 \leq p \leq q \leq \infty$. It holds that %for any $t \ge 1/2$,
  \begin{equation*}%\label{heat-operator-norm-1}
\|e^{-t \L}\|_{p \rightarrow q} \lesssim t^{-\delta(\frac{1}{p}-\frac{1}{q})}, \quad \forall t \geq \frac{1}{2}.
%\, 1 \leq p \leq q \leq \infty.
\end{equation*}
where %$2\delta:=\min\{2, \min_{1 \leq i \leq \ell} \delta_i\}.$
 $\delta:=\min_{1 \leq i \leq \ell}\{ \delta_i/2\}$
 % Ϊʲô��ҪС��1�� arxiv���Ⱥ��Ͻ���1/t��һ��. ��ʱû�У� ����С��1Ӧ������Ȼ��
 and $\delta_i$ is as in \eqref{rev-d}.
\end{prop}
\begin{proof}
%For any $t \ge 1/2$,
By the doubling property of $\mu_i$ and Lemma \ref{lem:volume-growth}, one has
%$$V_0(\sqrt t) %\gtrsim V_0(2 \sqrt t)
%= \min_{1\le i\le \ell} V_i(\sqrt t)
%\gtrsim \min_{1\le i\le \ell} V_i(2 \sqrt t)
%\gtrsim t^{\min\{n_i/2\}} \gtrsim t^{\min\{\delta_i/2\}} \gtrsim t^{\delta}. %, %\quad \forall t \ge \frac12.
%$$
%\[ V_0(\sqrt t) %\gtrsim V_0(2 \sqrt t)
%= \min_{1\le i\le \ell} V_i(\sqrt t)
%\gtrsim \min_{1\le i\le \ell} V_i(2 \sqrt t)
%    \gtrsim t^{\min\limits_{1\le i\le \ell}\{\frac{n_i}2\}} \gtrsim t^{\min\limits_{1\le i\le \ell}\{\frac{\delta_i}2\}} \gtrsim t^{\delta}.
%    %\gtrsim \sqrt{t}^{\min\limits_{1\le i\le \ell}\{n_i\}} \gtrsim \sqrt{t}^{\min\limits_{1\le i\le \ell}\{\delta_i\}} \gtrsim t^{\delta}.
%\]
%\[
%V_0(\sqrt t) %\gtrsim V_0(2 \sqrt t)
%= \min_{1\le i\le \ell} V_i(\sqrt t)
%\gtrsim \min_{1\le i\le \ell} V_i(2 \sqrt t)
%    \gtrsim t^{\min\limits_{1\le i\le \ell}\{n_i/2\}} \gtrsim t^{\min\limits_{1\le i\le \ell}\{\delta_i/2\}} \gtrsim t^{\delta}.
%\]
\[
V_0(\sqrt t) %\gtrsim V_0(2 \sqrt t)
= \min_{1\le i\le \ell} V_i(\sqrt t)
\gtrsim \min_{1\le i\le \ell} V_i(2 \sqrt t)
    \gtrsim {(\sqrt t)}^{\min\limits_{1\le i\le \ell}\{n_i\}}
    \gtrsim {(\sqrt t)}^{\min\limits_{1\le i\le \ell}\{\delta_i\}} \gtrsim t^{\delta}.
\]
%\[
%\frac{1}{V_0(\sqrt t)} %\gtrsim V_0(2 \sqrt t)
%= \frac{1}{\min_{1\le i\le \ell} V_i(\sqrt t)}
%\lesssim \frac{1}{\min_{1\le i\le \ell} V_i(2 \sqrt t)}
%\lesssim     \frac{1}{{(\sqrt t)}^{\min\limits_{1\le i\le \ell}\{n_i\}}}
%    \lesssim \frac{1}{{(\sqrt t)}^{\min\limits_{1\le i\le \ell}\{\delta_i\}}} \lesssim \frac1{t^{\delta}}.
%\]
For any $1 \leq i \leq \ell$, from \eqref{rev-d} and the non-collapsing condition, we deduce that
%\[
%    \frac{1}{V_i(x, \sqrt t)}
%    %\lesssim \frac{1}{V_i(x, 2 \sqrt t)}
%    \lesssim \frac{1}{t^{\delta_i/2}} \frac{1}{V_i(x,1)}
%    %\lesssim \max_{u \in M_i} \frac{1}{V(u,1)} \frac{1}{t^{\delta_i/2}}
%    \lesssim \frac{1}{t^{\delta_i/2}}
%    \lesssim \frac{1}{t^{\delta}}. %\quad
%    %\forall t \ge \frac12,
%\]
\[
    {V_i(x, \sqrt t)}
    %\lesssim \frac{1}{V_i(x, 2 \sqrt t)}
    \gtrsim {t^{\delta_i/2}} {V_i(x,1)}
    %\lesssim \max_{u \in M_i} \frac{1}{V(u,1)} \frac{1}{t^{\delta_i/2}}
    \gtrsim {t^{\delta_i/2}}
    \gtrsim {t^{\delta}},
    \quad \forall x \in E_i.
     %\quad
    %\forall t \ge \frac12,
\]
%\[
%    \frac{1}{V_i(x, \sqrt t)}
%    %\lesssim \frac{1}{V_i(x, 2 \sqrt t)}
%    \lesssim \frac{1}{V_i(x,1)} \frac{1}{t^{\frac{\delta_i}2}}
%    %\lesssim \max_{u \in M_i} \frac{1}{V(u,1)} \frac{1}{t^{\delta_i/2}}
%    \lesssim \frac{1}{t^{\frac{\delta_i}2}}
%    \lesssim \frac{1}{t^{\delta}}, %\quad
%    %\forall t \ge \frac12,
%\]
%where we used the fact that $V(x, 1)=V_i(x, 1)$ whenever $B_i(x, 1) \subset E_i$, otherwise $V(x, 1) \sim V_i(x, 1)$
%(see \cite[p. 1945]{gri-sal09} and Remark \ref{rem:thr} (ii)).
    Combining the above two estimates with Lemma \ref{lem:for-ultra}, it holds that
    %for any $x \in M$, %and $t \ge 1/2$,
    \[
        h_t(x,x)
        \lesssim {t^{-\delta}},
        \quad  \forall x\in M, %, \, t \geq \frac{1}{2},
    \]
    which together with the semigroup property (cf. \cite[(5.2)]{gri99-es}) gives that
    %for any $x, y \in M$, %and $t \ge 1/2$,
    \[
        h_t(x,y) \lesssim t^{-\delta},
        \quad \forall x,y \in M, %, \, t \ge\frac12,
    \]
   and hence
$$
\| e^{-t \L}\|_{1 \rightarrow \infty} \lesssim t^{-\delta}. %\quad \forall t \geq \frac{1}{2}.
$$
  From this, using the Riesz-Thorin interpolation theorem and the contraction property of the heat semigroup, we get the following ultracontractivity:
\begin{equation*}%\label{heat-operator-norm-1}
\|e^{-t \L}\|_{p \rightarrow q} \lesssim t^{-\delta(\frac{1}{p}-\frac{1}{q})},
%\quad \forall t \geq \frac{1}{2}, \, 1 \leq p \leq q \leq \infty,
\quad \forall t  \geq \frac{1}{2},
\end{equation*}
%for any $t \ge 1/2$,
which completes the proof.
%as desired.
%\\--------------------------------------------------------------------------\\
%The details is as follows:
%We have
%$$\|e^{-t\L}\|_{m \rightarrow m} \lesssim 1, $$ for any $1\leq m \leq \infty$;
%see Davies [Heat kernels and Spectral theorey, Theorem 1.3.3].
%Then
%\begin{equation*}
%\begin{dcases}
%    \frac{1-\theta}{m} + \frac{\theta}{1} = \frac{1}{p} \\
%    \frac{1-\theta}{m} + \frac{\theta}{\infty} =\frac{1}{q}.
%\end{dcases}
%\end{equation*}
%Therefore, $\theta = \frac{1}{p} - \frac{1}{q}$.
%\\----------------------------------------------------------
\end{proof}

\section{Mapping properties of the heat semigroup and its time derivative} \label{sec:map-proper} \hskip\parindent
In this section, we discuss the boundedness of $e^{-t\L}$ and $t\L e^{-t\L}$ from $L^p(F^{(R_i)}_i)$ to $L^2(E_j\setminus F^{(R_j)}_j)$,
%for any $1 \leq i, j\leq \ell$ and $t \ge R_j^2$,
which are obtained by %the heat kernel upper bound and the estimate of the heat kernel's time derivative, respectively.
the estimates of the heat kernel and its time derivative, respectively.
% which is straightforward to prove using \eqref{heat-mixed} in subsection \ref{sec:heat-ker-est}.
%and the boundedness of  from $L^p(F_i)$ to $L^2(E_j\setminus F_j)$, which follows mainly from the estimate of the inverse of the heat kernel time in subsection \ref{sec:time-der-heat-es}.
%{\color{blue}Besides, the $L^1$-estimate of the space derivative of the heat kernel is discussed.}
%%
%
%

\subsection{Mapping properties of the heat semigroup} \hskip\parindent
%{\color{blue}In this subsection, we discuss the boundedness of $e^{-t\L}$ and $t\L e^{-t\L}$ from $L^p(F_i)$ to $L^2(E_j\setminus F_j)$.
%They are deduced by the heat kernel upper bound \eqref{heat-mixed} in Subsection \ref{sec:heat-ker-est} and the estimate of  the heat kernel's time derivative in Subsection \ref{sec:time-der-heat-es}, respectively.}
% which is straightforward to prove using \eqref{heat-mixed} in subsection \ref{sec:heat-ker-est}.
%and the boundedness of  from $L^p(F_i)$ to $L^2(E_j\setminus F_j)$, which follows mainly from the estimate of the inverse of the heat kernel time in subsection \ref{sec:time-der-heat-es}.
%We first give an estimate that will be used in the proof of the Mapping properties.
%    The following proposition follows directly from
    %By the heat kernel upper bound \eqref{heat-mixed} in Subsection \ref{sec:heat-ker-est},
  The following mapping properties follow directly from Theorem \ref{thm:weighted-upper bound}.
\begin{prop}\label{map-heat-parabolic}
Let $1 \leq i, j \leq \ell$  and $ 1 \leq p < 2$. It holds that % and any $t \ge R_j^2$, it holds that:
 %and some given constants $c', C' >0$, there exists a constant $C>0$ such that for any $t>R_j^2$, it holds
$$\|e^{-t\L}\|_{L^p(F^{(R_i)}_i)\to L^2(E_j\setminus F^{(R_j)}_j)}
\lesssim \frac{R_j}{\sqrt t}\mu(F^{(R_i)}_i)^{\frac 12-\frac 1p},
\quad \forall t \ge R_j^2,$$
and
\[
    \|e^{-t\L}\|_{L^p(F^{(R_i)}_i)\to L^\infty(E_j\setminus F^{(R_j)}_j)}
   \lesssim   \frac{R_j^2}{t}  \mu(F^{(R_i)}_i)^{- \frac1p}, \quad \forall t \ge R_j^2.
\]
%and for $\widetilde F^{(R_j)}_j:=\{x\in E_j:\, c'R_j\le |x|\le C'R_j\}$ that
%$$\|e^{-t\L}\|_{L^p(F^{(R_i)}_i)\to L^2(\widetilde F^{(R_j)}_j)}\le
%\frac{CR_j^{2}}{t}\mu(F^{(R_i)}_i)^{\frac 12-\frac 1p},$$
\end{prop}

\begin{proof}
%Recall that, by Corollary \ref{cor:phi-H}, for any $x \in M$ and $t>0$, it holds
%\[
%    \phi_{i_x}(|x|) \widetilde{H}(x,t) \lesssim 1.
%\]
%Sometimes we  use it directly, without specifying which lemma it come from.

By Theorem \ref{thm:weighted-upper bound} and Remark \ref{rem:new-dist}, we obtain that for any $x \in E_j\setminus F^{(R_j)}_j$, $y \in F^{(R_i)}_i$ and $t \ge R_j^2$,
\begin{eqnarray*}
        h_t(x,y)
        &&\lesssim \phi_{j}(|x|)\phi_{i}(|y|)\left(\frac{\widetilde{H}(x,t)\widetilde{H}(y,t)}{\widetilde{V}_0(\sqrt t)}+\frac{\widetilde{H}(y,t)}{\widetilde{V}_{j}(\sqrt t)}+\frac{\widetilde{H}(x,t)}{\widetilde{V}_{i}(\sqrt t)}\right)
        e^{-c\frac{|x|^2 + |y|^2}{t}} \nonumber\\
        && \ +\frac{\phi_{j}(|x|)\phi_{i}(|y|)}{\sqrt{\widetilde{V}_{j}(x,\sqrt t)\widetilde{V}_{i}(y,\sqrt t)}} e^{-c\frac{d_{\emptyset}(x,y)^2}{t}}
        = : \mathrm{I} + \mathrm{II}.
\end{eqnarray*}

For the term I, when $n_j >2$, it follows from
Lemma \ref{H-equal-general}, the fact $|x| \ge \dist(x, E_0)> R_j$ and Lemma \ref{lem:volume-growth} (i) that
\[
    \widetilde{H}(x,t)
   \lesssim \frac{|x|^2}{V_j(|x|)}
   \lesssim \frac{R_j^2}{V_j(R_j)},
\]
which, together with the fact $\phi_j(|x|) \sim 1$ and
$\phi_i(|y|) \widetilde{H}(y,t) \lesssim 1$ (cf. Lemma \ref{lem:phi-r-equal} and Corollary \ref{cor:phi-H}), gives that
\begin{eqnarray*}
  \rm{I}
  % && \sim \phi_{j}(|x|)\phi_{i}(|y|)\left(\frac{\widetilde{H}(x,t)\widetilde{H}(y,t)}{\widetilde{V}_0(\sqrt t)}+\frac{\widetilde{H}(y,t)}{\widetilde{V}_{j}(\sqrt t)}+\frac{\widetilde{H}(x,t)}{\widetilde{V}_{i}(\sqrt t)}\right)
%    e^{-c\frac{|x|^2 + |y|^2}{t}}\\
% && \sim \phi_{i}(|y|)\left(\frac{\widetilde{H}(x,t)\widetilde{H}(y,t)}{\widetilde{V}_0(\sqrt t)}+\frac{\widetilde{H}(y,t)}{\widetilde{V}_{j}(\sqrt t)}+\frac{\widetilde{H}(x,t)}{\widetilde{V}_{i}(\sqrt t)}\right)
%    e^{-c\frac{|x|^2 + |y|^2}{t}}\\
%  && \sim \phi_{i}(|y|) \frac{\widetilde{H}(y,t)}{\widetilde{V}_0(\sqrt t)} \widetilde{H}(x,t) +
%     \phi_{i}(|y|) \frac{\widetilde{H}(y,t)}{\widetilde{V}_{j}(\sqrt t)}
%   +\frac{\phi_{i}(|y|)}{\widetilde{V}_{i}(\sqrt t)} \widetilde{H}(x,t)
%    e^{-c\frac{|x|^2 + |y|^2}{t}}\\
  && \lesssim \frac{1}{\widetilde{V}_0(\sqrt t)} \frac{R_j^2}{V_j(R_j)}
     +\frac{1}{\widetilde{V}_{j}(\sqrt t)}
     + \frac{\phi_{i}(|y|)}{\widetilde{V}_{i}(\sqrt t)} \frac{R_j^2}{V_j(R_j)} e^{-c\frac{|y|^2}{t}}.
\end{eqnarray*}
From this, Lemma \ref{V0-general-equal}, the fact
$t/\widetilde{V}_{j}(\sqrt t) \sim t/V_j(\sqrt t) \lesssim R_j^2/V_j(R_j)$ (cf. Remark \ref{rem:thr} (i) and Lemma \ref{lem:volume-growth} (i))
and Lemma \ref{lem:phi/v-e}, one concludes that
\[
    {\rm{I}}
    %\lesssim \frac{1}{t} \frac{R_j^2}{V_j(R_j)}
%        +\frac{1}{\widetilde{V}_{j}(\sqrt t)}
    \lesssim \frac{1}{t} \frac{R_j^2}{V_j(R_j)} + \frac{1}{t} \frac{R_j^2}{V_j(R_j)} +\frac{1}{t} \frac{R_j^2}{V_j(R_j)}
%    \lesssim \frac{1}{t} \frac{R_j^2}{V_j(R_j)}
%        +\frac{1}{{V}_{j}(\sqrt t)}
    \lesssim \frac{1}{t} \frac{R_j^2}{V_j(R_j)}.
\]
%where we used Remark \ref{rem:thr} (i) in the second inequality and Lemma \ref{lem:volume-growth} (i) in the last inequality.
%{\color{blue}where the last inequality is due to Remark \ref{rem:thr} (i) and Lemma \ref{lem:volume-growth} (i). }
When $n_j \leq 2$, we deduce from
Corollary \ref{cor:phi-H} that
\begin{eqnarray*}
  \rm{I}
%  && \sim \phi_{j}(|x|)\phi_{i}(|y|)\left(\frac{\widetilde{H}(x,t)\widetilde{H}(y,t)}{\widetilde{V}_0(\sqrt t)}+\frac{\widetilde{H}(y,t)}{\widetilde{V}_{j}(\sqrt t)}+\frac{\widetilde{H}(x,t)}{\widetilde{V}_{i}(\sqrt t)}\right)
%     e^{-c\frac{|x|^2 + |y|^2}{t}} \\
   && \lesssim \frac{1}{\widetilde{V}_0(\sqrt t)} + \frac{\phi_{j}(|x|)}{\widetilde{V}_j(\sqrt t)} e^{-c\frac{|x|^2}{t}}
     + \frac{\phi_{i}(|y|)}{\widetilde{V}_i(\sqrt t)} e^{-c\frac{|y|^2}{t}},
%    %& \lesssim \frac{1}{\widetilde{V}_0(\sqrt t)} + \frac{1}{t} + \frac{1}{t}
%     \lesssim \frac{1}{t}
%     \lesssim \frac{1}{t} \frac{R_j^2}{R_j^{n_j}}
%     \sim \frac{1}{t} \frac{R_j^2}{V_j(R_j)},
\end{eqnarray*}
which, together with Lemma \ref{V0-general-equal}, Lemma \ref{lem:phi/v-e}, the fact $R_j \ge 1$ and Lemma \ref{lem:volume-growth} (i), implies that
\[
    {\rm{I}}
    \lesssim \frac{1}{t}
     \lesssim \frac{1}{t} \frac{R_j^2}{R_j^{n_j}}
     \sim \frac{1}{t} \frac{R_j^2}{V_j(R_j)}.
\]

For the term II, %by Remark \ref{rem:new-dist}, we have $d_{\emptyset}(x,y) = \infty$ when $i \neq j $
%it follows from Remark \ref{rem:new-dist} that $d_{\emptyset}(x,y) = \infty$ when $i \neq j $.
%note that
when $i \neq j $, Remark \ref{rem:new-dist} yields that $d_{\emptyset}(x,y) = \infty$.
%Therefore, we only need consider the case $i = j$.
Therefore, it suffices to consider the case $i = j$.
Notice that, for any %$1 \leq i =j \leq \ell$,
$y \in F^{(R_j)}_j$ and $t \ge R_j^2$, we have
$$|y| \le
\dist(y,E_0) +\diam(E_0) \lesssim
%R_i =
R_j \lesssim \sqrt t.$$
When $n_j \ge 2$, by Lemma \ref{lem:V-tildeV},
%it follows  from Lemma \ref{lem:V-tildeV} that
one obtains  that
\begin{eqnarray*}
    \mathrm{II}
    && \lesssim \frac{\phi_{j}(|x|)\phi_{j}(|y|)}{\sqrt{\widetilde{V}_{j}(x,\sqrt t)\widetilde{V}_{j}(y,\sqrt t)}} e^{-c\frac{d(x,y)^2}{t}}
    %& =\frac{\phi_{j}(|x|)\phi_{j}(|y|)}{ \sqrt{\widetilde{V}_j(x , \sqrt t)}  \sqrt{\widetilde{V}_j(y, \sqrt t)}}    \exp\left(-c\frac{d(x,y)^2}{t}\right) \\
     \lesssim \frac{1}{ \sqrt{ {V}_j(x , \sqrt t) {V}_j(y, \sqrt t)}}    e^{-c\frac{d(x,y)^2}{t}}, %\\
%    && \lesssim \frac{1}{ {V}_j(x , \sqrt t) }
%     \lesssim \frac{1}{ {V}_j(\sqrt t) }
%     \lesssim \frac{1}{t} \frac{R_j^2}{V_j(R_j)},
\end{eqnarray*}
%where in the third inequality, the fourth inequality, we used the standard track of the doubling property (cf. \eqref{eq:change-centre}),
%and in the last inequality, we used the fact that $V_j(\sqrt t) / V_j(R_j) \gtrsim (\sqrt t/R_j)^{n_j} \gtrsim (\sqrt t/R_j)^2$.
which, together with the standard trick of doubling property (cf. \eqref{eq:change-centre})
and Lemma \ref{lem:volume-growth} (i), %the fact that $V_j(\sqrt t) / V_j(R_j) \gtrsim (\sqrt t/R_j)^{n_j} \gtrsim (\sqrt t/R_j)^2$,
implies that
\[
    \mathrm{II} \lesssim \frac{1}{ {V}_j(y , \sqrt t) }
     \lesssim \frac{1}{ {V}_j(\sqrt t) }
     \lesssim \frac{1}{t} \frac{R_j^2}{V_j(R_j)}.
\]
%Similarly, if $n_j <2$, by Lemma \ref{lem:V-tildeV} and \eqref{eq:change-centre}, it holds
Similarly, when $n_j <2$, it follows from Lemma \ref{lem:V-tildeV} and \eqref{eq:change-centre} that
\begin{eqnarray*}
    \rm{II}
     && \lesssim
     \frac{\phi_{j}(|x|)\phi_{j}(|y|)}{\sqrt{ \widetilde{V}_{j}(x,\sqrt t) \widetilde{V}_{j}(y,\sqrt t)}}    e^{-c\frac{d(x,y)^2}{t}}
      \lesssim \frac{1}{\sqrt{V_j(x,\sqrt t)}} \frac{\phi_j(|y|)}{\sqrt{\widetilde{V}_j(y, \sqrt t)}} e^{-c\frac{d(x,y)^2}{t}}\\
     && \lesssim \frac{1}{\sqrt{{V}_j(y,\sqrt t)}} \frac{\phi_j(|y|)}{\sqrt{\widetilde{V}_j(y, \sqrt t)}}
       \lesssim \frac{1}{\sqrt{{V}_j(\sqrt t)}} \frac{\phi_j(|y|)}{\sqrt{\widetilde{V}_j(\sqrt t)}},
%     && \lesssim \frac{R_j^{2-n_i}}{t^{(4-n_j)/4} t^{n_j/4}}
%         \lesssim \frac{1}{t} \frac{R_j}{V_j(R_j)}.
\end{eqnarray*}
%
%\begin{eqnarray*}
%    \rm{II}
%     && \sim \frac{\phi_{j}(|x|)}{\widetilde{V}_{j}(x,\sqrt t)^{1/2}} \frac{\phi_{j}(|y|)}{\widetilde{V}_{j}(y,\sqrt t)^{1/2}} e^{-c\frac{d(x,y)^2}{t}}
%      \lesssim \frac{\phi_j(|x|)}{\widetilde{V}_j(x,\sqrt t)^{1/2}} \frac{1}{V_j(y, \sqrt t)^{1/2}} e^{-c\frac{d(x,y)^2}{t}}\\
%     && \lesssim \frac{\phi_j(|x|)}{\widetilde{V}_j(x,\sqrt t)^{1/2}} \frac{1}{V_j(x, \sqrt t)^{1/2}}
%       \lesssim \frac{\phi_j(|x|)}{\widetilde{V}_j(\sqrt t)^{1/2}} \frac{1}{V_j(\sqrt t)^{1/2}},
%%     && \lesssim \frac{R_j^{2-n_i}}{t^{(4-n_j)/4} t^{n_j/4}}
%%         \lesssim \frac{1}{t} \frac{R_j}{V_j(R_j)}.
%\end{eqnarray*}
which together with %the facts
%$\phi_j(|x|) \sim |x|^{2-n_j} \lesssim R_j^{2-n_j}$ and $\widetilde{V}_j(\sqrt t) \sim t^{(4-n_j)/2}$
%%(cf.
Lemmas \ref{lem:volume-growth}, \ref{lem:phi-r-equal}, \ref{lem:tilde-V}
and the fact $|y| \lesssim R_j$ %)
implies that
\[
    \mathrm{II}
    \lesssim \frac{1}{t^{n_j/4}} \frac{|y|^{2-n_j}}{t^{(4-n_j)/4} }
    %\lesssim \frac{R_j^{2-n_j}}{t^{(4-n_j)/4} } \frac{1}{t^{n_j/4}}
    \lesssim \frac{R_j^{2-n_j}}{t}
    \sim \frac{1}{t} \frac{R_j^2}{V_j(R_j)}.
\]

In conclusion, combining the estimates of $\mathrm{I}$ and $\mathrm{II}$, we have for any  $t \ge R_j^2$,
\begin{equation*}
\|e^{-t\L}\|_{L^1(F^{(R_i)}_i)\to L^\infty(E_j\setminus F^{(R_j)}_j)}\lesssim \frac{1}{t} \frac{R_j^2}{V_j(R_j)}.
\end{equation*}
This together with the fact $\|e^{-t\L}\|_{L^1 \rightarrow L^1} \lesssim 1$ and the Riesz-Thorin interpolation theorem implies that
\[
    \|e^{-t\L}\|_{L^1(F^{(R_i)}_i)\to L^2(E_j\setminus F^{(R_j)}_j)} \lesssim \frac{1}{\sqrt t} \frac{R_j}{\sqrt{V_j(R_j)}}.
\]
%Combining each of the two inequalities above with the H\"{o}lder inequality, we deduce from the fact
%$\mu(F^{(R_i)}_i) \sim V_j(R_j)$ (cf. \eqref{eq:fifj-con}) that
%Combining each of the  two inequalities above
From the two inequalities above, the H\"{o}lder inequality and the fact
$\mu(F^{(R_i)}_i) \sim V_j(R_j)$ (cf. \eqref{eq:fifj-con}), we deduce that
\[
    \|e^{-t\L}\|_{L^p(F^{(R_i)}_i)\to L^\infty(E_j\setminus F^{(R_j)}_j)}
    \lesssim \frac{1}{t} \frac{R_j^2}{V_j(R_j)} \mu(F^{(R_i)}_i)^{1- \frac1p}
    \sim  \frac{R_j^2}{t}  \mu(F^{(R_i)}_i)^{- \frac1p}
\]
and
\begin{equation*} %\label{eq:1inf-p2}
    \|e^{-t\L}\|_{L^p(F^{(R_i)}_i)\to L^2(E_j\setminus F^{(R_j)}_j)}
    \lesssim \frac{1}{\sqrt t} \frac{R_j}{\sqrt{V_j(R_j)}} \mu(F^{(R_i)}_i)^{1- \frac1p}
    \sim \frac{R_j}{\sqrt t} \mu(F^{(R_i)}_i)^{\frac{1}{2}-\frac{1}{p}},
\end{equation*}
which completes the proof.
\end{proof}

\begin{rem} \rm
  In fact, from the proof above, the doubling property implies that
  the subsets $F^{(R_i)}_i$ and $E_j \setminus F^{(R_j)}_j$
  can be relaxed to the subsets $F^{(c_1 R_i)}_i$ and $E_j \setminus F^{(c_2 R_j)}_j$, for some given constants $c_1, c_2>0$.
  Moreover,  the assumption $t \ge R_j^2$ can also be relaxed to
  $t \ge c_3 R_j^2$, for some given constants $c_3>0$.
  The conclusion of Proposition \ref{map-heat-parabolic} still holds.
  % $t \ge cR_j^2$ for some given positive constant $c$, the estimates still hold.
  The same is valid for Proposition \ref{mapping-time-heat-parabolic-1} below.
\end{rem}

 We will also need the following mapping properties which follow directly from Theorem \ref{thm:sim-heat-es}.
\begin{prop} \label{prop:map-anot}
  Let $1 \leq i \leq \ell$. It holds that
  \[
      \|e^{-t \L}\|_{L^2(M \setminus E_i) \rightarrow L^{\infty}(E_i \setminus F^{(R_i)}_i)}
        \lesssim \frac{1}{\sqrt{V_i(R_i)}},
        \quad \forall t >0.
  \]
\end{prop}
\begin{proof}
 % For the term $\cg_{22}$, we  first claim that
%    \begin{equation} \label{eq:cg12cl}
%        \|e^{-t\L}\|_{L^1(M \setminus E_i) \rightarrow L^{\infty}(B_{ik})}
%        \lesssim \frac{1}{V_i(R_i)}.
%    \end{equation}
    Let $x \in E_i \setminus F^{(R_i)}_i$  and $y \in E_j$ ($0 \leq j \neq i \leq \ell$). Consider first the case  $t <1$.
    From Theorem \ref{thm:sim-heat-es} (i), we can write
    \begin{eqnarray*}
        h_{t}(x,y)
        &&
        \lesssim \frac{1}{V(x,\sqrt{t})} e^{-c\frac{d(x,y)^2}{t}}
        \lesssim \frac{1}{V(x,\sqrt{t})} e^{-c'\frac{d(x,o_i)^2}{t}} e^{-c'\frac{R_i^2}{t}} \\
        && %\lesssim \frac{1}{V(o_i,\sqrt{t})} e^{-c'\frac{R_i^2}{t}}
        \lesssim \frac{1}{V_i(\sqrt{t})}  e^{-c'\frac{R_i^2}{t}}
          \lesssim \frac{1}{V_i(R_i)} \lf( \frac{R_i}{\sqrt{t}}\r)^{N_i} e^{-c'\frac{R_i^2}{t}}
         \lesssim \frac{1}{V_i(R_i)},
    \end{eqnarray*}
  %  ========THE DETAILS OF SECOND INEQUALITY=========\\
%       (i) $d(y,x) \ge d(x,o_i)$ : \mbox{ by the facts that}
%      $$d(x,y) \ge \dist(x, E_0) \ge R_i \gtrsim \diam(E_0) $$
%      and
%      $$d(x,o_i) \leq \dist(x, E_0) + \diam(E_0) \lesssim d(x,y) + d(x,y)$$
%      (ii) $ d(x,y) \gtrsim R_i$\\
%    ============END=======================\\
    where the third inequality follows from
 %   Lemma \ref{lem:volume-growth} (iii)
 %   the fact
%%    $V(y, \sqrt{t})=V_i(y, \sqrt{t})$ whenever $B_i(y, \sqrt{t}) \subset E_i$, otherwise $V(y, \sqrt{t}) \sim V_i(y, \sqrt{t})$
      $V(x, \sqrt{t}) \sim V_i(x, \sqrt{t})$
    and a standard trick of doubling property
    (cf. \eqref{eq:change-centre}). %,  and the last inequality follows from
%    the fact that $V_i(R_i)/V_i(\sqrt{t})$ has at most polynomial growth w.r.t $R_i/\sqrt{t}$ by the doubling property.

    For the opposite case $t \ge 1$, note first that when $y \in E_0$, it holds $|y| \leq 2$ and hence
    $$\phi_0(|y|) \lesssim 1, \quad \forall y \in E_0. $$
    Then, for each $y \in E_j$ ($0\leq j \neq i \leq \ell$), it follows from Theorem \ref{thm:sim-heat-es} (iii)  and Corollary \ref{cor:phi-H} that
   \begin{eqnarray*}
               h_{t}(x,y)
        && \lesssim \phi_i(|x|) \phi_j(|y|) \lf( \frac{\widetilde{H}(x,t) \widetilde{H}(y,t)}{\widetilde{V}_0(\sqrt{t})}
                         +\frac{\widetilde{H}(y,t)}{\widetilde{V}_i(\sqrt{t})} + \frac{\widetilde{H}(x,t)}{\widetilde{V}_j(\sqrt{t})} \r)
                                      e^{-c\frac{|x|^2+|y|^2}{t}} \\
       && \lesssim
          \lf( \frac{\phi_i(|x|) \widetilde{H}(x,t) }{\widetilde{V}_0(\sqrt{t})}
                         +\frac{\phi_i(|x|)}{\widetilde{V}_i(\sqrt{t})} + \frac{\phi_i(|x|)  \widetilde{H}(x,t)\phi_j(|y|)}{\widetilde{V}_j(\sqrt{t})} \r)
                                      e^{-c\frac{|x|^2+|y|^2}{t}},
    \end{eqnarray*}
    where we used the fact that $\phi_0(|y|) \lesssim 1$ and $\widetilde{H}(y,t) \leq 1$ (see the definition of $\widetilde{H}$), when $j=0$.
    Besides, it holds by Lemma \ref{V0-general-equal}, Lemma \ref{lem:phi/v-e} and the fact $\phi_0(|y|) \lesssim 1$ when $j=0$ that
%    \begin{equation}
%      \frac{1}{\widetilde{V}_0(\sqrt{t})} e^{-c'\frac{R_i^2}{t}}
%                   + \frac{\phi_i(|u|)}{\widetilde{V}_i(\sqrt{t})} e^{-c'\frac{|u|^2}{t}}
%                       + \frac{\phi_j(|v|)}{\widetilde{V}_j(\sqrt{t})} e^{-c'\frac{|v|^2}{t}}
%                       \lesssim \frac{1}{t}.
%    \end{equation}
%    \begin{equation} \label{eq:no-c}
%      \frac{1}{\widetilde{V}_0(\sqrt{t})} e^{-c'\frac{R_i^2}{t}}
%      \lesssim \frac{1}{t},
%    \qquad
%       \frac{\phi_i(|u|)}{\widetilde{ V}_i(\sqrt{t})} e^{-c'\frac{|u|^2}{t}}
%       \lesssim \frac{1}{t},
%    \qquad
%       \frac{\phi_j(|v|)}{\widetilde{V}_j(\sqrt{t})}
%       e^{-c'\frac{|v|^2}{t}}
%       \lesssim \frac{1}{t}.
%    \end{equation}
       \begin{equation} \label{eq:no-c}
      \frac{1}{\widetilde{V}_0(\sqrt{t})}
      +
       \frac{\phi_i(|x|)}{\widetilde{ V}_i(\sqrt{t})} e^{-c'\frac{|x|^2}{t}}
      +
       \frac{\phi_j(|y|)}{\widetilde{V}_j(\sqrt{t})}
       e^{-c'\frac{|y|^2}{t}}
       \lesssim \frac{1}{t}.
    \end{equation}
    When $n_i \le 2$, using Corollary \ref{cor:phi-H} again and the fact $|x| \gtrsim R_i$, one has
    \begin{eqnarray*}
       h_{t}(x,y)
       && \lesssim \frac{1}{\widetilde{V}_0(\sqrt{t})} e^{-c'\frac{R_i^2}{t}}
                   + \frac{\phi_i(|x|)}{\widetilde{V}_i(\sqrt{t})} e^{-c'\frac{|x|^2}{t}} e^{-c'\frac{R_i^2}{t}}
                       + \frac{\phi_j(|y|)}{\widetilde{V}_j(\sqrt{t})} e^{-c'\frac{|y|^2}{t}} e^{-c'\frac{R_i^2}{t}}  \\
       && \lesssim \frac{1}{t}e^{-c'\frac{R_i^2}{t}}
          \lesssim \frac{1}{R_i^2}
          \lesssim \frac{1}{V_i(R_i)},
    \end{eqnarray*}
    where the second inequality follows from \eqref{eq:no-c}.
    When $n_i >2$, by Lemmas \ref{lem:phi-r-equal}, \ref{H-equal-general} and \ref{lem:volume-growth},  it holds $\phi_i(|x|) \sim 1$ and
    $\widetilde{H}(x,t) \sim |x|^2/V_i(|x|) \lesssim R_i^2 / V_i(R_i)$, which gives
    \begin{eqnarray*}
       h_{t}(x,y)
       && \lesssim \frac{1}{\widetilde{V}_0(\sqrt{t})} \frac{R_i^2}{V_i(R_i)} e^{-c'\frac{R_i^2}{t}}
                   + \frac{1}{\widetilde{V}_i(\sqrt{t})}  e^{-c'\frac{R_i^2}{t}}
                   + \frac{\phi_j(|y|)}{\widetilde{V}_j(\sqrt{t})} \frac{R_i^2}{V_i(R_i)} e^{-c'\frac{|y|^2}{t}} e^{-c'\frac{R_i^2}{t}}  \\
       && \lesssim \frac{1}{t} \frac{R_i^2}{V_i(R_i)} e^{-c'\frac{R_i^2}{t}}
                    + \frac{1}{\widetilde{V}_i(\sqrt{t})} e^{-c'\frac{R_i^2}{t}}
                    + \frac{1}{t} \frac{R_i^2}{V_i(R_i)} e^{-c'\frac{R_i^2}{t}} \\
     %  && \lesssim \frac{1}{V_i(R_i)} + \frac{1}{\widetilde{V}_i(\sqrt{t})} e^{-c'\frac{R_i^2}{t}}, \\
        && \lesssim \frac{1}{V_i(R_i)} +
        \frac{1}{V_i(R_i)} \frac{V_i(R_i)}{V_i(\sqrt{t})} \frac{V_i(\sqrt{t})}{\widetilde{V}_i(\sqrt{t})}
        e^{-c'\frac{R_i^2}{t}}
        \lesssim \frac{1}{V_i(R_i)},
    \end{eqnarray*}
     where the second inequality follows from \eqref{eq:no-c},
     and the last inequality follows from
      Remark \ref{rem:thr} (i) together with
    the fact that $V_i(R_i)/V_i(\sqrt{t})$ has at most polynomial growth w.r.t $R_i/\sqrt{t}$ by the doubling property.

    In conclusion, we have
       \begin{equation*} %\label{eq:cg12cl}
        \|e^{-t\L}\|_{L^1(M \setminus E_i) \rightarrow L^{\infty}(E_i \setminus F^{(R_i)}_i)}
        \lesssim \frac{1}{V_i(R_i)}.
    \end{equation*}
    Combining this with $\|e^{-t\L}\|_{L^{\infty}(M \setminus E_i) \rightarrow L^{\infty}(E_i \setminus F^{(R_i)}_i)} \lesssim 1$ and
    the Riesz-Thorin interpolation theorem,
    one concludes that
    \[
        \|e^{-t\L}\|_{L^{2}(M \setminus E_i) \rightarrow L^{\infty}(E_i \setminus F^{(R_i)}_i)} \lesssim \frac{1}{\sqrt{V_i(R_i)}},
    \]
    which completes the proof.
\end{proof}

\begin{rem} \rm
  In fact, from the proof above, the doubling property implies that
  the subset $E_i \setminus F^{(R_i)}_i$
  can be relaxed to the subset $E_i \setminus F^{(c R_i)}_i$, for some given constants $c>0$.
\end{rem}

\subsection{Mapping properties of the time derivative of the heat semigroup} \label{sec:der-mapping-pr} \hskip\parindent
%
%\subsection{Mapping properties of the heat semigroup's time derivative} \label{sec:der-mapping-pr} \hskip\parindent
%
%In this subsetion we will prove the boundedness of  $t\L e^{-t\L}$ from $L^p(F_i)$ to $L^2(E_j\setminus F_j)$.
%with the help of Propositions \ref{time-heat-parabolic} and \ref{time-heat-parabolic-2} in Subsection \ref{sec:time-der-heat-es}. %Moreover, we will use the results of Lemma \ref{phi-equal}: for any $z \in M$,
%       \begin{equation*}
%        \phi(z) \sim
%        \begin{dcases}
%            1, & \, \mbox{if} \ n_{i_{z}} >2, \\
%           \log(2+|z|), & \, \mbox{if} \ n_{i_{z}} =2, \\
%            |z|^{2-n_{i_x}}, & \, \mbox{if} \ n_{i_{x}}< 2
%        \end{dcases}
%    \end{equation*}
%and the Doob's transform.
%
%{\color{blue}In this subsection, we will deduce the estimates of the time derivative of the heat kernel and then use them to obtain the mapping property.}
%
%
%\subsubsection{Mapping properties of $t \L e^{-t \L}$ with large time} \hskip\parindent
%\subsubsection{Proof of the mapping properties} \hskip\parindent
By the estimates on the time derivative of the heat kernel on weighted manifold $\widetilde{M}$ (Proposition %\ref{time-heat-parabolic} and
 \ref{time-heat-parabolic-2}), we shall use Doob's transform,
the estimates of $\phi_i$ and $\widetilde{V}_i$ (Lemmas \ref{lem:phi-r-equal} and \ref{lem:tilde-V}) to deduce
%the estimates of the time derivative of the heat kernel, and then
the mapping properties of the heat semigroup. %as follows.
In what follows,
we set
$p_1$ as
%\begin{equation} \label{eq:p1-def}
%p_0:=\min\left\{2,\,\frac{N_k}{N_k-n_k+2} : \, 1\leq k \leq \ell, \, n_k>2\right\}.
%\end{equation}
\begin{equation*} %\label{eq:p1-def}
    p_1:=\min\left\{2,\, \min_{1\leq k \leq \ell, \, n_k >2} \frac{N_k}{N_k-n_k+2} \r\}.
\end{equation*}
%\[
%    p_0:=\min\left\{2,\, \min_{1\leq k \leq \ell, \atop n_k >2} \frac{N_k}{N_k-n_k+2} \r\}
%\]
%\[
%    p_0:=\min\left\{2,\, \min_{\substack{1\leq k \leq \ell, \\ n_k >2}} \frac{N_k}{N_k-n_k+2} \r\}
%\]
%Apparently $1<p_0\le 2$.

\begin{prop}\label{mapping-time-heat-parabolic-1}
%We set $p_0$ as
%\begin{equation} \label{eq:p1-def}
%p_0:=\min\left\{2,\,\frac{N_i}{N_i-n_i+2}\bigg|\, n_i>2\right\}.
%\end{equation} Apparently $1<p_0\le 2$.
Let $1\leq i,j \leq \ell$. For each $1 < p < p_1$, it holds
$$\|t\L e^{-t\L}\|_{L^p(F^{(R_i)}_i)\to L^2(E_j\setminus F^{(R_j)}_j)} \lesssim \mu(F^{(R_i)}_i)^{\frac 12-\frac 1p},
\quad \forall t \ge R_j^2.$$
\end{prop}

%We need the following Lemma. /
To prove Proposition \ref{mapping-time-heat-parabolic-1}, we first establish the following lemma.

\begin{lem} \label{lem:for-p2}
  Let $H$ be an operator defined on a space of measurable functions on $(M, \mu)$ and taking values in the set of all
  almost everywhere measurable functions on $(M, \mu)$.
  Assume that $E, F \subset M$, $\mu(E) < \infty$,
  $1 < p <2$ and $1 \leq q \leq p$.
  Suppose that the operator $H$ satisfies
  \[
    \|H\|_{L^p(E) \rightarrow L^p(F)} \lesssim_p 1
    %\quad \forall  1 < p <2,
  \]
  and
  \begin{equation*}%\label{eq:qinf}
    \|H\|_{L^q(E) \rightarrow L^\infty(F)} \lesssim \mu(E)^{-\frac{1}{q}}.
    %\quad \forall 1 \leq q \leq p.
  \end{equation*}
  Then we have
  %\[
%    \|T\|_{L^p(E) \rightarrow L^2(F)} \lesssim \mu(E)^{\lf(1-\frac{1}{p}\r) \lf(-\frac1q + 1 - \frac{q}p\r)}.
%  \]
   \[
    \|H\|_{L^p(E) \rightarrow L^2(F)} \lesssim_p \mu(E)^{\frac12- \frac{1}p}.
  \]
\end{lem}
\begin{proof}
  By the H\"{o}lder inequality, it holds that for any $f \in L^p(E)$,
%\[
%     \int_{F} |Tf|^2 \,d\mu
%      \lesssim \int_{F} |Tf|^{2-p} |Tf|^p  \, d\mu
%      \lesssim
%        \mu(E)^{-\frac{1}{q}(2-p)}
%         \|f\|^{2-p}_{L^q(E)}
%        \|f\|^{p}_{L^p(E)}
%     %&& \lesssim
%%     \|f\|^{p}_{L^p(E)}
%%     \mu(E)^{-\frac{1}{q}(2-p)}
%%     \|f\|^{2-p}_{L^p(E)} \mu(E)^{(2-p)(1-\frac q p)}    \\
%      \lesssim
%     \mu(E)^{(2-p) (-\frac1q + 1 - \frac{q}p)}
%     \|f\|^{2}_{L^p(E)}
%\]
  \begin{eqnarray*}
     \int_{F} |Hf|^2 \,d\mu
     && \lesssim \int_{F} |Hf|^{2-p} |Hf|^p  \, d\mu \\
    % && \lesssim \|Tf\|^{2-p}_{L^\infty(F)} \|Tf\|^p_{L^p(F)}
%         \lesssim_p  \|Tf\|^{2-p}_{L^q(E) \rightarrow L^\infty(F)} \|f\|^{2-p}_{L^q(E)}  \|f\|^{p}_{L^p(E)} \\
     && \lesssim_p
        \mu(E)^{-\frac{1}{q}(2-p)}
         \|f\|^{2-p}_{L^q(E)}
        \|f\|^{p}_{L^p(E)}
           \\
     %&& \lesssim
%     \|f\|^{p}_{L^p(E)}
%     \mu(E)^{-\frac{1}{q}(2-p)}
%     \|f\|^{2-p}_{L^p(E)} \mu(E)^{(2-p)(1-\frac q p)}    \\
     && \lesssim_p
     \mu(E)^{-\frac{1}{q}(2-p)}
     \mu(E)^{(\frac1q - \frac{1}p) (2-p) }
     \|f\|^{2-p}_{L^{p}(E)}  \|f\|^{p}_{L^p(E)} \\
     && \lesssim_p
     \mu(E)^{1- \frac{2}p}
     \|f\|^{2}_{L^p(E)},
  \end{eqnarray*}
  which finishes the proof.
%  \[
%    \|f\|_{L^q(E)}
%    = \lf(\int_E |f|^q d\mu \r)^{\frac1q}
%    \leq \lf(\int_E |f|^p d\mu \r)^{\frac1p} \lf(\int_E  d\mu \r)^{\frac1q - \frac{1}p}
%    \leq \|f\|_{L^p(E)} \mu(E)^{\frac1q-\frac{1}p}.
%  \]
  \end{proof}

\begin{proof}[Proof of Proposition \ref{mapping-time-heat-parabolic-1}]
%Recall that
%\[
%    V_i(R_i) \sim \mu(F^{(R_i)}_i) = \mu(F^{(R_j)}_j) \sim V_j(R_j).
%\]

%We will complete the proof in nine cases:

%(i) $n_i, n_j >2$;
%(ii) $n_i, n_j <2$;
%(iii) $n_i >2$, $n_j<2$;
%(iv) $n_i<2$, $n_j >2$;
%(v) $n_i=2$, $n_j>2$;
%(vi) $n_i=2$, $n_j <2$;
%(vii) $n_i=2=n_j$;
%(viii) $n_j=2$, $n_i >2$;
%(ix) $n_j=2$, $n_i <2$.
%The first four cases are proved with the help of Propositions \ref{time-heat-parabolic},
%and the last five cases are proved with the help of Propositions \ref{time-heat-parabolic-2}.

Let $1 \leq i, j \leq \ell$ and $t \ge R_j^2$.
Note that the classical Littlewood-Paley-Stein theory says that for any $1 < p <2$,
$$
\|t \L e^{-t \L}\|_{L^p(F^{(R_i)}_i) \rightarrow L^p(E_j \setminus F^{(R_j)}_j)}  \lesssim_p 1. %\quad \forall 1<p<2.
$$
By Lemma \ref{lem:for-p2}, it remains to show that the operator $t \L e^{-t \L}$ satisfies
$$
\|t \L e^{-t \L}\|_{L^1(F^{(R_i)}_i) \rightarrow L^\infty(E_j \setminus F^{(R_j)}_j)} \lesssim \mu(F^{(R_i)}_i)^{-1}
$$
or
$$
\|t \L e^{-t \L}\|_{L^p(F^{(R_i)}_i) \rightarrow L^\infty(E_j \setminus F^{(R_j)}_j)} \lesssim \mu(F^{(R_i)}_i)^{-\frac{1}{p}}.
$$

 Let $x \in E_j \setminus F^{(R_j)}_j$ and $y \in F^{(R_i)}_i$.
%We denote by $p'$ the H\"older conjugate index of $p$.
%Based on the ranges of $n_i$ and $n_j$,
We shall discuss the mapping properties of $t \L e^{-t \L}$ in the following nine cases, whose order is consistent with that of Proposition \ref{time-heat-parabolic-2}.

\textbf{Case 1: \boldmath $n_i, n_j >2$.}
If $i=j$, then for each $1 < p <2$, let us choose an $0 < \epsilon < 1/{8}$ small enough such that
$3 \epsilon /2 < 1/p'.$
%$
%    \frac32 \epsilon < \frac1{p'},
%$
Here and in what follows,
$p': = p/(p-1)$
%$p' = \frac{p}{p-1}$
is the H\"older conjugate of $p$.
It follows from Proposition \ref{time-heat-parabolic-2} (i), Lemma \ref{lem:phi-r-equal}, Lemma \ref{V0-general-equal} and Remark \ref{rem:thr} (i) that
%for any $x\in F^{(R_i)}_i$, $y\in E_i\setminus F^{(R_i)}_i$ and $t \ge R_i^2$,
\begin{eqnarray*}
|t\partial_t h_t(x,y)|\sim \phi_{i}(|x|)\phi_{i}(|y|) |t\partial_t\tilde h_t(x,y)|
%&&\lesssim \frac{1}{t} \frac{R_i^2}{\widetilde{V}_i(R_i)}\left(1+ \left(\frac{\widetilde{V}_i(R_i)}{R_i^2}\right)^{\frac{3}{2}\epsilon}
%\left(\frac{|x|^2}{\widetilde{V}_i(|x|)}\right)^{3\epsilon/2}\right)
%e^{-c\frac{d(x,y)^2}{ct}}\\
&&\lesssim  \frac{1}{{V}_i(R_i)}\left[1+ \left(\frac{{V}_i(R_i)}{R_i^2}\right)^{\frac{3}{2}\epsilon}
\left(\frac{|y|^2}{{V}_i(|y|)}\right)^{\frac{3}{2}\epsilon}\right].
%e^{-c\frac{d(x,y)^2}{t}}.
\end{eqnarray*}
%by Lemma \ref{lem:phi-r-equal} and Remark \ref{rem:thr} (i).
%Using the H\"older inequality,
Therefore, one has %for any $1< p <2$,
\begin{eqnarray*}
&& \|t\L e^{-t\L}(f\chi_{F^{(R_i)}_i})\|_{L^\infty(E_i\setminus F^{(R_i)}_i)} \\
%&&\lesssim \lf\|\int_{F^{(R_i)}_i} t\partial_t h_t(x, \cdot) f(x) \, d\mu(x) \r\|_{L^{\infty}(E_i \setminus F^{(R_i)}_i)} \\
&& \ \lesssim \frac{1}{V_i(R_i)}
\left[\|f\chi_{F^{(R_i)}_i}\|_1
         +\left(\frac{V_i(R_i)}{R_i^2}\right)^{\frac{3}{2}\epsilon} \int_{F^{(R_i)}_i}|f(y)|\left(\frac{|y|^2}{V_i(|y|)}\right)^{\frac{3}{2}\epsilon}\,d\mu(y)
\right]\\
%&& \ \lesssim  \frac{\|f\|_{L^p(F^{(R_i)}_i)}}{ V_i(R_i)}
%    \left[\mu(F^{(R_i)}_i)^{\frac1{p'}}
%        +\left(\frac{V_i(R_i)}{R_i^2}\right)^{\frac{3}{2}\epsilon}\left(\int_{F^{(R_i)}_i}\frac{|x|^{3\epsilon p'}}{V_i(|x|)^{3\epsilon p'/2}}\,d\mu(x)\right)^{\frac1{p'}}
%    \right],
&& \ \lesssim  \frac{\|f\|_{L^p(F^{(R_i)}_i)}}{ V_i(R_i)}
    \left\{\mu(F^{(R_i)}_i)^{\frac1{p'}}
        +\left(\frac{V_i(R_i)}{R_i^2}\right)^{\frac{3}{2}\epsilon}
        \left[\int_{F^{(R_i)}_i}\lf(\frac{|y|^2}{V_i(|y|)}\r)^{\frac32 \epsilon p'}\,d\mu(y)\right]^{\frac1{p'}}
    \right\},
%&&\lesssim  \frac{\|f\chi_{F^{(R_i)}_i}\|_p}{\mu(F^{(R_i)}_i)^{1/p}} .
\end{eqnarray*}
by the H\"older inequality.
Next, we split the domain of integration into %(not necessarily disjoint) sets%$F^{(R_i)}_i$ into
%$$
%F^{(R_i)}_i
%= \lf\{x \in E_i: \dist(x, E_0) \leq 2\r\}
%\bigcup \lf( \bigcup^{[\log_2R_i]}_{k=0} \{x\in E_i:2^{-k}R_i\le \dist(x,E_0)\le 2^{-k+1}R_i\} \r)
%$$
%$$
%F^{(R_i)}_i
%= \lf\{x \in E_i: \dist(x, E_0) \leq 2\r\}
%\bigcup \bigcup^{[\log_2R_i]}_{k=0} \{x\in E_i:2^{-k}R_i\le \dist(x,E_0)\le 2^{-k+1}R_i\}
%$$
$$
F^{(R_i)}_i
= \{y \in E_i: \dist(y, E_0) \leq 2\}
\cup \left( \mathop{\cup}\limits^{[\log_2R_i]}_{k=0} \{y \in E_i:2^{-k}R_i\le \dist(y,E_0)\le 2^{-k+1}R_i\} \right),
$$
%$$
%F^{(R_i)}_i
%= \{x \in E_i: \dist(x, E_0) \leq 2\}
%\cup \left( \cup^{[\log_2R_i]}_{k=0} \{x\in E_i:2^{-k}R_i\le \dist(x,E_0)\le 2^{-k+1}R_i\} \right)
%$$
%
%\[
%\{x \in E_i: \dist(x, E_0) \leq 2\}
%\]
%and
%\[
%\{x\in E_i:2^{-k}R_i\le \dist(x,E_0)\le 2^{-k+1}R_i\},
%\]
where %$0 \leq k \leq[\log_2R_i]$ and
$[\log_2R_i]$ denotes the integer part of $\log_2R_i$.
%and deduce that
%\begin{eqnarray*}
%\left(\int_{F^{(R_i)}_i}\frac{|x|^{3\epsilon p'}}{V_i(|x|)^{3\epsilon p'/2}}\,d\mu\right)^{1/p'}
% \lesssim  V_i(1)^{1/p'}
%   +   \sum_{k=0}^{[\log_2R_i]} \frac{(2^{-k}R_i)^{3\epsilon }}{V_i(2^{-k}R_i)^{3\epsilon /2}}  V_i(R_i)^{1/p'}
%%&& \lesssim
%%         1
%%       +
%%       \sum_{k=0}^{[\log_2R_i]} (2^{-k})^{3 \epsilon} R_i^{3 \epsilon} V_i(R_i)^{1/p' - 3 \epsilon/2}
%%          \\
% \lesssim    R_i^{3\epsilon}V_i(R_i)^{\frac{1}{p'}-3\epsilon/2}.
%\end{eqnarray*}
From Lemma \ref{lem:volume-growth} and the fact $3\epsilon /2 < 1/p'$,
%$\frac32  < \frac1{p'}$,
it holds that
\begin{eqnarray*} %\label{eq:decom}
%\left(\int_{F^{(R_i)}_i}\frac{|x|^{3\epsilon p'}}{V_i(|x|)^{3\epsilon p'/2}}\,d\mu\right)^{\frac1{p'}}
\left[\int_{F^{(R_i)}_i}\lf(\frac{|y|^2}{V_i(|y|)}\r)^{\frac32 \epsilon p'}\,d\mu(y)\right]^{\frac1{p'}}
%&& \ \le  \left(\int_{\{x\in E_i:\,\dist(x,E_0)\le 2\}} \,d\mu\right)^{\frac1{p'}}
%       +   \sum_{k=0}^{[\log_2R_i]}
%\left(\int_{\{x\in E_i:\,2^{-k}R_i\le \dist(x,E_0)\le 2^{-k+1}R_i\}} \frac{(2^{-k}R_i)^{3\epsilon p'}}{V_i(2^{-k}R_i)^{3\epsilon p'/2}} \,d\mu\right)^{\frac1{p'}}\\
&&  \lesssim
         V_i(1)^{\frac1{p'}-\frac{3}{2}\epsilon}
       +
       %\sum_{k=0}^{[\log_2R_i]} \frac{(2^{-k}R_i)^{3\epsilon }}{V_i(2^{-k}R_i)^{3\epsilon /2}}  V_i(2^{-k}R_i)^{\frac1{p'}}  \\
        \sum_{k=0}^{[\log_2R_i]} \lf(\frac{2^{-2k}R_i^2}{V_i(2^{-k}R_i)}\r)^{\frac32\epsilon}  V_i(2^{-k}R_i)^{\frac1{p'}}  \nonumber \\
&&  \lesssim
         V_i(1)^{\frac1{p'}-\frac{3}{2}\epsilon}
       +
       \sum_{k=0}^{[\log_2R_i]} 2^{-3\epsilon k} R_i^{3\epsilon} V_i(2^{-k}R_i)^{\frac1{p'}-\frac32 \epsilon}  \nonumber \\
&&  \lesssim    R_i^{3\epsilon}V_i(R_i)^{\frac1{p'}-\frac32 \epsilon}
 \lesssim  \left( \frac{R_i^2}{V_i(R_i)} \right)^{\frac32 \epsilon}  V_i(R_i)^{\frac1{p'}}.
\end{eqnarray*}
%Combining the above two estimates,
Therefore,  by the fact $V_i(R_i) \sim \mu(F^{(R_i)}_i)$ (cf. Lemma \ref{lem:volume-growth} (ii)),
we conclude that %for any $1 < p<2$,
%\begin{eqnarray*}%\label{case-1}
%\|t\L e^{-t\L}(f\chi_{F^{(R_i)}_i})\|_{L^\infty(E_i\setminus F^{(R_i)}_i)}
%\lesssim \frac{\|f\chi_{F^{(R_i)}_i}\|_p}{ V_i(R_i)}
%    \left(\mu(F^{(R_i)}_i)^{1/p'}
%        +\left(\frac{V_i(R_i)}{R_i^2}\right)^{\frac{3}{2} \epsilon} R_i^{3\epsilon}V_i(R_i)^{\frac{1}{p'}-3\epsilon/2} \r)
%\lesssim
% \frac{\|f\chi_{F^{(R_i)}_i}\|_p}{\mu(F^{(R_i)}_i)^{1/p}}.
%\end{eqnarray*}
\begin{eqnarray*} %\label{case-1}
\|t \L e^{-t \L}\|_{L^p(F^{(R_i)}_i) \rightarrow L^\infty(E_j \setminus F^{(R_j)}_j)}
&& \lesssim \frac{1}{ V_i(R_i)}
    \left[\mu(F^{(R_i)}_i)^{\frac1{p'}}
        + \left(\frac{V_i(R_i)}{R_i^2}\right)^{\frac{3}{2}\epsilon} \left( \frac{R_i^2}{V_i(R_i)} \right)^{\frac32 \epsilon}  V_i(R_i)^{\frac1{p'}} \r] \\
&& \lesssim \mu(F^{(R_i)}_i)^{-\frac1p},
 %\frac{1}{\mu(F^{(R_i)}_i)^{1/p}}.
\end{eqnarray*}
  which implies that
  for all $1\leq i=j \leq \ell$, $n_i, n_j >2$ and  each $1 < p<2$, %it holds
  \begin{equation*} %\label{eq:pinf-p2}
\|t \L e^{-t \L}\|_{L^p(F^{(R_i)}_i) \rightarrow L^2(E_j \setminus F^{(R_j)}_j)} \lesssim_p \mu(F^{(R_i)}_i)^{\frac{1}{2}-\frac{1}{p}},
\quad \forall  t \ge R_j^2.
  \end{equation*}
Let us consider another subcase $i \neq j$.
For each $1 < p < p_1$,  we can choose an $0<\epsilon< 1/8$ small enough
such that
\begin{equation}\label{epsilon-general}
p'(1-3\epsilon)(n_i-2)>N_i
\end{equation}
and
\begin{equation}\label{range-general-1}
\frac32\epsilon(N_\infty-2)<\frac {n_{0}-2}{p'},
\end{equation}
where %$N_\infty$ is defined in \eqref{N-infty} and
$n_0:=\min\{n_k:1\leq k \leq \ell, \, n_k>2\}>2.$
%\\-------------The reason why $\epsilon_0<\min\{\frac 18,\frac{1}{3p'}\}$-------------------\\
%
%    $\epsilon_0 < \frac18$: Therefore, we can use the previous Proposition. Notice that the estimate of the time derivative of heat kernel is due to [Davies] which request $0<\epsilon < \frac18$
%
%    $\epsilon_0<\frac{1}{3p'}$: This naturally holds, since $3 \epsilon p' < \frac{n_0-2}{N_{\infty}-2} \leq 1$. Sometimes, we use $3 \epsilon p'\leq 1$ directly in the following proofs.
%\\----------------End-----------------------------------------\\
%Above, we used the fact that
%\[\frac{1}{p_1} \geq \frac{N_i-n_i+2}{N_i}\]
%\[\frac{1}{p'_1}=1 - \frac{1}{p_1} < 1- (1 - \frac{n_i-2}{N_i}) = \frac{n_i-2}{N_i}\]
%which means that $p'>p'_1>\frac{N_i}{n_i-2}$. Therefore,  we can pick $\epsilon$ small enough to let
%$p'(1-3\epsilon)(n_i-2)>N_i$ holds for each $n_i >2$.
%
%On the other hand, if we pick $\epsilon=\epsilon(p, n_0, N_\infty)$ small enough, we have
%$3\epsilon(N_\infty-2)<\frac {n_{0}-2}{p'}$ holds.
%\\===========================\\
It follows from Proposition \ref{time-heat-parabolic-2} (ii),  Lemma \ref{lem:phi-r-equal}, Lemma \ref{V0-general-equal} and Remark \ref{rem:thr} (i) that
%for any $x\in F^{(R_i)}_i$, $y\in E_j\setminus F^{(R_j)}_j$,$t\ge R_j^2$,
\[
    |t\partial_t h_t(x,y)|
    \sim \phi_{j}(|x|)\phi_{i}(|y|) |t\partial_t\tilde h_t(x,y)|
\lesssim
\frac{1}{V_j(R_j)}
\left(\frac{{V}_j(R_j)}{R_j^2}\right)^{\frac32\epsilon}
    \lf(\frac{|y|^2}{{V}_i(|y|)} + \frac{R_j^2}{{V}_i(R_j)} \r)^{1-3\epsilon}.
\]
%{\color{gray}\begin{eqnarray*}
%|t\partial_t\tilde h_t(x,y)|
%\lesssim
%\frac{1}{V_j(R_j)}
%\left(\frac{{V}_j(R_j)}{R_j^2}\right)^{3\epsilon/2}
%\times
%\begin{cases}
%    \Big(\frac{|x|^2}{{V}_i(|x|)} \Big)^{1-3\epsilon}, & \mbox{if}\, |x| \le R_j, \vspace{0.5em} \\
%    \Big(\frac{R_j^2}{{V}_i(R_j)}\Big)^{1-3\epsilon}, & \mbox{if}\, |x| > R_j.
%\end{cases}
%\end{eqnarray*}}
%
%\begin{eqnarray*}
%|t\partial_t\tilde h_t(x,y)|
%\lesssim
%\begin{cases}
% \frac{1}{V_j(R_j)}
%\left(\frac{{V}_j(R_j)}{R_j^2}\right)^{3\epsilon/2}
%\left(\frac{|x|^2}{{V}_i(|x|)} \right)^{1-3\epsilon}
%e^{-c\frac{|x|^2+|y|^2}{t}}, & \mbox{if} |x| \le R_j, \\
% \frac{1}{V_j(R_j)}
%\left(\frac{{V}_j(R_j)}{R_j^2}\right)^{3\epsilon/2}
%\left(\frac{R_j^2}{{V}_i(R_j)}\right)^{1-3\epsilon}
%e^{-c\frac{|x|^2+|y|^2}{t}}, & \mbox{if} |x| > R_j.
%\end{cases}
%\end{eqnarray*}
%by Lemma \ref{lem:phi-r-equal} and Remark \ref{rem:thr} (i).
Using the H\"older inequality, one has %for any $1 < p < p_1$,
%\begin{eqnarray*}
%&&\left\|t\L e^{-t\L}(f\chi_{F^{(R_i)}_i})\right\|_{L^\infty(E_j\setminus F^{(R_j)}_j)}\\
%&&\lesssim  \frac{1}{V_j(R_j)} \left(\frac{V_j(R_j)}{R_j^{2}}\right)^{3\epsilon/2}\left[\left(\frac{R_j^2}{V_i(R_j)}\right)^{1-3\epsilon} \int_{\{x\in F^{(R_i)}_i: |x|>R_j\}} |f(x)|\,d\mu(x) \right.\\
%&& \quad \left. +\int_{\{x\in F^{(R_i)}_i: |x|\le R_j\}} \left(\frac{|x|^2}{V_i(|x|)}\right)^{1-3\epsilon} |f(x)| \,d\mu(x)\right]\\
%&&\lesssim  \frac{\|f\|_{L^p(F^{(R_i)}_i)}}{V_j(R_j)} \left(\frac{V_j(R_j)}{R_j^{2}}\right)^{3\epsilon/2}
%\left[\left(\frac{R_j^2}{V_i(R_j)}\right)^{1-3\epsilon}\mu(F^{(R_i)}_i)^{1/p'}+
%\left(\int_{F^{(R_i)}_i}\left(\frac{|x|^{2}}{V_i(|x|)}\right)^{p'(1-3\epsilon)}\,d\mu(x)\right)^{1/p'}\right].
%\end{eqnarray*}
\begin{eqnarray*}
&& \|t\L e^{-t\L}(f\chi_{F^{(R_i)}_i})\|_{L^\infty(E_j\setminus F^{(R_j)}_j)} \\
%&&\lesssim  \frac{1}{V_j(R_j)} \left(\frac{V_j(R_j)}{R_j^{2}}\right)^{3\epsilon/2} \left(\frac{R_j^2}{V_i(R_j)}\right)^{1-3\epsilon} \int_{\{x\in F^{(R_i)}_i: |x|>R_j\}} |f(x)|\,d\mu(x) \\
%&& \quad  + \frac{1}{V_j(R_j)}  \left(\frac{V_j(R_j)}{R_j^{2}}\right)^{3\epsilon/2}
%\int_{\{x\in F^{(R_i)}_i: |x|\le R_j\}} \left(\frac{|x|^2}{V_i(|x|)}\right)^{1-3\epsilon} |f(x)| \,d\mu(x) \\
&&\ \lesssim  \frac{1}{V_j(R_j)} \left(\frac{V_j(R_j)}{R_j^{2}}\right)^{\frac32\epsilon}
\lf[
\int_{F^{(R_i)}_i} \left(\frac{|y|^2}{V_i(|y|)}\right)^{1-3\epsilon} |f(y)| \,d\mu(y)
+ \left(\frac{R_j^2}{V_i(R_j)}\right)^{1-3\epsilon}
\int_{F^{(R_i)}_i}  |f(y)| \,d\mu(y)  \r] \\
&& \ \lesssim  \frac{\|f\|_{L^p(F^{(R_i)}_i)}}{V_j(R_j)} \left(\frac{V_j(R_j)}{R_j^{2}}\right)^{\frac32\epsilon}
    \left\{ \left[\int_{F^{(R_i)}_i} \left(\frac{|y|^{2}}{V_i(|y|)}\right)^{p'(1-3\epsilon)}\,d\mu(y)\right]^{\frac1{p'}}
           + \left(\frac{R_j^2}{V_i(R_j)}\right)^{1-3\epsilon} \mu(F^{(R_i)}_i)^{\frac1{p'}}   \right\}.
%&&\ \lesssim  \frac{\|f\|_{L^p(F^{(R_i)}_i)}}{V_j(R_j)} \mu(F^{(R_i)}_i)^{1/p'}
% \lesssim  \frac{\|f\|_{L^p(F^{(R_i)}_i)}}{V_j(R_j)} \mu(F^{(R_i)}_i)^{1/p}.
\end{eqnarray*}
From a classical annulus argument %(cf. \eqref{eq:decom})
, Lemma \ref{lem:volume-growth} and \eqref{epsilon-general}, we have%,i.e., $p'(1-3\epsilon)(n_i-2)>N_i$ that
%\[
%    \int_{F^{(R_i)}_i}\left(\frac{|x|^{2}}{V_i(|x|)}\right)^{p'(1-3\epsilon)}\,d\mu(x)
%    \lesssim \int_{F^{(R_i)}_i}\frac{1}{|x|^{p'(1-3\epsilon)(n_i-2)}}\,d\mu(x)
%    \lesssim 1.
%\]
\begin{eqnarray*}
 \int_{F^{(R_i)}_i}\left(\frac{|y|^{2}}{V_i(|y|)}\right)^{p'(1-3\epsilon)}\,d\mu(y)
%&& \lesssim \int_{F^{(R_i)}_i}\frac{1}{|y|^{p'(1-3\epsilon)(n_i-2)}}\,d\mu(y)\\
&&  \lesssim \int_{E_i}|y|^{-p'(1-3\epsilon)(n_i-2)}  \,d\mu(y)\\
&&  \lesssim \int_{E_i \cap B_i(o_i, 1)} |y|^{-p'(1-3\epsilon)(n_i-2)} \, d\mu(y) \\
&&  \  + \sum_{i=1}^{\infty} \int_{E_i \cap (B_i(o_i, 2^{i+1}) \setminus B_i(o_i, 2^{i}))} |y|^{-p'(1-3\epsilon)(n_i-2)} \, d\mu(y) \\
&& %\lesssim  1+ \sum_{i=1}^{\infty} \frac{1}{2^{ip'(1-3\epsilon)(n_i-2)}} \mu(E_i \cap B_i(o_i, 2^{i+1}))
   \lesssim  V_i(1)+ \sum_{i=1}^{\infty} \frac{1}{2^{ip'(1-3\epsilon)(n_i-2)}} V_i( 2^{i+1}) \\
&&  \lesssim  1+ \sum_{i=1}^{\infty} 2^{-i[p'(1-3\epsilon)(n_i-2)-N_i]}
    \lesssim 1.
\end{eqnarray*}
Besides, it follows from \eqref{range-general-1} and \eqref{eq:fifj-con} that%, i.e., $\frac{3}{2}\epsilon(N_\infty-2) <(n_{0}-2)/p'$  that
%\[
%   \left(\frac{V_j(R_j)}{R_j^{2}}\right)^{\frac32\epsilon}
%   %\lesssim  \left(\frac{R_j^{N_j}}{R_j^{2}}\right)^{3\epsilon/2}
%   \lesssim R_j^{\frac 32 \epsilon(N_j-2)}
%   \lesssim  R_j^{(n_0-2)/p'}
%   \lesssim  R_j^{n_{j}/p'}
%   \lesssim V_j(R_j)^{1/p'} \sim \mu(F^{(R_i)}_i)^{1/p'}.
%\]
\[
   \left(\frac{V_j(R_j)}{R_j^{2}}\right)^{\frac32\epsilon}
   %\lesssim  \left(\frac{R_j^{N_j}}{R_j^{2}}\right)^{3\epsilon/2}
   \lesssim R_j^{\frac 32 \epsilon(N_j-2)}
   \lesssim  R_j^{\frac{n_0-2}{p'}}
   \lesssim  R_j^{\frac{n_{j}}{p'}}
   \lesssim V_j(R_j)^{\frac1{p'}} \sim \mu(F^{(R_i)}_i)^{\frac1{p'}}.
\]
%{\color{gray}and $
%    \frac{3}{2} \epsilon (N_j-2) + (2-n_i)(1-3\epsilon)
%    < (n_{0}-2)/p' + (2-n_0)(1-3\epsilon)
%    %\leq (n_{0}-2) (\frac{1}{p'} + 1 - 3\epsilon)
%    < 0
%$
%which implies
%\[
%\left(\frac{V_j(R_j)}{R_j^{2}}\right)^{3\epsilon/2} \left(\frac{R_j^2}{V_i(R_j)}\right)^{1-3\epsilon}
%\lesssim R_j^{\frac 32 \epsilon(N_j-2)+(2-n_i)(1-3\epsilon)}
%\lesssim R_j^{(n_{0}-2)/p' + (2-n_0)(1-3\epsilon)}
%\lesssim 1.
%\]}
%The convention \eqref{range-general-1} also implies that
%The fact $p'>2$, {\color{blue}combined with \eqref{range-general-1}}, tells us that
By \eqref{range-general-1} again and the fact $p'>2$, it holds
$$
    %\frac{3}{2} \epsilon (N_j-2) + (2-n_i)(1-3\epsilon)
  \frac{3}2 \epsilon (N_j-2) + (2-n_i)(1-3\epsilon)
  %  < \frac{n_{0}-2}{p'} + (2-n_0)(1-3\epsilon)
    \le (n_{0}-2) \lf(\frac1{p'} - 1 +3\epsilon\r)<0.
$$
%by the fact $p' \ge 2$.
This estimate together with the fact $R_j \geq 1$ gives that
\[
\left(\frac{V_j(R_j)}{R_j^{2}}\right)^{\frac32\epsilon} \left(\frac{R_j^2}{V_i(R_j)}\right)^{1-3\epsilon}
\lesssim R_j^{\frac 32 \epsilon(N_j-2)+(2-n_i)(1-3\epsilon)}
%\lesssim R_j^{(n_{0}-2)/p' + (2-n_0)(1-3\epsilon)}
\lesssim 1.
\]
%Therefore, combining the above four estimates,
Based on the above arguments, we deduce from \eqref{eq:fifj-con} that %if $1\leq i\neq j \leq \ell$, $n_i, n_j >2$,
%for any $1 < p<p_0$,
%\[
%\left\|t\L e^{-t\L}(f\chi_{F^{(R_i)}_i})\right\|_{L^\infty(E_j\setminus F^{(R_j)}_j)}
%\lesssim  \frac{\|f\|_{L^p(F^{(R_i)}_i)}}{V_j(R_j)} \mu(F^{(R_i)}_i)^{1/p'}
%\lesssim  \frac{1}{\mu(F^{(R_i)}_i)^{1/p}} \|f\chi_{F^{(R_i)}_i}\|_{p}.
%\]
\[
\left\|t\L e^{-t\L}\right\|_{L^p(F^{(R_i)}_i) \rightarrow L^\infty(E_j \setminus F^{(R_j)}_j)}
\lesssim  \frac{1}{V_j(R_j)} \mu(F^{(R_i)}_i)^{\frac1{p'}}
%\lesssim  \frac{1}{\mu(F^{(R_i)}_i)^{1/p}},
\lesssim  \mu(F^{(R_i)}_i)^{-\frac1p},
\]
%
%Therefore,  for any $1<p<p_0$, by suitably choosing $\epsilon$, we deduce that
%\begin{eqnarray*}
%&&\left\|t\L e^{-t\L}(f\chi_{F^{(R_i)}_i})\right\|_{L^\infty(E_j\setminus F^{(R_j)}_j)}
%\le  \frac{C\widetilde{V}_0(R_j)}{\widetilde{V}_0(\sqrt t) \mu(F^{(R_i)}_i)^{1/p}} \|f\chi_{F^{(R_i)}_i}\|_{p}.
%\end{eqnarray*}
%Combining this with \eqref{case-1}, we obtain for any $1<p<p_0$,
%\begin{eqnarray*}
%&&\left\|t\L e^{-t\L}(f\chi_{F^{(R_i)}_i})\right\|_{L^\infty(E_j\setminus F^{(R_j)}_j)}
%\lesssim  \frac{1}{\mu(F^{(R_i)}_i)^{1/p}} \|f\chi_{F^{(R_i)}_i}\|_{p}.
%\end{eqnarray*}
%which together with the fact
%$
%\|t\L e^{-t\L}(f\chi_{F^{(R_i)}_i})\|_{L^p(E_i\setminus F^{(R_i)}_i)}
%\lesssim \|f\chi_{F^{(R_i)}_i}\|_{p}
%$
%and the Riesz-Thorin theorem,
%which further
%The same argument as from \eqref{case-1} to \eqref{eq:pinf-p2}
which further
yields that for all $1\leq i\neq j \leq \ell$, $n_i, n_j >2$
and each $1 < p<p_1$,
\[
\|t\L e^{-t\L}\|_{L^p(F^{(R_i)}_i)\to L^2(E_j\setminus F^{(R_j)}_j)}
%\lesssim  \mu(F^{(R_i)}_i)^{-\frac{1}{p}(1-\frac{p}{2})}
%\le \left(\frac{CR_j^2}{t}\right)^{1-\frac{p}{2}} \mu(F^{(R_i)}_i)^{\frac 12 - \frac 1p}
\lesssim_p \mu(F^{(R_i)}_i)^{\frac 12-\frac 1p},
\quad \forall t \ge R_j^2.
\]

 \textbf{Case 2: \boldmath $n_i, n_j<2$.} %$i=j$, by Lemma \ref{lem:phi-r-equal} and Lemma \ref{lem:tilde-V}, it holds $\phi_{i}(|x|) \sim |x|^{2-n_i}$, $\phi_{j}(|y|) \sim |y|^{2-n_i}$ and $\widetilde{V}_i(R_i) \sim R_i^{4-n_i}$. Then
%By Lemmas \ref{lem:phi-r-equal} and  \ref{lem:tilde-V}, we have
%\[
%    \phi_{i}(|x|) \sim |x|^{2-n_i}
%    , \quad
%    \phi_{j}(|y|) \sim |y|^{2-n_j}
%    , \quad
%    \widetilde{V}_j(R_j) \sim R_j^{4-n_j}
%    , \quad
%    \widetilde{V}_i(|x|) \sim |x|^{4-n_i}
%    \quad \mbox{and} \quad
%    \widetilde{V}_i(R_j) \sim R_j^{4-n_i}
%\]
If $i=j$, then using the estimates of $\phi_i$, $\widetilde{V}_i$ and $\widetilde{V}_0$ (cf.  Lemmas \ref{lem:phi-r-equal}, \ref{lem:tilde-V}  and \ref{V0-general-equal}), one deduces from
Proposition \ref{time-heat-parabolic-2} (i) that
%for any $x\in F^{(R_i)}_i$, $y\in E_i\setminus F^{(R_i)}_i$ and $t \ge R_i^2$,
\begin{eqnarray*}
|t\partial_t h_t(x,y)|
%&& \sim \phi_{i}(|x|)\phi_{j}(|y|)|t\partial_t\tilde h_t(x,y)| \\
&&\lesssim_\epsilon\frac{|x|^{2-n_i} |y|^{2-n_i}}{t} \frac{R_i^2}
{R_i^{4-n_i}}\left[1+ \left(\frac{R_i^{4-n_i}}{R_i^2}\right)^{\frac{3}{2}\epsilon}
\left(\frac{|y|^2}{|y|^{4-n_i}}\right)^{\frac{3}{2}\epsilon}\right]
e^{-c\frac{d(x,y)^2}{t}}\\
&&\sim_\epsilon \frac{|x|^{2-n_i}}{t} \frac{|y|^{2-n_i}}{R_i^{2-n_i}} \left[ 1 + \lf(\frac{R_i^{2-n_i}}{|y|^{2-n_i}}\r)^{\frac{3}{2} \epsilon} \right] e^{-c\frac{d(x,y)^2}{t}} \\
&& \sim_\epsilon \frac{|x|^{2-n_i}}{t}  \left[ \frac{|y|^{2-n_i}}{R_i^{2-n_i}}  +\left( \frac{|y|^{2-n_i}}{R_i^{2-n_i}}\right)^{1-\frac{3}{2}\epsilon} \right]e^{-c\frac{d(x,y)^2}{t}} \\
&& \lesssim_\epsilon\frac{|x|^{2-n_i}}{t}
e^{-c\frac{d(x,y)^2}{t}}
\lesssim_\epsilon\frac{1}{t^{n_i/2}}
 \lesssim_\epsilon\frac{1}{R_i^{n_i}},
\end{eqnarray*}
%where the penultimate inequality follows from the fact that
where we used the fact that %that $|x| \lesssim R_i$ and
\begin{equation} \label{eq:y-es}
    |x| \leq |y| + d(x,y)
    \lesssim \sqrt t  + d(x,y).
   % \lesssim
%    \begin{cases}
%       \sqrt t, & \,\mbox{if} \, d(x,y) \leq \sqrt t, \\
%       d(x,y), & \,\mbox{if} \, d(x,y) \ge \sqrt t.
%    \end{cases}
\end{equation}
%-----------------------\\
%Above, in the last inequality, we used the fact that if $d(x,y) \leq \sqrt t$, then $|y| \leq |x| + d(x,y) \leq C \sqrt t$, and hence
%\[
%    \frac{|y|^{2-n_i}}{t} \lesssim \frac{(\sqrt t)^{2-n_i}}{t} \lesssim \frac{1}{t^{n_i/2}};
%\]
%if $d(x,y) \geq \sqrt t$, then $|y| \leq |x| + d(x,y) \leq C d(x,y)$, and hence
%\[
%     \frac{|y|^{2-n_i}}{t} e^{-\frac{d(x,y)^2}{ct}}
%     \lesssim \frac{|y|^{2-n_i}}{t} \frac{(\sqrt t)^{2-n_i}}{|y|^{2-n_i}}
%     \lesssim \frac{1}{t^{n_i/2}}.
%\]
Therefore, by suitably choosing $\epsilon$, one has
\begin{equation*} %\label{eq:case2}
\|t\L e^{-t\L}\|_{L^1(F^{(R_i)}_i)\to L^\infty(E_j\setminus F^{(R_j)}_j)} \lesssim   V_i(R_i)^{-1} \sim \mu(F^{(R_i)}_i)^{-1},
\end{equation*}
     which implies that %we obtain that
     for all $1\leq i=j\leq \ell$, $n_i, n_j <2$ and each $1 < p < 2$,
\begin{equation*}%\label{eq:1inf-p2}
\|t\L e^{-t\L}\|_{L^p(F^{(R_i)}_i)\to L^2(E_j\setminus F^{(R_j)}_j)} \lesssim_p \mu(F^{(R_i)}_i)^{\frac 12-\frac 1p},
\quad \forall t \ge R_j^2.\end{equation*}
%{\color{blue}The same trick of deducing from \eqref{eq:case2} to \eqref{eq:1inf-p2} will be used repeatedly.}
%============DETAILS OF ABOVE==============\\
%  For any $f \in L^p$, it holds
%  \begin{eqnarray*}
%     \int_{E_j \setminus F^{(R_j)}_j} |t \L e^{-t\L} f |^2 d\mu
%     && \lesssim \int_{E_j \setminus F^{(R_j)}_j} |t \L e^{-t\L} f |^p |t \L e^{-t\L} f |^{2-p}  d\mu
%     \lesssim \|t \L e^{-t\L} f \|^{2-p}_{L^\infty(E_j \setminus F^{(R_j)}_j)} \|t \L e^{-t\L} f \|^p_{L^p(E_j \setminus F^{(R_j)}_j)}\\
%     && \lesssim \|t \L e^{-t \L}\|^{2-p}_{L^1(F^{(R_i)}_i) \rightarrow L^\infty(E_j \setminus F^{(R_j)}_j)}  \|f\|^{2-p}_{L^1(F^{(R_i)}_i)}
%              \|t \L e^{-t \L}\|^p_{L^p(F^{(R_i)}_i) \rightarrow L^p(E_j \setminus F^{(R_j)}_j)} \|f\|^{p}_{L^p(F^{(R_i)}_i)}   \\
%     && \lesssim \mu(F^{(R_i)}_i)^{p-2}  \|f\|^{2-p}_{L^p(F^{(R_i)}_i)} \mu(F^{(R_i)}_i)^{\frac{2-p}{p'}} \|f\|^{p}_{L^p(F^{(R_i)}_i)}   \\
%     && \lesssim \mu(F^{(R_i)}_i)^{1-\frac{2}{p}} \|f\|^{2}_{L^p(F^{(R_i)}_i)},
%  \end{eqnarray*}
%  where we used the fact that
%  \[
%    p-2 + \frac{2-p}{p'} = p-2 + 2(1- \frac1p) - p(1-\frac1p) = 1- \frac{2}{p},
%  \]
%  And hence
%  $$
%        \|t \L e^{-t \L}\|_{L^p(F^{(R_i)}_i) \rightarrow L^2(E_j \setminus F^{(R_j)}_j)} \lesssim \mu(F^{(R_i)}_i)^{\frac{1}{2}-\frac{1}{p}}.
%  $$
%\\============END OF DETAILS================\\

%When $i\neq j$, let us first fix an $0< \epsilon <1/8$ small enough such that
Let us consider another subcase $i\neq j$,  and fix an $0< \epsilon <1/8$ small enough such that

%We choose $\epsilon>0$
%%which depends only on $n_i,n_j$
%small enough such that
\begin{equation} \label{eq:epi-reg-<2}
3\epsilon(2-n_i) + \frac{3}{2}\epsilon(2-n_j)  \le 2-n_j.
\end{equation}
%by Lemmas \ref{lem:phi-r-equal} and  \ref{lem:tilde-V}, we have
%\[
%    \phi_{i}(|x|) \sim |x|^{2-n_i}
%    , \quad
%    \phi_{j}(|y|) \sim |y|^{2-n_j}
%    , \quad
%    \widetilde{V}_j(R_j) \sim R_j^{4-n_j}
%    , \quad
%    \widetilde{V}_i(|x|) \sim |x|^{4-n_i}
%    \quad \mbox{and} \quad
%    \widetilde{V}_i(R_j) \sim R_j^{4-n_i}
%\]
Proposition \ref{time-heat-parabolic-2} (ii) together with Lemmas \ref{lem:phi-r-equal}, \ref{lem:tilde-V} and \ref{V0-general-equal} yields  that
%for any $x\in F^{(R_i)}_i$, $y\in E_j\setminus F^{(R_j)}_j$ and $t \ge R_j^2$,
\begin{eqnarray*}
|t\partial_t h_t(x,y)|
%&&\sim \phi_{i}(|x|)\phi_{j}(|y|)|t\partial_t\tilde h_t(x,y)|\\
&&  \lesssim {\frac{|x|^{2-n_i} |y|^{2-n_j} R_j^2}{tR_j^{4-n_j}}}\left(\frac{R_j^{4-n_j}}{R_j^2}\right)^{\frac 32\epsilon}\left(\frac{|y|^2}{|y|^{4-n_i}}+\frac{R_j^2}{R_j^{4-n_i}}\right)^{1-3\epsilon}
    e^{-c\frac{|x|^2+|y|^2}{t}} \\
&&  \sim {\frac{|x|^{2-n_i} |y|^{2-n_j}}{t}} R_j^{n_j-2} R_j^{\frac32 \epsilon(2-n_j)} \left(|y|^{n_i-2} + R_j^{n_i-2}\right)^{1-3\epsilon}
    e^{-c\frac{|x|^2+|y|^2}{t}}.
\end{eqnarray*}
If $|y|\le R_j$, via \eqref{eq:epi-reg-<2}, one has
\begin{eqnarray*}
     |t\partial_t h_t(x,y)|
    &&  \lesssim {\frac{|x|^{2-n_i} |y|^{2-n_j}}{t}} R_j^{n_j-2} R_j^{\frac32 \epsilon(2-n_j)} |y|^{(3\epsilon-1)(2-n_i)}  e^{-c\frac{|x|^2}{t}} \\
    &&  \lesssim {\frac{|x|^{2-n_j}}{t}}
     |y|^{3\epsilon(2-n_i)}
     R_j^{n_j-2} R_j^{\frac32 \epsilon(2-n_j)}   e^{-c\frac{|x|^2}{t}} \\
    && \lesssim \frac{1}{t^{n_j/2}} R_j^{3\epsilon(2-n_i) + \frac32\epsilon(2-n_j) - (2-n_j)}
     \lesssim \frac{1}{t^{n_j/2}}
      \lesssim \frac{1}{R_j^{n_j}} \sim V_j(R_j)^{-1}.
\end{eqnarray*}
%where we used the trivial inequality $e^{-c\frac{|y|^2}{t}} \lesssim \frac{t^{1-n_j/2}}{|y|^{2-n_j}}$.
 Otherwise $|y| > R_j$, using \eqref{eq:epi-reg-<2} again, one obtains %and $\frac{3}{2}\epsilon(2-n_j)-3\epsilon(2-n_i)\le 2-n_j,$
\begin{eqnarray*}
     |t\partial_t h_t(x,y)|
    && \lesssim {\frac{|x|^{2-n_i} |y|^{2-n_j}}{t}}  R_j^{n_j-2} R_j^{\frac32 \epsilon(2-n_j)} R_j^{(3\epsilon-1)(2-n_i)}  e^{-c\frac{|x|^2+|y|^2}{t}}  \\
%    &  \lesssim \frac{1|x|^{2-n_i}|y|^{2-n_j}}{t}\frac{R_j^2}{R_j^{4-n_j}}
%    \left(\frac{R_j^{4-n_j}}{R_j^2}\right)^{\frac 32\epsilon}\left(\frac{R_j^2}{R_j^{4-n_i}}\right)^{1-3\epsilon}
%    e^{-\frac{|x|^2+|y|^2}{ct}} \\
%    &  \lesssim \frac{1|x|^{2-n_i}|y|^{2-n_j}}{t}   e^{-\frac{|x|^2+|y|^2}{ct}} \frac{R_j^2}{R_j^{4-n_j}}
%    \left(R_j^{2-n_j}\right)^{\frac 32\epsilon} \left(R_j^{n_i-2}\right)^{1-3\epsilon} \\
 %   &  \lesssim \frac{1}{t^{\frac{n_i}{2}+\frac{n_j}{2}-1}} R_j^{3\epsilon(2-n_i) + \frac{3}{2}\epsilon(2-n_j) - (2-n_j) + n_i-2} \\
     &&  \lesssim \frac{1}{t^{n_i/2 + n_j/2-1}} R_j^{3\epsilon(2-n_i) + \frac{3}{2}\epsilon(2-n_j) - (2-n_j)} R_j^{n_i-2} \\
    &&  \lesssim \frac{1}{R_j^{n_i+ n_j - 2}}  R_j^{n_i-2}
     \lesssim  \frac{1}{R_j^{n_j}} \sim V_j(R_j)^{-1}.
\end{eqnarray*}
%where we used the trivial inequality $e^{-c\frac{|x|^2+|y|^2}{t}} \lesssim \frac{t^{2-n_i/2-n_j/2}}{|x|^{2-n_i} |y|^{2-n_j}}$.
Hence, it holds
$$\|t\L e^{-t\L}\|_{L^1(F^{(R_i)}_i)\to L^\infty(E_j\setminus F^{(R_j)}_j)} \lesssim  V_j(R_j)^{-1} \sim \mu(F^{(R_i)}_i)^{-1},$$
which further implies that for all $1\leq i \neq j\leq \ell$, $n_i, n_j <2$ and each $1 < p < 2$,
$$\|t\L e^{-t\L}\|_{L^p(F^{(R_i)}_i)\to L^2(E_j\setminus F^{(R_j)}_j)}%\le \sqrt{\frac{CR_j^{n_i+n_j-2}}{t^{\frac{n_i}{2}+\frac{n_j}{2}-1}}+\frac{CR_j^{n_j}}{t^{n_j/2}}}\mu(F^{(R_i)}_i)^{\frac 12-\frac 1p}
\lesssim_p \mu(F^{(R_i)}_i)^{\frac 12-\frac 1p},
\quad \forall t \ge R_j^2. $$

\textbf{Case 3: \boldmath $n_i > 2$, $n_j<2$.} %then $R_i\le CR_j$ and  $|x| \lesssim R_i \lesssim R_j \lesssim \sqrt t$. By Lemmas  \ref{lem:phi-r-equal}  and \ref{lem:tilde-V}, we have
%$$\phi_{i}(|x|) \sim 1, \quad \phi_{j}(|y|) \sim |y|^{2-n_j} ,\quad \widetilde{V}_j(R_j) \sim R_j^{4-n_j} , \quad
%|x|^2 / \widetilde{V}_i(|x|) \lesssim 1 \quad \mbox{and} \quad
%R_j^2 / \widetilde{V}_i(R_j) \lesssim 1
%.$$
%We choose $\epsilon >0$ small enough such that $\frac{3}{2} \epsilon (2-n_j) < 2-n_j$.
%-------------\\
%Let us first fix an $0 < \epsilon < 1/8$ small enough such that
%$$\frac{3}{2} \epsilon (2-n_j) < 2-n_j.$$
Proposition \ref{time-heat-parabolic-2} (ii) together with Lemmas \ref{lem:phi-r-equal}, \ref{lem:tilde-V} and \ref{V0-general-equal} yields  that
%for any $x\in F^{(R_i)}_i$, $y\in E_j\setminus F^{(R_j)}_j$ and $t \ge R_j^2$,
\begin{eqnarray*}
|t\partial_th_t(x,y)|
%&&\sim \phi_{i}(|x|) \phi_{j}(|y|)|t\partial_t\tilde h_t(x,y)|
%&&\lesssim \phi_{i}(|x|) \phi_{j}(|y|) {\frac{R_j^2}{t\widetilde{V}_j(R_j)}} \left(\frac{\widetilde{V}_j(R_j)}{R_j^2}\right)^{\frac 32\epsilon}\left(\frac{|x|^2}{\widetilde{V}_i(|x|)}+\frac{R_j^2}{\widetilde{V}_i(R_j)}\right)^{1-3\epsilon}
%        e^{-\frac{|x|^2+|y|^2}{ct}} \\
 && \lesssim_\epsilon \frac{|x|^{2-n_j}}{t}  \frac{R_j^2 }{R_j^{4-n_j}}
        \left(\frac{R_j^{4-n_j}}{R_j^2}\right)^{\frac32 \epsilon}
       % \left(\frac{|y|^2}{\widetilde{V}_i(|y|)}+\frac{R_j^2}{\widetilde{V}_i(R_j)}\right)^{1-3\epsilon}
        e^{-c\frac{|x|^2}{t}}
%  && \lesssim \frac{|y|^{2-n_j}}{t}  \frac{R_j^2 }{R_j^{4-n_j}}
%        \left(\frac{R_j^{4-n_j}}{R_j^2}\right)^{\frac32 \epsilon}
%        e^{-c\frac{|y|^2}{t}} \\
   \lesssim_\epsilon \frac{|x|^{2-n_j}}{t}  \frac{1}{R_j^{2-n_j}}
        R_j^{\frac32 \epsilon(2-n_j)} e^{-c\frac{|x|^2}{t}}\\
  &&   \lesssim_\epsilon \frac{1}{t^{n_j/2}} R_j^{\frac 32\epsilon(2-n_j) - (2-n_j)}
      \lesssim_\epsilon \frac{1}{t^{n_j/2}} \lesssim_\epsilon \frac{1}{R_j^{n_j}} \sim_\epsilon V_j(R_j)^{-1}.
%&& \le\frac{C}{t^{n_j/2}}\frac{R_j^2}{R_j^{4-n_j}} \left(R_j^{2-n_j}\right)^{3\epsilon/2}
%    \le\frac{CR_j^{n_j}}{t^{n_j/2}}R_j^{-2+\frac 32\epsilon(2-n_j)} \\
%&&       \le \frac{CR_j^{n_j}}{t^{n_j/2}}R_j^{-n_j}.
\end{eqnarray*}
%where we used the trivial inequality
%%$e^{-c\frac{|y|^2}{t}} \lesssim \frac{t^{1-n_j/2}}{|y|^{2-n_j}}$
%$e^{-c\frac{|y|^2}{t}} \lesssim \frac{t^{1-n_j/2}}{|y|^{2-n_j}}$.
By suitably choosing $\epsilon$, one concludes
\[
    \|t\L e^{-t\L}\|_{L^1(F^{(R_i)}_i)\to L^\infty(E_j\setminus F^{(R_j)}_j)}
  %  \lesssim V_j(R_j)^{-1},
    \lesssim V_j(R_j)^{-1}
    \sim \mu(F^{(R_i)}_i)^{-1},
\]
which further implies that for all $1 \leq i, j \leq \ell$, $n_i >2$, $n_j<2$ and each $1<p<2$,
\[
     \|t\L e^{-t\L}\|_{L^2(F^{(R_i)}_i)\to L^p(E_j\setminus F^{(R_j)}_j)}
    \lesssim_p \mu(F^{(R_i)}_i)^{\frac{1}{2} - \frac1p},
\quad \forall t \ge R_j^2.
\]

\textbf{Case 4: \boldmath $n_i<2$, $n_j>2$.} %it holds $\phi_{i}(|x|) \sim |x|^{2-n_i}$ and $\phi_{j}(|y|) \sim 1$.
%by Lemmas \ref{lem:phi-r-equal} and  \ref{lem:tilde-V}, we have
%\[
%    \phi_{i}(|x|) \sim |x|^{2-n_i}
%    , \quad
%    \phi_{j}(|y|) \sim 1
%    , \quad
%    \widetilde{V}_j(R_j) \sim V_j(R_j)
%    , \quad
%    \widetilde{V}_i(|x|) \sim |x|^{4-n_i}
%    \quad \mbox{and} \quad
%    \widetilde{V}_i(R_j) \sim R_j^{4-n_i}
%\]
For each $1 < p<2$, let us choose an $0< \epsilon <1/8$ small enough such that
\begin{equation*}%\label{eq:epi-reg-<2>2}
3\epsilon(2-n_i) + \frac{3}{2}\epsilon (N_j - 2) \le (n_j-n_i)\frac{1}{p'},
\end{equation*}
and hence it holds
\begin{equation} \label{eq:rjrj}
  R_j^{3\epsilon(2-n_i)}R_j^{\frac{3}{2}\epsilon (N_j - 2)}
  \lesssim V_j(R_j)^{\frac1{p'}} V_i(R_j)^{-\frac1{p'}}.
\end{equation}
Using Lemma \ref{lem:phi-r-equal}, Remark \ref{rem:thr} (i), Lemmas \ref{lem:tilde-V} and \ref{V0-general-equal}, we deduce from
Proposition \ref{time-heat-parabolic-2} (ii) that %for any $x \in F^{(R_i)}_i$, $y \in E_j \setminus F^{(R_j)}_j$ and $t \ge R_j^2$, it holds
\begin{eqnarray*}
|t\partial_t h_t(x,y)|
%&&\sim \phi_{i}(|x|) \phi_{j}(|y|) |t\partial_t\tilde h_t(x,y)| \\
&&  \lesssim {\frac{|y|^{2-n_i}R_j^2 }{t}}  \frac{1}{V_j(R_j)} \left(\frac{{V}_j(R_j)}{R_j^2}\right)^{\frac 32\epsilon} \left(\frac{|y|^2}{|y|^{4-n_i}}+ \frac{R_j^2}{R_j^{4-n_i}}\right)^{1-3\epsilon}    e^{-c\frac{|x|^2+|y|^2}{t}} \\
&& \lesssim \frac{|y|^{2-n_i} R_j^2}{t}  R_j^{ \frac{3}{2}\epsilon(N_j-2)} V_j(R_j)^{-1}
\left(|y|^{n_i-2}+R_j^{n_i-2}\right)^{1-3\epsilon}    e^{-c\frac{|x|^2+|y|^2}{t}}.
\end{eqnarray*}
When $|y| \leq R_j$, %one has via \eqref{eq:rjrj} and the fact $t \ge R_j^2$ that
via \eqref{eq:rjrj} and the fact $t \ge R_j^2$, one has
\begin{eqnarray*}
  |t\partial_t h_t(x,y)|
&& \lesssim   |y|^{2-n_i} R_j^{ \frac{3}{2}\epsilon(N_j-2)} V_j(R_j)^{-1} |y|^{(3\epsilon-1)(2-n_i)}
  \lesssim   |y|^{3\epsilon(2-n_i)}  R_j^{ \frac{3}{2}\epsilon(N_j-2)} V_j(R_j)^{-1} \\
&& \lesssim R_j^{3\epsilon(2-n_i)}  R_j^{ \frac{3}{2}\epsilon(N_j-2)} V_j(R_j)^{-1}
  \lesssim  V_j(R_j)^{-\frac1{p}}  V_i(R_j)^{-\frac1{p'}}.
\end{eqnarray*}
When $ |y| > R_j$, using \eqref{eq:rjrj} again, we deduce that
\begin{eqnarray*}
 |t\partial_t h_t(x,y)|
&& \lesssim  {\frac{|y|^{2-n_i}R_j^2 }{t}}   R_j^{ \frac{3}{2}\epsilon(N_j-2)} V_j(R_j)^{-1} R_j^{(3\epsilon-1)(2-n_i)} e^{-c\frac{|y|^2}{t}}  \\
%& \lesssim \frac{1}{t^{n_i/2}}  R_j^{2 + 3\epsilon(2-n_i) - (2-n_i)} R_j^{-n_j+\frac{3}{2}\epsilon(N_j-2)} e^{-c\frac{|x|^2+|y|^2}{t}} \\
&&    \lesssim  |y|^{-n_i} R_j^{n_i} R_j^{3\epsilon(2-n_i)}  R_j^{ \frac{3}{2}\epsilon(N_j-2)} V_j(R_j)^{-1}
  \lesssim |y|^{-n_i} R_j^{n_i}   V_j(R_j)^{-\frac1{p}}  V_i(R_j)^{-\frac1{p'}}.
\end{eqnarray*}
%where we used the trivial inequality
%$e^{-c\frac{|x|^2}{t}} \lesssim \frac{t}{|x|^2}$.
%where $\delta$ satisfies $0 < \delta \leq n_i/2$ and $\frac{2\delta p}{p-1}>n_i$. For instance we can take $\delta=\frac{n_i}{2}$.
% The first condition makes $ \left(\frac{R_j^2}{t}\right)^{\frac{n_i}{2}-\delta} \lesssim 1$.
%  The Second condition makes the following series   $\sum_{k=0}^{\infty} (2^k)^{-2\delta} (2^k)^{n_i(1-\frac{1}{p})}$ converge.
%  Since $\frac{p}{p-1} > 2$, if $\delta > \frac{n_i}{4}$, then $\frac{2\delta p}{p-1}>n_i$.
%Since $n_i/2 - \delta <1$, we have $\frac{R_j^2}{t} \lesssim \left(\frac{R_j^2}{t}\right)^{n_i/2 -\delta}$.
Hence, it follows from the H\"{o}lder inequality that
\begin{eqnarray} \label{eq:s2b2}
\|t\L e^{-t\L}\|_{L^p(F^{(R_i)}_i)\to L^\infty(E_j\setminus F^{(R_j)}_j)}
&&   \lesssim
         V_j(R_j)^{-\frac1{p}}  V_i(R_j)^{-\frac1{p'}}  %R_j^{3\epsilon(2-n_i)-n_j+\frac32 \epsilon(N_j-2)}
      \left(\int_{F^{(R_i)}_i \cap \{y \in E_i: \, |y| \leq R_j \}}  \,  d\mu(y)\right)^{\frac1{p'}}  \notag \\
&&  \   + R_j^{n_i}  V_j(R_j)^{-\frac1{p}}  V_i(R_j)^{-\frac1{p'}}%\frac{R_j^{- n_i/p'}}{V_j(R_j)^{1/p}} %R_j^{- n_j/p - n_i/p'}
       \left(\int_{F^{(R_i)}_i \cap  \{y \in E_i: \, |y| > R_j\}} |y|^{-n_i p'} \,d\mu(y)\right)^{\frac1{p'}} \notag \\
&&     \lesssim  {V}_j(R_j)^{-\frac 1p}
        \sim \mu(F^{(R_i)}_i)^{-\frac 1p},
\end{eqnarray}
%---------------------------
%\begin{eqnarray*}
%\|t\L e^{-t\L}\|_{L^p(F^{(R_i)}_i)\to L^\infty(E_j\setminus F^{(R_j)}_j)}
%&&\lesssim
%         R_j^{3\epsilon(2-n_i)-n_j+\frac32 \epsilon(N_j-2)}
%      \left(\int_{F^{(R_i)}_i \cap \{x \in E_i; \, |x| \leq R_j \}} d\mu(x)\right)^{1 - \frac1 p}\\ %R_j^{n_i(1-\frac 1p)}\\
%&& \quad +\frac{R_j^{n_i}}{t^{n_i/2}} R_j^{3\epsilon(2-n_i)} R_j^{-n_j}    R_j^{\frac32 \epsilon(N_j-2)}
%       \left(\int_{F^{(R_i)}_i \cap  \{x\in E_i: R_j<|x|\}}\left(\frac{t}{|x|^2+|y|^2}\right)^{\frac{\delta p}{p-1}}\,d\mu\right)^{1-\frac 1p}\\
%&&\lesssim  R_j^{3\epsilon(2-n_i)} R_j^{-n_j}    R_j^{\frac32 \epsilon(N_j-2)} R_j^{n_i(1-\frac 1p)}\\
%&& \quad + \left(\frac{R_j^2}{t}\right)^{\frac{n_i}{2}-\delta}  R_j^{3\epsilon(2-n_i)} R_j^{-n_j}    R_j^{\frac32 \epsilon(N_j-2)} R_j^{n_i(1-\frac 1p)} \\
%&& \lesssim  R_j^{3\epsilon(2-n_i) + \frac32 \epsilon(N_j-2) -n_j + n_i(1-\frac 1p)}    \\
%&& \lesssim  R_j^{n_j (-\frac 1p)} \le C {V}_j(R_j)^{-\frac 1p},
%\end{eqnarray*}
%as soon as %$0<\delta\le \frac{n_i}{2}$,
%%$\frac{\delta p}{p-1}>n_i$
%$\epsilon$ small enough such that $3\epsilon(2-n_i) + \frac{3}{2}\epsilon (N_j - 2) \le (n_j-n_i)(1-\frac 1p)$.
%For instance we can take $\delta=\frac{n_i}{2}$.
%where in the last inequality we used
where the second inequality is due to%we used the fact that
\begin{equation*}
    \left(\int_{F^{(R_i)}_i \cap \{y \in E_i: \, |y| \leq R_j \}} \, d\mu(y)\right)^{\frac1{p'}}
    % \leq \left(\int_{\{x \in E_i: \, |x| \leq R_j \}} d\mu(x)\right)^{\frac1{p'}}
     \leq \left(\int_{\{y \in E_i: \dist(y, E_0) \leq 2R_j \}} \, d\mu(y)\right)^{\frac1{p'}}
     = \mu(F^{(R_j)}_i)^{\frac1{p'}}
     \sim V_i(R_j)^{\frac1{p'}}
%     \sim R_j^{n_i(\frac1{p'})}
\end{equation*}
and
\begin{eqnarray*}
      \left(\int_{F^{(R_i)}_i \cap \{y\in E_i: \, |y| > R_j\}} |y|^{-n_ip'} \,d\mu(y)\right)^{\frac1{p'}}
    && \leq \sum_{k=0}^{\infty} \left(\int_{\{y\in E_i: \, 2^kR_j<|y|\le 2^{k+1}R_j\}} |y|^{-n_ip'} \,d\mu(y)\right)^{\frac1{p'}} \\
%    & \lesssim \sum_{k=0}^{\infty} \left[\frac{t}{(2^kR_j)^2}\right]^{n_i/2} \left(\int_{\{x\in E_i: \dist(x,E_0) \le 2^{k+1}R_j\}} \,d\mu(x)\right)^{\frac1{p'}} \\
    && \lesssim \sum_{k=0}^{\infty} (2^kR_j)^{-n_i} V_i(2^k R_j)^{\frac1{p'}}
     %\lesssim \sum_{k=0}^{\infty} (2^k)^{-n_i} (2^k)^{n_i/p'} R_j^{n_i} V_i(R_j)^{\frac1{p'}} \\
      \lesssim R_j^{-n_i} V_i(R_j)^{\frac1{p'}}.
\end{eqnarray*}
By \eqref{eq:s2b2}, we conclude that for all $1\leq i, j\leq j$, $n_i <2$, $n_j>2$ and each $1 < p <2$,
%Hence, we have
%\[
%\|t\L e^{-t\L}\|_{L^p(F^{(R_i)}_i)\to L^\infty(E_j\setminus F^{(R_j)}_j)}
%\lesssim %\left(\frac{R_j^2}{t}\right)^{\frac{n_i}{2}-\delta}
%{V}_j(R_j)^{-\frac 1p},
%\]
%Combining this with $\|t\L e^{-t\L}\|_{L^p(F^{(R_i)}_i)\to L^p(E_j\setminus F^{(R_j)}_j)} \lesssim 1$ and the Riesz-Thorin interpolation theorem, where $\frac{2/p}{p} + \frac{1-2/p}{\infty} = \frac{1}{2}$, yields that
\begin{eqnarray*}
\|t\L e^{-t\L}\|_{L^p(F^{(R_i)}_i)\to L^2(E_j\setminus F^{(R_j)}_j)}
%&& \lesssim \|t\L e^{-t\L}\|^{1 - \frac2 p}_{L^p(F^{(R_i)}_i)\to L^\infty(E_j\setminus F^{(R_j)}_j)} \\
%&&\lesssim % \left(\frac{R_j^2}{t}\right)^{(\frac{n_i}{2}-\delta)(1-\frac p2)}
%{V}_j(R_j)^{\frac 12-\frac 1p}
\lesssim_p \mu(F^{(R_i)}_i)^{\frac 12-\frac 1p},
\quad \forall t \ge R_j^2.
\end{eqnarray*}

\textbf{Case 5: \boldmath $n_i=2$, $n_j<2$.} %It holds $R_i\le CR_j$ and  $|x| \lesssim R_i \lesssim R_j \lesssim \sqrt t$.
%By Lemmas  \ref{lem:phi-r-equal}  and \ref{lem:tilde-V}, we have
%$$\phi_{i}(|x|) \sim \log(2+|x|), \quad \phi_{j}(|y|) \sim |y|^{2-n_j} \quad \mbox{and} \quad \widetilde{V}_j(R_j) \sim R_j^{4-n_j}.$$
Note that in this case $R_i \lesssim R_j$ and hence $|y| \lesssim \sqrt t$.
%From this, using Lemmas \ref{lem:phi-r-equal} and \ref{V0-general-equal}, we conclude from Proposition \ref{time-heat-parabolic-2} (ii) that
From this, together with Proposition \ref{time-heat-parabolic-2} (iii), Lemma \ref{lem:phi-r-equal} and Lemma \ref{V0-general-equal}, we conclude that
%for any  $x \in F^{(R_i)}_i$, $y \in E_j \setminus F^{(R_j)}_j$ and $t \ge R_j^2$,
\begin{eqnarray*}
|t\partial_t  h_t(x,y)|
 %\sim \phi_{i}(|x|)\phi_{j}(|y|)|t\partial_t\tilde h_t(x,y)|
  \lesssim_\epsilon\frac{|x|^{2-n_j} \log(2+|y|)}{t \log^2(2+\sqrt t)} e^{-c\frac{|x|^2}{t}}
   \lesssim_\epsilon \frac{|x|^{2-n_j}}{t} e^{-c\frac{|x|^2}{t}}
  %\lesssim_\epsilon\frac{\log(2+|x|)|y|^{2-n_j}}{t \log^2(2+\sqrt t)} e^{-\frac{|y|^2}{ct}}
  %\lesssim_\epsilon\frac{1}{t^{n_j/2}\log(2+\sqrt t)}
  \lesssim_\epsilon\frac{1}{t^{n_j/2}} \lesssim_\epsilon\frac{1}{R_j^{n_j}}
  \sim_\epsilon V_j(R_j)^{-1}.
\end{eqnarray*}
%where we used the trivial inequality $e^{-c\frac{|y|^2}{t}} \lesssim \frac{t^{1-n_j/2}}{|y|^{2-n_j}}$.
Choosing $\epsilon$ suitably, it follows that
\[
    \|t\L e^{-t\L}\|_{L^1(F^{(R_i)}_i)\to L^\infty(E_j\setminus F^{(R_j)}_j)}
    \lesssim V_j(R_j)^{-1}
    \sim \mu(F^{(R_i)}_i)^{-1} ,
\]
which further implies that for all $1 \leq i, j \leq \ell$, $n_i =2$, $n_j <2$ and each $1 < p <2$,
\[
     \|t\L e^{-t\L}\|_{L^2(F^{(R_i)}_i)\to L^p(E_j\setminus F^{(R_j)}_j)}
    \lesssim_p \mu(F^{(R_i)}_i)^{\frac{1}{2} - \frac1p},
\quad \forall t \ge R_j^2.
\]
%\begin{align*}
%\|t\L e^{-t\L}\|_{L^p(F^{(R_i)}_i)\to L^\infty(E_j\setminus F^{(R_j)}_j)}
%&\le \frac{CR_j^{n_j}}{t^{n_j/2}} R_j^{-n_j} \left(\int_{F^{(R_i)}_i} d\mu\right)^{1 - \frac1 p} \\
%&\le \frac{CR_j^{n_j}}{t^{n_j/2}} \frac{1}{V_j(R_j)} \mu_i(F^{(R_i)}_i)^{1 - \frac1 p} \\
%&\le \frac{CR_j^{n_j}}{t^{n_j/2}} V_j(R_j)^{- \frac1 p}.
%\end{align*}
%This together with $$\|t\L e^{-t\L}\|_{L^p(F^{(R_i)}_i)\to L^p(E_j\setminus F^{(R_j)}_j)}\le C,$$
%and the Riesz-Thorin interpolation Theorem, where $\frac{2/p}{p} + \frac{1-2/p}{\infty} = \frac{1}{2}$,  gives that
%\begin{eqnarray*}
%\|t\L e^{-t\L}\|_{L^p(F^{(R_i)}_i)\to L^2(E_j\setminus F^{(R_j)}_j)}
%%\le \| t\L e^{-t\L} \|^{\frac{2-p}{2}}_{L^p(F^{(R_i)}_i)\to L^\infty(E_j\setminus F^{(R_j)}_j)}
%\le C \left(\frac{R_j^{n_j}}{t^{n_j/2}}\right)^{1-\frac p2}
%{V}_j(R_j)^{\frac 12-\frac 1p}\le C\mu(F^{(R_i)}_i)^{\frac 12-\frac 1p}.
%\end{eqnarray*}

\textbf{Case 6: \boldmath  $n_i=n_j=2$.}  %For any $x\in F^{(R_i)}_i$ and $y\in E_j\setminus F^{(R_j)}_j$,
%by Lemma \ref{lem:phi-r-equal}, we have
%\[
%    \phi_{i}(|x|) \sim \log(2+|x|) \quad \mbox{and} \quad
%    \phi_{j}(|y|) \sim \log(2+|y|).
%\]
Note that in this case $R_i \sim R_j$ and hence $|y| \lesssim \sqrt t $.
%Using Lemmas \ref{lem:phi-r-equal} and \ref{V0-general-equal}, it follows from Proposition \ref{time-heat-parabolic-2} (ii) that
It follows from Proposition \ref{time-heat-parabolic-2} (iv), Lemma \ref{lem:phi-r-equal} and Lemma \ref{V0-general-equal} that
%for any  $x \in F^{(R_i)}_i$, $y \in E_j \setminus F^{(R_j)}_j$ and $t \ge R_j^2$,
\begin{eqnarray*}
|t\partial_t h_t(x,y)|
 %\sim \phi_{i}(|x|)\phi_{j}(|y|)|t\partial_t\tilde h_t(x,y)|
 \lesssim_\epsilon \frac{\log(2+|x|)\log(2+|y|)}{t\log^2(2+\sqrt t)} e^{-c\frac{d(x,y)^2}{t}}
  \lesssim_\epsilon \frac{1}{t}
   \lesssim_\epsilon \frac{1}{R_j^2} \sim_\epsilon V_j(R_j)^{-1},
\end{eqnarray*}
where the second inequality is due to the fact $|y| \lesssim \sqrt t$ and \eqref{eq:y-es}.
%\[
%    |y| \leq |x| + d(x,y)
%    \lesssim
%    \begin{cases}
%       \sqrt t, & \mbox{if} \, d(x,y) \leq \sqrt t, \\
%       d(x,y), & \mbox{if} \, d(x,y) \ge \sqrt t.
%    \end{cases}
%\]
%Above, in the last inequality, we used the fact that if $d(x,y) \leq \sqrt t$, then $|y| \leq |x| + d(x,y) \lesssim \sqrt t$, and hence
%\[
%    \frac{\log(2+|y|)}{\log(2+ \sqrt t)} \lesssim 1;
%\]
%if $d(x,y) \geq \sqrt t$, then $|y| \leq |x| + d(x,y) \lesssim d(x,y)$, and hence
%\[
%     \frac{\log(2+|y|)}{\log(2+ \sqrt t)} e^{-\frac{d(x,y)^2}{ct}}
%     \lesssim 1.
%\]
Hence, by suitably choosing $\epsilon$, one has
$$\|t\L e^{-t\L}\|_{L^1(F^{(R_i)}_i)\to L^\infty(E_j\setminus F^{(R_j)}_j)} \lesssim  V_j(R_j)^{-1} \sim \mu(F^{(R_i)}_i)^{-1},$$
which further implies that for all $1 \leq i, j \leq \ell$, $n_i =n_j=2$ and each $1 < p <2$,
$$\|t\L e^{-t\L}\|_{L^p(F^{(R_i)}_i)\to L^2(E_j\setminus F^{(R_j)}_j)}
%\le \sqrt{\frac{CR_j^{2}}{t}}\mu(F^{(R_i)}_i)^{\frac 12-\frac 1p}
\lesssim_p \mu(F^{(R_i)}_i)^{\frac 12-\frac 1p},
\quad \forall t \ge R_j^2.$$

\textbf{Case 7: \boldmath $n_i=2$, $n_j>2$.}
%$$\phi_{i}(|x|) \sim \log(2+|x|), \quad \phi_{j}(|y|) \sim 1 \quad \mbox{and} \quad \widetilde{V}_j(R_j) \sim V_j(R_j).$$
For each $1 < p <2$, let us  choose an $0 < \epsilon < 1/8$ small enough such that $$\frac{3}{2}\epsilon (N_j-2) \le \frac{1}{p'}(n_j-2),$$
and hence
\begin{equation*}% \label{eq:cf}
  \frac{1}{V_j(R_j)} \left(\frac{{V}_j(R_j)}{R_j^2}\right)^{\frac{3}{2}\epsilon}
  \lesssim \frac{1}{V_j(R_j)} R_j^{ \frac{3}{2}\epsilon(N_j-2)}
  \lesssim \frac{1}{V_j(R_j)} V_j(R_j)^{\frac1{p'}} V_i(R_j)^{-\frac1{p'}}
  \lesssim V_j(R_j)^{-\frac1p} V_i(R_j)^{-\frac1{p'}}.
\end{equation*}
Using this %and the estimates of $\phi_i$,
%Lemma \ref{lem:phi-r-equal},
 we deduce from Proposition \ref{time-heat-parabolic-2} (v),
 Remark \ref{rem:thr} (i), Lemma \ref{lem:phi-r-equal} and Lemma \ref{V0-general-equal} that %for any  $x \in F^{(R_i)}_i$, $y \in E_j \setminus F^{(R_j)}_j$ and $t \ge R_j^2$,
 \begin{eqnarray*}
 |t\partial_t h_t(x,y)|
% \sim \phi_{i}(|x|) \phi_{j}(|y|) |t\partial_t\tilde h_t(x,y)| \\
 %&& \lesssim \frac{R_j^2}{t} \frac{\log(2+|y|) \log(2+R_j)}{\log(2+\sqrt t) \log^{3\epsilon}(2+R_j)} \frac{1}{{V}_j(R_j)}
%    \left(\frac{{V}_j(R_j)}{R_j^2}\right)^{\frac32\epsilon} \\
% && \ \times  \left(\frac{1}{\log(2+R_j)} + \frac{1}{\log(2+|y|)}\right)^{1-3\epsilon}
% && \lesssim \frac{R_j^2}{t} \frac{\log(2+R_j)}{\log(2+\sqrt t)}
%\frac{\log^{3\epsilon-1}(2+R_j) + \log^{3\epsilon-1}(2+|y|)}{\log^{3\epsilon}(2+R_j)} \\
%&& \ \times
%\frac{1}{ \widetilde V_j(R_j)} \left(\frac{\widetilde V_j(R_j)}{R_j^2}\right)^{\frac 32\epsilon}
%e^{-c\frac{|x|^2+|y|^2}{t}} \\
%&& \lesssim \frac{R_j^2}{t} \frac{\log(2+|x|) \log(2+R_j)}{\log(2+\sqrt t)}
%\frac{\log^{3\epsilon -1}(2+|x|) + \log^{3\epsilon -1}(2+R_j)}{\log^{3\epsilon}(2+R_j)} R_j^{ \frac{3}{2}\epsilon(N_j-2)} V_j(R_j)^{-1} e^{-c\frac{|x|^2+|y|^2}{t}} \\
&&  \lesssim \frac{R_j^2}{t} \frac{\log(2+|y|) \log(2+R_j)}{\log(2+\sqrt t)}
\frac{\log^{3\epsilon -1}(2+R_j) + \log^{3\epsilon -1}(2+|y|)}{\log^{3\epsilon}(2+R_j)}    \\
&& \ \times \frac{1}{V_j(R_j)} \left(\frac{{V}_j(R_j)}{R_j^2}\right)^{\frac{3}{2}\epsilon} e^{-c\frac{|x|^2+|y|^2}{t}}  \\
&&  \lesssim \frac{R_j^2}{t} \frac{\log(2+|y|) \log(2+R_j)}{\log(2+\sqrt t)}
\frac{\log^{3\epsilon -1}(2+R_j) + \log^{3\epsilon -1}(2+|y|)}{\log^{3\epsilon}(2+R_j)} \\
&& \ \times
V_j(R_j)^{-\frac1p} V_i(R_j)^{-\frac1{p'}} e^{-c\frac{|x|^2+|y|^2}{t}}.
\end{eqnarray*}
%\begin{eqnarray*}
%&& |t\partial_t h_t(x,y)| \\
%% \sim \phi_{i}(|x|) \phi_{j}(|y|) |t\partial_t\tilde h_t(x,y)| \\
%&& \ \lesssim \frac{R_j^2}{t} \frac{\log(2+|x|) \log(2+R_j)}{\log(2+\sqrt t) \log^{3\epsilon}(2+R_j)} \frac{1}{{V}_j(R_j)}
%    \left(\frac{{V}_j(R_j)}{R_j^2}\right)^{\frac32\epsilon}
%   \left(\frac{1}{\log(2+|x|)} + \frac{1}{\log(2+R_j)}\right)^{1-3\epsilon}
%e^{-c\frac{|x|^2+|y|^2}{t}} \\
%%&& \lesssim \frac{R_j^2}{t} \frac{\log(2+|x|) \log(2+R_j)}{\log(2+\sqrt t)}
%%\frac{\log^{3\epsilon -1}(2+|x|) + \log^{3\epsilon -1}(2+R_j)}{\log^{3\epsilon}(2+R_j)} R_j^{ \frac{3}{2}\epsilon(N_j-2)} V_j(R_j)^{-1} e^{-c\frac{|x|^2+|y|^2}{t}} \\
%&& \ \lesssim \frac{R_j^2}{t} \frac{\log(2+|x|) \log(2+R_j)}{\log(2+\sqrt t)}
%\frac{\log^{3\epsilon -1}(2+|x|) + \log^{3\epsilon -1}(2+R_j)}{\log^{3\epsilon}(2+R_j)} V_j(R_j)^{-\frac1p} V_i(R_j)^{-\frac1{p'}} e^{-c\frac{|x|^2+|y|^2}{t}}.
%\end{eqnarray*}
When $|y| \leq R_j$, by the fact $t \ge R_j^2$, one has
\begin{eqnarray*}
|t\partial_t h_t(x,y)|
 \lesssim \log(2+|y|) \frac{\log^{3\epsilon -1}(2+|y|)}{\log^{3\epsilon}(2+R_j)}  V_j(R_j)^{-\frac1p} V_i(R_j)^{-\frac1{p'}}
%&&\le\frac{C\log^{3\epsilon}(2+|x|)}{\widetilde{V}_0(\sqrt t)}\frac{\widetilde{V}_0(R_j)}{\widetilde{V}_j(R_j)}
%\left(\frac{\widetilde{V}_j(R_j)}{\widetilde{V}_0(R_j)}\right)^{3\epsilon/2}
%e^{-\frac{
%|x|^2+|y|^2}{ct}}\\
%&&\le\frac{C\log^{3\epsilon}(2+|x|)}{t \log^2(2 +\sqrt t)}\frac{R_j^2 \log^2(2+R_j)}{\widetilde{V}_j(R_j)}
%\left(\frac{\widetilde{V}_j(R_j)}{R_j^2 \log^2(2+R_j)}\right)^{3\epsilon/2}
%e^{-\frac{|x|^2+|y|^2}{ct}}\\
%&& \lesssim  R_j^{ \frac{3}{2}\epsilon(N_j-2)} V_j(R_j)^{-1}
\lesssim V_j(R_j)^{-\frac1p} V_i(R_j)^{-\frac1{p'}}.
\end{eqnarray*}
When $|y| > R_j$, one finds
\begin{eqnarray*}
|t\partial_t  h_t(x,y)|
&&\lesssim  \frac{R_j^2}{t} \frac{\log(2+|y|) \log(2+R_j) }{ \log(2+\sqrt t) }
   \frac{\log^{3\epsilon -1}(2+R_j)}{\log^{3\epsilon}(2 + R_j)}
   V_j(R_j)^{-\frac1p} V_i(R_j)^{-\frac1{p'}}  e^{-c \frac{|y|^2}{t}}\\
%&&\le\frac{C}{t} \frac{R_j^2}{{V}_j(R_j)} \frac{\log^2(2+R_j)}{\log^2(2 + \sqrt t)} \frac{\log(2+|x|)}{\log(2+R_j)}
%\left(\frac{{V}_j(R_j)}{R_j^2}\right)^{3\epsilon/2}
%e^{-\frac{(1-3\epsilon)
%|x|^2+|y|^2}{ct}}\\
%&&\lesssim \frac{R_j^2}{t} \frac{\log(2+|x|)}{\log(2+\sqrt t)} V_j(R_j)^{-1/p} V_i(R_j)^{-1/p'} e^{-\frac{|x|^2+|y|^2}{ct}} \\
&&\lesssim |y|^{-2} R_j^2 V_j(R_j)^{-\frac1p} V_i(R_j)^{-\frac1{p'}},
\end{eqnarray*}
where we used %the
%trivial inequality
%$e^{-c\frac{|x|^2}{t}} \lesssim \frac{t}{|x|^2}$ and
the fact that
%$\frac{\log(2+|x|)}{\log(2+\sqrt t)} e^{-c\frac{|x|^2}{t}} \lesssim 1$.
$ \log(2+|y|)  e^{-c\frac{|y|^2}{t}} \lesssim \log(2+\sqrt t) $ .
Combining the above two estimates and applying the approach similar to \eqref{eq:s2b2}, one concludes that
\begin{eqnarray*}%\label{eq:tlt6}
\|t\L e^{-t\L}\|_{L^p(F^{(R_i)}_i)\to L^\infty(E_j\setminus F^{(R_j)}_j)}
&&\lesssim V_j(R_j)^{-\frac1p} V_i(R_j)^{-\frac1{p'}}
         \left(\int_{F^{(R_i)}_i \cap \{y \in E_i: \, |y| \leq R_j \}} \, d\mu(y)\right)^{\frac1{p'}} \nonumber\\
&& \ + R_j^2 V_j(R_j)^{-\frac1p} V_i(R_j)^{-\frac1{p'}}
\left(\int_{F^{(R_i)}_i \cap \{y \in E_i: \, |y|>R_j\}} |y|^{-2p'} \,d\mu(y)\right)^{\frac1{p'}} \nonumber\\
&& %\lesssim V_j(R_j)^{-\frac1p} %V_i(R_j)^{-\frac1{p'}} V_i(R_j)^{\frac1{p'}}
%  + R_j^2 R_j^{-2}  V_j(R_j)^{-\frac1p} %V_i(R_j)^{-\frac1{p'}} V_i(R_j)^{\frac1{p'}} \nonumber\\
 \lesssim V_j(R_j)^{-\frac1p}
 \sim \mu(F^{(R_i)}_i)^{-\frac1p},
\end{eqnarray*}
which further implies that for all $1 \leq i, j \leq \ell$, $n_i =2$, $n_j >2$ and each $1 < p <2$,
\begin{eqnarray*}
\|t\L e^{-t\L}\|_{L^p(F^{(R_i)}_i)\to L^2(E_j\setminus F^{(R_j)}_j)}
\lesssim_p \mu(F^{(R_i)}_i)^{\frac 12-\frac 1p},
\quad \forall t \ge R_j^2.
\end{eqnarray*}

{\bf Case 8: \boldmath $n_i > 2$, $ n_j=2$.} %It holds $R_i\le cR_j$ and $|x| \lesssim R_i \lesssim R_j$.
%It follows from Lemma \ref{lem:tilde-V} (ii) that
%\[
%    |x|^2 / \widetilde{V}_i(|x|) \lesssim 1.
%\]
We deduce from Proposition \ref{time-heat-parabolic-2} (vi) and Lemma \ref{V0-general-equal} that
\begin{eqnarray*}
 |t\partial_t\tilde h_t(x,y)|
&&  \lesssim_\epsilon
 \frac{\log^{3\epsilon}(2+R_j) \log^{1-3\epsilon}(2+R_j)}{t\log^2(2+\sqrt t)}
 \left(\frac{|y|^2}{\widetilde V_i(|y|)} + \frac{R_j^2 }{\widetilde V_i(R_j)}\right)^{\frac 32 \epsilon} \\
&& \ \times \left( \frac{|y|^2}{\widetilde V_i(|y|)} + \frac{R_j^2 }{\widetilde{V}_i(R_j)} \right)^{1-3\epsilon}e^{-c\frac{|x|^2+|y|^2}{t}}  \\
&&  \lesssim_\epsilon
 \frac{1}{t\log(2+\sqrt t)}
 \left(\frac{|y|^2}{\widetilde V_i(|y|)} + \frac{R_j^2 }{\widetilde V_i(R_j)}\right)^{1-\frac 32 \epsilon} e^{-c\frac{|x|^2+|y|^2}{t}}.
\end{eqnarray*}
Combining this with
Lemma \ref{lem:phi-r-equal}, we have
%for any  $x \in F^{(R_i)}_i$, $y \in E_j \setminus F^{(R_j)}_j$ and $t \ge R_j^2$,
\begin{eqnarray*}
|t\partial_th_t(x,y)|
%\sim \log(2+|y|)|t\partial_t\tilde h_t(x,y)|
%&&\le\frac{C \log(2+|y|)}{t\log^2(2+\sqrt t)}\left(\frac{\widetilde{V}_0(|x|)}{{V}_i(|x|)}\right)^{\frac 32 \epsilon}
%\left(\frac{\log(2+R_j)|x|^2}{ V_i(|x|)}\right)^{1-3\epsilon}
%e^{-\frac{(1-3\epsilon)|x|^2+|y|^2}{ct}}\\
%&&\le\frac{C \log(2+|y|)}{t\log^2(2+\sqrt t)}\left(\frac{|x|^2 \log^2(2+|x|)}{{V}_i(|x|)}\right)^{\frac 32 \epsilon}
%\left(\frac{\log(2+R_j)|x|^2}{ V_i(|x|)}\right)^{1-3\epsilon} \\
%%e^{-\frac{|x|^2+|y|^2}{ct}}\\
\lesssim_\epsilon\frac{1}{t} \frac{\log(2+|x|)}{\log(2+ \sqrt t)}
\left(\frac{|y|^2}{\widetilde V_i(|y|)} + \frac{R_j^2 }{\widetilde V_i(R_j)}\right)^{1- \frac 32 \epsilon} e^{-c\frac{|x|^2}{t}}
\lesssim_\epsilon \frac{1}{t}
%\left(\frac{|x|^2}{ V_i(|x|)}\right)^{1-\frac32\epsilon}
  \lesssim_\epsilon\frac{1}{R_j^2}
  \sim_\epsilon V_j(R_j)^{-1},
\end{eqnarray*}
where in the second inequality we used the fact
$\log(2+|x|) e^{-c\frac{|x|^2}{t}} \lesssim \log(2+\sqrt t)$ and Lemma \ref{lem:tilde-V} (ii).
Hence, by suitably choosing $\epsilon$, one has
\[
    \|t\L e^{-t\L}\|_{L^1(F^{(R_i)}_i)\to L^\infty(E_j\setminus F^{(R_j)}_j)}
    \lesssim V_j(R_j)^{-1}
    \sim \mu(F^{(R_i)}_i)^{-1},
\]
which further implies that for all $1 \leq i, j \leq \ell$, $n_i >2$, $n_j=2$ and each $1 < p <2$,
\[
      \|t\L e^{-t\L}\|_{L^2(F^{(R_i)}_i)\to L^p(E_j\setminus F^{(R_j)}_j)}
      \lesssim_p \mu(F^{(R_i)}_i)^{\frac12 - \frac1p},
\quad \forall t \ge R_j^2.
\]
%\begin{align*}
%\|t\L e^{-t\L}\|_{L^p(F^{(R_i)}_i)\to L^\infty(E_j\setminus F^{(R_j)}_j)}
%&\le \frac{C R_j^2}{t} \frac{1}{R_j^2} \left(\int_{F^{(R_i)}_i} d\mu\right)^{1 - \frac1 p} \\
%&\le \frac{C R_j^2}{t} \frac{1}{V_j(R_j)} \mu_i(F^{(R_i)}_i)^{1 - \frac1 p} \\
%&\le \frac{C R_j^2}{t} V_j(R_j)^{- \frac1 p},
%\end{align*}
%which together with $$\|t\L e^{-t\L}\|_{L^p(F^{(R_i)}_i)\to L^p(E_j\setminus F^{(R_j)}_j)}\le C,$$
%and the Riesz-Thorin interpolation Theorem, where $\frac{2/p}{p} + \frac{1-2/p}{\infty} = \frac{1}{2}$  gives that
%\begin{eqnarray*}
%\|t\L e^{-t\L}\|_{L^p(F^{(R_i)}_i)\to L^2(E_j\setminus F^{(R_j)}_j)}
%%\le \| t\L e^{-t\L} \|^{\frac{2-p}{2}}_{L^p(F^{(R_i)}_i)\to L^\infty(E_j\setminus F^{(R_j)}_j)}
%\le C\left(\frac{R_j^{2}}{t}\right)^{1-\frac p2}
%{V}_j(R_j)^{\frac 12-\frac 1p}\le C\mu(F^{(R_i)}_i)^{\frac 12-\frac 1p}.
%\end{eqnarray*}
%\begin{eqnarray*}
%\|t\L e^{-t\L}\|_{L^p(F^{(R_i)}_i)\to L^2(E_j\setminus F^{(R_j)}_j)}
%%\le C \| t\L e^{-t\L} \|^{\frac{2-p}{2}}_{L^p(F^{(R_i)}_i)\to L^\infty(E_j\setminus F^{(R_j)}_j)}
%&&\le C\left(\frac{R_j^{2}}{t}\right)^{1-\frac p2}
%{V}_j(R_j)^{\frac 12-\frac 1p}\le C\mu(F^{(R_i)}_i)^{\frac 12-\frac 1p}.
%\end{eqnarray*}

\textbf{Case 9: \boldmath $n_i<2$, $n_j=2$.} % It holds $R_i\ge cR_j$.
%By Lemmas \ref{lem:phi-r-equal} and \ref{lem:tilde-V}, it holds
% $$\phi_{i}(|x|) \sim |x|^{2-n_i}, \quad \phi_{j}(|y|) \sim \log(2+|y|), \quad  \widetilde{V}_i(x) \sim |x|^{4-n_i} \quad \mbox{and} \quad \widetilde{V}_i(R_j) \sim R_j^{4-n_i}.$$
For each $1 < p <2$, let us fix an $0 < \epsilon < 1/8$ small enough such that
$$\frac{3}{2} \epsilon (2-n_i) < \frac{1}{p'}  (2-n_i).$$
Hence, it holds
\begin{equation} \label{ci1}
    R_j^{\frac32 \epsilon(2-n_i)}
    \leq R_j^{\frac{1}{p'}(2-n_i)}
    = R_j^{2- \frac2p} R_j^{-\frac{n_i}{p'}}
    \sim R_j^2 V_j(R_j)^{-\frac1p} V_i(R_j)^{-\frac1{p'}}.
\end{equation}
This further implies that
\begin{equation} \label{ci2}
    R_j^{(1-\frac 32\epsilon)(n_i-2)}
    \leq R_j^{n_i-2} R_j^{\frac 32\epsilon(2-n_i)}
    \lesssim R_j^{n_i} V_j(R_j)^{-\frac1p} V_i(R_j)^{-\frac1{p'}}.
\end{equation}
%Using Lemmas \ref{lem:phi-r-equal} and \ref{lem:tilde-V}, it follows from Proposition \ref{time-heat-parabolic-2} (ii) that
Using Proposition \ref{time-heat-parabolic-2} (vi) and Lemma \ref{V0-general-equal}, we get that
%for any  $x \in F^{(R_i)}_i$, $y \in E_j \setminus F^{(R_j)}_j$ and $t \ge R_j^2$,
\begin{eqnarray*}
 |t\partial_t\tilde h_t(x,y)|
  \lesssim_\epsilon
 \frac{1}{t\log(2+\sqrt t)}
 \left(\frac{|y|^2}{\widetilde V_i(|y|)} + \frac{R_j^2 }{\widetilde V_i(R_j)}\right)^{1-\frac 32 \epsilon} e^{-c\frac{|x|^2+|y|^2}{t}},
\end{eqnarray*}
which together with  Lemma \ref{lem:phi-r-equal} and Lemma \ref{lem:tilde-V} yields that
\begin{eqnarray*}
|t\partial_th_t(x,y)|
%&&\sim \phi_{i}(|x|) \phi_{j}(|y|) |t\partial_t\tilde h_t(x,y)|\nonumber\\
%&&\lesssim \frac{\phi_{i}(|x|)\phi_{j}(|y|)}{t\log(2+\sqrt t)} \left(\frac{|x|^2}{\widetilde{V}_i(|x|)} + \frac{R_j^2}{\widetilde{V}_i(R_j)}\right)^{1-\frac32 \epsilon} e^{-\frac{|x|^2+|y|^2}{ct}}\nonumber \\
&&\lesssim \frac{\log(2+|x|) |y|^{2-n_i}}{t\log(2+\sqrt t)} \left(\frac{|y|^2}{|y|^{4-n_i}} + \frac{R_j^2}{R_j^{4-n_i}}\right)^{1-\frac32 \epsilon}
e^{-c\frac{|x|^2+|y|^2}{t}}\nonumber\\
&&\lesssim \frac{|y|^{2-n_i}}{t} \left(|y|^{n_i-2} +R_j^{n_i-2}\right)^{1-\frac32 \epsilon}
e^{-c\frac{|x|^2+|y|^2}{t}},
\end{eqnarray*}
where %in the last inequality
we used the fact that %when $|y| \leq \sqrt t$, we have $\frac{\log(2+|y|)}{ \log(2+\sqrt t)} \leq 1$;
%hen $|y| \geq \sqrt t$, for any $m >0$, we have
$\log(2+|x|) e^{-c\frac{|x|^2}{t}} \lesssim  \log(2+\sqrt t).$
When $|y|\leq R_j$, via \eqref{ci1} and the fact $t \ge R_j^2$, one has
\[%\label{eq:tL-1}
|t\partial_th_t(x,y)|
 \lesssim \frac{1}{t} |y|^{2-n_i}  |y|^{(\frac32 \epsilon-1 )(2-n_i)}%|y|^{(1-\frac32 \epsilon )(n_i-2)}
\lesssim \frac{1}{t}  |y|^{\frac 32 \epsilon(2-n_i)}
\lesssim \frac{1}{t} R_j^{\frac 32 \epsilon(2-n_i)}
\lesssim V_j(R_j)^{-\frac1p} V_i(R_j)^{-\frac1{p'}}.
\]
%by the fact $t \ge R_j^2$.
When $|y| >R_j$, via \eqref{ci2}, it holds
\[
|t\partial_th_t(x,y)|
 \lesssim \frac{|y|^{2-n_i}}{t}  R_j^{(1-\frac 32\epsilon)(n_i-2)} e^{-c\frac{|y|^2}{t}}
   \lesssim  |y|^{-n_i} R_j^{n_i} V_j(R_j)^{-\frac1p} V_i(R_j)^{-\frac1{p'}}. %e^{-\frac{|x|^2+|y|^2}{ct}}.
\]
Combining the above two estimates and applying the approach similar to \eqref{eq:s2b2}, it holds
\begin{eqnarray*} %\label{eq:ix-1}
 \|t\L e^{-t\L}\|_{L^p(F^{(R_i)}_i)\to L^\infty(E_j\setminus F^{(R_j)}_j)}
&&  \lesssim V_j(R_j)^{-\frac1p} V_i(R_j)^{-\frac1{p'}}  \lf(\int_{F^{(R_i)}_i \cap \{y\in E_i: \, |y| \le R_j\}} \, d\mu(y)\r)^{\frac1{p'}}  \notag \\
&&  \  + R_j^{n_i} V_j(R_j)^{-\frac1p} V_i(R_j)^{-\frac1{p'}} \lf(\int_{F^{(R_i)}_i \cap \{y\in E_i: \, |y|> R_j\}} |y|^{-n_ip'} \, d\mu(y)\r)^{\frac1{p'}}\nonumber\\
&&  \lesssim  V_j(R_j)^{-\frac1p}
    \sim \mu(F^{(R_i)}_i)^{-\frac1p},
%&&\lesssim \|f\|_{L^p(F^{(R_i)}_i)} V_j(R_j)^{-1/p}
%+ \|f\|_{L^p(F^{(R_i)}_i)} R_j^{n_i/2} V_j(R_j)^{-1/p} V_i(R_i)^{-1/p'} R_j^{-\frac{n_i}2} V_i(R_i)^{-1/p'} \nonumber\\
%&&\lesssim \frac{1}{V_j(R_j)^{1/p}} \|f\|_{L^p(F^{(R_i)}_i)}.
\end{eqnarray*}
%where in the second inequality we used the fact that %$p'/2 >1$,
%%$\frac12 - \frac{1}{p'} >0$
%\[
%     \lf(\int_{\{x\in F^{(R_i)}_i:\,|x|\le R_j\}} \,d\mu(x)\r)^{1/p'}  \lesssim V_i(R_i)^{1/p'}
%\]
%and
%\begin{align*}
%    \left(\int_{\{x\in F^{(R_i)}_i:\,|x|\ge R_j\}} |x|^{- n_ip'/2}\,d\mu(x)\right)^{1/p'}
%    & \lesssim \sum_{k=0}^{\infty}\left(  \int_{\{x\in E_i:\,2^kR_j \leq |x| \leq 2^{k+1} R_j\}} |x|^{-n_ip'/2}\,d\mu(x)\right)^{1/p'} \\
%    & \lesssim \sum_{k=0}^{\infty} 2^{-kn_i/2} R_j^{-n_i/2} V_i(2^{k}R_j)^{1/p'}
%    %& \lesssim \sum_{k=0}^{\infty} 2^{k(-\frac{n_i}{2})} 2^{k \frac{n_i}{p'}} R_j^{-\frac{n_i}{2}} R_j^{\frac{n_i}{p'}} \\
%      \lesssim  R_j^{-n_i/2} V_i(R_j)^{1/p'}.
%\end{align*}
%Combing \eqref{eq:ix-1}
% with $\|t\L e^{-t\L}(f\chi_{F^{(R_i)}_i})\|_{L^p(E_j\setminus F^{(R_j)}_j)}\lesssim \|f\|_{L^p(F^{(R_i)}_i)}$, we can deduce that
%\begin{eqnarray*}
%\|t\L e^{-t\L}\|_{L^p(F^{(R_i)}_i)\to L^2(E_j\setminus F^{(R_j)}_j)}
%\lesssim \| t\L e^{-t\L} \|^{\frac{2-p}{2}}_{L^p(F^{(R_i)}_i)\to L^\infty(E_j\setminus F^{(R_j)}_j)}
%&&\lesssim %\left(\frac{R_j^{\frac{n_i}{2}}}{t^{n_i/4}}\right)^{1-\frac p2}
%{V}_j(R_j)^{\frac 12-\frac 1p} \lesssim \mu(F^{(R_i)}_i)^{\frac 12-\frac 1p}.
%\end{eqnarray*}
which further implies that for all $1 \leq i, j \leq \ell$, $n_i <2$, $n_j=2$ and each $1 < p <2$,
\[
      \|t\L e^{-t\L}\|_{L^2(F^{(R_i)}_i)\to L^p(E_j\setminus F^{(R_j)}_j)}
      \lesssim_p \mu(F^{(R_i)}_i)^{\frac12 - \frac1p},
\quad \forall t \ge R_j^2.
\]

 Collecting the estimates above, we finish the proof.
\end{proof}

\section{Proof of the main result} \label{sec:main-proof} \hskip\parindent
    In this section, we will show our main result, i.e.,
the Riesz transform $\nabla \L^{-1/2}$ is bounded on $L^q(M)$ for each
$1 < q<2$,
%Theorem \ref{main-result-parabolic}.
    % It is worth to note that
%     the $L^p$-Davies-Gaffney estimate,
%    the mapping properties of the heat semigroup together with its time derivative,
%    and the $L^1$ estimate of the space derivative of the heat kernel
%    have been established.
%     %in Subsection \ref{sec:DG-es} and Section \ref{sec:map-proper},
%    and hence
    %The argument in this section is similar to that in \cite[Section 4]{jiang-li-lin-2022}.
%    We include it for the sake of  clarity. %/completeness.
%whose proof is inspired by \cite{cd99,jiang-li-lin-2022}.
The methods of \cite{cd99,jiang-li-lin-2022} play important roles in our proof.
Let us outline the main approach for the convenience of the reader.

%    ��һ�� ֱ��˵(step 1)���������ݣ�
%    �����ڵ�һ�����棬 ����Щ���зֽ⡣
%    Ȼ����Subsection a, b, c���ֱ������ǽ��й��ơ�
%    ������ in subsection d�� ���ǽ�����֮ǰ���õĽ����� �õ���T����L^p �н��ԡ�
%
%    3.4�й����Ⱥ˵Ŀռ䵼���Ĺ��ƣ� ���ǽ�����part away from the center.
%    3.3��mapping property, etL, tLetL���н��ԣ����ǽ�������part around the center

    \textbf{Step 1: Reduction.}

    In Subsection \ref{sec:main-proof-red},
 %   by Lemma \ref{local-part} and
%    the Marcinkiewicz interpolation theorem, we convert the proof of the main result,  into the  proof of the weakly $L^p$-boundedness of $T$ for  $p$ close to $1$,
 %     where
 %   $$T:=\frac{1}{\sqrt \pi} \int_1^\infty\nabla e^{-t\L}\frac{\,dt}{\sqrt t}.$$
    we will show that the small time part of the Riesz transform
    \[
   S:= \frac{1}{\sqrt \pi} \int_0^1 \nabla e^{-s\L}\frac{\,ds}{\sqrt s},
    \]
is $L^p$-bounded for $1 <p <2$ and
%    we point out that
    hence our main task is reduced to show that the large time part of the Riesz transform
        $$T:=\frac{1}{\sqrt \pi} \int_1^\infty\nabla e^{-t\L}\frac{\,dt}{\sqrt t}$$
    is weakly $(p,p)$ bounded for some $1< p <2$ but close to $1$, by the natural $L^2$-boundedness of the Riesz transform and the the Marcinkiewicz interpolation theorem.
    %
  %
%
%    we first transform the $L^p$-boundedness of the Riesz transform into the weak $L^p$-boundedness of $T$ by using Lemma \ref{local-part} and
%    the Marcinkiewicz interpolation theorem, where
%    $$T:=\int_1^\infty\nabla e^{-t\L}\frac{\,dt}{\sqrt t}.$$
   %

    % By a simple calculation, we obtain that for any $\lambda > \|f\|_{p}$, it holds
    %A simple calculation leads to that for any $\lambda > \|f\|_{p}$, it holds
    It is simple to obtain that  for any $\lambda > \|f\|_{p}$,
    \[
        \mu(\{x \in M: |Tf| > \lambda\}) \lesssim \frac{\|f\|_p^p}{\lambda^p}.
    \]
    For the case $0 < \lambda \leq \|f\|_p$, we decompose $M$ into three parts as
 %   \[
%    M=E_0 \cup (\cup_{i=1}^\ell F^{(R_i)}_i) \cup \lf[\cup_{j=1}^\ell (E_j \setminus F^{(R_j)}_j)\r].
%    \]
    \[
    M=E_0 \cup (\mathop{\cup}\limits_{i=1}^\ell F^{(R_i)}_i) \cup \lf[\mathop{\cup}\limits_{j=1}^\ell (E_j \setminus F^{(R_j)}_j)\r].
    \]
%    \[
%    M=E_0 \cup (\mathop{\cup}\limits_{i=1}^\ell F^{(R_i)}_i) \cup \Big[\mathop{\cup}\limits_{j=1}^\ell (E_j \setminus F^{(R_j)}_j)\Big].
%    \]
   It is obvious that the first two parts satisfy
    \[
        \mu\lf(E_0\cup(\cup_{i=1}^\ell F^{(R_i)}_i)\r) \lesssim \frac{\|f\|_p^p}{\lambda^p}.
    \]
%    For the third part, we further decompose it into the center part, the part away from the center, the part around the center as
%%    \begin{eqnarray*}
%%    &&\mu\left(\{x\in \cup_{j=1}^\ell E_j\setminus F_j:|Tf|>(2\ell+1)\lambda\}\right)\\
%%    &&\le \mu\left(\{x\in \cup_{j=1}^\ell E_j\setminus F_j:|T(f\chi_{E_0})|>\lambda\}\right)
%%        +\sum_{i=1}^\ell \mu\left(\{x\in \cup_{j=1}^\ell E_j\setminus F_j:|T(f\chi_{E_i\setminus F_i})|>\lambda\}\right)\\
%%    &&\quad +\sum_{i=1}^\ell \mu\left(\{x\in \cup_{j=1}^\ell E_j\setminus F_j:|T(f\chi_{F_i})|>\lambda\}\right).
%%    \end{eqnarray*}
%    \begin{eqnarray*}
%&& \mu\left(\lf\{x\in \bigcup_{j=1}^\ell \, (E_j\setminus F^{(R_j)}_j):|Tf(x)|>(2\ell+1)\lambda\r\}\right) \\
%&& \ \le \mu\left(\lf\{x\in \bigcup_{j=1}^\ell \, (E_j\setminus F^{(R_j)}_j):|T(f\chi_{E_0})(x)|>\lambda\r\}\right) \\
%&& \ \ +\sum_{i=1}^\ell \mu\left(\lf\{x\in \bigcup_{j=1}^\ell \, (E_j\setminus F^{(R_j)}_j):|T(f\chi_{F^{(R_i)}_i})(x)|>\lambda\r\}\right)     \\
%&& \ \ +\sum_{i=1}^\ell \mu\left(\lf\{x\in \bigcup_{j=1}^\ell \, (E_j\setminus F^{(R_j)}_j):|T(f\chi_{E_i\setminus F^{(R_i)}_i})(x)|>\lambda\r\}\right).
%\end{eqnarray*}
For the third part $\cup_{j=1}^\ell (E_j \setminus F^{(R_j)}_j)$, we shall further decompose $f$ into %three parts as
\[
    f = f\chi_{E_0} +\sum_{i=1}^\ell f\chi_{F^{(R_i)}_i}+ \sum_{i=1}^\ell f\chi_{E_i\setminus F^{(R_i)}_i},
\]
and then treat them in the following steps.
 %   We shall estimate them in Subsections \ref{sec:es-on-center}, \ref{sec:es-awar-center} and \ref{sec:es-ard-center}, respectively.
%    Finally, in Subsection \ref{sec:com-prof}, combining the previous results, we obtain the $L^p$ boundedness of the Riesz transform for $1< p <2$.

    \textbf{Step 2: Estimate on the center.}

    In Subsection \ref{sec:es-on-center}, we can easily obtain that the center part  satisfies
    \[
        \mu\left(\lf\{x\in \bigcup_{j=1}^\ell \, (E_j\setminus F^{(R_j)}_j):|T(f\chi_{E_0})(x)|>\lambda\r\}\right) \lesssim \frac{\|f\|_p^p}{\lambda^p}.
    \]%which is obvious.

    \textbf{Step 3: Estimate on the part around the center.}

    In Subsection \ref{sec:es-ard-center},
     in order to deal with the part
    \[
        \mu\left(\lf\{x\in \bigcup_{j=1}^\ell \, (E_j\setminus F^{(R_j)}_j):|T(f\chi_{F^{(R_i)}_i})(x)|>\lambda\r\}\right),
    \]
     we shall use the boundedness of the operators $e^{-t \L}$ and $t \L e^{-t \L}$ from $L^p(F_i^{(R_i)})$ to $L^2(E_j \backslash F_j^{(R_j)})$ established in Section \ref{sec:map-proper}, and the $L^p$-Davies-Gaffney estimates from Subsection \ref{sec:DG-es}.

    \textbf{Step 4: Estimate on the part away from the center.}

    In Subsection \ref{sec:es-awar-center}, we shall handle the part
    \[
        \mu\left(\lf\{x\in \bigcup_{j=1}^\ell \, (E_j\setminus F^{(R_j)}_j):|T(f\chi_{E_i\setminus F^{(R_i)}_i})(x)|>\lambda\r\}\right),
    \]
    by incorporating some ideas from \cite{cd99} and using the Calder\'{o}n-Zygmund decomposition to decompose $f \chi_{E_i \backslash F_i^{(R_i)}}$ as $g_i+\sum_k b_{i k}$. %, where $g_i$ and $b_{i k}$ are supported on $E_i$.
    %We can easily estimate the part $g_i$.
  %  The estimates of the good part $g_i$ is obvious.
   The estimates of the good part
    \[ \mu\lf(\lf\{x\in \bigcup_{j=1}^\ell \, (E_j\setminus F^{(R_j)}_j):|Tg_i(x)|>\lambda\r\}\r)
        \lesssim \frac{\|f\|_p^p}{\lambda^p}\]
    is obvious.
    However, %for the part $\sum_{k}b_{ik}$, we need to further divide it into the on-diagonal part
    the bad part needs to be further divided into the on-diagonal part
    \[
        \mu\lf(\lf\{x\in E_i \setminus F^{(R_i)}_i: \lf|T\lf(\sum_k b_{ik}\r)(x)\r|>\lambda\r\}\r)
    \]
    and the off-diagonal part
    \[
        \mu\lf(\lf\{x\in \bigcup_{1\leq j\neq i \leq \ell} \, (E_j\setminus F^{(R_j)}_j):\lf|T\lf(\sum_k b_{ik}\r)(x)\r|>\lambda\r\}\r).
    \]
  %  \begin{align*}%\label{CZ-decompo-2}
%        & \sum_{i=1}^\ell \mu\lf(\lf\{x\in \bigcup_{j=1}^\ell \, (E_j\setminus F_j):\lf|T\lf(\sum_k b_{ik}\r)(x)\r| >2 \lambda\r\}\r)  \\
%        &  \ \leq  \sum_{i=1}^\ell \mu\lf(\lf\{x\in E_i \setminus F^{(R_i)}_i: \lf|T\lf(\sum_k b_{ik}\r)(x)\r|>\lambda\r\}\r) \\
%        & \ \ +\sum_{i=1}^\ell \mu\lf(\lf\{x\in \bigcup_{1\leq j\neq i \leq \ell} \, (E_j\setminus F^{(R_j)}_j):\lf|T\lf(\sum_k b_{ik}\r)(x)\r|>\lambda\r\}\r).
%    \end{align*}
%         \begin{eqnarray*}%\label{CZ-decompo-1}
%        && \sum_{i=1}^\ell \mu\left(\lf\{x\in \bigcup_{j=1}^\ell \, (E_j\setminus F^{(R_j)}_j):|T(f\chi_{E_i\setminus F^{(R_i)}_i})(x)|>3\lambda\r\}\right) \nonumber \\
%        && \ \leq \sum_{i=1}^\ell \mu\lf(\lf\{x\in \bigcup_{j=1}^\ell \, (E_j\setminus F^{(R_j)}_j):|Tg_i(x)|>\lambda\r\}\r) \\
%        && \ \
%            + \sum_{i=1}^\ell \mu\lf(\lf\{x\in E_i \setminus F^{(R_i)}_i:\lf|T\lf(\sum_k b_{ik}\r)(x)\r|>\lambda\r\}\r) \notag \\
%        && \ \ + \sum_{i=1}^\ell \mu\lf(\lf\{x\in \bigcup_{1\leq j\neq i \leq \ell} \, (E_j\setminus F^{(R_j)}_j):\lf|T\lf(\sum_k b_{ik}\r)(x)\r|>\lambda\r\}\r).
%    \end{eqnarray*}
    %{\color{blue}and} %̫��and �ɸ�Ϊthenʲô�ġ�
%    Then
%    the required
%    estimates will be deduced
    Then, we shall deduce the required  estimates
    by using the $L^p$-Davies-Gaffney estimates of the operators $\nabla e^{-t \L}$ and $\L e^{-t \L}$ established in Subsection \ref{sec:DG-es},
    %the heat kernel estimate in Section \ref{sec:map-proper}
    and the heat kernel estimate in Subsection \ref{sec:heat-ker-away-center-est}.
   %  For the diagonal part, we shall deduce the required estimate by using the $L^p$-boundedness (Davies-Gaffney estimate) of $\nabla e^{-t \L}$ and $\L e^{-t \L}$, and heat kernel estimate in Subsection XXX.
%     The $L^p$-Davies-Gaffney estimate of the operators $\nabla e^{-t \L}$ and $\L e^{-t \L}$ established in Section XXX is sufficient to yield the estimate on the off-diagonal part.

    \textbf{Step 5: Completion of the proof.}

    In Subsection \ref{sec:com-prof}, by combining previously obtained results, we show that $\nabla \L^{-1/2}$ is weakly $(p,p)$ bounded %/$(p, p)$ bounded
    for each $1<p<p_0$.
    An application of the Marcinkiewicz interpolation theorem gives the desired result.
 %
   %{\color{gray} We would like to mention that the $L^p$-boundedness of $\nabla e^{-t\L}$ (Proposition \ref{mapping-gradient-heat}) and the $L^p$-Davies-Gaffney estimate (Corollary \ref{davies-operators})
%    will be used in Subsections \ref{sec:es-awar-center} and \ref{sec:es-ard-center};
%    The estimate of the heat kernel's spatial gradient away from the center (Proposition \ref{est-integral-diagonal-general}) will be used in Subsection \ref{sec:es-awar-center};
%    The mapping property of $e^{-t\L}$ and $t\L e^{-t\L}$ (Propositions \ref{map-heat-parabolic} and \ref{mapping-time-heat-parabolic-1}) will be used in Subsection \ref{sec:es-ard-center}.}
%
   % \textbf{Step 2: center} ��subsection 2 �У� ���ǱȽ����׵صõ������Ĳ��ֵع���
%
%    \textbf{Step 3: away center} ��subsection 3�У� �������Ƚ���Calderon-Zygmund�ֽ⽫$fxx$�ֽ�Ϊgood part $g$��bad par $\sum b$. g�Ĺ����ǱȽ����׵õ��ġ� b�Ĺ�����Ҫ������ʽ�ٴη�Ϊ
%    off the diagonal part��on the diagonal part.
%    off the diagonal part: �Ⱥ��Ͻ�(ME), prop 2.1 (\nabla e-tL��Lp�н���)�� D-G estimate
%    on the diagonal: ����off�Ľ����� ʵ���Ͼ���(ME); proposition 3.10(�Ⱥ˿ռ䵼��Զ�����ĵĹ���)�� prop 2.1 (\nabla e-tL��Lp�н���)��D-G estimate
%
%    \textbf{Step 3: around center} Proposition 3.8, 3.9 (Mapping property of e-tL, tLe-tL), prop 2.1 (\nabla e-tL��Lp�н���)��D-G estimate
%
\subsection{Reduction} \label{sec:main-proof-red} \hskip\parindent
Recall that the Riesz transform
$$
\nabla \L^{-1 / 2}=\frac{1}{\sqrt{\pi}} \int_{0}^{\infty} \nabla e^{-s \L} \frac{\, ds}{\sqrt{s}}
$$
is naturally bounded on $L^2(M)$.
Hence for each $1 < q <2$, by the Marcinkiewicz interpolation theorem,
if we show that there exists some $1 < p < q$ such that $\nabla \L^{-1 / 2}$ is
weakly $(p,p)$ bounded, then the Riesz transform $\nabla \L^{-1 / 2}$
is bounded on $L^q(M)$.
%
%
%
%xxxxxx
To this end, we shall decompose the Riesz transform into the the small time part and the large time part.
%We will prove that the small time part is $L^p$-bounded blow.
The $L^p$-boundedness of the small time part can be obtained with the help of \cite[Theorem 1.2]{cd99}; see Proposition \ref{local-part} below.
Then the main difficulty is to %show that  the big time part is weakly $(p,p)$ bounded,
consider the large time part,
which costs us a huge effort in the rest of Section \ref{sec:main-proof}.
%
%
%The Marcinkiewicz interpolation theorem implies that the $L^p$-boundedness of $\nabla \L^{-1 / 2}$ for any $1 <p <2$
%%the main result
%is obtained by
%the weakly $L^{q}$-boundedness of $\nabla \L^{-1 / 2}$
%%for some $0 < q < p$.
%for $0 < q < p$.

%\subsubsection{Small time part of the Riesz transform} \hskip\parindent
%In this subsection, we shall prove that
%the small time part of the Riesz transform, i.e.,
%\[
%    \frac{1}{\sqrt \pi} \int_0^1 \nabla e^{-s\L}\frac{\,ds}{\sqrt s},
%\]
%is bounded on $L^p(M)$ for $1<p<2$.
%This is ensured by the fact that
% the local Riesz transform
%$$\nabla(1+\L)^{-1/2}=\frac{1}{\sqrt{\pi}} \int_0^\infty\nabla e^{-s-s\L}\frac{\,ds}{\sqrt s}$$
%is $L^p$-bounded under the assumptions of Theorem \ref{main-result-parabolic}.
%
%
%In this subsection,
\subsubsection{Local estimates} \hskip\parindent
To begin with,
we shall prove that
the small time part of the Riesz transform, i.e.,
\[
   S= \frac{1}{\sqrt \pi} \int_0^1 \nabla e^{-s\L}\frac{\,ds}{\sqrt s},
\]
is bounded on $L^p(M)$ for each $1<p<2$.
This result is actually deduced from %/ensured by
the fact that the local Riesz transform
$$\nabla(1+\L)^{-1/2}=\frac{1}{\sqrt{\pi}} \int_0^\infty\nabla e^{-s-s\L}\frac{\,ds}{\sqrt s}$$
is $L^p$-bounded under the assumptions of Theorem \ref{main-result-parabolic}.
%In fact,
%%suppose that each $M_i$ satisfies $(D)$ and {\color{blue}$(UE)$},
%%\[
%%xxxxxxxxxxxxxxx
%%\]
%from Lemma \ref{lem:volume-growth} (iii)
%and the small time heat kernel Gaussian upper bound (cf. \cite[Corollary 4.7]{gri-sal09})),
%we can use \cite[Theorem 1.2]{cd99} to deduce that the local Riesz transform $\nabla(1+\L)^{-1/2}$ is bounded on $L^p(M)$ for $1<p\le 2$.
More precisely, we have
\begin{prop}\label{local-part}
%Assume that $M=M_1\#\cdots \# M_\ell$, and each $M_i$ is a doubling manifold satisfying {\color{blue}the upper Gaussian bound}.
Under the assumptions of Theorem \ref{main-result-parabolic}, the small time part of the Riesz transform, i.e.,
the operator
$$S=\frac{1}{\sqrt \pi} \int_0^1 \nabla e^{-s\L}\frac{\,ds}{\sqrt s},$$
is bounded on $L^p(M)$ for each $1 < p \leq 2$.
\end{prop}
\begin{proof}
The proof is adapted from \cite[Lemma 2.9]{jiang-li-lin-2022}. We give a brief argument of this proof for completeness and the convenience of the
reader.

First, we obtain from Proposition \ref{mapping-gradient-heat} and the semigroup property that for any $1<p\le 2$,
\begin{eqnarray*}
\left\|\int_1^\infty \nabla e^{-s-s\L}\frac{\,ds}{\sqrt s}\right\|_{p\to p}&&\le
\int_1^\infty \left\|\nabla e^{-s-s\L}\right\|_{p\to p}\frac{\,ds}{\sqrt s}\\
&&\le \int_1^\infty e^{-s}\left\|\nabla e^{-\L}\right\|_{p\to p}\left\|e^{-(s-1)\L}\right\|_{p\to p}\frac{\,ds}{\sqrt s}\\
&&\lesssim \int_1^\infty e^{-s}\frac{\,ds}{\sqrt s}
\lesssim 1.
\end{eqnarray*}
%When $p=2$, the local Riesz transform $\nabla(1+\L)^{-1/2}$ is naturally bounded on $L^2(M)$. Moreover, for $1<p<2$, by assumption
%that $\nabla(1+\L)^{-1/2}$ is bounded on $L^p(M)$.
%This together with the fact that the local Riesz transform $\nabla(1+\L)^{-1/2}$ is bounded on $L^p(M)$ for $1 < p \leq 2$, implies that

Next, from Lemma \ref{lem:volume-growth} (iii)
and {the small time heat kernel Gaussian upper bound} (cf. Theorem \ref{thm:sim-heat-es} (i)),
we can use \cite[Theorem 1.2]{cd99} to deduce that the local Riesz transform $\nabla(1+\L)^{-1/2}$ is bounded on $L^p(M)$ for $1<p\le 2$,
which further implies that
\begin{eqnarray*}
\frac{1}{\sqrt \pi}\int_0^1 \nabla e^{-s-s\L}\frac{\,ds}{\sqrt s} = \nabla(1+\L)^{-1/2}-\frac{1}{\sqrt \pi}\int_1^\infty \nabla e^{-s-s\L}\frac{\,ds}{\sqrt s}
\end{eqnarray*}
is bounded on $L^p(M)$ for $1 < p  \leq 2$.

On the other hand, by Proposition \ref{mapping-gradient-heat}, it holds
\begin{eqnarray*}
\left\|\int_0^1 \nabla e^{-s-s\L}\frac{\,ds}{\sqrt s}-\int_0^1 \nabla e^{-s\L}\frac{\,ds}{\sqrt s}\right\|_{p\to p}
&&=\left\|\int_0^1(1-e^{-s}) \nabla e^{-s\L}\frac{\,ds}{\sqrt s}\right\|_{p\to p}\\
&&\lesssim \int_0^1(1-e^{-s}) \frac{1}{\sqrt s}\frac{\,ds}{\sqrt s}
\lesssim 1.
\end{eqnarray*}
%Notice that $1-e^{-s} \lesssim s$, for $s \in (0, 1]$.
We conclude therefore that the small part of the Riesz transform %$\int_0^1 \nabla e^{-s\L}\frac{\,ds}{\sqrt s}$
 is bounded on $L^p(M)$ for $1 < p \leq 2$,
%Moreover, since $\nabla\L^{-1/2}$ is bounded on $L^2(M)$, we also have $\int_1^\infty \nabla e^{-s\L}\frac{\,ds}{\sqrt s}$ is bounded on $L^2(M)$.
which completes the proof.
\end{proof}

\begin{rem} \rm
  Referring to the proof of Lemma \ref{lem:volume-growth} (iii)
  and Remark \ref{rem:sma-ue}, %the setting of \cite[Corollary 4.7]{gri-sal09}),
  the assumptions of Proposition \ref{local-part}
  can be relaxed to that
   each $M_i$ satisfies $(D)$ and $(UE)$.
    %the heat kernel satisfies the Gaussian upper bound
%\begin{equation} \tag{$UE$}
%h_{i, t}(x, y) \lesssim \frac{1}{V_i(x, \sqrt{t})} e^{-c\frac{d^2(x, y)}{t}}, \quad \forall t>0, \, x, y \in M_i.
%\end{equation}
%Indeed, if each $M_i$ satisfies $(D)$ and $(UE)$,
%%then Lemma \ref{lem:volume-growth} (iii)
%%and the small time heat kernel Gaussian upper bound hold
%then the local Riesz transform $\nabla(1+\L)^{-1/2}$ is $L^p$-bounded,
%which further deduces that the operator $\int_0^1 \nabla e^{-s\L}\frac{\,ds}{\sqrt s}$ is $L^p$-bounded.
\end{rem}

%Since $\nabla\L^{-1/2}$ is naturally bounded on $L^2(M)$,
%by Lemma \ref{local-part}, we have
By Proposition \ref{local-part} and the fact that $\nabla\L^{-1/2}$ is naturally bounded on $L^2(M)$, we have
%we also have $\int_1^\infty \nabla e^{-s\L}\frac{\,ds}{\sqrt s}$ is bounded on $L^2(M)$.
%This completes the proof.
\begin{cor}\label{cor:L2-T}
   The large time part of the Riesz transform %$\int_1^\infty \nabla e^{-s\L}\frac{\,ds}{\sqrt s}$
   is bounded on $L^2(M)$.
\end{cor}
\subsubsection{Further reduction} \hskip\parindent
Since we showed in Proposition  \ref{local-part} that the operator $S$ %$
%  \int_0^1 \nabla e^{-t \L} \frac{\,dt}{\sqrt{t}}
%$
is weakly $(p,p)$ bounded for any $1 < p <2$,
it reduces to show that there exists $1 < p <q$ such that
$$T=\frac{1}{\sqrt \pi} \int_1^\infty\nabla e^{-t\L}\frac{\,dt}{\sqrt t}$$
is weakly $(p,p)$ bounded, i.e., for each $f \in L^p(M)$, it holds
\begin{equation} \label{weak-T}
     \mu\left(\{x \in M: |Tf(x)| > \lambda\}\right) \lesssim \frac{\|f\|_p^p}{ \lambda^p},
    \quad \forall  \lambda > 0.
\end{equation}
In fact, we shall show  that \eqref{weak-T} holds %/%is valid
for each $1< p <p_0$, where $p_0$ is defined as
\begin{equation}\label{eq:p0-def}
    p_0:=\min\left\{p_1, \, \frac{N_\infty}{N_\infty-2}\right\} = \min\left\{2,\,  \min_{1\leq k \leq \ell, \, n_k >2} \frac{N_k}{N_k-n_k+2}, \, \frac{N_\infty}{N_\infty-2} \r\}.
\end{equation}

Let us begin with the following simple observation.
\begin{lem} \label{lem:big-lam}
      Let $1 < p <2$. For each $f \in L^p(M)$, it holds that
  \begin{equation*}%\label{eq:weak-T}
    \mu\left(\{x \in M: |Tf(x)| > \lambda\}\right) \lesssim \frac{\|f\|_p^p}{ \lambda^p},
    \quad \forall  \lambda > \|f\|_p.
\end{equation*}
\end{lem}
\begin{proof}
  By the Chebyshev inequality and the Minkowski inequality, we have
\[
    \mu\left(\left\{x \in M:\, |Tf(x)|>\lambda\right\}\right)
\le \frac{1}{\pi \lambda^2 }\left\|\int_1^\infty\nabla e^{-t\L}f\frac{\,dt}{\sqrt t}\right\|^2_2
\le  \frac{1}{\pi \lambda^2 }\left( \int_1^\infty \|\nabla e^{-t\L}f\|_2 \frac{\,dt}{\sqrt t}\right)^2.
\]
It follows from the semigroup property, Proposition \ref {mapping-gradient-heat} and the ultracontractivity of the heat semigroup (cf. Proposition \ref{prop:ultcon}) that for any $t \ge 1$,
\[
    \|\nabla e^{-t\L}f\|_2
    = \|\nabla e^{-\frac{1}{2}t\L} \circ e^{-\frac{1}{2}t\L} f \|_2
   \lesssim   \frac{1}{\sqrt t} \|e^{-\frac{1}{2}t\L} f \|_2
%    \lesssim   \frac{1}{\sqrt t} \frac{1}{t^{\delta(\frac 1p - \frac 12)}} \|f\|_p.
     \lesssim   \frac{1}{\sqrt t} {t^{-\delta(\frac 1p - \frac 12)}} \|f\|_p.
\]
Therefore, it holds
\[
    \mu\left(\left\{x \in M:\, |Tf(x)|>\lambda\right\}\right)
    \lesssim \frac{\|f\|_p^2}{\lambda^2}
%    \lesssim \frac{\|f\|_p^p}{\lambda^p} \left( \frac{\|f\|_p}{\lambda} \right)^{2-p}
     \lesssim \frac{\|f\|_p^p}{\lambda^p},
\]
which completes the proof.
\end{proof}
%Ϊ��֤��xxx�����н��ԣ� ����ֻ��֤��xxx����xxxstill holds. In such case, xxxxxxx.
%��M�ֽ����������ϡ�
%By xxx, it holods
%$$ǰ�������ϵĹ���xxxx.$$
%For the term xxxx, ���ǰ�����Ϊ������������.
%���ǽ��ѵ�һ������XXX�� �ڶ�������xxx,  ����������xxx, ����������С���������ǽ��й��ơ�
%{\color{blue}Now, it suffices to show that for any $1 < p <p_0$, $f \in L^p$ and $0<\lambda \leq \|f\|_p$, it still holds
%$\mu\left(\{x: |Tf| > \lambda\}\right) \lesssim \|f\|_p^p/ \lambda^p$.}
 Now, it is then enough to prove that \eqref{weak-T} holds for the case
$0<\lambda \leq \|f\|_p$ with $1 < p <p_0$.
In such case, we decompose $M$ as
%\[
%    M=E_0 \bigcup \lf(\bigcup_{i=1}^\ell F^{(R_i)}_i\r) \bigcup \lf[\bigcup_{j=1}^\ell (E_j \setminus F^{(R_j)}_j)\r],
%\]
%\[
%    M=E_0 \cup (\cup_{i=1}^\ell F^{(R_i)}_i) \cup \lf[\cup_{j=1}^\ell (E_j \setminus F^{(R_j)}_j)\r],
%\]
 \[
    M=E_0 \cup (\mathop{\cup}\limits_{i=1}^\ell F^{(R_i)}_i) \cup \lf[\mathop{\cup}\limits_{j=1}^\ell (E_j \setminus F^{(R_j)}_j)\r],
    \]
%\[
%    M=E_0 \cup \lf(\mathop{\cup}\limits_{i=1}^\ell F^{(R_i)}_i\r)
%    \cup \lf[\mathop{\cup}\limits_{j=1}^\ell (E_j \setminus F^{(R_j)}_j)\r],
%\]
%\[
%    M=E_0 \cup  (\mathop{\cup}\limits_{i=1}^\ell F^{(R_i)}_i)
%    \cup [\mathop{\cup}\limits_{j=1}^\ell (E_j \setminus F^{(R_j)}_j)],
%\]
%\[
%    M=E_0 \bigcup \bigg(\bigcup_{i=1}^\ell F^{(R_i)} \bigg) \bigcup \bigg[\bigcup_{j=1}^\ell (E_j \setminus F^{(R_j)}_j)\bigg],
%\]
where
%\begin{equation}\label{eq:F_i-set}
%F_i^{(R_i)}:=\{x \in E_i: \dist(x, E_0) \leq 2 R_i\} \text{ with } R_i \ge 1 \text { such that }
%      \mu(F_i^{(R_i)})
%      = 100 \ell \frac{\|f\|_p^p}{ \lambda^p}.
% %     = 100 \ell \|f\|_p^p / \lambda^p.
%\end{equation}
$
F_i^{(R_i)} = \{x \in E_i: \dist(x, E_0) \leq 2 R_i\}
$
with $R_i \ge 1$ such that
      \begin{equation} \label{eq:F_i-set}
      \mu(F_i^{(R_i)})
%      = 100 \ell \frac{\|f\|_p^p}{ \lambda^p}
      = 100 \ell \frac{\|f\|_p^p}{ \lambda^p}.
      \end{equation}
By the convention $\mu(E_0) =1$ (cf. Subsection \ref{sec:not-basic}), it holds
%assumed in Subsection \ref{sec:not-basic}, it holds
\begin{equation} \label{est-small-lambda-3}
\mu\lf(E_0\cup(\cup_{i=1}^\ell F^{(R_i)}_i)\r) = 1 + \sum_{i=1}^\ell \mu(F^{(R_i)}_i)
\lesssim \frac{\|f\|_p^p}{\lambda^p}.
%\lesssim {\color{blue}\frac{\|f\|_p^p}{\lambda^p}}.
\end{equation}
%\begin{eqnarray}\label{est-small-lambda-3}
%\mu\lf(E_0\cup\lf(\mathop{\cup}\limits_{i=1}^\ell F^{(R_i)}_i\r)\r) = 1 + \sum_{i=1}^\ell \mu(F^{(R_i)}_i)
%\lesssim \frac{\|f\|_p^p}{\lambda^p}.
%%\lesssim {\color{blue}\frac{\|f\|_p^p}{\lambda^p}}.
%\end{eqnarray}
    Therefore, it suffices to prove that for such $\lambda$, %each  $1 < p <p_0$,
    \begin{equation*}
      \mu\left(\lf\{x\in \bigcup_{j=1}^\ell \, (E_j\setminus F^{(R_j)}_j):|Tf(x)|> \lambda\r\}\right)
      \lesssim \frac{\|f\|_p^p}{\lambda^p}.
 %     \quad \forall 0<\lambda < \|f\|_p, \, f \in L^p.
    \end{equation*}
 %
%       \begin{equation*}
%      \mu\left(\lf\{x\in \bigcup_{1\leq j \leq \ell} \lf(E_j\setminus F^{(R_j)}_j\r):|Tf(x)|> \lambda\r\}\right)
%      \lesssim \frac{\|f\|_p^p}{\lambda^p}.
% %     \quad \forall 0<\lambda < \|f\|_p, \, f \in L^p.
%    \end{equation*}
%
%
%    \begin{equation*}
%      \mu\left(\lf\{x\in {\cup}_{j=1}^\ell (E_j\setminus F^{(R_j)}_j):|Tf(x)|> \lambda\r\}\right)
%      \lesssim \frac{\|f\|_p^p}{\lambda^p}.
% %     \quad \forall 0<\lambda < \|f\|_p, \, f \in L^p.
%    \end{equation*}
%
%       \begin{equation*}
%      \mu\left(\lf\{x\in \mathop{\cup}\limits_{1 \leq j\leq \ell} (E_j\setminus F^{(R_j)}_j):|Tf(x)|> \lambda\r\}\right)
%      \lesssim \frac{\|f\|_p^p}{\lambda^p}.
% %     \quad \forall 0<\lambda < \|f\|_p, \, f \in L^p.
%    \end{equation*}
%
For the simplicity of notation, we replace $\lambda$ by $(2\ell +1) \lambda$ (the same trick will be used repeatedly) and write
\begin{eqnarray*}
&& \mu\left(\lf\{x\in \bigcup_{j=1}^\ell \, (E_j\setminus F^{(R_j)}_j):|Tf(x)|>(2\ell+1)\lambda\r\}\right) \\
&& \ \le \mu\left(\lf\{x\in \bigcup_{j=1}^\ell \, (E_j\setminus F^{(R_j)}_j):|T(f\chi_{E_0})(x)|>\lambda\r\}\right) \\
&& \ \ +\sum_{i=1}^\ell \mu\left(\lf\{x\in \bigcup_{j=1}^\ell \, (E_j\setminus F^{(R_j)}_j):|T(f\chi_{F^{(R_i)}_i})(x)|>\lambda\r\}\right)     \\
&& \ \ +\sum_{i=1}^\ell \mu\left(\lf\{x\in \bigcup_{j=1}^\ell \, (E_j\setminus F^{(R_j)}_j):|T(f\chi_{E_i\setminus F^{(R_i)}_i})(x)|>\lambda\r\}\right)
   = : \mathcal{C}+\mathcal{A}+\mathcal{W}.
\end{eqnarray*}
%\begin{eqnarray*}
%&&\mu\left(\lf\{x\in \bigcup_{j=1}^\ell \lf(E_j\setminus F^{(R_j)}_j\r):|Tf(x)|>(2\ell+1)\lambda\r\}\right)\\
%&& \ \le \mu\left(\lf\{x\in \bigcup_{j=1}^\ell \lf(E_j\setminus F^{(R_j)}_j\r):|T(f\chi_{E_0})(x)|>\lambda\r\}\right)
%        +\sum_{i=1}^\ell \mu\left(\lf\{x\in \bigcup_{j=1}^\ell \lf(E_j\setminus F^{(R_j)}_j\r):|T(f\chi_{F^{(R_i)}_i})(x)|>\lambda\r\}\right)     \\
%&& \ \  +\sum_{i=1}^\ell \mu\left(\lf\{x\in \bigcup_{j=1}^\ell \lf(E_j\setminus F^{(R_j)}_j\r):|T(f\chi_{E_i\setminus F^{(R_i)}_i})(x)|>\lambda\r\}\right)
%    = : \mathcal{C}+\mathcal{A}+\mathcal{W}.
%\end{eqnarray*}
We shall call the first term as the center part,
the second term as the part around the center, and the last term as the part away from the center, respectively,
and treat them in the following subsections.

\subsection{Estimate on the center}\label{sec:es-on-center} \hskip\parindent
 First we establish the following lemma.
 \begin{prop} \label{es:center-part}
   Let $1 < p <2$. For each $f \in L^p(M)$, it holds that
   \[
     \mathcal{C}
     = \mu\left(\lf\{x\in \bigcup_{j=1}^\ell \, (E_j\setminus F^{(R_j)}_j):|T(f\chi_{E_0})(x)|>\lambda\r\}\right)
     \lesssim \frac{\|f\|_p^p}{\lambda^p},
     \quad \forall 0<\lambda \leq \|f\|_p.
   \]
 \end{prop}
\begin{proof}
  %since $\mu(E_0) = 1$,  we conclude via the Minkowski inequality, Proposition \ref {mapping-gradient-heat} and \eqref{heat-operator-norm-1} that
  %Applying the approach
  Similar to the proof of Lemma \ref{lem:big-lam}, one has
\begin{eqnarray*}
\mathcal{C}
&&\leq \frac{1}{\lambda^p} \|T(f\chi_{E_0}) \|_p^p
 \lesssim \frac{1}{\lambda^p}\left[\int_1^\infty \|\nabla e^{-t\L}(f\chi_{E_0})\|_p\frac{\,dt}{\sqrt t}\right]^p\\
&&\lesssim \frac{1}{\lambda^p}\left[\int_1^\infty  \| e^{-\frac{t}{2}  \L}(f\chi_{E_0})\|_p\frac{\,dt}{t}\right]^p
% \lesssim \frac{1}{\lambda^p}\left[\int_1^\infty \frac{1}{t^{\delta(1- \frac 1p)}} \|f\chi_{E_0}\|_1\frac{\,dt}{t}\right]^p \\
  \lesssim \frac{1}{\lambda^p}\left[\int_1^\infty {t^{-\delta(1- \frac 1p)}} \|f\chi_{E_0}\|_1\frac{\,dt}{t}\right]^p \\
&&\lesssim \frac{\|f\chi_{E_0}\|_1^p}{\lambda^p} \lesssim \frac{\|f\|_p^p}{\lambda^p},
\end{eqnarray*}
%where in the last inequality we used the  convention $\mu(E_0) =1$.
where the last inequality is due to the H\"older inequality and the convention $\mu(E_0) =1$.
This completes the proof.
\end{proof}
%
%
%==============\\
%The part on $E_0$ is easier, since $\mu(E_0) = 1$,  we conclude via the Minkowski inequality, Proposition \ref {mapping-gradient-heat} and \eqref{heat-operator-norm-1} that
%\begin{eqnarray*}
%\mu\left(\{x\in \cup_{j=1}^\ell E_j\setminus F_j:|T(f\chi_{E_0})|>\lambda\}\right)
%&&\le \frac{1}{\lambda^p}\left\|T(f\chi_{E_0})\right\|_p^p\\
%&&\le \frac{1}{\lambda^p}\left[\int_1^\infty \|\nabla e^{-t\L}(f\chi_{E_0})\|_p\frac{\,dt}{\sqrt t}\right]^p\\
%&&\lesssim \frac{1}{\lambda^p}\left[\int_1^\infty  \| e^{-\frac{t}{2}  \L}(f\chi_{E_0})\|_p\frac{\,dt}{t}\right]^p\\
%&&\lesssim \frac{1}{\lambda^p}\left[\int_1^\infty \frac{1}{t^{\delta(1- \frac 1p)}} \|f\chi_{E_0}\|_1\frac{\,dt}{t}\right]^p\\
%&&\lesssim \frac{\|f\chi_{E_0}\|_1^p}{\lambda^p} \lesssim \frac{\|f\|_p^p}{\lambda^p}.
%\end{eqnarray*}

\subsection{Estimate on the part around the center}\label{sec:es-ard-center}\hskip\parindent
For the part around the center, we have
\begin{prop}\label{es:around-center}
    Let $1 <p <2$. For each $f \in L^p(M)$, it holds that
  \[
  \ca=
  \sum_{i=1}^{\ell} \mu\left(\left\{x \in \bigcup_{j=1}^{\ell}\, (E_j\setminus F^{(R_j)}_j):
  |T(f \chi_{F_i^{(R_i)}})(x)|> \lambda\right\}\right)
  \lesssim \frac{\|f\|_p^p}{\lambda^p}, \quad \forall 0< \lambda \leq \|f\|_p.
  \]
\end{prop}
\begin{proof}
   %\textbf{(a) Decomposition:}
   For each $1\leq i,j \leq \ell$, we write
\begin{eqnarray*}%\label{key-est-around-neck}
&&\mu\left(\left\{x\in  E_j\setminus F^{(R_j)}_j:\, |T(f\chi_{F^{(R_i)}_i})(x)|> \lambda\right\}\right) \notag \\
&& \ \le \mu\left(\left\{x\in  E_j\setminus F^{(R_j)}_j:\, \lf|T\lf(f\chi_{F^{(R_i)}_i}-e^{-R_j^2\L}(f\chi_{F^{(R_i)}_i})\r)(x)\r|> \frac{\lambda}2\right\}\right)\nonumber\\
&& \ \ +\mu\left(\left\{x\in  E_j\setminus F^{(R_j)}_j:\, \lf|T\lf((f\chi_{F^{(R_i)}_i})\r)(x)\r|> \frac{\lambda}2\right\}\right)
=: \ca^{ij}_{1} + \ca^{ij}_2.
% =: {\color{blue} K_{j,i,1} + K_{j,i,2} + K^{ji}_{1} + K^1_{ji} + K^{(2)}_{ji} + K_{ji1}}.
\end{eqnarray*}

   \textbf{Estimation of \boldmath $\ca^{ij}_1$:}  %it holds
%\[
%        \mu(F^{(2R_j)}_j)  \sim \mu(F^{(R_j)}_j) \sim \frac{\|f\|_{p}^p}{\lambda^{p}},  \quad
%        \dist(E_j\setminus F^{(2R_j)}_j, F^{(R_i)}_i) \ge 2R_j.
%\]
%For any $1 \leq i, j \leq \ell$,
One can write
\[
 \ca^{ij}_1
 \le \mu(F^{(2R_j)}_j)
 + \mu\left(\left\{x \in  E_j \setminus F^{(2R_j)}_j:\, \lf|T\lf(f\chi_{F^{(R_i)}_i}-e^{-R_j^2\L}(f\chi_{F^{(R_i)}_i})\r)(x)\r|> \frac{\lambda}2 \right\}\right).
% & \lesssim \|f\|_p^p/ \lambda^p
% + \mu\left(\left\{x \in  E_j \setminus F^{(2R_j)}_j:\, |T(f\chi_{F_i}-e^{-R_j^2\L}f\chi_{F_i})(x)|>\frac12\lambda\right\}\right)
\]
For the first term, note that
$$\mu(F^{(2R_j)}_j)  \sim V_j(2R_j) \sim \mu(F^{(R_j)}_j) \sim \frac{\|f\|_{p}^p}{\lambda^{p}},$$
and hence it remains to estimate the second term.
%Using the Chebyshev inequality, the Minkowski inequality and the fact
%\[
%    \dist(E_j\setminus F^{(2R_j)}_j, F^{(R_i)}_i) \ge 2R_j,
%\]
%we deduce from Corollary \ref{cor:com-DG} that
By the the Chebyshev inequality, the Minkowski inequality, Corollary \ref{cor:com-DG} and the fact
\[
    \dist(E_j\setminus F^{(2R_j)}_j, F^{(R_i)}_i) \ge 2R_j,
\]
we obtain that
\begin{eqnarray*}%\label{key-est-around-neck-3}
&&\mu\left(\left\{x\in  E_j\setminus F^{(2R_j)}_j:\, \lf|T\lf(f\chi_{F^{(R_i)}_i}-e^{-R_j^2\L}(f\chi_{F^{(R_i)}_i})\r)(x)\r|> \frac{\lambda}2\right\}\right)\nonumber\\
&&\ \lesssim \frac{1}{\lambda^p}\int_{E_j\setminus F^{(2R_j)}_j}
                 \lf|T\lf(f\chi_{F^{(R_i)}_i}-e^{-R_j^2\L}(f\chi_{F^{(R_i)}_i})\r)\r|^p\,d\mu\nonumber\\
&&\ \lesssim \frac{1}{\lambda^p}\left[\int_1^\infty\int_0^{R_j^2}
         \|\nabla \L e^{-(s+t)\L}(f\chi_{F^{(R_i)}_i})\|_{L^p(E_j\setminus F^{(2R_j)}_j)}\frac{\,ds\,dt}{\sqrt t}\right]^p\nonumber\\
%&&\ \lesssim \frac{1}{\lambda^p}\left[\int_1^{\infty}\int_0^{R_j^2} \frac{\|f\chi_{F^{(R_i)}_i}\|_p}{(s+t)^{3/2}} e^{-\frac{\dist(E_j \setminus F^{(2R_j)}_j , F^{(R_i)}_i)^2}{s+t}} \frac{\,ds\,dt}{\sqrt t}\right]^p\nonumber\\
&&\ \lesssim \frac{\|f\chi_{F^{(R_i)}_i}\|^p_p}{\lambda^p}
     \left[\int_{1}^\infty\int_0^{R_j^2} \frac{1}{(s+t)^{3/2}} e^{-c\frac{R_j^2}{s+t}} \frac{\,ds\,dt}{\sqrt t}\right]^p
%&&\ \lesssim \frac{1}{\lambda^p}\left[\int_{R_j^2}^\infty\int_0^{R_j^2} \frac{\|f\chi_{F^{(R_i)}_i}\|_p}{(s+t)^{3/2}} \frac{\,ds\,dt}{\sqrt t}+\int_1^{R_j^2}\int_0^{R_j^2} \frac{\|f\chi_{F^{(R_i)}_i}\|_p}{(s+t)^{3/2}} e^{-c\frac{R_j^2}{s+t}} \frac{\,ds\,dt}{\sqrt t}\right]^p
  \lesssim \frac{\|f\|_p^p}{\lambda^p},
\end{eqnarray*}
since the integral here can be controlled by
\[
    \frac{1}{R_j^3} \int_0^{R_j^2}\int_0^{R_j^2} \frac{\,ds\,dt}{\sqrt t}
    +
    \int_{R_j^2}^\infty\int_0^{R_j^2} \frac{1}{t^{3/2}}  \frac{\,ds\,dt}{\sqrt t}
    \lesssim 1.
\]

    \textbf{Estimation of \boldmath $\ca^{ij}_2$:}
    Using the Chebyshev inequality and the Minkowski inequality, one has
\begin{eqnarray}\label{first-part-ard}
  \ca^{ij}_2
  && \lesssim  \frac{1}{\lambda^2}\int_{E_j\setminus F^{(R_j)}_j}\lf|T\lf(e^{-R_j^2\L}(f\chi_{F^{(R_i)}_i})\r)\r|^2\,d\mu \notag \\
  %&& \lesssim  \frac{1}{\lambda^2}\int_{E_j\setminus F^{(R_j)}_j} \left|\int_{1}^{\infty} \nabla e^{-t\L} \lf(e^{-R_j^2\L}(f\chi_{F^{(R_i)}_i})\r) \frac{\,dt}{\sqrt t} \right|^2\,d\mu \nonumber \\
  &&\lesssim \frac{1}{\lambda^2}\left(\int_1^\infty \|\nabla e^{-(t+R_j^2)\L}(f\chi_{F^{(R_i)}_i}) \|_{L^2(E_j\setminus F^{(R_j)}_j)}\frac{\,dt}{\sqrt t}\right)^2.
\end{eqnarray}
%\begin{equation}
%  K_{j,i,2}
%   \lesssim  \frac{1}{\lambda^2}\int_{E_j\setminus F^{(R_j)}_j}\lf|T\lf(e^{-R_j^2\L}(f\chi_{F^{(R_i)}_i})\r)\r|^2\,d\mu
%%  && \lesssim  \frac{1}{\lambda^2}\int_{E_j\setminus F^{(R_j)}_j} \left|\int_{1}^{\infty} \nabla e^{-t\L} \lf(e^{-R_j^2\L}(f\chi_{F^{(R_i)}_i})\r) \frac{\,dt}{\sqrt t} \right|^2\,d\mu \nonumber \\
% \lesssim \frac{1}{\lambda^2}\left[\int_1^\infty \left\|\nabla e^{-(t+R_j^2)\L}(f\chi_{F^{(R_i)}_i})\right\|_{L^2(E_j\setminus F^{(R_j)}_j)}\frac{\,dt}{\sqrt t}\right]^2.
%\end{equation}
    %Next, we shall estimate $\|\nabla e^{-(t+R_j^2)\L}(f\chi_{F_i^{(R_i)}})\|_{L^2(E_j\setminus F^{(R_j)}_j)}$ as follows.

Let $\psi_j$ be a Lipschitz function on $M$ with $\supp \psi_j\subset \{x \in E_j:\,\dist(x,E_j\setminus F^{(R_j)}_j)<R_j\}$ such that $\psi_j\equiv 1$ on $E_j\setminus F^{(R_j)}_j$  and $|\nabla \psi_j|\le C/R_j$.
%Note that for any $x \in \supp \psi_j$, it holds
%$
%    \dist(x, E_0) \geq \dist(E_0, E_j \setminus F^{(R_j)}_j) - d(x, E_j \setminus F^{(R_j)}_j)  > R_j,
%$
%which means that
Note that it holds
$\dist(\supp \psi_j,\,E_0)>R_j,$
and hence
$$ \supp \psi_j \subset E_j \setminus F^{(R_j/2)}_j.$$
% ��ʦ�� �������ǲ�ȷ���Ƿ������ڶ���psi��ʱ��ֱ��˵֧���������������档 ������ǰ������4.5�б�����ԭ����һ���ġ�
For the simplicity of notation, denote $f\chi_{F_j^{(R_j)}}$ and $t + R_j^2$ by $g$ and $s$.
We use the Caccioppoli
argument as follows %{\color{blue}for $g = f\chi_{F_j^{(R_j)}}$ and $s=t + R_j^2$,}
\begin{eqnarray*}
\int_{M}|\nabla e^{-s\L}g|^2\psi_j^2\,d\mu
&& =\int_M \nabla e^{-s\L}g \cdot \nabla  (\psi_j^2e^{-s\L}g)\,d\mu
   -2\int_M (\nabla e^{-s\L}g \cdot \nabla  \psi_j) \psi_j e^{-s\L}g\,d\mu\\
&& \le \int_M (\L e^{-s\L}g) (\psi_j^2e^{-s\L}g)\,d\mu
    +\frac{1}{2}\int_{M}|\nabla e^{-s\L}g|^2\psi_j^2\,d\mu \\
&& \     +2\int_{M} |\nabla \psi_j(x)|^2 |e^{-s\L}g(x)|^2\,d\mu(x).
\end{eqnarray*}
This fact, together with H\"older's inequality %, the fact $|\psi_j| \lesssim 1$ and $|\nabla \psi_j| \lesssim 1 /R_j$  implies
and the properties of $\psi_j$, implies that
\begin{eqnarray*}
&&\int_{E_j\setminus F^{(R_j)}_j}|\nabla e^{-s\L}g|^2\,d\mu\le \int_{M}|\nabla e^{-s\L}g|^2\psi_j^2\,d\mu\\
&& \ \le 2\left(\int_M |\L e^{-s\L}g|^2\psi_j^2\,d\mu\right)^{\frac12} \left(\int_M|\psi_j e^{-s\L}g|^2\,d\mu\right)^{\frac12}
       + 4 \int_{M} |\nabla \psi_j|^2 |e^{-s\L}g|^2\,d\mu \\
&& \  \lesssim \left(\int_{E_j \setminus F^{(R_j/2)}_j} |\L e^{-s\L}g|^2 \,d\mu\right)^{\frac12}
                \left(\int_{E_j \setminus F^{(R_j/2)}_j} |e^{-s\L}g|^2\,d\mu\right)^{\frac12}
       + \frac{1}{R_j^2} \int_{F^{(R_j)}_j \setminus F^{(R_j/2)}_j}  |e^{-s\L}g|^2\,d\mu.
\end{eqnarray*}
 Combining this with Propositions \ref{map-heat-parabolic} and \ref{mapping-time-heat-parabolic-1},
 we deduce that % for $1<p<2$,
\begin{eqnarray*}%\label{L2-nabla-f}
&& \|\nabla e^{-(t+R_j^2)\L}(f\chi_{F^{(R_i)}_i}) \|^2_{L^2(E_j\setminus F^{(R_j)}_j)}\nonumber \\
%&& \ \lesssim \frac{1}{t+R_j^2}\left(\int_{E_j\setminus F^{(R_j/2)}_j} |(t+R_j^2)\L e^{-(t+R_j^2)\L}f\chi_{F^{(R_i)}_i}|^2 \,d\mu\right)^{1/2} \left(\int_{E_j\setminus F^{R_j/2}_j} |e^{-(t+R_j^2)\L}f\chi_{F^{(R_i)}_i}|^2\,d\mu\right)^{1/2}\nonumber \\
%&& \ \quad  +\frac{1}{R_j^2}\int_{F^{(R_j)}_j \setminus F^{(R_j/2)}_j} (e^{-(t+R_j^2)\L}f\chi_{F^{(R_i)}_i})^2\,d\mu \nonumber \\
&& \ \lesssim \frac{\|(f\chi_{F^{(R_i)}_i})\|^2_{p}}{t+R_j^2} \frac{R_j}{\sqrt t + R_j} \mu(F^{(R_i)}_i)^{1 - \frac 2p}
    + \frac{1}{R_j^2} \|e^{-(t+R_j^2) \L} (f \chi_{F_i^{(R_i)}}) \|_{L^{\infty} (E_j \setminus F^{(R_j/2)}_j)}^2 \mu(F^{(R_j)}_j) \nonumber \\
&& \
\lesssim \frac{\|(f\chi_{F^{(R_i)}_i})\|_{p}^2}{t+R_j^2} \frac{R_j}{\sqrt t+R_j} \mu(F^{(R_i)}_i)^{1-\frac 2p}+\frac{\|(f\chi_{F^{(R_i)}_i})\|_{p}^2}{R_j^2}\left(\frac{R_j^{2}}{t +R_j^2}\right)^2 \mu(F^{(R_i)}_i)^{1-\frac 2p} \nonumber \\
%&& \
%\lesssim \frac{\|(f\chi_{F^{(R_i)}_i})\|_{p}^2}{t+R_j^2} \frac{R_j}{\sqrt t+R_j} \mu(F^{(R_i)}_i)^{1-\frac 2p}+\frac{\|(f\chi_{F^{(R_i)}_i})\|_{p}^2}{t+R_j^2} \frac{R_j^{2}}{t +R_j^2} \mu(F^{(R_i)}_i)^{1-\frac 2p} \nonumber\\
&& \ \lesssim \frac{\|(f\chi_{F^{(R_i)}_i})\|_{p}^2}{t+R_j^2}\frac{R_j}{\sqrt t+R_j} \mu(F^{(R_i)}_i)^{1-\frac 2p}.
\end{eqnarray*}

Inserting this estimate into \eqref{first-part-ard} and using \eqref{eq:F_i-set}, we conclude that
\begin{eqnarray*} %\label{key-est-around-neck-2}
  \ca^{ij}_2
  && \lesssim \frac{\|(f\chi_{F^{(R_i)}_i})\|_{p}^2}{\lambda^2 } \mu(F^{(R_i)}_i)^{1-\frac 2p}\left[\int_1^\infty \left(\frac{R_j}{\sqrt t+R_j}\right)^{\frac12}\frac{\,dt}{\sqrt{t+R_j^2}\sqrt t}\right]^2\nonumber\\
 && \lesssim \frac{\|f\|_{p}^2}{\lambda^2 } \mu(F^{(R_i)}_i)^{1-\frac 2p}
   % \lesssim \frac{\|f\|_{p}^2}{\lambda^2 } \left( \frac{\|f\|^p_p}{\lambda^p} \right)^{1-\frac 2p}
   \lesssim \frac{\|f\|_{p}^p}{\lambda^p },
\end{eqnarray*}
since the integral here can be controlled by
$$
    %\int_1^\infty \left(\frac{R_j}{\sqrt t+R_j}\right)^{1/2}\frac{\,dt}{\sqrt{t+R_j^2}\sqrt t}
%    \leq
    \int_0^{R_j^2} \frac{1}{R_j} \frac{\,dt}{\sqrt t}
    + \int_{R_j^2}^{\infty} \lf( \frac{R_j}{\sqrt t} \r)^{\frac12} \frac{\, dt}{t}
         %+ \int_{R_j^2}^{\infty} \frac{R_j^{1/2}}{t^{1/4}} \frac{\, dt}{t}
    \lesssim 1.
$$

Combing the estimates of $\ca^{ij}_1$ and $\ca^{ij}_2$, and summing over $1\leq i \leq \ell$ and $1 \leq j \leq \ell$, we have
\[
    \ca
    \leq \sum_{i=1}^{\ell} \sum_{j=1}^{\ell} (\ca^{ij}_1+ \ca^{ij}_2)
    \lesssim \frac{\|f\|_{p}^p}{\lambda^p }.
\]
This completes the proof.
\end{proof}

\subsection{Estimate on the part away from the center}\label{sec:es-awar-center} \hskip\parindent
  The main aim of this subsection is to prove the following proposition.
\begin{prop} \label{es:away-center}
 Let $1<p<p_0$, where $p_0$ is as in \eqref{eq:p0-def}. For each $f \in L^p(M)$, it holds that
\[
    \mathcal{W}=
    \sum_{i=1}^{\ell}
       \mu\lf(\lf\{
         x \in \bigcup_{j=1}^{\ell} \, (E_j\setminus F^{(R_j)}_j):
         |T(f \chi_{E_i \setminus F_i^{(R_i)}})(x)| > \lambda\r\}\r) \lesssim \frac{\|f\|_p^p}{\lambda^p},
         \quad \forall 0<\lambda \leq \|f\|_p.
\]
\end{prop}
 To this end, we shall decompose $\mathcal{W}$ into three parts with the help of the Calder\'on-Zygmund decomposition in Subsection \ref{sec:cz-decom}, and then estimate them in Subsections \ref{sec:es-of-g}-\ref{sec:off-diag}.

\subsubsection{Calder\'on-Zygmund decomposition } \label{sec:cz-decom} \hskip\parindent
For each $1 \leq i \leq \ell$, since $M_i$ is a doubling manifold,
we may run the Calder\'on-Zygmund decomposition (cf. \cite{cw77, cd99}) of $f\chi_{E_i\setminus F^{(R_i)}_i}$ on $M_i$ at the height $\delta \lambda$,
where $\delta >1$ is large enough to be fixed later,
and then obtain that
$$f\chi_{E_i\setminus F^{(R_i)}_i}=g_i + \sum_{k}b_{ik}.$$
Here $\supp b_{ik} \subset B_{ik} = B(x_{ik}, r_{ik})$, $\{B_{ik}\}$ has the bounded overlap property (i.e. $\sum_k \chi_{B_{ik}} \lesssim 1$), $|g_i|\le \delta \lambda$,
%where the supports of $\{b_{ik}\}$ is bounded overlap, $|g_i|\le C_{CZ}\lambda$,
\begin{equation} \label{g-es}
  \int_{M_i} |g_i|^p d\mu_i
  \lesssim \int_{M_i} |f \chi_{E_i \setminus F^{(R_i)}_i}|^p d\mu_i
  = \int_M |f \chi_{E_i \setminus F^{(R_i)}_i}|^p d \mu,
\end{equation}
\begin{equation} \label{CZ-1} %\tag{\textsl{CZ-1}}
\fint_{B_{ik}}|b_{ik}|^p\,d\mu_i \lesssim (\delta\lambda)^p,
\end{equation}
and
\begin{equation} \label{CZ-2} %\tag{\textsl{CZ-2}}
\sum_{k}\mu_i(B_{ik})
\lesssim \frac{1}{(\delta\lambda)^p} \int_{M_i}|f\chi_{E_i\setminus F^{(R_i)}_i}|^p\,d\mu_i
\lesssim \frac{\|f\|_p^p}{(\delta\lambda)^p}.
\end{equation}

 \textbf{Claim.}
%Fixed a large enough $\delta$, %depending only on the manifold itself,
For a fixed $\delta >1$ large enough,
it holds that
\begin{equation}\label{isolate-balls}
\dist(B_{ik},E_0)\ge \frac{1}{2}(R_i+r_{ik}).
\end{equation}
\textit{Proof of the Claim.}
%Let us prove the claim.
When $r_{ik}<R_i/2$,
%since $B_{ik}\cap (E_i\setminus F_i) \neq\emptyset$,
%we can pick a point $u \in B_{ik}\cap (E_i\setminus F^{(R_i)}_i)$,
for each $z \in B_{ik}\cap (E_i\setminus F^{(R_i)}_i)$,
it holds
$$\dist(B_{ik}, E_0) \geq \dist(z, E_0) - \diam(B_{ik}) \ge 2R_i - 2r_{ik} %> 2R_i - R_i = R_i
%\geq (R_i + r_{ik})/2,
\geq \frac12(R_i + r_{ik}),
$$
which means that \eqref{isolate-balls} holds.
%\\----------------------------------------------------\\
%The details of $\dist(B_{ik}, E_0) \geq \dist(y, E_0) - \diam(B_{ik})$ is as follows:
%For any $y \in B_{ik}\cap E_i\setminus F_i$, $z \in E_0$ and $x \in B_{ik}$, we have
%$$ d(y,z) \leq d(y,x) + d(x,z) \leq \diam(B_{ik}) + d(x,z),$$
%and hence
%$$d(y,E_0) =  \inf\limits_{z \in E_0} d(y,z) \leq  \diam(B_{ik})+ \inf\limits_{z \in E_0} d(x,z) =  \diam(B_{ik})+ d(x,E_0),$$ for any $x \in B_{ik}$.
%Moreover,
%$$ d(y, E_0)  \leq \diam(B_{ik}) + \inf\limits_{x \in B_{ik}} d(x,E_0) \leq \diam(B_{ik}) + \dist(B_{ik},E_0).$$
%-------------------------------------------------------------\\

When $r_{ik}\ge R_i/2$,  %we argue via the proof of contradiction. %we conclude the claim via arguing by contradiction.
%
%
%Otherwise $r_{ik}\ge R_i/2$, the proof is by contradiction.
%If the claim fails,
% prove it by negation
if  the claim is not true,
then
$$
\dist(B_{ik},E_0)<\frac12(R_i+r_{ik})
<\frac32r_{ik},
%< \frac{3}{2}r_{ik},
$$
%$$
%\dist(B_{ik},E_0)<\frac12(R_i+r_{ik})
%<\frac32r_{ik},
%%< \frac{3}{2}r_{ik},
%$$
which together with the convention $\diam(E_0) =2$ implies that
$$
    d(x_{ik}, o_i) \leq r_{ik} + \dist(B_{ik},E_0)+\diam(E_0) \lesssim r_{ik}.
$$
%\\=======THE REASON FOR THE FIRST INEQUALITY=========\\
%    For any $z \in B_{ik}$, $y \in E_0$, it holds
%    \[
%        d(x_{ik}, o_i)
%        \leq d(x_{ik},z) + d(z,y) + d(o_i,y)
%    \]
%    Taking the infimum of $z \in B_{ik}$, $y \in E_0$, we have
%    \[
%        d(x_{ik}, o_i)  \leq r_{ik} + \dist (B_{ik} , E_0) + \diam(E_0).
%    \]
%=========END=======================================\\
Combining this with the doubling condition, one concludes
\begin{equation*} %\label{eq:claim-con}
   \mu(F^{(R_i)}_i) \sim V_i(R_i) \lesssim V_i(o_i,2r_{ik}) \lesssim V_i(x_{ik},r_{ik}) \lesssim \mu_i(B_{ik}).
\end{equation*}
%\begin{equation*}
%\mu(B_{ik})\ge CV_i(x_{ik},3r_{ik})\ge CV_i(o_i,3r_{ik})\ge CV_i(R_i),
%\end{equation*}
Besides, it follows from \eqref{CZ-2} and the convention $\mu(F^{(R_i)}_i) = 100\ell \|f\|_p^p / \lambda^p$ that
$$
\mu_i(B_{ik}) \leq
\sum_{k}\mu_i(B_{ik})
\lesssim  \frac{\mu(F^{(R_i)}_i)}{\delta^p}.
%\lesssim \frac{\mu(F^{(R_i)}_i)}{\delta^p},
$$
%which contradicts with \eqref{eq:claim-con} as soon as we choose $\delta$ large enough.
%which contradicts with \eqref{cz-measure} as soon as we choose $\delta$ large enough.
%Therefore \eqref{isolate-balls} holds.                                           $\hfill \square$
The two inequalities above lead to a contradiction as soon as we choose $\delta$ large enough.
Hence the claim is proved.  $\hfill \square$ %/ we conclude the claim.

\begin{rem}
\rm
  By the triangle inequality, \eqref{isolate-balls} implies that  %(see \cite{jiang-li-lin-2022}) %(see \cite[(4.11) and (4.12)]{jiang-li-lin-2022})
   \begin{equation}\label{isolate-balls-1}
%\dist(B_{ik},E_0) \le
%(R_i+r_{ik})/2
\frac12 (R_i+r_{ik}) + r_{ik}
\le d(x_{ik},o_i)
  \le 3\dist(B_{ik},E_0);
\end{equation}
see \cite[(4.11) and (4.12)]{jiang-li-lin-2022}.
%see \cite{jiang-li-lin-2022}.
Besides, \eqref{isolate-balls} also tells us that
 $\supp b_{i k} \subset E_i$ (hence $\mu(B_{ik}) = \mu_i(B_{ik})$), which together with the identity $f \chi_{E_i \setminus F_i^{(R_i)}} = g_i+ \sum_k b_{i k}$ implies that
 $\supp g_i \subset E_i$. Since $\mu=\mu_i$ on $E_i$, the integrals of $g_i$ and $b_{ik}$ on $M_i$ are the same as on $M$, i.e.,
$$
\int_{M_i}|g_i|^p \, d\mu_i=\int_M|g_i|^p \, d\mu, \quad \int_{B_{ik}}|b_{i k}|^p \, d\mu_i = \int_{B_{ik}}|b_{ik}|^p \, d\mu.
$$
%and
%\begin{equation}\label{isolate-balls-2}
%  (R_i+r_{ik})/2 \le  d(x_{ik},o_i).
%\end{equation}
 %Notice that \eqref{isolate-balls} implies that $\dist(B_{ik},E_0)\le d(x_{ik},o_i) \le r_{ik} + \dist(B_{ik},E_0)+\diam(E_0)\le 3\dist(B_{ik},E_0)$
% and  $\frac{1}{2}(R_i+r_{ik})+r_{ik}\le \dist(B_{ik},E_0)+r_{ik}\le d(x_{ik},o_i)$ hold and then \eqref{isolate-balls-1} \eqref{isolate-balls-2}
\end{rem}
%
%This claim further implies that
%\begin{equation}\label{isolate-balls-1}
%\dist(B_{ik},E_0)\le d(x_{ik},o_i) \le  \dist(B_{ik},E_0)+\diam(E_0)+r_{ik}\le 3\dist(B_{ik},E_0),
%\end{equation}
%and
%\begin{equation}\label{isolate-balls-2}
%\frac{1}{2}(R_i+r_{ik})+r_{ik}\le \dist(B_{ik},E_0)+r_{ik}\le d(x_{ik},o_i).
%\end{equation}
%Note that \eqref{isolate-balls} also implies that
%\[
%    \supp b_{ik}, \, \supp g_i \subset E_i.
%\]
%%\\---------------------------------------------------\\
%%The details of the last inequality:
%%
%%Assume that $\ell_{x_{ik}, o_i} \cap \partial B_{ik} =c$, $\dist(B_{ik} , E_0) = d(a,b)$.
%%Then
%%$$ d(a,b) \leq d(c, o_i) \leq d(x_{ik} , o_i) - r_{ik}.$$
%%
%%We used this inequality \eqref{isolate-balls-2} below, but it seems that the inequality
%%$$ \frac{1}{2}(R_i+r_{ik}) \le \dist(B_{ik},E_0) \le d(x_{ik},o_i)$$ is enough.
%%\\-------------------------------------------------\\

    Using the Calder\'on-Zygmund decomposition, one writes
     \begin{eqnarray*}%\label{CZ-decompo-1}
        && \sum_{i=1}^\ell \mu\left(\lf\{x\in \bigcup_{j=1}^\ell \, (E_j\setminus F^{(R_j)}_j):|T(f\chi_{E_i\setminus F^{(R_i)}_i})(x)|>3\lambda\r\}\right) \nonumber \\
        && \ \leq \sum_{i=1}^\ell \mu\lf(\lf\{x\in \bigcup_{j=1}^\ell \, (E_j\setminus F^{(R_j)}_j):|Tg_i(x)|>\lambda\r\}\r) \\
        && \ \
            + \sum_{i=1}^\ell \mu\lf(\lf\{x\in E_i \setminus F^{(R_i)}_i:\lf|T\lf(\sum_k b_{ik}\r)(x)\r|>\lambda\r\}\r) \notag \\
        && \ \ + \sum_{i=1}^\ell \mu\lf(\lf\{x\in \bigcup_{1\leq j\neq i \leq \ell} \, (E_j\setminus F^{(R_j)}_j):\lf|T\lf(\sum_k b_{ik}\r)(x)\r|>\lambda\r\}\r).
    \end{eqnarray*}
   %====================
%    \[
%    \mu\lf(\lf\{x\in \bigcup_{\substack{1 \leq j \leq \ell \\ j \neq i}} \lf(E_j\setminus F^{(R_j)}_j\r):\lf|T\lf(\sum_k b_{ik}\r)\r|>\lambda\r\}\r)
%    \]
%     \[
%    \mu\lf(\lf\{x\in \bigcup_{1 \leq j \leq \ell,\, j \neq i} \lf(E_j\setminus F^{(R_j)}_j\r):\lf|T\lf(\sum_k b_{ik}\r)\r|>\lambda\r\}\r)
%    \]
%    \begin{align*}%\label{CZ-decompo-1}
%        & \sum_{i=1}^\ell \mu\left(\lf\{x\in \bigcup_{j=1}^\ell (E_j\setminus F^{(R_j)}_j):|T(f\chi_{E_i\setminus F^{(R_i)}_i})|>3\lambda\r\}\right) \nonumber \\
%        & \ \leq \sum_{i=1}^\ell \mu\lf(\lf\{x\in \bigcup_{j=1}^\ell (E_j\setminus F^{(R_j)}_j):|T(g_i)|>\lambda\r\}\r)
%        + \sum_{i=1}^\ell \mu\lf(\lf\{x\in \bigcup_{1\leq j\neq i \leq \ell} (E_j\setminus F^{(R_j)}_j):\lf|T\lf(\sum_k b_{ik}\r)\r|>\lambda\r\}\r) \notag\\
%        & \ \
%            + \sum_{i=1}^\ell \mu\lf(\lf\{x\in E_i \setminus F^{(R_i)}_i:\lf|T\lf(\sum_k b_{ik}\r)\r|>\lambda\r\}\r).
%    \end{align*}
%    ====================
    We shall call the first term as the good part, the second term as the bad parts on the diagonal, and the
    last term as the bad parts off the diagonal,  respectively, and then handle them in the following subsections.

%    ====================\\
%    Using the Calder\'on-Zygmund decomposition, we split the part away from the center $f\chi_{E_i\setminus F_i}$ into its good part $g$ and its bad part $\sum_{k}b_{ik}$ as follows
%    \begin{align}\label{CZ-decompo-1}
%        & \sum_{i=1}^\ell \mu\left(\{x\in \cup_{j=1}^\ell E_j\setminus F_j:|T(f\chi_{E_i\setminus F_i})|>\lambda\}\right) \nonumber \\
%        & \leq \sum_{i=1}^\ell \mu(\{x\in \cup_{j=1}^\ell E_j\setminus F_j:|T(g_i)|>\lambda/2\})
%         + \sum_{i=1}^\ell \mu(\{x\in \cup_{j=1}^\ell E_j\setminus F_j:|T(\sum_{k}b_{ik})|>\lambda/2\})
%    \end{align}
%    We can easily obtain the estimate of the term involving $g$; see \S\ref{sec:es-of-g} for details.
%    However, with respect to the estimate of the term involving $\sum_{k}b_{ik}$, we need to further divide it into an off-diagonal part and an on-diagonal part as follows
%    \begin{align}\label{CZ-decompo-2}
%        & \sum_{i=1}^\ell \mu(\{x\in \cup_{j=1}^\ell E_j\setminus F_j:|T(\sum_{k}b_{ik})|>\lambda\}) \nonumber \\
%        & \leq \sum_{i=1}^\ell \mu(\{x\in \cup_{j\neq i} E_j\setminus F_j:|T(\sum_k b_{ik})|>\lambda\})
%           + \sum_{i=1}^\ell \mu(\{x\in E_i \setminus F_i:|T(\sum_k b_{ik})|>\lambda\}),
%    \end{align}
%    and deduce the required
%    estimate by using the $L^p$-boundedness of $\nabla e^{-t\L}$, the $L^p$-Davies-Gaffney estimate of $\L e^{-t\L}$, the heat kernel estimate in Section \ref{sec:map-proper} and Subsection \ref{sec:heat-ker-away-center-est}.
%    Below, we will estimate them separately.

\subsubsection{Estimate of the good part} \label{sec:es-of-g} \hskip\parindent
    We can easily obtain the following lemma.
\begin{lem}
  Let $1 < p <2$. For each $f \in L^p(M)$, it holds that
  \[
  \sum_{i=1}^\ell \mu\lf(\lf\{x\in \bigcup_{j=1}^\ell \, (E_j\setminus F^{(R_j)}_j):|Tg_i(x)|>\lambda\r\}\r)
  \lesssim \frac{\|f\|_p^p}{\lambda^p}, \quad
  \forall 0<\lambda \leq \|f\|_p.
  \]
\end{lem}
\begin{proof}
The $L^2$-boundedness of $T$ (cf. Corollary \ref{cor:L2-T}) together with the fact $|g_i| \lesssim \lambda$  implies that
\begin{eqnarray*}%\label{es-g}
\sum_{i=1}^\ell \mu\lf(\lf\{x\in \bigcup_{j=1}^\ell \, (E_j\setminus F^{(R_j)}_j):|Tg_i(x)|>\lambda\r\}\r)
&&\lesssim \sum_{i=1}^\ell\frac{1}{\lambda^2}\int_{M}|g_i|^2\,d\mu \\
% \lesssim \sum_{i=1}^\ell \frac{1}{\lambda^p}\int_{E_i} \left(\frac{|g_i|}{\lambda}\right)^{2-p} |g_i|^p\,d\mu \nonumber\\
&& \lesssim \sum_{i=1}^\ell\frac{1}{\lambda^p}\int_{E_i}|g_i|^p\,d\mu
  \lesssim \frac{\|f\|_p^p}{\lambda^p},
\end{eqnarray*}
because of \eqref{g-es}. This completes the proof.
\end{proof}

\subsubsection{Estimate of the bad parts on the diagonal}\label{sec:on-diag}\hskip\parindent
Recall that
$$
    T = \frac{1}{\sqrt{\pi}} \int_1^{\infty} \nabla e^{-t \L}  \frac{\, dt}{\sqrt{t}}
   % = \frac{1}{\sqrt{\pi}} \int_0^{\infty} \nabla e^{-t \L}  \frac{\, dt}{\sqrt{t}}
%       - \frac{1}{\sqrt{\pi}} \int_1^{\infty} \nabla e^{-t \L}  \frac{\, dt}{\sqrt{t}}
    = \nabla \mathcal{L}^{-1/2} - \frac{1}{\sqrt{\pi}} \int_0^{1} \nabla e^{-t \L}  \frac{\, dt}{\sqrt{t}}.
$$
%
%$$
%T=\int_1^{\infty} \nabla e^{-t \L}  \frac{\, dt}{\sqrt{t}},
%  \qquad \nabla \mathcal{L}^{-1/2} = \frac{1}{\sqrt{\pi}} \int_0^{\infty} \nabla e^{-t \L}  \frac{\, dt}{\sqrt{t}} .
%$$
For each $1\leq i\leq \ell$, we write %split the bad part on the diagonal into two parts
\begin{eqnarray*}
&& \mu\left(\left\{x\in  E_i\setminus F^{(R_i)}_i:\left|T\left(\sum_k  b_{ik}\right)(x)\right|>3\lambda\right\}\right) \\
&&\ \le \mu\left(\left\{x\in E_i\setminus F^{(R_i)}_i:\left|\nabla \L^{-1/2}\left(\sum_k b_{ik}\right)(x)\right|> 2\lambda\right\}\right) \\
&& \  \ + \mu\left(\left\{x\in E_i\setminus F^{(R_i)}_i:\left|\int_0^1\nabla e^{-t\L}\left(\sum_k b_{ik}\right)(x)\frac{\,dt}{\sqrt t}\right|>  \sqrt\pi \lambda \right\}\right)
   =: \mathcal{S}_1 + \cs_2.
\end{eqnarray*}
%\begin{eqnarray*}
%&&  \mu\left(\left\{x\in  E_i\setminus F^{(R_i)}_i:\left|T\left(\sum_k  b_{ik}\right)(x)\right|>\lambda\right\}\right) \\
%&&\ \le \mu\left(\left\{x\in E_i\setminus F^{(R_i)}_i:\left|\nabla \L^{-1/2}\left(\sum_k b_{ik}\right)(x)\right|> \frac{\lambda}{2}\right\}\right) \\
%&& \  \ + \mu\left(\left\{x\in E_i\setminus F^{(R_i)}_i:\left|\int_0^1\nabla e^{-t\L}\left(\sum_k b_{ik}\right)(x)\frac{\,dt}{\sqrt t}\right|>\frac{\lambda}{2}\right\}\right) \\
%&& \ \le \mu\left(\left\{x\in E_i\setminus F^{(R_i)}_i:\left|\nabla \L^{-1/2}\left(\sum_k e^{-r_{ik}^2\L} b_{ik}\right)(x)\right|> \frac{\lambda}{2}\right\}\right) \\
%&&\ +\mu\left(\left\{x\in E_i\setminus F^{(R_i)}_i:\left|\nabla \L^{-1/2}\left(\sum_k (1-e^{-r_{ik}^2\L}) b_{ik}\right)(x)\right|>\frac{\lambda}{2}\right\}\right) \\
%&& \  \ + \mu\left(\left\{x\in E_i\setminus F^{(R_i)}_i:\left|\int_0^1\nabla e^{-t\L}\left(\sum_k b_{ik}\right)(x)\frac{\,dt}{\sqrt t}\right|>\frac{\lambda}{2}\right\}\right)
% =: \mathcal{S}_1 + \cs_2 + \cs_3 \\
%\end{eqnarray*}

The estimation of the term $\cs_1$ is somewhat complicated. We shall decompose it into two parts as follows and then estimate them respectively.
\begin{eqnarray*}
 \cs_1
&&=\mu\left(\left\{x\in E_i\setminus F^{(R_i)}_i:\left|\nabla \L^{-1/2}\left(\sum_k b_{ik}\right)(x)\right|>  2\lambda\right\}\right)\nonumber\\
&&\le \mu\left(\left\{x\in E_i\setminus F^{(R_i)}_i:\left|\nabla \L^{-1/2}\left(\sum_k (1-e^{-r_{ik}^2\L}) b_{ik}\right)(x)\right|> {\lambda}\right\}\right) \\
&&\ +
\mu\left(\left\{x\in E_i\setminus F^{(R_i)}_i:\left|\nabla \L^{-1/2}\left(\sum_k e^{-r_{ik}^2\L} b_{ik}\right)(x)\right|> {\lambda}\right\}\right)
    = : \cs_{11} + \cs_{12}.
\end{eqnarray*}

We begin with the estimates of  the term $\cs_{11}$.

\begin{lem}
Let $1<p<p_0$, where $p_0$ is as in \eqref{eq:p0-def}.
%We have that for all $f \in L^p(M)$ and all $0<\lambda \leq \|f\|_p$,
For each $f \in L^p(M)$, it holds that
$$
\cs_{11}
=\mu\left(\left\{x \in E_i \setminus F_i^{(R_i)}
   :\left|\nabla \L^{-1/2}\left(\sum_k(1-e^{-r_{ik}^2 \L}) b_{ik}\right)(x)\right|>{\lambda}\right\}\right)
 \lesssim \frac{\|f\|_p^p}{\lambda^p},
 \quad \forall 0<\lambda \leq \|f\|_p.
$$
\end{lem}
\begin{proof}
   % \textbf{Step 1: Decomposition.}
   \textbf{Step 1.} In this step, we decompose the estimation of $\cs_{11}$ into four parts and handle the first part.

    As in \cite[pp. 1160-1161]{cd99}, using \eqref{isolate-balls}, we write
     % The role of \eqref{isolate-balls}: The reason for choosing the piecewise points of integration
\begin{eqnarray*}
 \sqrt{\pi} \nabla \L^{-1/2}(1-e^{-r_{ik}^2\L})
&&=\int_0^{\infty}\nabla e^{-t\L}\frac{\,dt}{\sqrt t}-\int_0^{\infty}\nabla e^{-(t+r_{ik}^2)\L}\frac{\,dt}{\sqrt t} \\
%=\int_0^{\infty}\nabla e^{-t\L}\frac{\,dt}{\sqrt t}-\int_{r^2_{ik}}^{\infty}\nabla e^{-t\L}\frac{\,dt}{\sqrt {t-r_{ik}^2}} \\
&& =\int_{0}^\infty\left[\frac{1}{\sqrt t}- \frac{\chi_{(r_{ik}^2,\infty)}(t) }{\sqrt{t-r_{ik}^2}}\right]\nabla e^{-t\L}\,dt\\
&&=\lf\{\int_{0}^{R_i^2} + \int_{R_i^2}^{4d(x_{ik},o_i)^2} +\int_{4d(x_{ik},o_i)^2}^\infty\r\} \cdots \,dt\\
&& =:T_{ik1}+T_{ik2}+T_{ik3}.
\end{eqnarray*}

It holds then
\[
    \cs_{11}
    \leq \mu(\cup_m 2B_{im}) + \sum_{j=1}^3 \cs_{11j},
\]
where
\[
    \cs_{11j}
    : = \mu\left(\left\{x\in E_i\setminus (F^{(R_i)}_i\cup (\cup_m 2B_{im}) ):\left|\sum_k T_{ikj}b_{ik}(x)\right|>\frac{\sqrt \pi}3\lambda\right\}\right),
    \quad 1 \leq j \leq 3.
\]

    The estimation of $\mu(\cup_m 2B_{im})$ is standard.
    Indeed, via the doubling property and \eqref{CZ-2}, one has
    $$\mu(\cup_m 2B_{im}) \lesssim \frac{\|f\|_p^p}{\lambda^p}.$$
    %$$\mu(\cup_m 2B_{im}) \lesssim \frac{\|f\|_p^p}{\lambda^p}.$$

    %\textbf{(b) Next, we estimate \boldmath $\cs_{111}$:}

    \textbf{Step 2.} In this step, we estimate the term $\cs_{111}$.

     Recall that \eqref{isolate-balls} says
     $$B_{ik} \subset E_i \setminus \{z \in E_i: \dist(z, E_0) < R_i/2 \} = E_i \setminus F^{(R_i/4)}_i.$$
     Then for any $0 < t \leq R_i^2$ and $y \in B_{ik}$, by Proposition \ref{est-integral-diagonal-general}, it holds
$$\int_{E_i \setminus (F^{(R_i)}_i \cup 2B_{ik})}|\nabla h_t(x,y)|\,d\mu(x)
\le \int_{E_i\setminus (F^{(R_i/4)}_i \cup B(y,r_{ik}) )} |\nabla h_t(x,y)|\,d\mu(x)
\lesssim \frac{1}{\sqrt t}e^{-c\frac{r_{ik}^2}{t}}.$$
Hence, it follows from Fubini's theorem that
\begin{eqnarray*}
    \int_{E_i \setminus (F^{(R_i)}_i \cup 2B_{ik})} |\nabla e^{-t\L}b_{ik}(x)| \, d\mu(x)
    &&= \int_{E_i \setminus (F^{(R_i)}_i \cup 2B_{ik})} \left|\int_{B_{ik}} \nabla h_t(x,y) b_{ik}(y) \, d\mu(y)\right| \, d\mu(x)\\
    &&\leq \int_{B_{ik}} |b_{ik}(y)| \int_{E_i \setminus (F^{(R_i)}_i \cup 2B_{ik})} | \nabla h_t(x,y) | \, d\mu(x) \, d\mu(y)\\
    && \lesssim \frac{1}{\sqrt t}e^{-c\frac{r_{ik}^2}{t}} \|b_{ik}\|_{1}.
\end{eqnarray*}
Therefore, following the argument in \cite[p. 1161]{cd99}, we deduce that
\begin{eqnarray*}
 \cs_{111}
&& \lesssim \frac{1}{\lambda} \int_{E_i\setminus (F^{(R_i)}_i\cup \lf(\cup_m 2B_{im}\r))} \left|\sum_k T_{ik1}b_{ik}\right| \,d\mu \\
%&&\lesssim \frac{1}{\lambda} \int_{E_i\setminus (F^{(R_i)}_i\cup 2B_{ik})} \sum_k \int_{0}^{R_i} \left|\frac{1}{\sqrt t}-\frac{\chi_{\{t>r_{ik}^2\}}}{\sqrt{t-r_{ik}^2}}\right| |\nabla e^{-t\L}b_{ik}| \,dt \,d\mu \\
&& \lesssim \frac{1}{\lambda}\sum_k\int_{0}^{R_i^2}\int_{E_i\setminus (F^{(R_i)}_i \cup 2B_{ik})}
\left|\frac{1}{\sqrt t}-\frac{\chi_{(r_{ik}^2,\infty)}(t)}{\sqrt{t-r_{ik}^2}}\right| |\nabla e^{-t\L}b_{ik}|\,d\mu\,dt\\
&&\lesssim \frac{1}{\lambda}\sum_k \|b_{ik}\|_1 \int_{0}^{R_i^2} \frac{1}{\sqrt t}e^{-c\frac{r_{ik}^2}{t}}\left|\frac{1}{\sqrt t}-\frac{\chi_{(r_{ik}^2,\infty)}(t)}{\sqrt{t-r_{ik}^2}}\right|\,dt\\
%&&\lesssim \frac{1}{\lambda}\sum_k\|b_{ik}\|_1 \left[\int_{0}^{r_{ik}^2} \frac{1}{\sqrt t}e^{-c\frac{r_{ik}^2}{t}}\frac{1}{\sqrt t}\,dt+\int_{r_{ik}^2}^\infty \frac{1}{\sqrt t} \lf(\frac{1}{\sqrt{t-r_{ik}^2}} - \frac{1}{\sqrt t} \r)\,dt\right]\\
&&\lesssim \frac{1}{\lambda}\sum_k\|b_{ik}\|_1
  % \lesssim \sum_k\mu(B_{ik})\\
 \lesssim \frac{\|f\|_p^p}{\lambda^p},
\end{eqnarray*}
where we have used  the H\"older inequality, \eqref{CZ-1} and \eqref{CZ-2} in the last inequality.
% The details of the integration is bounded can be seen in \cite[p. 1161]{cd99}. The details is as follows.
%\\======Details of the integration==============\\
%\begin{eqnarray*}
%  && \int_{0}^{R_i^2} \frac{1}{\sqrt t}e^{-c\frac{r_{ik}^2}{t}}\left|\frac{1}{\sqrt t}-\frac{\chi_{\{t>r_{ik}^2\}}}{\sqrt{t-r_{ik}^2}}\right|\|b_{ik}\|_1\,dt \\
%  && \ \lesssim
%  \int_{0}^{r_{ik}^2} \frac{1}{\sqrt t}e^{-c\frac{r_{ik}^2}{t}}\frac{1}{\sqrt t}\,dt
%  +
%  \int_{r_{ik}^2}^\infty \frac{1}{\sqrt t} \lf(\frac{1}{\sqrt{t-r_{ik}^2}} - \frac{1}{\sqrt t} \r)\,dt
%\end{eqnarray*}
%It holds
%\[
%    \int_{0}^{r_{ik}^2} \frac{1}{\sqrt t}e^{-c\frac{r_{ik}^2}{t}}\frac{1}{\sqrt t}\,dt \lesssim 1
%\]
%and
%\begin{eqnarray*}
%  I:= \int_{r_{ik}^2}^\infty \frac{1}{\sqrt t} \lf(\frac{1}{\sqrt{t-r_{ik}^2}} - \frac{1}{\sqrt t} \r)\,dt
%  \lesssim
%  \int_{0}^\infty \frac{1}{\sqrt{t+r_{ik}^2}} \lf(\frac{1}{\sqrt{t}} - \frac{1}{\sqrt{t+r_{ik}^2}} \r)\,dt
%\end{eqnarray*}
%Setting $t = v r_{ik}^2$
%\begin{eqnarray*}
%    I
%    && \lesssim
%    \int_{0}^\infty \lf(\frac{1}{r_{ik}^2\sqrt{v(v+1)}} - \frac{1}{r_{ik}^2( v +1)}\r)  r_{ik}^2 \, dv
%    \lesssim
%    \int_{0}^\infty \lf(\frac{1}{\sqrt{v(v+1)}} - \frac{1}{v +1}\r)  \, dv \\
%    && \lesssim
%    \int_{0}^\infty \frac{1}{\sqrt v (v+1)} \frac{1}{\sqrt{v+1} + \sqrt v}  \, dv \\
%    && \lesssim
%    \int_{0}^1 \frac{1}{\sqrt v} \, dv
%     + \int_1^{\infty} \frac{1}{\sqrt v v} \frac{1}{\sqrt v} \, dv
%    \lesssim 1.
%\end{eqnarray*}
%
%===========End of details================

 %\textbf{(c) Aim now at \boldmath $\cs_{113}$:}
 \textbf{Step 3.} In this step, we estimate the term $\cs_{113}$.

From the Chebyshev inequality and the Minkowski inequality,
% the Beppo Levi convergence theorem and the simple fact that $(\sum_j a_j)^\theta \leq \sum_j a_j^\theta$, $0 \leq \theta <1$, for any sequence of positive reals $\{a_j\}_j$,
we get %/can conclude that
\begin{eqnarray*}
\cs_{113}
&&\le  \frac{1}{\lambda^p}\int_{E_i\setminus F^{(R_i)}_i} \left|\sum_k T_{ik3}b_{ik}\right|^p\,d\mu\\
&&\lesssim  \frac{1}{\lambda^p}\int_{E_i\setminus F^{(R_i)}_i} \left|\int_0^\infty\sum_k  \chi_{(4d(x_{ik},o_i)^2,\infty)}(t) \left|\frac{1}{\sqrt t}-\frac{\chi_{(r_{ik}^2,\infty)}(t)}{\sqrt{t-r_{ik}^2}}\right| |\nabla  e^{-t\L}b_{ik}|\,dt \right|^p\,d\mu\\
&&\lesssim  \frac{1}{\lambda^p}\left[\sum_k\int_{4d(x_{ik},o_i)^2}^\infty \left|\frac{1}{\sqrt t}-\frac{\chi_{(r_{ik}^2,\infty)}(t)}{\sqrt{t-r_{ik}^2}}\right| \left(\int_{E_i\setminus F^{(R_i)}_i}|\nabla  e^{-t\L}b_{ik}|^p\,d\mu\right)^{1/p}\,dt\right]^p.
\end{eqnarray*}
%================DETAILS=====================
%\begin{eqnarray*}
%    && \int_{E_i\setminus F_i} \left|\int_0^\infty\sum_k  \chi_{(4d(x_{ik},o_i)^2,\infty)}(t) \left|\frac{1}{\sqrt t}-\frac{\chi_{\{t>r_{ik}^2\}}}{\sqrt{t-r_{ik}^2}}\right| |\nabla  e^{-t\L}b_{ik}|\,dt \right|^p\,d\mu \\
%    && \quad \leq \left[ \int_0^\infty\sum_k \left(\int_{E_i\setminus F_i} \chi_{(4d(x_{ik},o_i)^2,\infty)}(t) \left|   \frac{1}{\sqrt t}-\frac{\chi_{\{t>r_{ik}^2\}}}{\sqrt{t-r_{ik}^2}}\right|^p \left|\nabla  e^{-t\L}b_{ik} \right|^p \,d\mu \right)^{1/p} \,dt \right]^p \\
%    && \quad \leq \left[ \int_0^\infty\sum_k  \chi_{(4d(x_{ik},o_i)^2,\infty)}(t)  \left| \frac{1}{\sqrt t}-\frac{\chi_{\{t>r_{ik}^2\}}}{\sqrt{t-r_{ik}^2}}\right| \left(\int_{E_i\setminus F_i}  \left|\nabla  e^{-t\L}b_{ik} \right|^p \,d\mu \right)^{1/p} \,dt \right]^p \\
%    && \quad \leq \left[ \sum_k \int_0^\infty   \chi_{(4d(x_{ik},o_i)^2,\infty)}(t) \left|\frac{1}{\sqrt t}-\frac{\chi_{\{t>r_{ik}^2\}}}{\sqrt{t-r_{ik}^2}}\right| \left(\int_{E_i\setminus F_i}  \left|\nabla  e^{-t\L}b_{ik} \right|^p \,d\mu \right)^{1/p} \,dt \right]^p \\
%    &&\quad \leq \left[ \sum_k \int_{4d(x_{ik},o_i)^2}^\infty   \left|\frac{1}{\sqrt t}-\frac{\chi_{\{t>r_{ik}^2\}}}{\sqrt{t-r_{ik}^2}}\right| \left(\int_{E_i\setminus F_i}  \left|\nabla  e^{-t\L}b_{ik} \right|^p \,d\mu \right)^{1/p} \,dt \right]^p
%\end{eqnarray*}
%======================END======================\\
This together with the mapping property of $\nabla e^{-t\L}$ (cf. Proposition \ref{mapping-gradient-heat}) gives that
\[
    \cs_{113}
    \lesssim
    \frac{1}{\lambda^p}\left[\sum_k\|b_{ik}\|_p \int_{4d(x_{ik},o_i)^2}^\infty\left|\frac{1}{\sqrt t}-\frac{\chi_{(r_{ik}^2,\infty)}(t)}{\sqrt{t-r_{ik}^2}}\right|  \frac{\,dt}{\sqrt t}\right]^p,
\]
and in view of \eqref{isolate-balls-1}, the last integral here equals
\[
  \int_{4d(x_{ik}, o_i)^2}^{\infty}
 \frac{r_{ik}^2}{\sqrt{t}+\sqrt{t-r_{i k}^2}} \frac{1}{\sqrt{t-r_{i k}^2} \sqrt{t}} \frac{\,dt}{\sqrt{t}}
 \lesssim \int_{4d(x_{ik}, o_i)^2}^{\infty} \frac{r_{ik}^2}{\sqrt{t}} \frac{1}{d(x_{ik}, o_i)} \frac{\,dt}{t}
 %\lesssim r_{ik}^2 \int_{4d(x_{ik}, o_i)^2}^{\infty} \frac{1}{t^2} \,dt
 \lesssim \frac{r_{i k}^2}{d(x_{i k}, o_i)^2},
\]
%%In conclusion, noticing that $1 <p <p_0$,
%Consequently, by repeating the argument in \eqref{eq:sum-r-d}, we obtain that
%\[
%    \cs_{113}
%    \lesssim
%    \frac{1}{\lambda^p}\left(\sum_k \frac{r_{ik}^2}{d(x_{i k}, o_i)^2}\|b_{ik}\|_p\right)^p
%    \lesssim \frac{\|f\|_p^p}{\lambda^p}.
%\]
%
%=====================\\
%Consequently,  we obtain that
which yields that
\[
    \cs_{113}
    \lesssim
    \frac{1}{\lambda^p}\left(\sum_k \frac{r_{ik}^2}{d(x_{i k}, o_i)^2}\|b_{ik}\|_p\right)^p.
\]
Next, notice that $p' > N_\infty/2$ provided $1 < p <p_0$. From the doubling property and \eqref{isolate-balls-1}, one has
$$\left(\frac{r_{ik}^2}{d(x_{ik},o_i)^2}\right)^{p'}
\le \left(\frac{r_{ik}}{d(x_{ik},o_i)}\right)^{N_i}\
\lesssim \frac{V_i(x_{ik},r_{ik})}{V_{i}(x_{ik},d(x_{ik},o_i))}
\lesssim \frac{\mu(B_{ik})}{V_{i}(o_i,d(x_{ik},o_i))}
\lesssim \frac{\mu(B_{ik})}{V_i(R_i)},$$
which together with \eqref{CZ-2}, \eqref{eq:F_i-set}
and Lemma \ref{lem:volume-growth} (ii) implies that
\[
\sum_k \left(\frac{r_{ik}^2}{d(x_{ik},o_i)^2}\right)^{p'}
\lesssim \frac{\sum_k\mu(B_{ik})}{V_i(R_i)}
\lesssim 1.
\]
By the H\"older inequality, one concludes
%\begin{eqnarray}    \label{eq:sum-r-d}
%\cg_1
%&& \lesssim \frac{1}{\lambda^p}\left(\sum_k \frac{r_{ik}^2}{d(x_{ik},o_i)^2}\|b_{ik}\|_p \right)^p
%\lesssim \frac{1}{\lambda^p} \left[\sum_k \left(\frac{r_{ik}^2}{d(x_{ik},o_i)^2}\right)^{p'}\right]^{\frac{p}{p'}}\left(\sum_{k}\|b_{ik}\|_p^p\right) \notag \\
%&& \lesssim \frac{1}{\lambda^p} \lf( \lambda^p \sum_k \mu(B_{ik}) \r)
% \lesssim \frac{\|f\|_p^p}{\lambda^p},
%\end{eqnarray}
\begin{eqnarray} \label{eq:sum-r-d}
\cs_{113}
&& \lesssim \frac{1}{\lambda^p}\left(\sum_k \frac{r_{ik}^2}{d(x_{ik},o_i)^2}\|b_{ik}\|_p \right)^p \nonumber \\
&& \lesssim \frac{1}{\lambda^p} \left[\sum_k \left(\frac{r_{ik}^2}{d(x_{ik},o_i)^2}\right)^{p'}\right]^{\frac{p}{p'}}\left(\sum_{k}\|b_{ik}\|_p^p\right) \nonumber \\
&& \lesssim \frac{1}{\lambda^p} \lf( \lambda^p \sum_k \mu(B_{ik}) \r)
 \lesssim \frac{\|f\|_p^p}{\lambda^p},
\end{eqnarray}
where we have used %the properties of Calder\'on-Zygmund decomposition \eqref{CZ-1} and \eqref{CZ-2}.
\eqref{CZ-1} in the penultimate inequality and \eqref{CZ-2} in the last inequality.

%\textbf{(d) Finally, we turn to \boldmath $\cs_{112}$:}
\textbf{Step 4.} In this step, we estimate the term $\cs_{112}$.

We split the set $E_i\setminus (F^{(R_i)}_i\cup (\cup_m 2B_{im}))$ into
$$G_{1k}:=\left\{x\in E_i\setminus (F^{(R_i)}_i\cup (\cup_m 2B_{im})): \dist(x,E_0)\le \frac12\dist(B_{ik},E_0)\right\}$$
and
$$G_{2k}:=\left\{x\in E_i\setminus (F^{(R_i)}_i \cup (\cup_m 2B_{im})): \dist(x,E_0)> \frac12\dist(B_{ik},E_0)\right\}.$$
From the triangle inequality and \eqref{isolate-balls-1}, it holds
\begin{equation} \label{eq:g1k}
    \dist(G_{1k},B_{ik}) \ge \dist(B_{ik}, E_0) - \dist(G_{1k}, E_0) \ge \frac 12\dist(B_{ik},E_0)\ge \frac 16d(x_{ik},o_i)
\end{equation}
and
\begin{equation} \label{eq:g2k}
\dist(G_{2k},E_0)\ge \frac 12\dist(B_{ik},E_0)\ge \frac 16d(x_{ik},o_i).
\end{equation}

%Notice that
Using the Chebyshev inequality, one writes
\begin{eqnarray*}
\cs_{112}
&&  = \mu\left(\left\{x\in E_i \setminus (F^{(R_i)}_i\cup (\cup_m 2B_{im})) :\left|\sum_k (\chi_{G_{1k}}(x) + \chi_{G_{2k}}(x)) T_{ik2}b_{ik}(x)\right|>\frac{\sqrt\pi}3\lambda\right\}\right) \\
&& \lesssim  \frac{1}{\lambda^p}\int_{E_i\setminus F^{(R_i)}_i} \left|\sum_k \chi_{G_{1k}}(x)T_{ik2}b_{ik}(x)\right|^p\,d\mu
       + \frac{1}{\lambda}\int_{E_i\setminus F^{(R_i)}_i}  \left|\sum_k \chi_{G_{2k}}(x)T_{ik2}b_{ik}(x)\right|\,d\mu \\
&&   = : \cs_{1121} + \cs_{1122}.
\end{eqnarray*}

    For the term $\cs_{1121}$, we argue as in the estimation of $\cs_{113}$,
     by using the Davies-Gaffney estimate
(cf. Corollary \ref{davies-operators}) instead of Proposition \ref{mapping-gradient-heat}, and conclude that
\begin{eqnarray*}
\cs_{1121}
&&  \lesssim \frac{1}{\lambda^p}\int_{E_i\setminus F^{(R_i)}_i} \left|\int_0^\infty \sum_k \chi_{(R_i^2,4d(x_{ik},o_i)^2)}(t)\left|\frac{1}{\sqrt t}-\frac{\chi_{(r_{ik}^2,\infty)}(t)}{\sqrt{t-r_{ik}^2}}\right| |\chi_{G_{1k}}\nabla  e^{-t\L}b_{ik}| \,dt\right|^p\,d\mu\\
&&\lesssim \frac{1}{\lambda^p}\left[\sum_k\int_{R_i^2}^{4d(x_{ik},o_i)^2}\left|\frac{1}{\sqrt t}-\frac{\chi_{(r_{ik}^2,\infty)}(t)}{\sqrt{t-r_{ik}^2}}\right| \left(\int_{G_{1k}}|\nabla  e^{-t\L}b_{ik}|^p\,d\mu\right)^{1/p}\,dt\right]^p\\
&&\lesssim  \frac{1}{\lambda^p} \left[\sum_k \|b_{ik}\|_p  \int_{R_i^2}^{4d(x_{ik},o_i)^2}  \left|\frac{1}{\sqrt t}-\frac{\chi_{(r_{ik}^2,\infty)}(t)}{\sqrt{t-r_{ik}^2}}\right| e^{-c\frac{\dist(G_{1k},B_{ik})^2}{t}} \frac{\,dt}{\sqrt t}\right]^p.
\end{eqnarray*}
In view of \eqref{eq:g1k}, the last integral can be controlled by
\begin{eqnarray*}
  && \int_{R_i^2}^{4d(x_{ik}, o_i)^2}
    \left|\frac{1}{\sqrt{t}} - \frac{\chi_{(r_{ik}^2,\infty)}(t)}{\sqrt{t-r_{ik}^2}}\right| e^{-c\frac{d(x_{ik}, o_i)^2}{t}} \frac{\,dt}{\sqrt t} \\
 % && \lesssim \int_0^{r_{ik}^2} e^{-c\frac{d(x_{ik}, o_i)^2}{t}} \frac{\,dt}{t}
%        + \int_{r_{ik}^2}^{4d(x_{ik}, o_i)^2} e^{-c\frac{d(x_{ik}, o_i)^2}{t}}
%            \frac{r_{ik}^2}{\sqrt{t}+\sqrt{t-r_{i k}^2}} \frac{1}{\sqrt{t-r_{i k}^2} \sqrt{t}} \frac{\, dt}{\sqrt{t}} \\
%  && \lesssim \int_0^{r_{ik}^2} e^{-c\frac{d(x_{ik}, o_i)^2}{t}} \frac{\,dt}{t}
%        + \int_{r_{ik}^2}^{4d(x_{ik}, o_i)^2} e^{-c\frac{d(x_{ik}, o_i)^2}{t}}
%            \frac{r_{ik}^2}{t^{3/2}}  \frac{\, dt}{\sqrt{t-r_{i k}^2}} \\
  && \ \lesssim \int_0^{r_{ik}^2} \frac{\,dt}{d(x_{ik}, o_i)^2}
     +\int_{r_{ik}^2}^{4d(x_{ik}, o_i)^2} \frac{r_{ik}^2}{d(x_{ik}, o_i)^3} \frac{\, dt}{\sqrt{t-r_{ik}^2}}
%  && \lesssim  \frac{r_{i k}^2}{d(x_{i k}, o_i)^2} + \frac{r_{ik}^2}{d(x_{ik}, o_i)^3} \lf. \sqrt{t - r_{i k}^2}\r|^{4d(x_{ik}, o_i)^2}_{r_{ik}^2}
    \lesssim \frac{r_{i k}^2}{d(x_{i k}, o_i)^2}.
\end{eqnarray*}
Consequently, using \eqref{eq:sum-r-d}, we obtain that
\[
    \cs_{1121}
    \lesssim
    \frac{1}{\lambda^p}\left(\sum_k \frac{r_{ik}^2}{d(x_{i k}, o_i)^2}\|b_{ik}\|_p\right)^p
    \lesssim \frac{\|f\|_p^p}{\lambda^p}.
\]

    For the term $\cs_{1122}$, we shall use the same argument as in the estimation of $\cs_{111}$.
    It follows from  \eqref{isolate-balls-1} that
$$
    B_{ik} \subset
    \left\{z \in E_i:  \dist(z, E_0) \ge d(x_{ik}, o_i)/3 \right\} = E_i \setminus F_i^{(d(x_{ik},o_i)/6)}.
$$
Then for any $0 < t \leq 4d(x_{ik},o_i)^2$ and $y \in B_{ik}$, by \eqref{eq:g2k} and Proposition \ref{est-integral-diagonal-general}, it holds
\begin{equation*}
\int_{G_{2k}} |\nabla h_t(x,y)|\,d\mu(x)
 \le \int_{ E_i \setminus (F_i^{(d(x_{ik},o_i)/12)} \cup B(y, r_{ik}))} |\nabla h_t(x,y)|\,d\mu(x)
 \lesssim \frac{1}{\sqrt t}e^{-c\frac{r_{ik}^2}{t}},
\end{equation*}
which implies that
\begin{eqnarray*}
\cs_{1122}
&& \lesssim \frac{1}{\lambda}\sum_k\int_{R_i^2}^{4d(x_{ik},o_i)^2}\int_{G_{2k}}
\left|\frac{1}{\sqrt t}-\frac{\chi_{(r_{ik}^2,\infty)}(t)}{\sqrt{t-r_{ik}^2}}\right| |\nabla e^{-t\L}b_{ik}|\,d\mu\,dt\\
&& \lesssim \frac{1}{\lambda}\sum_k \|b_{ik}\|_1 \int_{R_i^2}^{4d(x_{ik},o_i)^2} \frac{1}{\sqrt t}e^{-c\frac{r_{ik}^2}{t}}\left|\frac{1}{\sqrt t}-\frac{\chi_{(r_{ik}^2,\infty)}(t)}{\sqrt{t-r_{ik}^2}}\right|\,dt
  \lesssim \frac{\|f\|_p^p}{\lambda^p}.
\end{eqnarray*}
%{\color{blue}which completes the proof. }
%Based on the above arguments, we obtain the desired result.
The above four steps give the desired estimates.
\end{proof}

We now turn to the term $\cs_{12}$.

\begin{lem}  \label{lem:s11}
Let $1<p<2$.
%It holds that for all $f \in L^p(M)$ and all $0<\lambda \leq \|f\|_p$,
For each $f \in L^p(M)$, it holds that
$$
\cs_{12}
=\mu\left(\left\{x \in E_i \setminus F_i^{(R_i)}:
     \left|\nabla \L^{-1/2}\left(\sum_k e^{-r_{ik}^2 \L} b_{ik}\right)(x)\right|>{\lambda}\right\}\right) \lesssim \frac{\|f\|_p^p}{\lambda^p},
     \quad \forall 0<\lambda \leq \|f\|_p.
$$
\end{lem}
\begin{proof}
 By the Chebyshev inequality and the natural $L^2$-boundedness of $\nabla \L^{-1/2}$, we obtain that
\begin{eqnarray*}
    \cs_{12}% \lesssim \frac{1}{\lambda^2} \big\| \sum_k e^{-r_{ik}^2\L} b_{ik}\big\|_2^2
      \lesssim \frac{1}{\lambda^2}\int_M \left|\sum_k e^{-r_{ik}^2\L}b_{ik}\right|^2\,d\mu
      \sim \frac{1}{\lambda^2} \lf\{\int_{E_i}+ \int_{M \setminus E_i} \r\}\cdots \,d\mu
       =: \cs_{121}+ \cs_{122}.
\end{eqnarray*}
%where
%\[
% \cs_{121}: = \frac{1}{\lambda^2} \int_{E_i}\left|\sum_k e^{-r_{ik}^2\L}b_{ik}\right|^2\,d\mu
%\]
%and
%\[
%\cs_{122}: =      \frac{1}{\lambda^2}  \sum_{0 \leq j\neq i \leq \ell} \int_{E_j}\left|\sum_k e^{-r_{ik}^2\L}b_{ik}\right|^2\,d\mu.
%\]
%One writes
%\begin{align*}
%\frac{1}{\lambda^2}\int_M \left|\sum_k e^{-r_{ik}^2\L}b_{ik}\right|^2\,d\mu
%& = \frac{1}{\lambda^2} \lf\{\int_{E_i}+ \int_{M \setminus E_i} \r\} \cdots \,d\mu \\
%&    =: \cs_{121}+ \cs_{122}.

    Let us estimate the term $\cs_{121}$ first. By \eqref{isolate-balls}, one deduces from Lemma \ref{lem:es-away} that
    \[
        h_{r_{ik}^2}(x,y)
        \lesssim  \frac{1}{{V}_{i}(x,r_{ik})}
                   e^{-c\frac{d(x,y)^2}{r_{ik}^2}}, \quad
        \forall x \in E_i, \, y \in B_{ik} \subset E_i.
    \]
    Let $\mathcal{M}_i$ denotes the centered Hardy-Maximal function on $M_i$.
    Then the same argument as Coulhon-Duong \cite[pp. 1158-1160]{cd99} yields  that
       \begin{eqnarray*}
        \lf(\int_{E_i} \left|\sum_k e^{-r_{ik}^2\L}b_{ik}\right|^2\,d\mu\r)^{1/2}
       &&  = \sup_{\|\psi\|_{L^2(M_i)}\le 1} \int_{M_i} \lf(\sum_k e^{-r_{ik}^2\L}b_{ik}\r) \psi \,d\mu \\
       &&  \lesssim \lambda \sup_{\|\psi\|_{L^2(M_i)} \le 1}\int_{M_i}\mathcal{M}_i\psi \lf(\sum_{k}\chi_{B_{ik}}\r) \,d\mu \\
       &&  \lesssim \lambda  \sup_{\|\psi\|_{L^2(M_i)} \le 1} \|\mathcal{M}_i\psi\|_{L^2(M_i)} \big\|\sum_{k}\chi_{B_{ik}}\big\|_{L^2(M_i)} \\
       && \lesssim \lambda\left(\sum_{k}\mu(B_{ik})\right)^{1/2}.
    \end{eqnarray*}
%=================DETAILS OF THE LAST INEQUALTIY================\\
%    Since $\{B_{ik}\}$ has bounded overlap property, set
%    \[
%        \sum_{k} \chi_{B_{ik}} \leq N,
%    \]
%    for some $N >0$.
%    Then, it holds
%    \[
%        \big\|\sum_{k}\chi_{B_{ik}}\big\|^2_{L^2(M_i)}
%        = \int_{M_i} (\sum_{k} \chi_{B_{ik}})^2 d\mu_i
%        = \int_{\cup_k B_{ik}} (\sum_{k} \chi_{B_{ik}})^2 d\mu_i
%        \leq N^2 \mu(\cup_k B_{ik})
%        \leq N^2 \sum_k \mu(B_{ik}).
%    \]
%    Hence, it follow that
%    \[
%        \|\sum_{k}\chi_{B_{ik}}\big\|_{L^2(M_i)}
%         \leq N \left(\sum_{k}\mu(B_{ik})\right)^{1/2}.
%    \]
%================END OF DETAILS===============================\\
    Combining this with \eqref{CZ-2}, one has
    \[
        \cs_{121} \lesssim \frac{\lambda^2}{\lambda^2} \sum_{k}\mu(B_{ik})
        \lesssim \frac{\|f\|_p^p}{\lambda^p}.
    \]

%    We now turn to $\cg_{12}$. We claim that
  For the term $\cs_{122}$, by the fact $B_{ik} \subset E_i \setminus F^{(R_i/4)}_i$ (see \eqref{isolate-balls}) and Proposition \ref{prop:map-anot}, we have
    \[
        \|e^{-r^2_{ik}\L}\|_{L^{2}(M \setminus E_i) \rightarrow L^{\infty}(B_{ik})} \lesssim \frac{1}{\sqrt{V_i(R_i)}}.
    \]
    Hence, for any $1 \leq i \leq \ell$ and $\psi \in L^2(M)$ with $\|\psi\|_2 \le 1$, it holds
    \begin{eqnarray*}
     \lf|\sum_{0 \leq j\neq i \leq \ell} \int_{E_j} \lf( \sum_k e^{-r_{ik}^2\L}b_{ik}\r) \psi \,d\mu\r|
    && = \left|\int_{M \setminus E_i} \lf(\sum_k e^{-r_{ik}^2\L}b_{ik}\r)\psi\,d\mu\right|\\
    &&\lesssim \sum_k \int_{M}  |b_{ik}| \, e^{-r_{ik}^2\L}(|\psi\chi_{M\setminus E_i}|) \,d\mu\\
    && \lesssim    \sum_{k} \int_{B_{ik}} \frac{|b_{ik}|}{\sqrt{V_i(R_i)}}\,d\mu
     \lesssim \lambda \frac{ \sum_k\mu(B_{ik})}{\sqrt{V_i(R_i)}},
    %\lesssim \lambda \sqrt{V_i(R_i)}.
\end{eqnarray*}
    where in the last inequality we used the H\"older inequality and \eqref{CZ-1}.
   % This together with \eqref{CZ-2} and $V_i(R_i) \sim \mu(F^{(R_i)}_i) \sim \|f\|_p^p / \lambda^p$ yields that
   % This, together with the fact (cf. \eqref{CZ-2}, \eqref{eq:F_i-set} and Lemma \ref{lem:volume-growth} (ii))
%    $$\sum_k\mu(B_{ik}) \lesssim \frac{\|f\|_p^p}{ \lambda^p} \sim \mu(F^{(R_i)}_i) \sim V_i(R_i),$$  yields that
%
    This, together with the fact
    $\sum_k\mu(B_{ik}) \lesssim \|f\|_p^p/\lambda^p \sim \mu(F^{(R_i)}_i) \sim V_i(R_i)$ (see \eqref{CZ-2}, \eqref{eq:F_i-set} and Lemma \ref{lem:volume-growth} (ii)), yields that
    \[
        \cs_{122} \lesssim \frac{1}{\lambda^2} \lf(\lambda \frac{ \sum_k\mu(B_{ik})}{\sqrt{V_i(R_i)}}\r)^2
        \lesssim \frac{\|f\|^p_p}{\lambda^p},
    \]
    which finishes the proof.
\end{proof}

%Let us begin with the following lemma.
%The estimation of $\cs_2$ is relatively simple and we can easily obtain the following result.
%Let us begin with the estimation of $\cs_1$.
%Let us estimate $\cs_1$ first.
%For the term $\cs_2$, we can easily obtain the following result.
Our task now is to estimate the term $\cs_2$.
\begin{lem} Let $1<p<2$.
%We have uniformly for all $f \in L^p(M)$ and all $0<\lambda \leq \|f\|_p$ that
For each $f \in L^p(M)$, it holds that
$$
\cs_2
=\mu\left(\left\{x \in E_i \setminus F_i^{(R_i)}:
     \left|\int_0^1 \nabla e^{-t \L}\left(\sum_k b_{i k}\right)(x) \frac{\, dt}{\sqrt{t}}\right|>\sqrt\pi \lambda\right\}\right) \lesssim \frac{\|f\|_p^p}{\lambda^p},
     \quad \forall 0<\lambda \leq \|f\|_p.
$$
\end{lem}
\begin{proof}
Recall that the operator $\int_0^1\nabla e^{-t\L}\frac{\,dt}{\sqrt t}$
is bounded on $L^p(M)$ (cf. Proposition \ref{local-part}).
Then, from the Chebyshev inequality and the fact that the
supports of $\{b_{ik}\}_{k}$ are of bounded overlap, it holds
\[
 \cs_2
 \lesssim \frac{1}{\lambda^p}\big\|\sum_kb_{ik}\big\|_{p}^p
 \lesssim \frac{1}{\lambda^p} \sum_k \|b_{ik}\|_{p}^p
 \lesssim \frac{\|f\|_p^p}{\lambda^p},
\]
%========THE DETAILS OF THE FIRST INEQUALITY=============
%\begin{equation*}%\label{cz-bp}
%    \int_{M_i} \left|\sum_k b_{ik}\right|^p d\mu
%    \lesssim \int_{M_i} \sum_k |b_{ik}|^p d\mu
%    \lesssim \sum_k \int_{B_{ik}}  |b_{ik}|^p d\mu.
%\end{equation*}
%where in second inequality we used the Beppo Levi Theorem or Zhou minqiang [ShiBianHanShuLun Theorem 4.6, Corrally 4.16, Theorem 4.27 (Tonelli Theorem)]
%\\==========END OF DETAILS===========================\\
where we have used \eqref{CZ-1} and \eqref{CZ-2} in the last inequality.
This completes the proof.
\end{proof}

\subsubsection{Estimate of the bad parts off the diagonal}\label{sec:off-diag}\hskip\parindent
For each $1\leq i\leq \ell$,
%we split the `bad' part off the diagonal into two parts:
one writes
\begin{eqnarray*}%\label{eq:off-dia-split}
&& \mu\left(\left\{x\in \bigcup_{1 \leq j\neq i \leq \ell}  (E_j\setminus F^{(R_j)}_j):\left|T\left(\sum_k b_{ik}\right)(x)\right|>2\lambda\right\}\right) \notag \\
&& \ \le
     \mu\left(\left\{x\in \bigcup_{1 \leq j\neq i\leq \ell}  (E_j\setminus F^{(R_j)}_j)
              :\left|T\left(\sum_k (1-e^{-r_{ik}^2\L}) b_{ik}\right)(x)\right|>\lambda\right\}\right) \nonumber\\
&& \ \ +
    \mu\left(\left\{x\in \bigcup_{1 \leq j\neq i \leq \ell}  (E_j\setminus F^{(R_j)}_j):
                       \left|T\left(\sum_k e^{-r_{ik}^2\L}b_{ik}\right)(x)\right|>\lambda\right\}\right)
   =: \mathcal{G}_1 + \mathcal{G}_2.
\end{eqnarray*}

%\textbf{Estimation of the term $\cg_1$:} We have

For the term $\cg_1$, %recall that $1 < p_0 \leq 2$ is given in \eqref{eq:p0-def}. We have
we have
\begin{lem}
Let $1<p<p_0$, where $p_0$ is as in \eqref{eq:p0-def}.
%It holds uniformly for all $f \in L^p(M)$ and $0<\lambda \leq \|f\|_p$ that
For each $f \in L^p(M)$, it holds that
$$
\cg_1=\mu\left(\left\{x \in \bigcup_{1 \leq j \neq i \leq \ell} (E_j \setminus F_j^{(R_j)})
                :\left|T\left(\sum_k(1-e^{-r_{i k}^2 \mathcal{L}}) b_{ik}\right)(x)\right|>{\lambda}\right\}\right)
     \lesssim \frac{\|f\|_p^p}{\lambda^p},
     \quad \forall 0<\lambda \leq \|f\|_p.
$$
\end{lem}
\begin{proof}
  Using the Chebyshev inequality, we have
    \begin{eqnarray*}
    \cg_1
    && \lesssim \frac{1}{\lambda^p} \int_{\bigcup_{1\leq j\neq i \leq \ell} (E_j\setminus F^{(R_j)}_j) } \left|T\left(\sum_k (1-e^{-r_{ik}^2\L}) b_{ik}\right)\right|^p\,d\mu\\
    && \lesssim \frac{1}{\lambda^p} \int_{\bigcup_{1 \leq j\neq i \leq \ell} (E_j\setminus F^{(R_j)}_j)} \left|\sum_k \left(\int_1^\infty\int_0^{r_{ik}^2}|\nabla\L e^{-(s+t)\L} b_{ik}|\frac{\,ds\,dt}{\sqrt t}\right)\right|^p\,d\mu,
    \end{eqnarray*}
  which together with the Minkowski inequality and Corollary \ref{cor:com-DG} gives that
\begin{eqnarray*}
\cg_1
&& \lesssim \frac{1}{\lambda^p}\left[ \sum_k \int_1^\infty\int_0^{r_{ik}^2}
       \left(\int_{\bigcup_{1 \leq j\neq i \leq \ell} (E_j\setminus F^{(R_j)}_j) } |\nabla \L e^{-(t+s)\L}b_{ik}|^p\,d\mu\right)^{1/p}\frac{\,ds\,dt}{\sqrt t}\right]^p \\
&& \lesssim \frac{1}{\lambda^p}\left[\sum_k\int_1^\infty\int_0^{r_{ik}^2}\|b_{ik}\|_pe^{-c\frac{\dist(B_{ik},E_0)^2}{t+s}} \frac{\,ds\,dt}{(t+s)^{3/2}\sqrt t}\right]^p.
\end{eqnarray*}
From \eqref{isolate-balls-1}, the double integral can be controlled by
\begin{eqnarray*}
&&\int_1^\infty\int_0^{r_{ik}^2} e^{-c\frac{d(x_{ik},o_i)^2}{t+s}} \frac{\,ds\,dt}{(t+s)^{3/2}\sqrt t}\\
&&\ \le \int_0^{d(x_{ik},o_i)^2}\int_0^{r_{ik}^2}e^{-c\frac{d(x_{ik},o_i)^2}{t+s}} \frac{\,ds\,dt}{(t+s)^{3/2}\sqrt t}
         +\int_{d(x_{ik},o_i)^2}^\infty\int_0^{r_{ik}^2} \frac{\,ds\,dt}{(t+s)^{3/2}\sqrt t}\\
%&&\lesssim \int_1^{d(x_{ik},o_i)^2}\int_0^{r_{ik}^2} \frac{1}{d(x_{ik}, o_i)^3} \frac{1}{\sqrt t}   \,ds\,dt
%        +\int_{d(x_{ik},o_i)^2}^\infty\int_0^{r_{ik}^2} \frac{1}{t^2} \,ds\,dt \\
&&\ \lesssim \frac{1}{d(x_{ik}, o_i)^3}  \int_0^{d(x_{ik},o_i)^2}  \frac{1}{\sqrt t} \, dt \int_0^{r_{ik}^2}  \,ds
        + \int_{d(x_{ik},o_i)^2}^\infty \frac{1}{t^2}  \,dt \int_0^{r_{ik}^2}  \,ds
  \lesssim \frac{r_{ik}^2}{d(x_{ik},o_i)^2}.
\end{eqnarray*}
Consequently, using \eqref{eq:sum-r-d}, we obtain that
\[
\cg_1 \lesssim  \frac{1}{\lambda^p}
 \left(\sum_k \frac{r_{ik}^2}{d(x_{ik},o_i)^2}\|b_{ik}\|_p \right)^p
 \lesssim \frac{\|f\|_p^p}{\lambda^p}.
\]
This finishes the proof.
\end{proof}

%\textbf{Estimation of the term $\cg_2$:} Recall that $p_0 \in (1,2]$ is given in \eqref{eq:p0-def}. We have

We now turn to the term $\cg_2$.

\begin{lem} \label{prop:es-cg1}
  Let $1 < p <2$. %It holds uniformly for all $f \in L^p(M)$ and
%  $0 < \lambda \leq \|f\|_p$ that
For each $f \in L^p(M)$, it holds that
  \[
        \cg_2=\mu\left(\left\{x\in \bigcup_{1 \leq j\neq i \leq \ell}  (E_j\setminus F^{(R_j)}_j):
                       \left|T\left(\sum_k e^{-r_{ik}^2\L}b_{ik}\right)(x)\right|>\lambda\right\}\right)
        \lesssim \frac{\|f\|^p_p}{\lambda^p},  %, \quad
        \quad \forall 0< \lambda \leq \|f\|_p. %, \, f \in L^p.
  \]
\end{lem}
\begin{proof}
  The $L^2$-boundedness of $T$ implies that
%\begin{eqnarray*}%\label{est-b-off}
%\cg_2
%   && \lesssim \frac{1}{\lambda^2}\int_M \left|\sum_k e^{-r_{ik}^2\L}b_{ik}\right|^2\,d\mu\nonumber\\
%&& \sim \frac{1}{\lambda^2} \int_{E_i}\left|\sum_k e^{-r_{ik}^2\L}b_{ik}\right|^2\,d\mu +
%     \frac{1}{\lambda^2}  \sum_{0 \leq j\neq i \leq \ell} \int_{E_j}\left|\sum_k e^{-r_{ik}^2\L}b_{ik}\right|^2\,d\mu
%    =: \cg_{21} + \cg_{22}.
%\end{eqnarray*}
\[
\cg_2
 \lesssim \frac{1}{\lambda^2}\int_M \left|\sum_k e^{-r_{ik}^2\L}b_{ik}\right|^2\,d\mu
 \lesssim \frac{\|f\|_p^p}{\lambda^p},
\]
where the last inequality is obtained in the proof of Lemma \ref{lem:s11}.
This finishes the proof.
\end{proof}

\subsection{Completion of the proof}\label{sec:com-prof}\hskip\parindent
%A combination of the previous steps shows that for
%$$1<p<\min\left\{p_0,\frac{N_\infty}{N_\infty-2}\right\},$$
%the operator
%$$T=\int_1^\infty \nabla e^{-t\L}\frac{\,dt}{\sqrt t}$$
%is weakly $L^p$-bounded. This together with Lemma \ref{local-part} shows that
%$\nabla\L^{-1/2}$ weakly $L^p$-bounded. An application of the  Marcinkiewicz
%interpolation theorem completes the proof of Theorem \ref{main-result-parabolic}.
%
%Here is the details.
Combining the estimates from the previous three subsections, i.e., the estimates on the center (Proposition \ref{es:center-part}), the estimates near the center (Proposition \ref{es:around-center}), and the estimates
on the part away from the center (Proposition \ref{es:away-center}),
we finally conclude that %for any $f \in L^p(M)$ and $0<\lambda \leq \|f\|_p$, it holds
for each $f \in L^p(M)$ with $1 < p < p_0$, where $p_0$ is as in \eqref{eq:p0-def},
\begin{eqnarray*}
&&\mu\lf(\lf\{x\in \bigcup_{j=1}^\ell \, (E_j\setminus F^{(R_j)}_j):|Tf(x)|>(2\ell+1)\lambda\r\}\r)\\
&&\ \le \mu\lf(\lf\{x\in \bigcup_{j=1}^\ell \, (E_j\setminus F^{(R_j)}_j):|T(f\chi_{E_0})(x)|>\lambda\r\}\r) \\
&&\ \ +\sum_{i=1}^\ell \mu\lf(\lf\{x\in \bigcup_{j=1}^\ell \, (E_j\setminus F^{(R_j)}_j):|T(f\chi_{F^{(R_i)}_i})(x)|>\lambda\r\}\r)\\
&& \ \ +\sum_{i=1}^\ell
     \mu\lf(\lf\{x\in \bigcup_{j=1}^\ell \, (E_j\setminus F^{(R_j)}_j):|T(f\chi_{E_i\setminus F^{(R_i)}_i})(x)|>\lambda\r\}\r)
 \lesssim \frac{\|f\|_p^p}{\lambda^p},
\quad \forall 0 < \lambda \leq  \|f\|_p,
\end{eqnarray*}
%\begin{eqnarray*}
%&&\mu\lf(\lf\{x\in \bigcup_{j=1}^\ell \, (E_j\setminus F^{(R_j)}_j):|Tf(x)|>(2\ell+1)\lambda\r\}\r)\\
%&&\ \le \mu\lf(\lf\{x\in \bigcup_{j=1}^\ell \, (E_j\setminus F^{(R_j)}_j):|T(f\chi_{E_0})(x)|>\lambda\r\}\r)
%   +\sum_{i=1}^\ell \mu\lf(\lf\{x\in \bigcup_{j=1}^\ell \, (E_j\setminus F^{(R_j)}_j):|T(f\chi_{F^{(R_i)}_i})(x)|>\lambda\r\}\r)\\
%&& \ \ +\sum_{i=1}^\ell
%     \mu\lf(\lf\{x\in \bigcup_{j=1}^\ell \, (E_j\setminus F^{(R_j)}_j):|T(f\chi_{E_i\setminus F^{(R_i)}_i})(x)|>\lambda\r\}\r) \\
%&& \lesssim \frac{\|f\|_p^p}{\lambda^p},
%\quad \forall 0 < \lambda \leq  \|f\|_p,
%\end{eqnarray*}
%\begin{eqnarray*}
%\mu\lf(\lf\{x\in \bigcup_{j=1}^\ell \, (E_j\setminus F^{(R_j)}_j):|Tf(x)|>(2\ell+1)\lambda\r\}\r)
%&& \le \mu\lf(\lf\{x\in \bigcup_{j=1}^\ell \, (E_j\setminus F^{(R_j)}_j):|T(f\chi_{E_0})(x)|>\lambda\r\}\r) \\
%&& \ +\sum_{i=1}^\ell \mu\lf(\lf\{x\in \bigcup_{j=1}^\ell \, (E_j\setminus F^{(R_j)}_j):|T(f\chi_{F^{(R_i)}_i})(x)|>\lambda\r\}\r)\\
%&&  \ +\sum_{i=1}^\ell
%     \mu\lf(\lf\{x\in \bigcup_{j=1}^\ell \, (E_j\setminus F^{(R_j)}_j):|T(f\chi_{E_i\setminus F^{(R_i)}_i})(x)|>\lambda\r\}\r) \\
%&& \lesssim \frac{\|f\|_p^p}{\lambda^p},
%\quad \forall 0 < \lambda \leq  \|f\|_p,
%\end{eqnarray*}
%
%for each $1< p <p_0$, where $p_0$ is as in \eqref{eq:p0-def}.
which together with \eqref{est-small-lambda-3} yields %for any $\lambda \leq \|f\|_p$,
\[
\mu(\{x\in M:|Tf(x)|>(2\ell+1)\lambda\})\lesssim \frac{\|f\|_p^p}{\lambda^p},
\quad \forall 0 < \lambda \leq \|f\|_p.
\]
By this and Lemma \ref{lem:big-lam}, we see that the operator $T$ is weakly $(p,p)$ bounded for each $1 < p <p_0$.
%
%Recall that
%$$T=\int_1^\infty\nabla e^{-t\L}\frac{\,dt}{\sqrt t}.$$
%On the other hand, by Lemma \ref{local-part},
%$$\int_0^1\nabla e^{-t\L}\frac{\,dt}{\sqrt t}$$
%is bounded on $L^p(M)$ for all $1<p<2$ since $\nabla(1+\L)^{-1/2}$ is $L^p$-bounded by \cite{cd99}.
%We finally conclude that $\nabla \L^{-1/2}$ is weakly $L^p$ bounded for each $p$ satisfying
%\begin{equation*}
%1<p<\min\left\{p_0,\frac{N_\infty}{N_\infty-2}\right\},
%\end{equation*}
%Based on the discussion in subsection \ref{sec:main-proof-red}, we can conclude that $\nabla \L^{-1/2}$
%is bounded on $L^p(M)$ for all $1<p<2$. % via the Marcinkiewicz interpolation theorem.
%\hl{as we have said} in Subsection \ref{sec:main-proof-red}, then xxxxxx

%Then, as stated %/the argument
%in the beginning of Subsection \ref{sec:main-proof-red}, recall that
%$$
%T=\frac{1}{\sqrt{\pi}} \int_1^{\infty} \nabla e^{-t \L} \frac{\,dt}{\sqrt{t}}
%$$
%and on the other hand, by Lemma \ref{local-part},
%$$
%\frac{1}{\sqrt{\pi}} \int_0^1 \nabla e^{-t \L} \frac{\,dt}{\sqrt{t}}
%$$
%is bounded on $L^p(M)$ for all $1<p<2$.
Moreover, note that
$$
    \nabla \L^{-1/2} - T
    =  \nabla \L^{-1/2} - \frac{1}{\sqrt{\pi}} \int_1^{\infty} \nabla e^{-t \L} \frac{\,dt}{\sqrt{t}}
    = \frac{1}{\sqrt{\pi}} \int_0^1 \nabla e^{-t \L} \frac{\,dt}{\sqrt{t}}
$$
is bounded $L^p(M)$ for all $1<p<2$; see Proposition \ref{local-part}.
We finally conclude that $\nabla \L^{-1/2}$ is weakly $(p,p)$ bounded for each $1<p<p_0$,
and hence the Riesz transform $\nabla \L^{-1/2}$ is bounded on $L^q(M)$ for all $1< q <2$ by the Marcinkiewicz interpolation theorem.
This completes the proof of Theorem \ref{main-result-parabolic}.

\end{document}